\documentclass{amsart}
\usepackage{amsrefs}
\pdfoutput=1
\usepackage{tikz}
\usetikzlibrary{decorations.pathreplacing}
\usepackage{mathrsfs}
\usepackage{stmaryrd}
\usepackage{enumerate}

\usepackage[hyperindex=true, breaklinks=true, frenchlinks=false, hyperfigures=true, colorlinks=true, citecolor=blue]{hyperref}
\usepackage[nameinlink]{cleveref}

\theoremstyle{plain}
\newtheorem{theorem}{Theorem}[section]
\newtheorem{thm}[theorem]{Theorem} 
\crefname{thm}{Thm.}{Thms.} \Crefname{thm}{Theorem}{Theorems}
\newtheorem{corollary}[theorem]{Corollary}
\Crefname{corollary}{Cor.}{Cors.}

\newtheorem{prop}[theorem]{Proposition}

\newtheorem{lem}[theorem]{Lemma}
\newtheorem{conj}[theorem]{Conjecture}
\newtheorem{ques}[theorem]{Question}

\theoremstyle{remark}
\newtheorem{rem}[theorem]{Remark}

\theoremstyle{definition}
\newtheorem{defn}[theorem]{Definition}
\newtheorem{exmp}[theorem]{Example}

\crefname{prop}{Prop.}{Props.} \Crefname{prop}{Proposition}{Propositions}
\crefname{cor}{Cor.}{Cors.} \Crefname{cor}{Corollary}{Corollaries}
\crefname{rem}{Rem.}{Rems.} \Crefname{rem}{Remark}{Remarks}
\crefname{defn}{Def.}{Defs.} \Crefname{defn}{Definition}{Definitions}
\crefname{exmp}{Exm.}{Exms.} \Crefname{exmp}{Example}{Examples}

\crefname{ques}{Ques.}{Ques.} \Crefname{ques}{Question}{Questions}

\newcommand{\N}{\mathbb{N}}

\DeclareMathOperator{\adj}{adj}
\newcommand{\ds}[1]{\displaystyle{#1}}
\newcommand{\abs}[1]{\left|{#1}\right|}
\renewcommand{\scr}[1]{\mathscr{#1}}
\newcommand{\vs}[1]{\left[#1\right]}
\newcommand{\esym}[2]{\left\langle#1;#2\right\rangle}

\newcommand{\set}[1]{\left\{#1\right\}}
\newcommand{\floor}[1]{\left\lfloor#1\right\rfloor}

\newcommand{\nth}{n^\text{th}}
\newcommand{\lolam}{\underline{\lambda}}
\newcommand{\hilam}{\overline{\lambda}}
\newcommand{\induced}[1]{\left\langle #1\right\rangle}
\newcommand{\gam}{\gamma}		
		
\newcommand{\sig}{\sigma}		

\newcommand{\bigmonouc}{\scr{G}_{4,4}\cup\scr{G}_{3+,5}\cup\scr{G}_{3,6}}

\newcommand{\ft}{\mathbf{f}}
\newcommand{\wft}[1]{\mathbf{w}_{#1}}
\newcommand{\bft}[1]{\mathbf{b}_{#1}}
\newcommand{\nft}[1]{\mathbf{n}_{#1}}

\title{Growth of Face-Homogeneous Tessellations}
\author{Stephen J.~Graves}
\address{Department of Mathematics\\
University of Texas at Tyler\\
Tyler, TX 75799}
\thanks{Much of the material in this work comes from the doctoral dissertation \cite{Graves09} of the first author written under the supervision of the second author.}
\email{sgraves@uttyler.edu}
\author{Mark E.~Watkins}
\address{Department of Mathematics\\
Syracuse University\\
Syracuse, NY 13244-1150}
\thanks{The second author was partially supported by a grant from the Simons Foundation (\#209803 to Mark E. Watkins).}
\email{mewatkin@syr.edu}

\date{29 June 2017}


\begin{document}
\maketitle

\begin{abstract}\sloppypar
A tessellation of the plane is \emph{face-homogeneous} if for some integer $k\geq3$ there exists a cyclic sequence $\sigma=[p_0,p_1,\ldots,p_{k-1}]$ of integers $\geq3$ such that, for every face $f$ of the tessellation, the valences of the vertices incident with $f$ are given by the terms of $\sigma$ in either clockwise or counter-clockwise order.  When a given cyclic sequence $\sigma$ is realizable in this way, it may determine a unique tessellation (up to isomorphism), in which case $\sigma$ is called \emph{monomorphic}, or it may be the valence sequence of two or more non-isomorphic tessellations (\emph{polymorphic}).    A tessellation whose faces are uniformly bounded in the hyperbolic plane but not uniformly bounded in the Euclidean plane is called a {\em hyperbolic tessellation}.   Such tessellations are well-known to have exponential growth.  We seek the face-homogeneous hyperbolic tessellation(s) of slowest growth and show that the least growth rate of such monomorphic tessellations is the ``golden mean,'' $\gamma=(1+\sqrt{5})/2$, attained by the sequences $[4,6,14]$ and $[3,4,7,4]$.  A polymorphic sequence may yield non-isomorphic tessellations with different growth rates.  However, all such tessellations found thus far grow at rates greater than $\gamma$.
\end{abstract}

\setcounter{section}{-1}
\section{Introduction\label{chap:Intro}}
\nocite{Graves09}
It has long been known that there are finitely many homogeneous tessellations of the Euclidean plane; they all have quadratic growth.  However, in the hyperbolic plane, for various definitions of ``homogeneity,'' infinitely many homogeneous tessellations are realizable, and their growth, if not quadratic, is always exponential.  Presently we will give a rigorous definition of growth rate, but for now one should think of this parameter intuitively as the asymptotic rate at which additional tiles (or faces) accrue with respect to some chosen center of a tessellation.  In this schema, all  Euclidean tessellations have growth rate equal to 1, and hyperbolic tessellations have growth rate strictly greater than 1.  The first author has shown by construction in \cite{Graves11} that, given any $\epsilon>0$, there exists a tessellation of the hyperbolic plane with growth rate exactly $1+\epsilon$.  In general, these latter tessellations have few if any combinatorial or geometric symmetries.  The question then becomes one of determining the growth rates of hyperbolic tessellations when some sort of homogeneity is imposed.  In particular, subject to a homogeneity constraint, how small can the gap be between quadratic and exponential growth?

\sloppypar{In a seminal work \cite{GrunShep87}, Gr\"{u}nbaum and Shephard defined a graph to be \emph{edge-homogeneous with edge-symbol $\langle p,q;k,\ell\rangle$} if every edge is incident with vertices of valence $p$ and $q$ and faces of covalence $k$ and $\ell$.  They proved that the parameters of an edge-symbol uniquely  determine an edge-homogeneous tessellation up to isomorphism.  }

The notion of homogeneity was extended by Moran \cite{Moran97}.  She defined a tessellation to be \emph{face-homogeneous with valence sequence $[p_0,\ldots,p_{k-1}]$} if every face is a $k$-gon incident with vertices of valences $p_0,\ldots,p_{k-1}$ in either clockwise or counter-clockwise consecutive order. Unfortunately, no uniqueness property analogous to the Gr\"{u}nbaum-Shephard result holds in general for face-homogeneous tessellations. 

Moran's work on growth rates of face-homogeneous tessellations led the authors (together with T.\,Pisanski) to return to edge-homogeneous tessellations and conclusively determine their growth rates.  In \cite{GPW} they determined the growth rate of any edge-homogeneous tessellation as a function of its edge-symbol and proved that the least growth rate for an exponentially-growing, edge-homogeneous tessellation is $\frac12(3+\sqrt{5})\approx2.618$.

\sloppypar{The goal of this article is to obtain an analogous result for face-homogeneous tessellations.  Our main result is that if a face-homogeneous tessellation with exponential growth is determined up to isomorphism by its valence sequence, then its growth rate is at least $\frac12(1+\sqrt5)$, namely the ``golden mean.'' Moreover,  we determine exactly the valence sequences for which this golden mean is realized.  A significant by-product of our investigation is an abundance of machinery for computing the growth rates of many classes of face-homogeneous planar tessellations.}

 \Cref{chap:prelims} consists of six sections.  Following some general definitions concerning infinite graphs in the plane, we present  (\Cref{sec:Bilinski}) a system for labeling sets of vertices and sets of faces of a tessellation; such a labeling is called a ``Bilinski diagram.''  \Cref{sec:FaceHom} presents the notion of face-homogeneity and associated notation.  Polynomial and exponential growth, defined on the one hand with respect to the standard graph-theoretic metric, and on the other hand with respect to the notion of angle excess, appear in \Cref{sec:poly/exp}.   \Cref{sec:growth-form} presents a rigorous theoretical treatment of growth rate with respect to regional distance in a Bilinski diagram.  \Cref{sec:edge-hom} concludes the Preliminaries  with a review of the completely resolved case of edge-homogeneous tessellations, summarizing results from \cite{GrunShep87} and \cite{GPW}. 
 
In \Cref{sec:Accretion} we lay out our method for filling in the formulas obtained in \Cref{sec:growth-form} while introducing the notion of a transition matrix.  Analogous to a Markov process, this matrix encodes for given $n\geq 1$ how many faces of each possible configuration are ``begotten'' at regional distance $n+1$ from the root of a Bilinski diagram by a face at regional distance $n$ from the root.  The maximum modulus of the eigenvalues of the transition matrix are key to the growth rate of $T$.

\Cref{sec:MonoUC} applies the machinery of \Cref{sec:Accretion} to the significant class of valence sequences that are \emph{monomorphic}, i.e., that are uniquely realizable as a face-homogeneous tessellation and whose Bilinski diagrams are in a certain sense well-behaved, called \emph{uniformly concentric}.  It  is shown in \Cref{thm:GrowthOrder} that for such valence sequences, the partial order defined in \Cref{sec:FaceHom} is preserved by their growth rates. The six classes of monomorphic sequences of lengths 3, 4, and 5 whose Bilinski diagrams are not uniformly concentric are identified in \Cref{sec:MonoNCVS}, where it is proved that they are indeed monomorphic.  The exhaustive proof that this list is complete is contained in the Appendix \cite{appx}. Finally, we present in \Cref{sec:MainResult} the main result of the paper, that the least growth rate of a face-homogeneous tessellation with monomorphic valence sequence is the golden mean $\frac12(1+\sqrt{5})$.

Those valence sequences (described as  \emph{polymorphic}) which admit multiple non-isomorphic tessellations are alive and well in \Cref{sec:PolymorphicVS}.  A general sufficient condition for polymorphism is given.  The difficulties posed by polymorphism are illustrated by an example; the polymorphic sequence $[4,4,6,8]$ is considered in some depth in \Cref{appx_a}.  In particular, we see by this example that two different tessellations having the same (polymorphic) valence sequence may well have different growth rates. We conclude the chapter with some conjectures in \Cref{sec:Conjectures}.

The appendix \cite{appx} alluded to above is to be found with this article on the arXiv. All references therein are to results in the present paper.
\medskip


\section{Preliminaries\label{chap:prelims}}

\subsection{Tessellations}
For a graph $\Gamma$, the symbol $V(\Gamma)$ denotes the vertex set of $\Gamma$. If $M$ is a planar embedding of $\Gamma$, we call $M$ a \emph{plane map} and denote by $F(M)$ the set of faces of $M$. 

A graph $\Gamma$ is \emph{infinite} if its vertex set $V(\Gamma)$ is infinite; $\Gamma$ is \emph{locally finite} if every vertex has finite valence.  A graph is \emph{$3$-connected} if there is no set of fewer than three vertices whose removal disconnects the graph.  It is well known that if the underlying graph $\Gamma$ of a plane map $M$ is 3-connected (as is generally the case in this work), then every automorphism of $\Gamma$ induces a permutation of $F(M)$ that preserves face-vertex incidence and can be extended to a homeomorphism of the plane.  Thus we  tend to abuse language and speak of ``the faces of $\Gamma$.''  When a plane map is 3-connected, every edge is incident with exactly two distinct faces.  In this case, the number of edges (and hence of vertices) incident with a given face is its \emph{covalence}.  A map is \emph{locally cofinite} if the covalence of every face is finite.

An \emph{accumulation point} of an infinite plane map $M$ is a point $x$ in the plane such that every open disk of positive radius (in either the Euclidean or hyperbolic metric) containing $x$ intersects infinitely many map objects, be they faces, edges, or vertices.  A map is 1-\emph{ended} when the deletion of any finite submap leaves exactly one infinite component. 

\begin{defn} A \emph{tessellation}\index{tessellation} is an infinite plane map 
that is $3$-connected, locally finite, locally cofinite, 1-ended, and also admits no accumulation point.
\end{defn}

In the terminology of Gr\"unbaum and Shepherd's exhaustive work \cite{TilesPats} on tilings of the plane, a tessellation $T$ is \emph{normal}\index{normal} if there is an embedding of $T$ in the plane and radii $0<r<R$ under a specific metric such that the boundary of each face lies within some annulus with inner radius $r$ and outer radius $R$.  A \emph{Euclidean tessellation} is tessellation that is normal with respect to the Euclidean metric, and a \emph{hyperbolic tessellation} is one that is normal with respect to the hyperbolic metric but not with respect the Euclidean metric.
\medskip

\subsection{Bilinski diagrams}\label{sec:Bilinski}

A very useful tool for computing ``growth rate" is what we have
called a  \emph{Bilinski diagram}, because these diagrams were first used by Stanko Bilinski in his dissertation \cite{Bil48, Bil49}.

\begin{defn}\label{Bilinski}
Let $M$ be a map that is rooted at some vertex 
$x$.   Define a sequence of sets $\{U_n :  n \ge 0\}$  of vertices and a sequence of
sets  $\{F_n: n \ge 0\}$ of faces of $M$ inductively as follows.\\
\noindent$\bullet$   Let $U_0=\{x\}$ and let $F_0=\emptyset$.\\  
\noindent$\bullet$   For $n\geq1$, let $F_n$ denote the set of faces of $M$ not in $F_{n-1}$ that are incident with some vertex in  $U_{n-1}$.\\
\noindent$\bullet$  For $n\geq1$, let $U_n$ denote the set of vertices of $M$ not in $U_{n-1}$ that are incident with some face in  $F_n$.\\
The stratification of $M$ determined by the set-sequences $\{U_n\}$ and $\{F_n\}$
is called \emph{the  Bilinski diagram of $M$ rooted at $x$.} 
In a similar way one can define a Bilinski diagram of 
$M$ rooted at a face $f$. In this case $U_0 = \emptyset$ and $F_0 = \{f\}$.  Given a Bilinski diagram of $T$, the induced submap $\langle F_n\rangle$ of $T$ is its $n^\text{th}$ \emph{corona}.

A Bilinski diagram is \emph{concentric} if each subgraph $\langle U_n\rangle$ induced by $U_n$ ($n\geq1$) is a 
cycle; otherwise the Bilinski diagram is {\em non-concentric}.  If a plane map yields a concentric Bilinski diagram regardless of which vertex or face is designated as its root, then the map is \emph{uniformly concentric}; analogously a map which for every designated root yields a non-concentric Bilinski diagram is \emph{uniformly non-concentric}. 
\end{defn}

To answer the question as to which tessellations are uniformly concentric we state a sufficient condition and a necessary condition.  Let $\mathscr{G}_{a,b}$ denote the class of tessellations all of whose vertices have valence at least $a$ and all of whose faces have covalence at least $b$. Let $\mathscr{G}_{a+,b}$ be the subclass of $\mathscr{G}_{a,b}$ of tessellations with no adjacent $a$-valent vertices.

\begin{prop}[\cite{BruceWat04} Corollary 4.2; \cite{NieWat97} Theorem 3.2]\label{prop:UnifConc} Every tessellation $T\in\mathscr{G}_{3,6}\cup\mathscr{G}_{3+,5}\cup\mathscr{G}_{4,4}$ is uniformly concentric, and in every Bilinski diagram of $T$, for all $n\geq1$, every face in $F_n$ is incident with at most two edges in $\langle U_{n-1}\rangle$.  
\end{prop}

\begin{prop}[\cite{BruceWat04} Theorem 5.1]\label{prop:UnifConcForbidden} If an infinite planar map admits any of the following configurations, then the map is not uniformly concentric:
\begin{enumerate}
\item a $3$-valent vertex incident with a $3$-covalent face;
\item a $4$-valent vertex incident with two nonadjacent $3$-covalent faces;
\item a $4$-covalent face incident with two nonadjacent $3$-valent vertices;
\item an edge incident with two $3$-valent vertices and two $4$-covalent faces;
\item an edge incident with two $4$-valent vertices and two $3$-covalent faces.
\end{enumerate}
\end{prop}
\medskip

\subsection{Face-homogeneity and realizability}\label{sec:FaceHom}

Let $k\geq3$ be an integer and let an equivalence relation be defined on the set of ordered $k$-tuples $(p_0,p_1,\ldots,p_{k-1})$ of positive integers whereby 
\begin{itemize}
\item $(p_0,p_1,\ldots,p_{k-1})\equiv (p_1,p_2,\ldots,p_{k-1},p_0)$, and
\item $(p_0,p_1,\ldots,p_{k-1})\equiv (p_{k-1},p_{k-2},\ldots,p_0)$.
\end{itemize}
The equivalence class of which $(p_0,\ldots,p_{k-1})$ is a member is the \emph{cyclic sequence} $\vs{p_0,\ldots,p_{k-1}}$, and $k$ is its \emph{length}.
There is a natural partial order $\leq$ on the set of cyclic sequences:
\[ \vs{p_0,\ldots,p_{k-1}} \leq \vs{q_0,\ldots,q_{\ell-1}} \]
if and only if $k\leq \ell$ and there exists a cyclic subsequence $q_{i_0},q_{i_1},\ldots,q_{i_{k-1}}$ occurring in either order in $[q_0,q_1,\ldots,q_{\ell-1}]$ such that $p_j\leq q_{i_j}$ for each $j\in\{0,\ldots,k-1\}$.  We write $\sigma_1<\sigma_2$ if $\sigma_1\leq\sigma_2$ but $\sigma_1\neq \sigma_2$, where $\sigma_1$ and $\sigma_2$ are cyclic sequences.

\begin{exmp} Let $\sigma_1=[4,6,8,10]$, $\sigma_2=[6,8,12,4]$, and $\sigma_3=[10,8,12,6,4]$. Then $\sigma_1<\sigma_2$ and $\sigma_1<\sigma_3$, but $\sigma_2$ and $\sigma_3$ are not comparable.
\end{exmp}

\begin{defn} Let $\sigma=\vs{p_0,p_1,\ldots,p_{k-1}}$ be a cyclic sequence of integers $\geq 3$. Then $\sigma$ is the \emph{valence sequence} of a $k$-covalent face $f$ of a tessellation $T$ if the valences of vertices incident with $f$ in clockwise or counter-clockwise order are $p_0,p_1,\ldots,p_{k-1}$. If every face of $T$ has the same valence sequence $\sigma$, then $T$ is \emph{face-homogeneous} and $\sig$ is the \emph{valence sequence of $T$}.  Thus, to say briefly that a tessellation $T$ has valence sequence $\sigma$ implies that $T$ is face-homogeneous.
\end{defn}

\begin{defn}  
Let the cyclic sequnce $\sigma$ be realizable as the valence sequence of a tessellation.  If every tessellation having valence sequence $\sigma$ is uniformly concentric, then we say that $\sigma$ is \emph{uniformly concentric}.  Otherwise $\sigma$ is \emph{non-concentric}.  If every tessellation having valence sequence $\sigma$ is non-concentric, then $\sigma$ is {\em uniformly non-concentric}.
\end{defn}

\textbf{Notation.} By convention, when distinct letters are used to represent terms in a cyclic sequence (e.g. $[p,p,q,r,q]$), the values corresponding to distinct letters are all presumed to be distinct; that is, $p\neq q\neq r\neq p$. Moreover, if some term in the cyclic sequence is given as an integer (usually $3$ or $4$), then the terms given by letters are presumed to be greater than that integer. For example, if $\sig=\vs{4,p,q}$, then we understand that $p,q>4$ and $p\neq q$.  When using subscripts in the general form $[p_0,\ldots,p_{k-1}]$, we do not make this assumption.

\begin{rem}  Not all cyclic sequences are realizable as vertex sequences of face-homogeneous tessellations of the plane.  For instance, the map with valence sequence $[3,3,3]$ (the tetrahedron) is a tessellation of the sphere but not of the plane. More importantly, there are many cyclic sequences for which no face-homogeneous map exists at all.  For instance, the valence sequence $[4,5,6,p]$ for any $p\geq3$ is not realizable, because in any such map the valences of the neighbors of a 5-valent vertex in cyclic order would have to alternate between 4 and 6.
However, this does not generalize to all cyclic sequences containing a subsequence $\vs{p,q,r}$ where $q$ is odd and $p\neq r$; for instance, $\vs{5,4,5,6,5,8}$ is realizable. 
\end{rem}
 
\begin{conj}
Suppose $\sigma$ is the valence sequence of a face-homogeneous tessellation and that $\sig$ contains $[p,q,r]$ as a subsequence, with $q$ odd and $p\neq r$. Then $\sigma$ must contain at least three terms equal to $q$.
\end{conj}
\medskip

\subsection{Polynomial versus exponential growth}\label{sec:poly/exp}

Let $x$ be a vertex of a connected graph $\Gamma$.  For each nonnegative integer $n$, the \emph{ball of radius $n$ about $x$} is the set of vertices of $\Gamma$ at distance $\leq n$ from $x$, written
\begin{equation}\label{eq:n-Ball}
B_n(x) = \{v\in V(\Gamma):d(x,v)\leq n\},
\end{equation}
where $d(-,-)$ is the {\em standard graph-theoretic metric}, that is, $d(u,v)$ is the length of a shortest path with terminal vertices $u$ and $v$.

\begin{defn}\label{defn:poly_expgrowth} 
An infinite, locally finite, connected graph $\Gamma$ has \emph{exponential growth} if for some vertex $x\in V(\Gamma)$ there exist real numbers $\alpha>1$ and $C>0$ such that, for all $n>0$, one has $|B_n(x)|>C\alpha^n$; otherwise $\Gamma$ has \emph{subexponential growth}.  We say that  $\Gamma$ has \emph{polynomial growth} of degree $d\in\N$ if there exist  positive constants $C_1$ and $C_2$  such that $C_1n^d\leq|B_n(x)| \leq C_2n^d$ for all but finitely many $n$.
\end{defn}

For example, the graph underlying the square lattice in the plane has quadratic growth ($d=2$).  If $x$ is any vertex, then  $|B_n(x)|=2n^2+2n+1$ for all $n\geq1$, and one can set $C_1=2$ and $C_2=3$.

Continuing the notation of \Cref{eq:n-Ball} and \Cref{defn:poly_expgrowth}, we consider the generating function
\begin{equation}\label{eq:BallGenFn}
\beta_x(z)=\sum_{n=0}^\infty\abs{B_n(x)}z^n
\end{equation}
We denote the radius of convergence of $\beta_x(z)$ by $R_B$ and define the {\em ball-growth rate of $\Gamma$ about $x$} to be the reciprocal of $R_B$.
 
If $\Gamma$ has exponential growth, then we have
\begin{equation}\label{eq:ExpoGrowth}
\beta_x(z)\geq\sum_{n=0}^\infty C\alpha^nz^n=\frac{C}{1-\alpha z},
\end{equation}
where $\alpha>1$ is the supremum of values for which the series of \Cref{eq:BallGenFn} converges. The convergence is absolute if and only if $|z|<1/\alpha<1$. If $\Gamma$ has polynomial growth of degree $d$, then 
$$C_1\sum_{n=0}^\infty n^d z^n\leq\sum_{n=0}^\infty\abs{B_n(x)}z^n\leq C_2\sum_{n=0}^\infty n^d z^n.$$
By the ``ratio test,'' the first and third series converge if and only if $|z|<1$.  These computations yield the following.

\begin{prop}\label{prop:ExpoPoly}  Let $R_B$ denote the radius of convergence of the generating function of \Cref{eq:BallGenFn}.  Then $R_B<1$ if and only if $\Gamma$ has exponential growth, and $R_B=1$ if and only if $\Gamma$ has polynomial growth.  Moreover, $R_B$ is independent of the vertex $x$ about which $|B_n(x)|$ is determined.
\end{prop}

It will be seen in the next subsection (see \Cref{thm:InvariantGrowth}) that the value of $R_B$ in independent of the choice of the root vertex $x$.
\medskip

It is well known (for example, see \cite{TilesPats}) that there exist exactly eleven face-homogeneous Euclidean tessellations, namely the Laves nets.  Their valence sequences $[p_0,\ldots,p_{k-1}]$ correspond to integer solutions of the equation
\[ \sum_{i=0}^{k-1}\frac1{p_i} = \frac{k-2}2. \]

A necessary condition for the existence of a face-homogeneous hyperbolic tessellation with valence sequence $[p_0,\ldots,p_{k-1}]$ is that the inequality
\begin{equation}\label{eq:HyperbolicCondition}
    \sum_{i=0}^{k-1}\frac1{p_i} < \frac{k-2}2
\end{equation}
hold.  This condition is not sufficient, because as we have seen, not every such integer solution is realizable as a valence sequence. 

\begin{defn}\label{def:AngleExcess} The \emph{angle excess} of a cyclic sequence $\sigma=[p_0,\ldots,p_{k-1}]$ is given by
\[ \eta(\sigma) = \left(\sum_{i=0}^{k-1} \frac{p_i-2}{p_i}\right)-2. \]
\end{defn} 

Motivation for this definition comes from Descartes' notion of angular defect in the Euclidean plane. 
When $\eta(\sigma)>0$, there are too many faces incident at a vertex for the faces to be regular $k$-gons in the Euclidean plane. 

\begin{prop}\label{thm:HypCond} For a cyclic sequence $\sigma=[p_0,\ldots,p_{k-1}]$, inequality (\ref{eq:HyperbolicCondition}) is equivalent to 
\begin{equation}
    \eta(\sigma)>0
\end{equation} and is a necessary condition for $\sigma$ to be a valence sequence of a face-homogeneous hyperbolic tessellation.
\end{prop}

Angle excess provides a quick gauge of the growth behavior of a tessellation with valence sequence $\sigma$.  If  $\eta(\sigma)<0$, the tessellation is finite. If $\eta(\sigma)=0$, the tessellation is one of the Laves nets and has polynomial growth of degree 2. If $\eta(\sigma)>0$, the tessellation has exponential growth. Additionally, we have the following comparison result.

\begin{prop}\label{sigma&eta} Let $\sigma_1$ and $\sigma_2$ be cyclic sequences that are comparable in the partial order.  Then $\sigma_1<\sigma_2$ if and only if $\eta(\sigma_1)<\eta(\sigma_2)$.
\end{prop}
\begin{proof}
Suppose that $\sigma_1<\sigma_2$, where $\sigma_1=[p_0,\ldots,p_{k-1}]$ and $\sigma_2=[q_0,\ldots,q_{\ell-1}]$. By definition there exist $q_{i_0},\ldots,q_{i_{k-1}}$ with $p_j\leq q_{i_j}$ for all $j=0,\ldots k-1$.
So
\begin{equation}\label{eq:OrderedEta}
\eta(\sigma_1)=\sum_{j=0}^{k-1}\frac{p_j -2}{p_j} \leq
        \sum_{j=0}^{k-1}\frac{q_{i_j} -2}{q_{i_j}} 
    \leq \sum_{i=0}^{\ell-1} \frac{q_i -2}{q_i} = \eta(\sigma_2).
\end{equation}
If $k=\ell$, then  $p_j< q_{i_j}$ for some $j$ and the first inequality in (\ref{eq:OrderedEta}) is strict. If $k<\ell$, the second inequality in (\ref{eq:OrderedEta}) is strict. Since $\sigma_1\neq\sigma_2$, at least one such strict inequality must hold.
\qquad\end{proof}
\medskip

\subsection{Growth formulas}\label{sec:growth-form} 

In \Cref{defn:poly_expgrowth}, the standard graph-theoretical metric was used to define polynomial and exponential growth of a connected graph.  However, to measure growth rates of tessellations, it is more convenient to use the notion of ``regional distance;"  we will count the number of graph objects in the $\nth$ corona of a Bilinski diagram centered at a given vertex $x$, and our working definition of ``growth rate\rq\rq{} will be the following.

\begin{defn}\label{defn:GrowthRate} Let $T$ be a tessellation labeled as a Bilinski diagram rooted at a vertex $x$.  Let $R$ be the radius of convergence of the power series
\begin{equation}
\varphi_x(z) = \sum_{i=1}^\infty|F_i|z^i.
\end{equation}
When $0<R<\infty$, we define the \emph{growth rate} of $T$ (with respect to $x$) to be $\gamma(T) = 1/R$.
\end{defn} 

Although it was shown in \cite{GPW} (see pages 3--4) that, for any connected planar map with bounded covalences, the above definition of growth rate is equivalent to the growth rate with respect to the standard graph-theoretic metric, we need to show that said growth rate is independent of the root of the Bilinski diagram in question.

\begin{thm}\label{thm:InvariantGrowth} The growth rate $\gamma(T)$ of a face-homogeneous tessellation $T$ computed by means of a Bilinski diagram is invariant under the choice of the root of the diagram.
\end{thm}
\begin{proof}  Choose an arbitrary vertex $x$ of $T$ and consider a Bilinski diagram rooted at $x$.  Recall that the sequences $\set{U_n(x):0\leq n\in\mathbb{Z}}$ and $\set{F_n(x):1\leq n\in\mathbb{Z}}$ constitute the conventional labeling of $T$ as a Bilinski diagram with root vertex $x$,  As $T$ is face-homogeneous, all faces are $k$-covalent for some $k\geq3$.  Hence for any $n\geq1$ and any vertex $v\in U_{n+1}(x)$ there exists a vertex $u\in U_n$ such that $d(u,v)\leq \floor{\frac{k}{2}}$.  By induction on $n$, we obtain $d(x,v)\leq (n+1)\floor{\frac{k}{2}}$, yielding
\begin{equation}\label{eqn:U-bddabove}
\bigcup_{i=0}^n U_i(x) \subseteq B_{n\floor{k/2}}(x)
\end{equation}
and similarly,
\begin{equation}\label{eqn:U-bddbelow}
B_n(x)\subseteq \bigcup_{i=0}^n U_i(x).
\end{equation}
In addition to the power series $\varphi_x(z)$ of \Cref{defn:GrowthRate} with radius of convergence $R_F$, we require the power series $\ds{ \upsilon_x(z) = \sum \abs{U_n(x)}z^n}$  with radius of convergence $R_U$.
Writing 
\[
\Upsilon_x(z) 
    = \frac{\upsilon_x(z)}{1-z} 
    = \sum_{n=0}^\infty \left(\sum_{i=0}^n \abs{U_i(x)}\right) z^n 
    = \sum_{n=0}^\infty \abs{\bigcup_{i=0}^n U_i(x)}z^n,
\]
we have that the radius of convergence of $\Upsilon_x(z)$ equals $\min\set{R_U,1} \leq R_B$ by \Cref{eqn:U-bddabove} (where $R_B$ is as in \Cref{prop:ExpoPoly}). But similarly by \Cref{eqn:U-bddbelow} we have that $R_B \leq \min\set{R_U,1}$. Hence the radii of convergence of $\Upsilon_x(z)$ and $\beta_x(z)$ are equal, for any choice of root vertex $x$.

If $p$ is the maximum valence of the vertices in $T$, each vertex is also incident with at most $p$ faces, while each face is incident with $k$ vertices, giving
\[
\abs{U_n(x)} \leq k\abs{F_{n+1}(x)} \leq pk \abs{U_{n+1}(x)}
\]
for each $n\geq 0$, or equivalently,
\[
\frac1k \abs{U_n(x)} \leq \abs{F_{n+1}(x)} \leq \frac{p}{k} \abs{U_{n+1}(x)}.
\]
Hence the radii of convergence of $\upsilon_x(z)$ and $\varphi_x(z)$ are equal, and more importantly, $R_F = R_B$; that is, the rate of ball-growth equals the rate of growth when the Bilinski diagram is labeled from a vertex $x$. 

Finally, it follows from \Cref{prop:ExpoPoly} that ball-growth rates computed about distinct vertices are asymptotically equal in locally finite, connected, infinite graphs. Hence the radii of convergence of $\varphi_x(z)$, $\beta_x(z)$, $\beta_y(z)$, and $\varphi_y(z)$ are equal for all $x,y\in V$. That is to say, the growth rate of the graph is independent of the choice of root vertex.
\end{proof}

{\bf Notation.}  The subscript on the symbol $\varphi$ of \Cref{defn:GrowthRate} has now been shown to be superfluous and will henceforth be suppressed.
\medskip

Consider the function $\tau:\N_0\to\N_0$, (where $\N_0=\{0,1,2,\ldots\}$) given by
$$\tau(n) = \sum_{i=1}^n |F_i|.$$
The quantity 
\begin{equation}\label{eq:LimitGrowth}
\lim_{n\to\infty} \frac{\tau(n+1)}{\tau(n)}.
\end{equation}
was the definition of the growth rate of a face-homogeneous tessellation used by Moran \cite{Moran97} provided that this limit exists, in which case she called the tessellation \emph{balanced}.   Moran\rq{}s limit fails to converge only when there exist subsequences of the sequence $\left\{\frac{\tau(n+1)}{\tau(n)}\right\}_{n=1}^\infty$ with distinct limits. 

The following proposition shows that the parameters of a face-homogeneous tessellation determine an upper bound for the limit in \Cref{eq:LimitGrowth}.

\sloppypar{\begin{thm}\label{thm:BoundedRatio} Let $T$ be a face-homogeneous tessellation with valence sequence $[p_0,\ldots,p_{k-1}]$, labeled as a Bilinski diagram. Then 
\[ \limsup_{n\to\infty}\frac{\tau(n+1)}{\tau(n)} \leq 1 + \sum_{i=0}^{k-1}p_i -2k<\infty. \]
\begin{proof}
By hypothesis, each face of the tessellation shares an incident vertex with exactly 
\[ \sum_{i=0}^{k-1}(p_i-2)=\sum_{i=0}^{k-1}p_i-2k\]
 other faces. So for $n>0$,
\[ |F_{n+1}| \leq |F_n|\left(\sum_{i=0}^{k-1}p_i-2k \right), \] 
which in turn gives that for all $n>0$, 
\begin{align*}
\frac{\tau(n+1)}{\tau(n)} &\leq 1+\frac{|F_n|}{\sum_{i=0}^n|F_i|}\left(\sum_{i=0}^{k-1}p_i -2k\right) \\
    &\leq 1+\sum_{i=0}^{k-1}p_i-2k<\infty,
\end{align*}
 since $T$ is locally finite.
\qquad\end{proof}
\end{thm}}

By the ``ratio test" of elementary calculus, the above proof implies that in the case of a ``balanced'' tessellation, Moran's definition of growth rate concurs with \Cref{defn:GrowthRate}, and 
\[ \frac1R=\limsup_{n\to\infty}\frac{\tau(n+1)}{\tau(n)} = \lim_{n\to\infty}\frac{\tau(n+1)}{\tau(n)}. \]

The definition of growth rate in terms of the radius of convergence of a power series also allows us to prove the following result, which is essential in many comparisons of growth rates of various tessellations.

\begin{lem}[Comparison Lemma]\label{lem:CoronaSize}\label{thm:comp} Let $T_1$ and $T_2$ be tessellations, and  for $i=1,2$ let $|F_{i,n}|$ be the number of faces in the $n^\text{th}$ corona of a Bilinski diagram of $T_i$. Suppose that for some $N\in\N$, we have $|F_{1,n}|\leq |F_{2,n}|$ for all $n\geq N$. Then $\gamma(T_1)\leq \gamma(T_2)$.
\begin{proof}
Let 
\[ \phi_1(z) = \sum_{n=0}^{\infty} |F_{1,n}|z^n, \quad
    \phi_2(z) = \sum_{n=0}^{\infty} |F_{2,n}|z^n,  \] 
and for $i\in\{1,2\}$, let $R_i$ be the radius of convergence of $\phi_i(z)$ about $0$. Then since $|F_{1,n}|\leq |F_{2,n}|$ for sufficiently large $n$, and 
\[\limsup_{n\to\infty}\sqrt[n]{|F_{i,n}|}=\frac1{R_i}=\gamma(T_i),\] 
we have $\gamma(T_1)\leq \gamma(T_2)$.
\qquad\end{proof}
\end{lem}

\subsection{The edge-homogeneous case}\label{sec:edge-hom}
We conclude our presentation of preliminary material with a quick review of what is known about growth rates of  edge-homogeneous tessellations, as this case has been completely resolved and its consequences turn out to be useful here and there in attacking the present problem.  The point of departure here is the following   classification theorem of Gr\"{u}nbaum and Shephard.  (Edge-symbols were defined in \Cref{chap:Intro}.)

\begin{prop}[\cite{GrunShep87} Theorem 1]\label{thm:UniqueEdgeSymbol} Let $p,q,k,\ell\geq 3$ be integers. There exists an edge-homogeneous, $3$-connected, finite or $1$-ended map with edge-symbol $\langle p,q;k,\ell\rangle$ if and only if exactly one of the following holds:
\begin{enumerate}
\item all of $p,q,k,\ell$ are even;
\item $k=\ell$ is even and at least one of $p,q$ is odd;
\item $p=q$ is even and at least one of $k,\ell$ is odd;
\item $p=q$, $k=\ell$, and all are odd.
\end{enumerate}
Such a tessellation is edge-transitive, and the parameters $p,q,k,\ell$ determine the tessellation uniquely up to homeomorphism of the plane. If $p=q$, then the tessellation is vertex-transitive. If $k=\ell$, then it is face-transitive. \end{prop}

Following up on the Gr\"{u}nbaum-Shephard result, the authors together with T. Pisanski completely determined the growth rates of all edge-homogeneous tessellations.  Their main result is the following.

\begin{prop}[\cite{GPW} Theorem 4.1]\label{thm:EdgeGrowth} Let the function \[g:\{t\in\N:t\geq 4\}\to[1,\infty)\]be given by
\begin{equation}
    g(t) = \frac12\left(t-2+\sqrt{(t-2)^2-4}\right).
\end{equation}
Let $T$ be an edge-homogeneous tessellation with edge-symbol $\langle p,q;k,\ell\rangle$, and let 
\begin{equation}
    t = \left(\frac{p+q}2-2\right)\left(\frac{k+\ell}2-2\right).
\end{equation}
Then exactly one of the following holds:
\begin{enumerate}
\item    the growth rate of $T$ is $\gamma(T)=g(t)$; or
\item    the edge-symbol of $T$ or its planar dual is $\esym{3,p}{4,4}$ with $p\geq 6$, and the growth rate of $T$ is $\gamma(T)=g(t-1)$.
\end{enumerate}
\end{prop}

Observe that each value of $t\geq4$ corresponds to  only finitely many edge-homogeneous tessellations and that pairs of planar duals correspond to the same value of $t$.  As the growth rates of edge-homogeneous tessellations are determined by an increasing function in one variable, the following is immediate.

\begin{corollary} The least growth rate of an edge-homogeneous hyperbolic tessellation is $(3+\sqrt{5})/2$.  This value is attained only by the tessellations with edge-symbols $\langle 3,3;7,7\rangle$, $\langle4,4;4,5\rangle$, $\langle3,7;4,4\rangle$, and their planar duals.
\end{corollary}

\emph{Remark.} 
It is evident from \Cref{thm:UniqueEdgeSymbol} that if a tessellation is both edge- and face-homogeneous, then its edge-symbol and valence sequence have, respectively, either the forms $\esym{p,p}{k,k}$ and $\vs{p,p,\ldots,p}$ or the forms $\esym{p,q}{k,k}$ and $\vs{p,q,\ldots,p,q}$, the latter pair being possible only when $k$ is even.

We mention that, by an argument similar to the proof of \Cref{thm:BoundedRatio}, one easily obtains the following upper bound for the growth rate of an edge-homogeneous tessellation.

\begin{prop}\label{prop:BoundedRatio} Let $T$ be an edge-homogeneous tessellation with edge-symbol $\langle p,q;k,\ell\rangle$. Then for any labeling of $T$ as a Bilinski diagram, one has
\[ \lim_{n\to\infty}\frac{\tau(n+1)}{\tau(n)} \leq 1+\max\{pk,qk,p\ell,q\ell\}.\]
\end{prop}

\section{Accretion and Monomorphic Valence Sequences}
\subsection{Accretion}\label{sec:Accretion}

Given an arbitrary face-homogeneous tessellation $T$ with valence sequence ${\sigma=[p_0,p_1,\ldots,p_{k-1}]}$, we wish to apply \Cref{defn:GrowthRate} to determine its growth rate.   Letting $T$ be labeled as a Bilinski diagram, we require a means to evaluate $|F_n|$ for all $n\in\N$.  This is done inductively; each face $f\in F_n$ ``begets" a certain number of facial ``offspring" in $F_{n+1}$, and that number is determined by the configuration of $f$ within $\induced{F_n}$, that is, what the valences are of the vertices incident with $f$ (in the rotational order of $\sigma$) that belong, respectively, to $U_{n-1}$ and more importantly to $U_n$.

A class of identically configured faces (in any corona) is a \emph{face type}, and is denoted by $\ft_i$ for some range of $i=1,\ldots,r$. The benefit of using face types is that we can define an $r$-dimensional column vector $\vec{v}_n$, called the \emph{$\nth$ distribution vector}, which lists the frequency with which each face type occurs in the $\nth$ corona. Thus, if $\vec{j}$ is the $r$-dimensional vector of $1$s, then $|F_n|=\vec{j}\cdot\vec{v}_n$  via the standard dot product.

\Cref{fig:offspring} depicts a face $f\in F_n$ of some tessellation and the faces in $F_{n+1}$ which are determined by the face type of $f$. These faces are called the \emph{offspring} of $f$, and the figure is accordingly called the \emph{offspring diagram} for $f$.
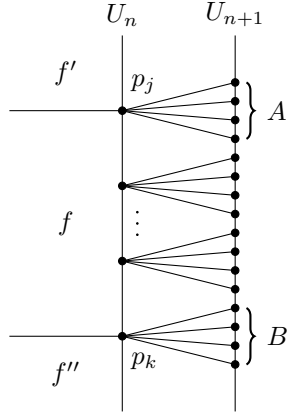
\begin{figure}[htbp]
\[
\begin{tikzpicture}
[
    bvert/.pic = {\draw [fill=black] (0,0) circle [radius=0.05];},
]
\foreach \a in {0,1,2,3}{
    \pic at (0,-1.5+\a) {bvert};
    \foreach \b in {-1.5,-.5,.5,1.5}{
        \draw (0,-1.5+\a) -- (1.5,-1.5+\a+.25*\b) pic {bvert};
        }
    }
\draw (0,-2.5) -- (0,2.5) 
    (1.5,-2.5) -- (1.5,2.5) 
    (-1.5,1.5) -- (0,1.5) 
    (0,-1.5) -- (-1.5,-1.5);
\node [above] at (0,2.5) {$U_n$};
\node [above] at (1.5,2.5) {$U_{n+1}$};
\node [above right] at (0,1.6) {$p_j$};
\node [below right] at (0,-1.6) {$p_k$};
\node at (-.75,0) {$f$};
\node at (-.75,2) {$f'$};
\node at (-.75,-2) {$f''$};
\foreach \a in {-.15,0,.15}{
    \draw [fill] (.2,\a) circle [radius=.01];
    }
\draw [thick, decorate, decoration={brace, amplitude=3pt}] 
    (1.65,1.875) -- (1.65,1.125) node [midway, right=4pt] {$A$};
\draw [thick, decorate, decoration={brace, amplitude=3pt}] 
    (1.65,-1.125) -- (1.65,-1.875) node [midway, right=4pt] {$B$};
\end{tikzpicture}
\]
\caption{\label{fig:offspring}A face $f$ in $F_n$ of a tessellation $T$, along with the offspring of $f$ in $F_{n+1}$.}
\end{figure}
As the vertex labeled $p_{j}$ is incident with both faces $f$ and $f'\in F_n$, one-half of those faces in $F_{n+1}$ labeled as $A$ in the figure count as offspring of $f$, and one-half are counted as offspring of $f'$. Similarly, half of the faces labeled by $B$ count as offspring of $f$ and half as offspring of $f''$.  All those faces between labels $A$ and $B$ in \Cref{fig:offspring} are wholly offspring of $f$. Those faces which are offspring of $f$, or offspring of offspring of $f$, and so on, are called collectively \emph{descendants} of~$f$.

\begin{defn}\label{subsec:GenFaceTypes}
With respect to the labeling of a Bilinski diagram, each vertex incident with a face $f\in F_n$ lies in $U_{n-1}$ or $U_n$. The pattern of valences of vertices in $U_{n-1}$ and in $U_n$ determines the \emph{face type} of $f$.
The three face types occurring most routinely are called \emph{wedges}, \emph{bricks}, and \emph{notched bricks}.  A face $f$ in $F_n$ is a \emph{wedge} if it is incident with exactly one vertex in $U_{n-1}$. The face $f$ is a \emph{brick} if it incident with exactly two adjacent vertices in $U_{n-1}$ and at least two vertices in $U_n$. Finally, $f$ is a \emph{notched brick} if it is incident with three consecutive vertices of $U_{n-1}$, of which the middle vertex is $3$-valent, and $f$ is incident with two or more vertices in $U_n$. For a given labeling of a tessellation $T$ as a Bilinski diagram, the face types of $T$ are indexed $\ft_1,\ldots,\ft_r$ for some $r\in\N$; we explain the method by which indices are assigned after the statement of \Cref{thm:GrowthOrder}.
\end{defn}

An algorithm by which one can describe the faces, corona by corona, of a tessellation labeled as a Bilinski diagram is called an \emph{accretion rule}. Often some homogeneous system of recurrence relations determines such an accretion rule.  In this case,  the $\nth$ distribution vector $\vec{v}_n$ defined above has the property that the $j^\text{th}$ component of $\vec{v}_n$ is the number of faces of type $\ft_j$ in the $\nth$ corona. We then encode the system of recurrences into a \emph{transition matrix} $M$ such that $\vec{v}_{n+1}=M\vec{v}_n$ holds for all $n\geq 1$. When $M=\left[m_{i,j}\right]$ is such a matrix, the entry $m_{i,j}$ is the number of faces of type $\ft_j$ that are offspring of a face of type $\ft_i$.
We require the following result from \cite{GPW}.

\begin{prop}[\cite{GPW}, Theorem 3.1]\label{prop:OGF} Let $T$ be a tessellation labeled as a Bilinski diagram with accretion rule specified by the transition matrix $M$ and first distribution vector  $\vec{v}_1$. Then the ordinary generating function for the sequence $\left\{\abs{F_n}\right\}_{n=1}^\infty$ is
\begin{equation}
\varphi(z) = \abs{F_0}+z\left( \vec{j}\cdot(I-zM)^{-1}\vec{v}_1 \right),
\end{equation} where $I$ is the identity matrix and $\vec{j}$ is the vector of $1$s.
\end{prop}

By using \Cref{defn:GrowthRate}, we can prove the following more directly than we did in Theorem 3.4 of \cite{GPW}.

\begin{thm}\label{thm:Eigenvalue} 
If $M$ is the transition matrix of a tessellation $T$ and $\Lambda$ is the maximum modulus of an eigenvalue of $M$, then $\gamma(T)=\Lambda$.
\begin{proof}
We can write the generating function $\varphi(z)$ of \Cref{prop:OGF} as a rational function $u(z)/v(z)$, with $v(z)$ determined entirely by $(I-zM)^{-1}$.
Specifically, using Cramer's rule where $r$ denotes the order of $M$, we have
\begin{equation}\label{eq:MatInv}
\begin{split}
(I-zM)^{-1} &= \frac1{\det(I-zM)}\adj(I-zM) \\
    &= \frac1{(-z)^r \det(M-\frac1zI)}\adj(I-zM) \\
    &= \frac1{(-z)^r\chi(\frac1z)}\adj(I-zM)
\end{split}\end{equation}
where $\chi(\frac1z)$ is the characteristic polynomial (in $\frac1z$) of $M$. Entries of the adjoint $\adj(I-zM)$ are polynomials in $z$ of degree at most $r-1$, and so $v(z)=(-z)^r\chi(\frac1z)$. As $\chi(\frac1z)$ is a polynomial in $\frac1z$ of degree exactly $r$, $v(z)$ has a nonzero constant term and the roots of $v$ occur precisely at the roots of $\chi(\frac1z)$. These are precisely the reciprocals of the eigenvalues of $M$. Thus the minimum modulus of a pole of $\varphi(z)$ is $1/\Lambda$. As this is the definition of the radius of convergence of a power series expanded about $0$, we have $\gamma(T)=\Lambda$.
\qquad\end{proof}
\end{thm}

\subsection{Monomorphic, Uniformly Concentric Sequences}\label{sec:MonoUC}
 
As we have already remarked, valence sequences of face-homogeneous tessellations are unlike edge-symbols of  edge-homogeneous tessellations in two significant ways:  (i) the requirements for realizability of an edge-symbol are simpler and less stringent than the realizability criteria for a cyclic sequence, and (ii)  two or more non-isomorphic face-homogeneous tessellations may share a common valence sequence.  This latter property motivates the following definition.

\begin{defn} \label{defn:MonoPoly}
Let $\sigma$ be a cyclic sequence.  If there exists (up to isomorphism) a unique face-homogeneous tessellation with valence sequence $\sigma$, then we say that $\sigma$ is \emph{monomorphic}.  If there exist at least two (non-isomorphic) tessellations with valence sequence $\sigma$, then $\sigma$ is \emph{polymorphic}. 
\end{defn}

\begin{prop}[Moran \cite{Moran97}]\label{coval3} All realizable cyclic sequences of length $3$ are monomorphic.
\end{prop}

A second property of interest is whether a given valence sequence is uniformly concentric.  These two properties thus yield four classes of valence sequences.  Not surprisingly, the class most amenable to an elegant and simple accretion rule consists of those that are both monomorphic and uniformly concentric.

One can find in \cite{SiagWat07} a complete classification of cyclic sequences of length $k$ for $3\leq k\leq5$ in terms of \Cref{defn:MonoPoly} which will help us to narrow our investigation.  (It is actually the equivalent dual problem that is treated in \cite{SiagWat07}, and the term ``covalence sequence'' is used.  In the present work we have opted to follow Moran \cite{Moran97}, speaking rather in terms of ``valence sequences.'')

We now turn to considering the relative growth rates of tessellations with monomorphic valence sequences. 
The ideal condition would be to have that the partial order on cycic sequences is mirrored by the natural order on growth rates: that is, if $T_1$ and $T_2$ are tessellations with valence sequences $\sig_1\leq \sig_2$, then $\gam(T_1)\leq \gam(T_2)$. For monomorphic, uniformly concentric valence sequences, this is precisely the case, as stated below in \Cref{thm:GrowthOrder}. In order to prove the theorem, we now demonstrate the necessary machinery via the following example, which can be readily generalized.

\begin{exmp}Consider $T_1$ and $T_2$ to be  face-homogeneous tessellations with monomorphic valence sequences $\sig_1=\vs{4,5,4,5}$ and $\sig_2=\vs{4,6,6,4,5}$, respectively, both labeled as face-rooted Bilinski diagrams. Note that $\sig_1<\sig_2$. We continue the convention that $F_{i,n}$ denotes the set of faces of the $\nth$ corona of $T_i$. (The reader may follow Figures 2 through 7.)  Starting with $T_1$, we construct by induction a sequence $\{T'_j:j\in\N\}$ of tessellations such that:
\begin{enumerate}
\item $T'_0=T_1$ as a base for the induction, 
\item if we denote by $F'_{j,n}$ the set of faces in the $\nth$ corona of $T'_j$, then  for each $j\in\N$, the unions of the first $n$ coronas of $T'_j$ satisfy 
\[\induced{\bigcup_{n=1}^j F'_{j,n}}\cong\induced{\bigcup_{n=1}^j F_{2,n}}\]
as induced subgraphs, and
\item $\abs{F_{1,n}}\leq \abs{F'_{j,n}}$ for all $n\in\N$.
\end{enumerate}

\begin{figure}[htbp]
\[
\begin{tikzpicture}
[
    bvert/.pic = {\draw [fill=black] (0,0) circle [radius=0.05];},
    wvert/.pic = {\draw [fill=white] (0,0) circle [radius=0.05];}
]
\draw (0,0) circle [radius=2] circle [radius=3];
\foreach \a in {0,1,2,3}{
    \draw ({-45+90*\a}:1) -- ({-45+90*(\a+1)}:1);
    }
\foreach \a in {-45,135}{
    \foreach \b in {-45,0,45}{
        \draw (\a:1) -- (\a+\b:2);
        }
    \foreach \b in {-22.5,22.5}{
        \draw (90+\a:1) -- (90+\a+\b:2);
        }
    \foreach \b in {-1,1}{
        \draw ({\a+\b*45}:2) -- ({\a+\b*(45+45/11)}:3);
        \foreach \c in {1,3,5}{
            \draw (\a+22.5*\b:2) -- ({\a+\b*90*\c/11}:3);
            }
        \draw (\a+90:2) -- (\a+90+\b*90/11:3);
        \foreach \c in {2,4}{
            \draw (\a+90+\b*22.5:2) -- (\a+90+\b*\c*90/11:3);
            }
        }
    \draw (\a:2) -- (\a:3);
    }
\foreach \a in {0,1}{
    \draw 
        ({-45+180*\a}:1) pic {wvert}
        ({45+180*\a}:1) pic {bvert};
    }
\foreach \a in {0,1,...,7}{
    \draw
        ({-45+45*\a}:2) pic {bvert}
        ({-22.5+45*\a}:2) pic {wvert};
    }
\foreach \a in {0,1,...,21}{
    \draw
        ({-45+180*\a/11}:3) pic {wvert}
        ({-45+90/11+180*\a/11}:3) pic {bvert};
    }
\end{tikzpicture}
\]
\caption[$T_1$ in the construction of the Growth Comparison Theorem]{The first three coronas of $T_1$}
\end{figure}
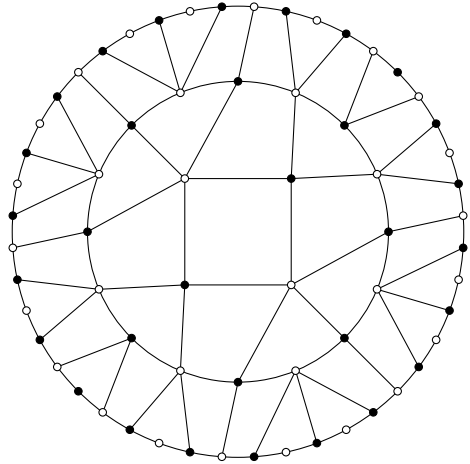

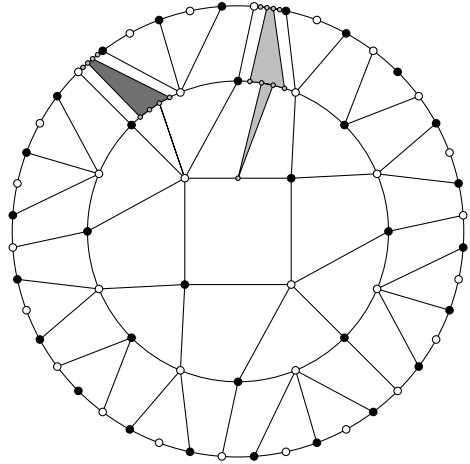
\begin{figure}[htbp]
\[
\begin{tikzpicture}
[
    bvert/.pic = {\draw [fill=black] (0,0) circle [radius=0.05];},
    wvert/.pic = {\draw [fill=white] (0,0) circle [radius=0.05];},
    dvert/.pic = {\draw [fill=darkgray] (0,0) circle [radius=0.03];},
    gvert/.pic = {\draw [fill=lightgray] (0,0) circle [radius=0.03];}
]
\draw (0,0) circle [radius=2] circle [radius=3];
\foreach \a in {0,1,2,3}{
    \draw ({-45+90*\a}:1) -- ({-45+90*(\a+1)}:1);
    }
\foreach \a in {-45,135}{
    \foreach \b in {-45,0,45}{
        \draw (\a:1) -- (\a+\b:2);
        }
    \foreach \b in {-22.5,22.5}{
        \draw (90+\a:1) -- (90+\a+\b:2);
        }
    \foreach \b in {-1,1}{
        \draw ({\a+\b*45}:2) -- ({\a+\b*(45+45/11)}:3);
        \foreach \c in {1,3,5}{
            \draw (\a+22.5*\b:2) -- ({\a+\b*90*\c/11}:3);
            }
        \draw (\a+90:2) -- (\a+90+\b*90/11:3);
        \foreach \c in {2,4}{
            \draw (\a+90+\b*22.5:2) -- (\a+90+\b*\c*90/11:3);
            }
        }
    \draw (\a:2) -- (\a:3);
    }
\filldraw [fill=lightgray] 
    (0,.707) -- (76.5:2) -- (81:2) -- (76.5:2) -- (72:2) -- 
    (81:3) -- (873/11:3) -- (81:3) -- (909/11:3) -- (927/11:3) --
    (909/11:3) -- (85.5:2) -- (81:2) -- (0,.707);
\filldraw [fill=darkgray!75]
    (135:1) -- (121.5:2) -- (117:2) --     (1431/11:3) --     
    (1449/11:3) --     (130.5:2) -- (126:2) -- (121.5:2) -- (135:1);
\draw
    (1413/11:3) pic {gvert}
    (1467/11:3) pic {gvert}
    (121.5:2) pic {gvert}
    (117:2) pic {gvert}
    (1431/11:3) pic {gvert}
    (1449/11:3) pic {gvert}
    (130.5:2) pic {gvert}
    (126:2) pic {gvert};
\draw 
    (0,.707) pic {gvert}  
    (76.5:2) pic {gvert} 
    (72:2) pic {gvert}  
    (81:3) pic {gvert} 
    (873/11:3) pic {gvert} 
    (927/11:3) pic {gvert} 
    (909/11:3) pic {gvert} 
    (85.5:2) pic {gvert} 
    (81:2) pic {gvert};
\foreach \a in {0,1}{
    \draw 
        ({-45+180*\a}:1) pic {wvert}
        ({45+180*\a}:1) pic {bvert};
    }
\foreach \a in {0,1,...,7}{
    \draw
        ({-45+45*\a}:2) pic {bvert}
        ({-22.5+45*\a}:2) pic {wvert};
    }
\foreach \a in {0,1,...,21}{
    \draw
        ({-45+180*\a/11}:3) pic {wvert}
        ({-45+90/11+180*\a/11}:3) pic {bvert};
    }
\end{tikzpicture}
\]
\caption[$T'_1$ in the construction of the Growth Comparison Theorem]{The first three coronas of $T'_1$. The dark gray region is a subgraph inserted by augmentation of the valence of a vertex from 5 to 6; the light gray region is a subgraph inserted while interpolating a $6$-valent vertex along an incident edge. These insertions continue throughout all coronas of $T'_1$.}
\end{figure}

To construct $T'_1$ from $T'_0$, the valence sequence of the root face of $T'_0$ must change from $\sig_1$ to $\sig_2$. To do so, we augment the valence of a $5$-valent vertex $v\in U_1$ to $6$-valent and then subdivide an edge of $\induced{U_1}$ incident with $v$ by inserting a $6$-valent vertex. Augmentation and interpolation are both performed via the insertion of an infinite ``cone" as follows. We choose a sequence of edges $e_2, e_3, e_4, \ldots$, with $e_i\in\induced{U_i}$, such that $e_2$ and $v$ are incident with a common face in $F_1$, and for each $i\geq 2$, $e_i$ and $e_{i+1}$ are incident with a common face in $F_i$. On each of these edges we interpolate vertices, and we insert edges connecting vertices between $U_i$ and $U_{i+1}$ ensuring that every face so created has covalence $5$. Furthermore, if a created face is incident  only with interpolated vertices, then its valence sequence is $\sig_2$. This insertion is well-defined precisely because $\sig_2$ is monomorphic, {i.e.}, the vertices and edges may be inserted in exactly one way. 

The resulting tessellation after the procedure just described  is denoted by $T'_1$. Faces of $T'_1$ fall into three classes: first, there are faces which have valence sequence $\sig_1$ and in $T'_0$ were not incident with any part of the inserted cone; second, there are those faces with valence sequence $\sig_2$ that have been inserted; finally, there are faces which are incident with newly inserted edges but which have neither valence sequence $\sig_1$ nor $\sig_2$. A face $f$ in this third class has covalence equal to the length of $\sig_2$, but some vertices incident with $f$ have valences from $\sig_1$. These faces may occur in all coronas outward from the first corona.
\begin{figure}[htbp]
\[\begin{tikzpicture}
[
    scale = 1,
    bvert/.pic = {\draw [fill=black] (0,0) circle [radius=0.05];},
    wvert/.pic = {\draw [fill=white] (0,0) circle [radius=0.05];},
    dvert/.pic = {\draw [fill=darkgray] (0,0) circle [radius=0.05];},
    gvert/.pic = {\draw [fill=lightgray] (0,0) circle [radius=0.05];}
]
\draw (0,4) -- (10/28,4) (10-10/28,4) -- (10,4) (2,2) -- (3,2) (6,2) -- (8,2)
    (0,4) -- (2,2) -- (5,0) -- (8,2) (7,2) -- (10,4);
\filldraw [thick,fill=darkgray!75]
     (5,0) -- (5,2) -- (3,2) -- 
     (5/7,4) -- (5/14,4) -- (135/14,4) -- (65/7,4) -- 
     (6,2) -- (5,2);
\draw [thick]
    (25/14,4) -- (3,2) -- (40/14,4)
    (50/14,4) -- (4,2) -- (65/14,4)
    (75/14,4) -- (5,2) -- (90/14,4) (5,2) -- (105/14,4)
    (6,2) -- (115/14,4);
\draw (0,4) pic {wvert}
    (10,4) pic {bvert}
    (5,0) pic {gvert}
    (2,2) pic {bvert}
    (3,2) pic {gvert}
    (4,2) pic {gvert}
    (5,2) pic {gvert}
    (6,2) pic {gvert}
    (7,2) pic {wvert}
    (8,2) pic {bvert};
\foreach \a in {1,2,...,27}{
    \draw (10*\a/28,4) pic {gvert};
    }
\foreach \a in {2,5,8,12,14,16,18,20,22,24}{
    \node [above] at (10*\a/28,4) {$4$};
    }
\foreach \a in {13,17,19,23}{
    \node [above] at (10*\a/28,4) {$5$};
    }
\foreach \a in {1,3,4,6,7,9,10,11,15,21,25,26,27}{
    \node [above] at (10*\a/28,4) {$6$};
    }
\node [minimum size=.05] (marked) [below] at (0,4) {$\uparrow$};
\node [above] at (0,4) {$5$};
\node [above] at (10,4) {$4$};
\node [below left] at (2,2) {$4$};
\node [below] at (3,2) {$5$};
\node [below] at (4,2) {$4$};
\node [below left] at (5,2) {$6$};
\node [below] at (6,2) {$4$};
\node [below] at (7,2) {$5$};
\node [below right] at (8,2) {$4$};
\node [below] at (5,0) {$6$};
\end{tikzpicture}\]
\caption[Effect of increasing a valence in the Growth Comparison Theorem]{An expanded view of the subgraph inserted when increasing the valence of a $5$-valent vertex to $6$-valent. Note that the $5$-valent vertex in the upper left, marked by the arrow, is disrupting the valence sequence of the white face with which it is incident; if the marked vertex were $6$-valent, that face would have valence sequence $\sig_2=\vs{4,6,6,4,5}$.}
\end{figure}
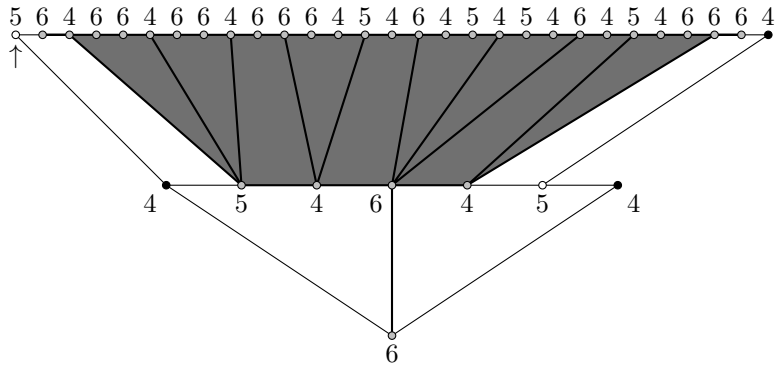
\begin{figure}[htbp]
\[\begin{tikzpicture}
[
    scale = 1,
    bvert/.pic = {\draw [fill=black] (0,0) circle [radius=0.05];},
    wvert/.pic = {\draw [fill=white] (0,0) circle [radius=0.05];},
    dvert/.pic = {\draw [fill=darkgray] (0,0) circle [radius=0.05];},
    gvert/.pic = {\draw [fill=lightgray] (0,0) circle [radius=0.05];}
]
\draw (1/4,4) -- (0,4) -- (7/4,2) -- (9/4,2) (7/4,2) -- (4,0) -- (6,0) -- (33/4,2) -- (31/4,2) (33/4,2) -- (10,4) -- (39/4,4);
\draw [thick] (1/4,4) -- (39/4,4) (9/4,2) -- (31/4,2);
\foreach \x in {2,5,8}{
    \draw [thick] (9/4,2) -- (\x/4,4);
    }
\draw [thick] (11/4,2) -- (10/4,4);
\foreach \x in {12,15,18}{
    \draw [thick] (13/4,2) -- (\x/4,4);
    }
\filldraw [fill=lightgray, draw=black, thick] (15/4,2) -- (27/4,2) -- (25/4,4) -- (20/4,4) -- (15/4,2);
\foreach \x in {27,30,33}{
    \draw [thick] (29/4,2) -- (\x/4,4);
    }
\draw [thick] (31/4,2) -- (35/4,4) (31/4,2) -- (38/4,4);
\foreach \x in {11,17,23,29}{
    \draw [thick] (5,0) -- (\x/4,2);
    }
\draw (4,0) pic {dvert} (5,0) pic {gvert} 
    (6,0) pic {bvert} (7/4,2) pic {bvert} 
    (33/4,2) pic {wvert} (0,4) pic {wvert} 
    (10,4) pic {bvert};
\foreach \x in {9,11,...,31}{
    \draw (\x/4,2) pic {gvert};
    }
\foreach \x in {1,2,...,20}{
    \draw (\x/4,4) pic {gvert};
    }
\foreach \x in {25,26,...,39}{
    \draw (\x/4,4) pic {gvert};
    }
\node [minimum size=.05] (marked) [below] at (0,4) {$\uparrow$};
\foreach \x in {1,3,4,6,7,9,10,11,13,14,16,17,19,20,27,33,37,38,39}{
    \node [above] at (\x/4,4) {\small$6$};
    }
\foreach \x in {2,5,8,12,15,18,26,28,30,32,34,36,40}{
    \node [above] at (\x/4,4) {\small$4$};
    }
\foreach \x in {0,25,29,31,35}{
    \node [above] at (\x/4,4) {\small$5$};
    }
\foreach \x in {7,11,15,19,23,27,31}{
    \node [below] at (\x/4,2) {\small$4$};
    }
\foreach \x in {9,13,21,25,33}{
    \node [below] at (\x/4,2) {\small$5$};
    }
\node [below] at (17/4,2) {\small$6$};
\node [below] at (29/4,2) {\small$6$};
\node [below] at (4,0) {\small$6$};
\node [below] at (5,0) {\small$6$};
\node [below] at (6,0) {\small$4$};

\end{tikzpicture}
\]
\caption[Effect of interpolating a vertex in the Growth Comparison Theorem]{An expanded view of the subgraph inserted when interpolating a $6$-valent vertex along an edge incident to the root. Again note the marked $5$-valent vertex in the upper left. (The large shaded region represents a number of faces of valence sequence $\vs{4,6,6,4,5}$ which are too dense to draw nicely in the Euclidean plane.)}
\end{figure}
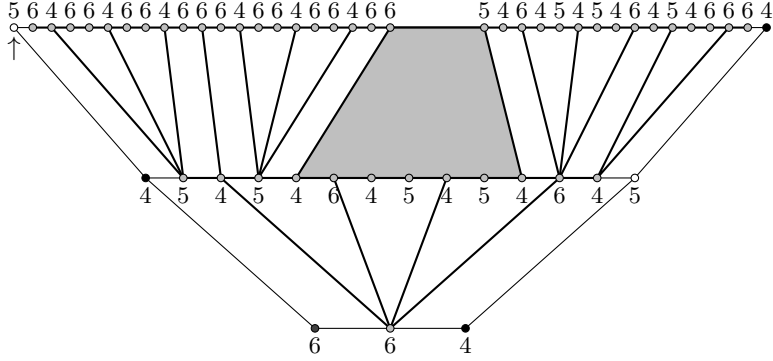

We compare now the tessellations $T_1$, $T'_1$, and $T_2$. In each case, the $0^\text{th}$ corona contains only the root face.
So from our construction, \[\abs{F_{1,0}}=\abs{F'_{1,0}}=\abs{F_{2,0}},\,\text{and for all }n\in\N_0,\,\abs{F_{1,n}}\leq \abs{F'_{1,n}},\] as we have inserted faces into every corona outward from the first.

We construct $T'_2$ from $T'_1$ just as we created $T'_1$ from $T'_0=T_1$; there is, however, one additional type of interpolation which may occur. Specifically, a vertex must be interpolated in an edge incident with two adjacent faces in $F'_{1,1}$. In \Cref{fig:46645-collapsed-v2}, an example of such an edge is marked with an arrow. This obstacle proves to be minor, as the necessary interpolation is shown in \Cref{fig:46645-insert} -- rather than interpolating a vertex on an edge incident with vertices in both $U_1$ and $U_2$, the vertex and its two neighbors are interpolated in $U_2$, replacing a $(5,4,5)$-path in $\induced{U_2}$ with a $(5,4,6,4,5)$-path.

\begin{figure}[htbp]
\[
\begin{tikzpicture}
[
    bvert/.pic = {\draw [fill=black] (0,0) circle [radius=0.05];},
    wvert/.pic = {\draw [fill=white] (0,0) circle [radius=0.05];},
    dvert/.pic = {\draw [fill=darkgray] (0,0) circle [radius=0.03];},
    gvert/.pic = {\draw [fill=lightgray] (0,0) circle [radius=0.03];}
]
\draw (0,0) circle [radius=2] circle [radius=3];
\foreach \a in {0,1,2,3}{
    \draw ({-45+90*\a}:1) -- ({-45+90*(\a+1)}:1);
    }
\draw (-45:1) -- (0:2);
\draw (0:2) -- (90/22:3);
\foreach \a in {-45,135}{
    \foreach \b in {-45,0}{
        \draw (\a:1) -- (\a+\b:2);
        }
    \foreach \b in {-22.5,22.5}{
        \draw (90+\a:1) -- (90+\a+\b:2);
        }
    \draw ({\a-45}:2) -- ({\a-(45+45/11)}:3);
    \foreach \b in {-1,1}{
        \foreach \c in {1,3,5}{
            \draw (\a+22.5*\b:2) -- ({\a+\b*90*\c/11}:3);
            }
        \draw (\a+90:2) -- (\a+90+\b*90/11:3);
        \foreach \c in {2,4}{
            \draw (\a+90+\b*22.5:2) -- (\a+90+\b*\c*90/11:3);
            }
        }
    \draw (\a:2) -- (\a:3);
    }
\filldraw [thick, fill=lightgray] 
    (0,.707) -- (76.5:2) -- (81:2) -- (76.5:2) -- (72:2) -- 
    (81:3) -- (873/11:3) -- (81:3) -- (909/11:3) -- (927/11:3) --
    (909/11:3) -- (85.5:2) -- (81:2) -- (0,.707);
\filldraw [thick, fill=darkgray!75]
    (135:1) -- (121.5:2) -- (117:2) --     (1431/11:3) --     
    (1449/11:3) --     (130.5:2) -- (126:2) -- (121.5:2) -- (135:1);
\filldraw [fill=darkgray!50, thick]
    (135:1) -- (180:2) node [midway, below] {$\uparrow$} -- (675/4:2) -- (2007/11:3) -- (2043/11:3) -- (765/4:2) -- (180:2);
\draw
    (1413/11:3) pic {gvert}
    (1467/11:3) pic {gvert}
    (121.5:2) pic {gvert}
    (117:2) pic {gvert}
    (1431/11:3) pic {gvert}
    (1449/11:3) pic {gvert}
    (130.5:2) pic {gvert}
    (126:2) pic {gvert};
\draw 
    (0,.707) pic {gvert}  
    (76.5:2) pic {gvert} 
    (72:2) pic {gvert}  
    (81:3) pic {gvert} 
    (873/11:3) pic {gvert} 
    (927/11:3) pic {gvert} 
    (909/11:3) pic {gvert} 
    (85.5:2) pic {gvert} 
    (81:2) pic {gvert};
\draw
    (675/4:2) pic {gvert}
    (765/4:2) pic {gvert}
    (1971/11:3) pic {gvert}
    (2007/11:3) pic {gvert}
    (2043/11:3) pic {gvert}
    (189:3) pic {gvert};
\foreach \a in {0,1}{
    \draw 
        ({-45+180*\a}:1) pic {wvert}
        ({45+180*\a}:1) pic {bvert};
    }
\foreach \a in {0,1,...,7}{
    \draw
        ({-45+45*\a}:2) pic {bvert}
        ({-22.5+45*\a}:2) pic {wvert};
    }
\foreach \a in {0,1,2,3,4,5,6,7,8,9,10,11,12,13,15,16,17,18,19,20,21}{
    \draw
        ({-45+180*\a/11}:3) pic {wvert}
        ({-45+90/11+180*\a/11}:3) pic {bvert};
    }
\draw (-135-4*90/11:3) pic {bvert};
\end{tikzpicture}
\]
\caption{\label{fig:46645-collapsed-v2}Beginning the construction of $T'_2$ from $T'_1$.}
\end{figure}
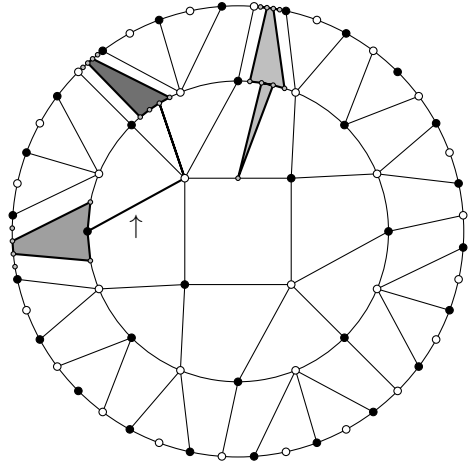

\begin{figure}[htbp]
\[
\begin{tikzpicture}
[
    bvert/.pic = {\draw [fill=black] (0,0) circle [radius=0.07];},
    wvert/.pic = {\draw [fill=white] (0,0) circle [radius=0.07];},
    dvert/.pic = {\draw [fill=lightgray] (0:.08) -- (60:.08) -- (120:.08) -- (180:.08) -- (240:.08) -- (300:.08) -- (0:.08);},
    gvert/.pic = {\draw [fill=lightgray] (0,0) circle [radius=0.07];},
]
\draw (0,3) -- (4,3) (2,3) -- (2,0) (0,3/2) -- (4,3/2);
\draw (6,3) -- (6.4,3) (9.6,3) -- (10,3) (6,3/2) -- (7,3/2) (9,3/2) -- (10,3/2);
\filldraw [thick, fill=darkgray!50]
    (8,0) -- (8,3/2) -- (7,3/2) -- (6.8,3) -- (6.4,3) -- (9.6,3) -- (9.2,3) -- (9,3/2) -- (8,3/2);
\draw [thick] (7,3/2) -- (7.2,3)
    (8,3/2) -- (7.6,3)
    (8,3/2) -- (8,3)
    (8,3/2) -- (8.4,3)
    (9,3/2) -- (8.8,3);
\draw (0,3) pic {bvert} (4,3) pic {bvert} (2,3/2) pic {bvert}
    (0,3/2) pic {wvert} (2,3) pic {wvert} (4,3/2) pic {wvert}
    (2,0) pic {dvert};
\node [above] at (0,3) {$4$};
\node [above] at (2,3) {$5$};
\node [above] at (4,3) {$4$};
\node [below] at (0,3/2) {$5$};
\node [below right] at (2,3/2) {$x$};
\node [below] at (4,3/2) {$5$};
\node [right] at (2,0) {$6$};
\foreach \x in {1,2,...,9}{
    \draw (6+\x*4/10,3) pic {gvert};
    }
\draw (7,3/2) pic {gvert} (8,3/2) pic {gvert} (9,3/2) pic {gvert};
\draw (6,3) pic {bvert} (10,3) pic {bvert} (6,3/2) pic {wvert} (10,3/2) pic {wvert} (8,0) pic {dvert};
\foreach \x in {1,2,4,6,8,9}{
    \node [above] at (6+4*\x/10,3) {$6$};
    }
\node [above] at (6,3) {$4$};
\node [above] at (10,3) {$4$};
\foreach \x in {3,5,7}{
    \node [above] at (6+4*\x/10,3) {$5$};
    }
\node [below] at (6,3/2) {$5$};
\node [below] at (7,3/2) {$4$};
\node [below right] at (8,3/2) {$x$};
\node [below] at (9,3/2) {$4$};
\node [below] at (10,3/2) {$5$};
\node [right] at (8,0) {$6$};
\end{tikzpicture}
\]
\caption[Effect of interpolating a vertex between adjacent faces of $F'_{j,n}$ in the Growth Comparison Theorem]{\label{fig:46645-insert}In the diagram to the left, the $(4,6)$-edge at the bottom must have a $6$-valent vertex interpolated, along with the attendant subgraph. However, we wish to avoid non-concentricity; hence the single $4$-valent vertex $x$ is expanded to a $(4,6,4)$-path as in the diagram on the right.}
\end{figure}
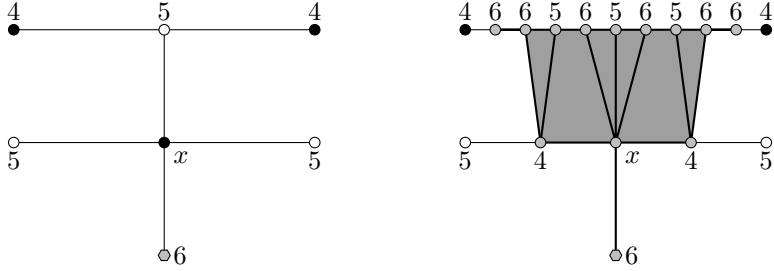

We continue by induction; suppose a tessellation $T'_j$ has been created by this process. Then in the $j^\text{th}$ corona, there are finitely many faces which require a finite number of vertices to have their valences increased and a finite number of edges along which we must interpolate a vertex. This creates $T'_{j+1}$ such that 
\[\abs{F_{1,n}}\leq \abs{F'_{j+1,n}} = \abs{F_{2,n}}\] for $n<j+1$, as the first $j$ coronas are comprised only of faces with valence sequence $\sig_2$. Furthermore, 
\[\abs{F_{1,n}}\leq \abs{F'_{j+1,n}}\] for all $n\in \N$. In this manner we can construct an infinite sequence of tessellations, namely $\{T'_j:j\in\N\}$, with the properties that $\abs{F_1,n}\leq \abs{F'_{j,n}}$ for any $j,n\in\N_0$, and $\abs{F'_{j,n}}=\abs{F_{2,n}}$ whenever $j>n$.
\end{exmp}

In the previous example, we constructed the sequence in the process of transforming $T_1$ with valence sequence $\vs{4,5,4,5}$ into $T_2$ with valence sequence $\vs{4,6,6,4,5}$; however, the process of creating $\{T'_j:j\in\N_0\}$ is identical in any case where $T_1$ and $T_2$ are face-homogeneous and uniformly concentric with monomorphic valence sequences $\sig_1$ and $\sig_2$, respectively, where $\sig_1<\sig_2$. Thus by \Cref{thm:comp}, we obtain the following result.

\begin{thm}[Growth Comparison Theorem]\label{thm:GrowthOrder} Let $\sig_1$ and $\sig_2$ be monomorphic valence sequences realized by tessellations $T_1,T_2\in\scr{G}_{4,4}\cup\scr{G}_{3+,5}\cup\scr{G}_{3,6}$, with $\sig_1<\sig_2$. Then $\gam(T_1)\leq \gam(T_2)$.\qquad\endproof
\end{thm}

Our convention is to index the face types ($\ft_1,\ldots,\ft_r$ for some $r$) in the following order: first wedges, then bricks, then notched bricks, and finally, other face types if any. A wedge in $F_n$ with face type $\ft_i$ is incident with a $p_{i-1}$-valent vertex in $U_{n-1}$, for $i=1,\ldots,k$. Similarly, the indices of face types of bricks begin with a brick in $F_n$ incident with a $p_0$-valent vertex and a $p_{k-1}$-valent vertex in $U_{n-1}$. A new index $\ft_j$ is not introduced if there is some $\ft_i$ for $i<j$ with the same configuration of vertices in $U_{n-1}$ and $U_n$, up to orientation. For example, the valence sequence $\vs{4,6,8,8,6,4}$ yields seven face types $\ft_1,\ldots,\ft_7$, of which $\ft_1$, $\ft_2$, and $\ft_3$ are wedge types and $\ft_4$ through $\ft_7$ are brick types.

When a monomorphic sequence $\vs{p_0,\ldots,p_{k-1}}$ is realized by a tessellation in $\bigmonouc$, then every face, with respect to \emph{any} Bilinski diagram, can be only a wedge, a brick, or a notched brick.  The indexing of face types when ${p_i\neq p_j}$ for $i\neq j$ allows a stricter labeling which we can use in several other cases. A face $f$ in $F_n$ is a wedge of type $\wft{i}$ when the vertex incident with $f$ in $U_{n-1}$ corresponds to valence $p_{i-1}$ in $\sig$. If instead $f$ is a brick with incident vertices in $U_{n-1}$ corresponding to valences $p_{i-1}$ and $p_{i-2}$ (indices here taken modulo $k$), then $f$ has face type $\bft{i}$. Finally, if $f$ a notched brick whose incident vertices in $U_{n-1}$ have valences $p_i$, $p_{i-1}=3$, and $p_{i-2}$, then $f$ has face type $\nft{i}$. It is important to note that if $p_{i-1}\neq 3$, then faces of type $\nft{i}$ never occur as offspring. This stricter labeling is used explicitly only for the few theorems which follow, by which we determine the number of offspring of each instance of these general face types. Furthermore, we demonstrate a first application of the accretion rules and half-counting of faces that were introduced in \Cref{sec:Accretion}. 

\textbf{Notation:} Let $T$ be a face-homogeneous tessellation with valence sequence $\sigma$, labeled as a Bilinski diagram. We denote by $\Omega(\ft)$  the number of faces in $F_{n+1}$ that are counted as offspring of a single face of face type $\ft$ in $F_n$, for any $n>0$. For ${T\in\bigmonouc}$ we let $\Omega(\wft{i})$, $\Omega(\bft{i})$, and $\Omega(\nft{i})$ denote the number of offspring of a single wedge, brick, or notched brick of , respectively, of the given type.

\begin{lem}\label{thm:NumOffs} For a face-homogeneous tessellation in  $\bigmonouc$ with monomorphic valence sequence $\sig=\vs{p_0,\ldots,p_{k-1}}$, one has for $i\in\{1,\ldots,k\}$,
\begin{align}
    \label{eqn:WedgeNumOff}
    \Omega(\wft{i}) &= \frac{p_{i-2}+p_i}2 - 2k+3+\sum_{j\notin I_1}p_j,\textrm{ and} \\
    \label{eqn:BrickNumOff}
    \Omega(\bft{i}) &= \frac{p_{i-3}+p_i}2 - 2k+5 + \sum_{j\notin I_2}p_j,
\end{align}
where $I_1 = \{i-2,i-1,i\}$ and $I_2=\{i-3,i-2,i-1,i\}$. Also, when $p_{i-1}=3$,
\begin{align}
    \label{eqn:N-BrickNumOff}
    \Omega(\nft{i}) &= \frac{p_{i-3}+p_{i+1}}2-2k+7+\sum_{j\notin I_3}p_j
\end{align}
with $I_3 = \{i-3,i-2,i-1,i,i+1\}$.
\end{lem}

\begin{proof}
The reader is referred to the three offspring diagrams shown in \Cref{fig:UAUC-offspring}.

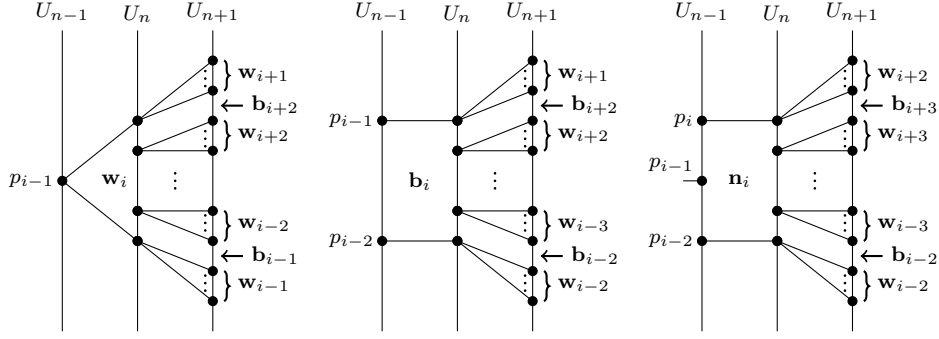
\begin{figure}[htbp]
{\footnotesize
\newcommand{\cs}{1.00}
\newcommand{\bs}{4.25}
\begin{tikzpicture}
[
    bvert/.pic = {\draw [fill=black] (0,0) circle [radius=0.06];},
    wvert/.pic = {\draw [fill=white] (0,0) circle [radius=0.05];},
    vell/.pic = {\foreach \y in {-1,0,1} \draw [fill=black] (0,\y/10) circle [radius=.01];},
]
\foreach \x in {0,1,2}{
    \foreach \y in {0,1,2}{
        \draw (\bs*\x+\cs*\y,0) -- (\bs*\x+\cs*\y,4);
        }
    \foreach \a in {-1,1}{
        \foreach \b in {0,1}{
            \foreach \c in {0,1}{
                \draw (\bs*\x+\cs,{2+\a*(.4+.4*\b)}) -- (\bs*\x+2*\cs, {2+\a*.4*(1+2*\b+\c)});
                \draw (\bs*\x+\cs,{2+\a*(.4+.4*\b)}) pic {bvert} (\bs*\x+2*\cs, {2+\a*.4*(1+2*\b+\c)}) pic {bvert};
                }
            \draw (\bs*\x+1.9*\cs, {2+\a*.4*(1.4+2*\b)}) pic {vell};
            }
        }
    \node [above] at (\bs*\x,4) {$U_{n-1}$};
    \node [above] at (\bs*\x+\cs,4) {$U_{n}$};
    \node [above] at (\bs*\x+2*\cs,4) {$U_{n+1}$};
    \draw (\x*\bs+1.5*\cs,2) pic {vell};
    }
\foreach \x in {-.8,.8}{
    \draw (0,2) -- (\cs,2+\x)
        (\bs,2+\x) -- (\bs+\cs,2+\x)
        (2*\bs,2+\x) -- (2*\bs+\cs,2+\x);
    }
\draw (2*\bs-\cs/4,2) -- (2*\bs,2);
\draw (0,2) pic {bvert};
\draw (2*\bs,2) pic {bvert};
\foreach \a in {-.8,.8}{\foreach \b in {1,2}{
    \draw (\b*\bs,2+\a) pic {bvert};
    }}
\node at (.7*\cs,2) {$\wft{i}$};
\node at (\bs+.5*\cs,2) {$\bft{i}$};
\node at (2*\bs+.5*\cs,2) {$\nft{i}$};
\node [left] at (0,2) {$p_{i-1}$};
\node [left] at (\bs,2.8) {$p_{i-1}$};
\node [left] at (\bs,1.2) {$p_{i-2}$};
\node [left] at (2*\bs,2.8) {$p_i$};
\node [above left] at (2*\bs,2) {$p_{i-1}$};
\node [left] at (2*\bs,1.2) {$p_{i-2}$};
\draw [thick, decorate, decoration={brace, amplitude=2pt}] 
    (2*\cs+.15,3.6) -- (2*\cs+.15,3.2) node [midway, right=2pt] {$\wft{i+1}$};
\draw [thick, decorate, decoration={brace, amplitude=2pt}] 
    (2*\cs+.15,2.8) -- (2*\cs+.15,2.4) node [midway, right=2pt] {$\wft{i+2}$};
\draw [thick, decorate, decoration={brace, amplitude=2pt}] 
    (2*\cs+.15,1.6) -- (2*\cs+.15,1.2) node [midway, right=2pt] {$\wft{i-2}$};
\draw [thick, decorate, decoration={brace, amplitude=2pt}] 
    (2*\cs+.15,0.8) -- (2*\cs+.15,0.4) node [midway, right=2pt] {$\wft{i-1}$};
\node [minimum size=.05] (marked) at (2*\cs,3) {};
\draw [->, thick] (2*\cs+.4,3) node [right] {$\bft{i+2}$} -- (marked);
\node [minimum size=.05] (marked) at (2*\cs,1) {};
\draw [->, thick] (2*\cs+.4,1) node [right] {$\bft{i-1}$} -- (marked);
\draw [thick, decorate, decoration={brace, amplitude=2pt}] 
    (\bs+2*\cs+.15,3.6) -- (\bs+2*\cs+.15,3.2) node [midway, right=2pt] {$\wft{i+1}$};
\draw [thick, decorate, decoration={brace, amplitude=2pt}] 
    (\bs+2*\cs+.15,2.8) -- (\bs+2*\cs+.15,2.4) node [midway, right=2pt] {$\wft{i+2}$};
\draw [thick, decorate, decoration={brace, amplitude=2pt}] 
    (\bs+2*\cs+.15,1.6) -- (\bs+2*\cs+.15,1.2) node [midway, right=2pt] {$\wft{i-3}$};
\draw [thick, decorate, decoration={brace, amplitude=2pt}] 
    (\bs+2*\cs+.15,0.8) -- (\bs+2*\cs+.15,0.4) node [midway, right=2pt] {$\wft{i-2}$};
\node [minimum size=.05] (marked) at (\bs+2*\cs,3) {};
\draw [->, thick] (\bs+2*\cs+.4,3) node [right] {$\bft{i+2}$} -- (marked);
\node [minimum size=.05] (marked) at (\bs+2*\cs,1) {};
\draw [->, thick] (\bs+2*\cs+.4,1) node [right] {$\bft{i-2}$} -- (marked);
\draw [thick, decorate, decoration={brace, amplitude=2pt}] 
    (2*\bs+2*\cs+.15,3.6) -- (2*\bs+2*\cs+.15,3.2) node [midway, right=2pt] {$\wft{i+2}$};
\draw [thick, decorate, decoration={brace, amplitude=2pt}] 
    (2*\bs+2*\cs+.15,2.8) -- (2*\bs+2*\cs+.15,2.4) node [midway, right=2pt] {$\wft{i+3}$};
\draw [thick, decorate, decoration={brace, amplitude=2pt}] 
    (2*\bs+2*\cs+.15,1.6) -- (2*\bs+2*\cs+.15,1.2) node [midway, right=2pt] {$\wft{i-3}$};
\draw [thick, decorate, decoration={brace, amplitude=2pt}]  
    (2*\bs+2*\cs+.15,0.8) -- (2*\bs+2*\cs+.15,0.4) node [midway, right=2pt] {$\wft{i-2}$};
\node [minimum size=.05] (marked) at (2*\bs+2*\cs,3) {};
\draw [->, thick] (2*\bs+2*\cs+.4,3) node [right] {$\bft{i+3}$} -- (marked);
\node [minimum size=.05] (marked) at (2*\bs+2*\cs,1) {};
\draw [->, thick] (2*\bs+2*\cs+.4,1) node [right] {$\bft{i-2}$} -- (marked);
\end{tikzpicture}
}
\caption[Offspring diagrams for general faces.]{\label{fig:UAUC-offspring} Offspring diagrams for the three general face types (respectively wedges, bricks, and notched bricks) of a tessellation with monomorphic, uniformly concentric valence sequence $[p_0,\ldots,p_{k-1}]$.}
\end{figure}

Letting $i\in\{1,\ldots,k\}$, the first diagram applies when $p_{i-1}\geq4.$  If also $p_{i-2},p_i\geq4$ as  in the diagram, then we have
\begin{align*}
    \Omega(\wft{i}) &= \frac{p_{i-2}-4}2+\frac{p_i-4}2+k-2+\sum_{j\notin I_1}(p_j-3) \\
        &= \frac{p_{i-2}+p_i}2 - 2k+3+\sum_{j\notin I_1}p_j.
\end{align*}
If instead $p_{i-2}=3$, then the number of wedge offspring of $\wft{i}$ is 
\[\frac{p_i-4}2+\sum_{j\notin I_1}(p_j-3),\]
the number of brick offspring is $k-3$, and the number of notched brick offspring is $\frac12$. Thus
when $p_{i-2}=3$, 
\begin{align*}
\Omega(\wft{i}) &= \frac12 + \frac{p_i-4}2+k-3+\sum_{j\notin I_1}(p_j-3) \\
    &= -\frac12 + \frac{p_i-4}2 +k-2+\sum_{j\notin I_1}(p_j-3) \\
    &= \frac{p_{i-2}-4}2+\frac{p_i-4}2+k-2+\sum_{j\notin I_1}(p_j-3)
\end{align*}  as before; likewise when $p_i=3$. Analogous arguments hold for the offspring of bricks and notched bricks.    
\end{proof}

The process of establishing an accretion rule and accompanying transition matrices is considerably simplified for tessellations in $\mathscr{G}_{4,4}$ by virtue of the absence of notched bricks.   By applying the following lemma and \Cref{thm:Eigenvalue}, one can then compute the growth rate explicitly of any monomorphic valence sequence realizable in $\mathscr{G}_{4,4}$.  Recall that by \Cref{prop:UnifConc}, all such valence sequences are uniformly concentric.
 
\begin{lem}\label{lem:UAUC-matrix}\label{thm:UAUC-matrix} Let $[p_0,\ldots,p_{k-1}]$ be the  monomorphic valence sequence for a tessellation $T\in\mathscr{G}_{4,4}$. Then $T$ has an accretion rule which admits the block transition matrix
\[ M= \begin{bmatrix}A & B \\ C & D \end{bmatrix}, \]
with $A=(a_{i,j})$, $B=(b_{i,j})$, $C=(c_{i,j})$, and $D=(d_{i,j})$ given by

\begin{align}
a_{i,j} &= \begin{cases} 
    \begin{minipage}{2cm}$0$\end{minipage} & j-i = 0 \\ 
    \frac12(p_{i-1}-4) & j-i\in\{1,k-1\} \pmod{k} \\ 
    p_{i-1} - 3 & \text{otherwise},
\end{cases} \\
b_{i,j} &= \begin{cases}
    \begin{minipage}{2cm}$0$\end{minipage} & j-i \in\{0,1\} \pmod{k} \\
    \frac12(p_{i-1}-4) & j-i \in \{2,k-1\} \pmod{k} \\
    p_{i-1}-3 & \text{otherwise},
\end{cases}\\
c_{i,j} &= \begin{cases}
    \begin{minipage}{2cm}$0$\end{minipage} & j-i \in \{0,1\} \pmod{k}\\
    1 & \text{otherwise},
\end{cases}\\
d_{i,j} &= \begin{cases}
    \begin{minipage}{2cm}$0$\end{minipage} & j-i \in \{0,1,k-1\} \pmod{k} \\
    1 & \text{otherwise},
\end{cases}
\end{align} for $i,j\in\{1,\ldots,k\}$.
\begin{proof} Since all general face types are wedges or bricks, we need demonstrate only that the entries $a_{i,j}$ and $c_{i,j}$ correspond to numbers of offspring of the $k$ face types in wedge configurations and that the entries $b_{i,j}$ and $d_{i,j}$ correspond to numbers of offspring of the $k$ face types in brick configurations.

\begin{figure}[htbp]
\[
\newcommand{\cs}{1.5}
\newcommand{\bs}{4.25}
\begin{tikzpicture}
[
    bvert/.pic = {\draw [fill=black] (0,0) circle [radius=0.06];},
    wvert/.pic = {\draw [fill=white] (0,0) circle [radius=0.06];},
    vell/.pic = {\foreach \y in {-1,0,1} \draw [fill=black] (0,\y/10) circle [radius=.01];},
]
\foreach \x in {0,1,2}{
    \draw (\x*\cs,0) -- (\x*\cs,4);
    }
\foreach \a in {-1,1}{
    \foreach \b in {1,2}{
        \draw (\cs,2+\a*4/7) pic {bvert} -- (2*\cs, 2+\a*\b*2/7) pic {bvert};
        }
    \foreach \b in {3,4,5,6}{
        \draw (\cs, 2+8*\a/7) pic {bvert} -- (2*\cs,2+\a*\b*2/7) pic {bvert};
        }
    }
\draw (0,2) pic {bvert} -- (\cs, 22/7)
    (0,2) -- (\cs, 6/7);
\draw [thick, decorate, decoration={brace, amplitude=4pt}] 
    (2*\cs+.15, 26/7) -- (2*\cs+.15, 20/7) node [midway, right=2pt] {$\frac12(p_i-4)$ faces of type $\wft{i+1}$};
\draw [thick, decorate, decoration={brace, amplitude=2pt}] 
    (2*\cs+.15, 18/7) -- (2*\cs+.15, 16/7) node [midway, right=2pt] {$p_{i+1}-3$ faces of type $\wft{i+2}$};
\draw [thick, decorate, decoration={brace, amplitude=2pt}] 
    (2*\cs+.15, 12/7) -- (2*\cs+.15, 10/7) node [midway, right=2pt] {$p_{i-3}-3$ faces of type $\wft{i-2}$};
\draw [thick, decorate, decoration={brace, amplitude=4pt}] 
    (2*\cs+.15, 8/7) -- (2*\cs+.15, 2/7) node [midway, right=4pt] {$\frac12(p_{i-2}-4)$ faces of type $\wft{i-1}$};
\draw (\cs+.3,2) pic {vell} (2*\cs+.3,2) pic {vell};
\node [left] at (0,2) {$p_{i-1}$};
\node [below left] at (\cs,6/7) {$p_{i-2}$};
\node [above left] at (\cs,22/7) {$p_i$};
\node at (2*\cs/4,2) {$\wft{i}$};
\node [above] at (0,4) {$U_{n-1}$};
\node [above] at (\cs,4) {$U_{n}$};
\node [above] at (2*\cs,4) {$U_{n+1}$};
\end{tikzpicture}
\]
\caption[Offspring of a wedge in $T\in\scr{G}_{4,4}$.]{\label{fig:UAUC-g44matrix-wedge}Offspring of a $\wft{i}$ face in a tessellation ${T\in\scr{G}_{4,4}}$, where $i\in\{1,\ldots,k\}$.}
\end{figure}

The offspring of wedges of type $\wft{i}$ are shown in \Cref{fig:UAUC-g44matrix-wedge}, and the offspring of a brick of type $\bft{i}$ is shown in \Cref{fig:UAUC-g44matrix-brick}.
\begin{figure}[htbp]
\[
\newcommand{\cs}{1.5}
\newcommand{\bs}{4.25}
\begin{tikzpicture}
[
    bvert/.pic = {\draw [fill=black] (0,0) circle [radius=0.06];},
    wvert/.pic = {\draw [fill=white] (0,0) circle [radius=0.06];},
    vell/.pic = {\foreach \y in {-1,0,1} \draw [fill=black] (0,\y/10) circle [radius=.01];},
]
\foreach \x in {0,1,2}{
    \draw (\x*\cs,0) -- (\x*\cs,4);
    }
\foreach \a in {-1,1}{
    \draw (0, 2+\a*8/7) pic {bvert} -- (\cs, 2+\a*8/7);
    \foreach \b in {1,2}{
        \draw (\cs,2+\a*4/7) pic {bvert} -- (2*\cs, 2+\a*\b*2/7) pic {bvert};
        }
    \foreach \b in {3,4,5,6}{
        \draw (\cs, 2+8*\a/7) pic {bvert} -- (2*\cs,2+\a*\b*2/7) pic {bvert};
        }
    }
\draw [thick, decorate, decoration={brace, amplitude=4pt}] 
    (2*\cs+.15, 26/7) -- (2*\cs+.15, 20/7) node [midway, right=2pt] {$\frac12(p_i-4)$ faces of type $\wft{i+1}$};
\draw [thick, decorate, decoration={brace, amplitude=2pt}] 
    (2*\cs+.15, 18/7) -- (2*\cs+.15, 16/7) node [midway, right=2pt] {$p_{i+1}-3$ faces of type $\wft{i+2}$};
\draw [thick, decorate, decoration={brace, amplitude=2pt}] 
    (2*\cs+.15, 12/7) -- (2*\cs+.15, 10/7) node [midway, right=2pt] {$p_{i-4}-3$ faces of type $\wft{i-3}$};
\draw [thick, decorate, decoration={brace, amplitude=4pt}] 
    (2*\cs+.15, 8/7) -- (2*\cs+.15, 2/7) node [midway, right=4pt] {$\frac12(p_{i-3}-4)$ faces of type $\wft{i-2}$};
\draw (\cs+.3,2) pic {vell} (2*\cs+.3,2) pic {vell};
\node [left] at (0, 22/7) {$p_{i-1}$};
\node [left] at (0, 6/7) {$p_{i-2}$};
\node [below left] at (\cs,6/7) {$p_{i-3}$};
\node [above left] at (\cs,22/7) {$p_i$};
\node at (2*\cs/4,2) {$\bft{i}$};
\node [above] at (0,4) {$U_{n-1}$};
\node [above] at (\cs,4) {$U_{n}$};
\node [above] at (2*\cs,4) {$U_{n+1}$};
\end{tikzpicture}
\]
\caption[Offspring of a brick in $T\in\scr{G}_{4,4}$.]{\label{fig:UAUC-g44matrix-brick}Offspring of a $\bft{i}$ face in a tessellation ${T\in\scr{G}_{4,4}}$, where $i\in\{1,\ldots,k\}$.}
\end{figure}
The ordering of face types is {$\wft{1}$,\,$\wft{2}$,\,$\ldots$,\,$\wft{k}$,\,$\bft{1}$,\,$\ldots$,\,$\bft{k}$}.
Recalling that the $(i,j)$-entry of a transition matrix $M$ is the number of faces of the $i^\text{th}$ indexed type which are produced in $F_{n+1}$ as offspring of a face of the $j^\text{th}$ indexed type in $F_n$, it is straightforward to verify from these two offspring diagrams that the entries of $M$ are correct.
\qquad\end{proof}
\end{lem}

\begin{rem}\label{rem:MonoUCsumup}
We emphasize the breadth of this class of monomorphic, uniformly concentric valence sequences.  In addition to the many monomorphic face-homogeneous  tessellations in $\mathscr{G}_{3,6}\cup\mathscr{G}_{3+,5}\cup\mathscr{G}_{4,4}$, there are many with covalence 3 (cf. \Cref{coval3}).
By \Cref{thm:UniqueEdgeSymbol}, all edge-transitive tessellations of constant covalence are included, except for those of the with valence sequence $[3,p,3,p]$ (edge-symbol $\langle 3,p;4,4\rangle$), as they are not uniformly concentric.  By \Cref{prop:UnifConc}, a $k$-covalent tessellation $T$ is uniformly concentric whenever $k\geq6$.  If $k\geq7$ and if $\sigma$ is monomorphic, then $\sigma\geq[3,3,3,3,3,3,3]$.  In that case, \Cref{thm:GrowthOrder} and \Cref{thm:EdgeGrowth} tell us that $\sigma$ has growth rate at least $\gamma([3,3,3,3,3,3,3])=\frac12(3+\sqrt{5})$.
\end{rem}
\medskip

\subsection{Monomorphic Non-Concentric Sequences}\label{sec:MonoNCVS}

The purpose of this  section is to characterize the six forms of monomorphic, non-concentric valence sequences with positive angle excess.  These sequences give rise to face types other than wedges, bricks, and notched bricks, and so the foregoing methods cannot be applied to compute their growth rates. 

An interesting situation arises when a tessellation is not \emph{uniformly} concentric but nonetheless, by prudent selection of the root, admits \emph{some} Bilinski diagram that \emph{is} concentric.  To illustrate this point, we examine sequences of the form $\vs{4,p,q}$.

\begin{exmp}\label{example:GrowthOf-4pq}
Consider the valence sequence $\sigma=[4,p,q]$ with $4<p<q$, where $\frac1p+\frac1q<\frac14$, and let $T$ be a face-homogeneous tessellation with valence sequence $\sigma$.  For $\sigma$ to be realizable, clearly $p$ and $q$ must be even.  Note as well that the inequality \eqref{eq:HyperbolicCondition} is satisfied.  While $\sigma$ is monomorphic and admits a concentric Bilinski diagram, $\sigma$ is not uniformly concentric
(cf. the second case of \Cref{prop:UnifConcForbidden}).

When a Bilinski diagram of $T$ admits a $4$-valent vertex $v_0\in U_n$ (for some $n$) adjacent to the vertices $u_1,u_2\in U_{n-1}$ and $v_1,v_2\in U_n$, then the diagram is not concentric; the vertices $v_1$ and $v_2$ must also be adjacent, as $T$ is $3$-covalent. Hence $\langle\{v_0,v_1,v_2\}\rangle$ is a cycle within $\langle U_n\rangle$, causing the Bilinski diagram to be non-concentric.  However, it is possible to avoid this configuration by choosing the root of the Bilinski diagram to be either a $p$-valent or a $q$-valent vertex. When so labeled, only four face types occur, as demonstrated by the offspring diagrams in \Cref{fig:4pq-offspring}.

\begin{figure}[htbp]

\[\newcommand{\cs}{1.5}\newcommand{\bs}{4.5}\newcommand{\hyt}{3}
\begin{tikzpicture}
[
    bvert/.pic = {\filldraw [draw=black, fill=black] (0,0) circle [radius=0.06];},
    wvert/.pic = {\filldraw [draw=black, fill=white] (0,0) circle [radius=0.06];},
    sqvert/.pic = {\newcommand{\s}{.09}\draw [fill=white] (\s,\s) -- (-\s,\s) -- (-\s,-\s) -- (\s,-\s) -- (\s,\s);},
    vell/.pic = {\foreach \y in {-1,0,1} \draw [fill=black] (0,\y/10) circle [radius=.01];},
]
\foreach \a in {0,1}{
    \foreach \b in {0,1,2}{
        \draw (\a*\bs+\b*\cs,0) -- (\a*\bs+\b*\cs,\hyt);
        }
    \node [above] at (\a*\bs,\hyt) {$U_{n-1}$};
    \node [above] at (\a*\bs+\cs,\hyt) {$U_n$};
    \node [above] at (\a*\bs+2*\cs,\hyt) {$U_{n+1}$};    
    }
\draw (\cs,2*\hyt/3) -- (0,\hyt/3) -- (\cs,\hyt/16) -- (2*\cs,\hyt/16);
\foreach \a in {2,9,16,30}{
    \draw (\cs,2*\hyt/3) -- (2*\cs,\a/32*\hyt);
    }    
\draw (\bs+\cs,\hyt-2*\hyt/3) -- (\bs+0,\hyt-\hyt/3) -- (\bs+\cs,\hyt-\hyt/16) -- (\bs+2*\cs,\hyt-\hyt/16);
\foreach \a in {2,9,16,30}{
    \draw (\bs+\cs,\hyt-2*\hyt/3) -- (\bs+2*\cs,\hyt-\a/32*\hyt);
    }

\draw (0,\hyt/3) pic {wvert} (\cs,\hyt/16) pic {bvert} (\cs,2*\hyt/3) pic {sqvert} (2*\cs,\hyt/16) pic {wvert} 
    (2*\cs,\hyt*9/32) pic {bvert} (2*\cs,\hyt/2) pic {wvert} (2*\cs,15/16*\hyt) pic {wvert};
\draw (2*\cs-.3, 23*\hyt/32) pic {vell};
\node at (.7*\cs,\hyt/3) {$\ft_{1}$};
\node at (1.3*\cs,19*\hyt/96) {$\ft_3$};
\draw [thick, decorate, decoration={brace, amplitude=4pt}]
    (.15 + 2*\cs,15*\hyt/16) -- (.15 + 2*\cs,\hyt/16) node [midway,right=3pt] {$\ft_2$};

\draw (\bs,2*\hyt/3) pic {sqvert} (\bs+\cs,\hyt/3) pic {wvert} (\bs+\cs,15*\hyt/16) pic {bvert}
    (\bs+2*\cs,\hyt/16) pic {sqvert} (\bs+2*\cs,\hyt/2) pic {sqvert} (\bs+2*\cs,23*\hyt/32) pic {bvert}
    (\bs+2*\cs,15*\hyt/16) pic {sqvert};
\draw (\bs + 2*\cs-.3, 9*\hyt/32) pic {vell};
\node at (\bs + .7*\cs,2*\hyt/3) {$\ft_{2}$};
\node at (\bs + 1.3*\cs,77*\hyt/96) {$\ft_4$};
\draw [thick, decorate, decoration={brace, amplitude=4pt}]
    (\bs + .15 + 2*\cs,15*\hyt/16) -- (\bs + .15 + 2*\cs,\hyt/16) node [midway,right=3pt] {$\ft_1$};

\draw (2*\bs,\hyt/2+.5) pic {bvert}; \node [right] at (2*\bs,\hyt/2+.5) {: $4$-valent};
\draw (2*\bs,\hyt/2) pic {wvert}; \node [right] at (2*\bs,\hyt/2) {: $p$-valent};
\draw (2*\bs,\hyt/2-.5) pic {sqvert}; \node [right] at (2*\bs,\hyt/2-.5) {: $q$-valent};
\end{tikzpicture}
\]
\caption{\label{fig:4pq-offspring}Offspring diagrams for a concentric tessellation with valence sequence $\vs{4,p,q}$.}
\end{figure}
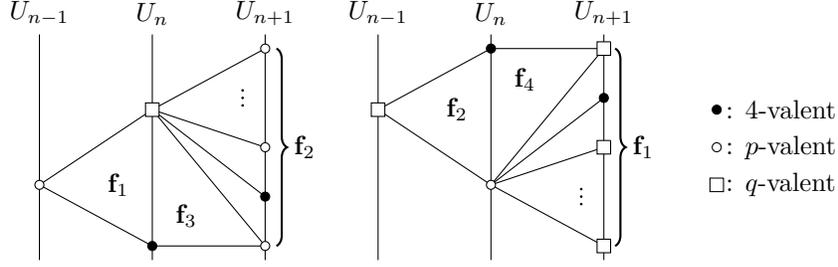

One sees here that if the root is taken to be a $p$-valent vertex, the first corona consists entirely of faces of type $\ft_1$, which produce in turn only offspring of types $\ft_2$ and $\ft_3$.  Similarly, given a $q$-valent root, the first corona consists entirely of faces of type $\ft_2$, which produce in turn only offspring of types $\ft_1$ and $\ft_4$.  The non-concentric configuration described above can never be produced among the descendants of faces of types $\ft_1$ or $\ft_2$.

Inspection of \Cref{fig:4pq-offspring} gives the first and second columns of the transition matrix $M$; the third and fourth columns, corresponding to $\ft_3$ and $\ft_4$, merit further explanation. A face of type $\ft_3$  in $F_{n+1}$ has a $p$-valent vertex in $U_{n+1}$; this vertex is incident with $p-5$ faces of type $\ft_1$ in $F_{n+2}$. So the behavior of a face of type $\ft_3$ is effectively to collapse one of the faces in $U_{n+2}$ of type $\ft_1$ begotten by the adjacent face of type $\ft_2$. Faces of type $\ft_4$ behave similarly, collapsing a face of type $\ft_2$ . These considerations give us 
\[ M = \begin{bmatrix}
0 & \frac12(p-4) & -1 & 0 \\
\frac12(q-4) & 0 & 0 & -1 \\
1 & 0 & 0 & 0 \\
0 & 1 & 0 & 0
\end{bmatrix} \]
as the transition matrix $M$ for this accretion rule for $T$.  As the characteristic equation for $M$ is of degree 4, it can be solved to determine that the maximum modulus of an eigenvalue of $M$ is
\[ \Lambda = \frac14\sqrt{[2(p-4)(q-4)-16]+2\sqrt{(p-4)^2(q-4)^2-16(p-4)(q-4)}}. \]
By \Cref{thm:Eigenvalue} and \Cref{thm:InvariantGrowth}, $\Lambda$ is the growth rate of $T$.  This quantity can be minimized by minimizing $pq$ subject to the initial conditions $\frac1p+\frac1q<\frac14$ and that $p$ and $q$ be even.  We shall see in \Cref{sec:MainResult} the role played by this example.
\end{exmp}

 Growth rate formulas for each of the other five classes are derived in the Appendix.
\begin{thm}\label{thm:UC-classify} Let $\sigma$ be a valence sequence such that  $\eta(\sigma)>0$.  Then $\sigma$ is both monomorphic and non-concentric if and only if $\sigma$ is of one of the following six forms:
\begin{itemize}
\item $\vs{3,p,p}$, with $p\geq14$ and even;
\item $\vs{4,p,q}$, with $p$ and $q$ both even, $4< p<q$, and $\frac1p+\frac1q<\frac14$;
\item $\vs{3,p,3,p}$, with $p\geq 7$;
\item $\vs{3,p,4,p}$, with $p\geq 5$ and even;
\item $\vs{3,3,p,3,p}$, with $p\geq 5$; or
\item $\vs{3,3,p,3,q}$, with $p,q\geq 4$ and $\frac1p+\frac1q<\frac12$.
\end{itemize}
\end{thm}
\begin{proof}
The parity conditions and the inequalities bounding the parameters in each case are minimal such that $\sig$ be indeed realizable as a tessellation with ${\eta(\sig)>0}$.

As noted in \Cref{rem:MonoUCsumup}, all valence sequences of length at least 6 are uniformly concentric. Furthermore, by \Cref{coval3}, all valence sequences of length 3 are monomorphic. Valence sequences $\vs{3,p,p}$, $\vs{4,p,q}$, and $\vs{3,p,3,p}$ give rise to tessellations exemplifying cases 1, 2, and 4 respectively of \Cref{prop:UnifConcForbidden}, and hence cannot be uniformly concentric. As a face-homogeneous tessellation with valence sequence $\vs{3,p,3,p}$ is also edge-transitive, the sequence must be monomorphic. The proof that the sequence $\vs{3,p,4,p}$ is be monomorphic and uniformly non-concentric is given in the Appendix, where the growth rate of a corresponding tessellation is determined.  

We now prove that $[3,3,p,3,p]$ is monomorphic for all $p\geq5$.  As a $3$-valent vertex is incident with a common face with any two of its neighbors, every $3$-valent vertex must be adjacent to at least two $p$-valent vertices; otherwise some face would be incident with a $(3,3,3)$-path. Consider a $p$-valent vertex $v_1$. By face-homogeneity, $v_1$ is adjacent to some $3$-valent vertex $u_1$, with $u_1$ adjacent in turn to a $3$-valent vertex $u_2$ which is not adjacent to $v_0$. But then the other vertex adjacent to $u_1$ must be a $p$-valent vertex $v_2$. This forces the pattern of valences at regional distance 1 from $v_1$ to be $(3,3,p,\ldots,3,3,p)$; as $v_1$ was arbitrary, this must be the pattern of valences at regional distance 1 from any $p$-valent vertex. As every vertex is at regional distance 1 from some $p$-valent vertex, $\vs{3,3,p,3,p}$ must be monomorphic; the first two coronas of a tessellation with this valence sequence rooted at a $p$-valent vertex is depicted in \Cref{fig:33p3p-firstCorona}. Furthermore, this local configuration to a $p$-valent vertex forces the local behaviors to a $(3,3)$-edge and a $3$-valent vertex shown in \Cref{fig:33p3p}. Hence when a $3$-valent vertex $v_0$ is taken as the root of the Bilinski diagram of a tessellation with valence sequence $\vs{3,3,p,3,p}$, a pendant vertex occurs in $\induced{U_3}$. This is shown in \Cref{fig:33p3p-nonconc}. So $\vs{3,3,p,3,p}$ is monomorphic but not uniformly concentric; the argument for $\vs{3,3,p,3,q}$ is analogous.

We have shown these six forms to be both monomorphic and non-concentric; that these are the only such valence sequences is proved via the exhaustive examination of cases in the Appendix.
\qquad\end{proof}

\begin{figure}[htbp]
\[
\begin{tikzpicture}
[
    bvert/.pic = {\draw [fill=black] (0,0) circle [radius=0.05];},
    wvert/.pic = {\draw [fill=white] (0,0) circle [radius=0.05];}
]
\draw (-60:1) arc [radius=1, start angle=-60, end angle=240];
\draw (-60:2) arc [radius=2, start angle=-60, end angle=240];
\foreach \a in {1,2,...,6}{
    \draw (0,0) -- ({-60*(1+\a)}:1);
    }
\foreach \a in {-1,0,1,2,3}{
    \draw ({60*\a+20}:1) -- ({60*\a+12}:2);
    \draw [fill=lightgray] ({60*\a+40}:1) -- 
        ({60*\a+36}:2) arc [radius=2, start angle={60*\a+36}, end angle={60*(1+\a)}]
        ({60*(1+\a)}:2) -- ({60*\a+40}:1);
    }
\draw (240:2) arc [radius=2, start angle=240, end angle=246];
\draw (-66:2) arc [radius=2, start angle=-66, end angle=-60];
\draw (0,0) pic {wvert};
\foreach \a in {-1,0,1,2,3}{
    \draw ({60*\a}:1) pic {bvert}
        ({60*\a+20}:1) pic {bvert}
        ({60*\a+40}:1) pic {wvert};
    }
\draw (240:1) pic {bvert};
\draw (-60:2) pic {bvert};
\foreach \a in {-1,0,1,2,3}{
    \draw ({60*\a+12}:2) pic {wvert};
    \foreach \b in {2,3,5}{
        \draw ({60*\a+\b*12}:2) pic {bvert};
        }
}
\node at (0,-1) {$\ldots$};
\node at (0,-2) {$\ldots$};
\draw (3,.25) pic {bvert}; \node [right] at (3.1,.25) {: $3$-valent};
\draw (3,-.25) pic {wvert}; \node [right] at (3.1,-.25) {: $p$-valent};
\end{tikzpicture}
\]
\caption[First two coronas of {$[3,3,p,3,p]$}.]{\label{fig:33p3p-firstCorona}  The first two coronas of a tessellation with valence sequence $[3,3,p,3,p]$ rooted at a $p$-valent vertex. Each shaded region indicates $p-3$ faces in $F_2$ all having the same face type.}
\end{figure}

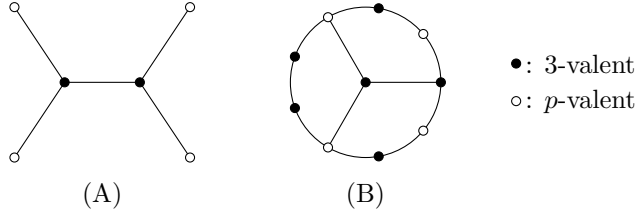
\begin{figure}[htbp]
\[
\begin{tikzpicture}
[
    scale = 1,
    bvert/.pic = {\draw [fill=black] (0,0) circle [radius=0.06];},
    wvert/.pic = {\draw [fill=white] (0,0) circle [radius=0.06];},
    afig/.pic = {
        \draw (0,0) -- (2/3,1) -- (0,2) (2/3,1) -- (5/3,1) (7/3,0) -- (5/3,1) -- (7/3,2);
        \draw (0,0) pic {wvert} (0,2) pic {wvert} (7/3,0) pic {wvert} (7/3,2) pic {wvert}
            (2/3,1) pic {bvert} (5/3,1) pic {bvert};},
    bfig/.pic = {
        \draw (0,0) circle [radius=1] pic {bvert};
        \foreach \a in {0,120,240}{\draw (0,0) -- (\a:1);}
        \foreach \a in {0,80,160,200,280}{\draw (\a:1) pic {bvert};}
        \foreach \a in {40,120,240,320}{\draw (\a:1) pic {wvert};}
        }
]
\node at (0,0) {(A)};
\draw (-7/6,.5) pic {afig};
\node at (3.5,0) {(B)};
\draw (3.5,1.5) pic {bfig};
\draw (5.5,1.75) pic {bvert}; \node [right] at (5.5,1.75) {: $3$-valent};
\draw (5.5,1.25) pic {wvert}; \node [right] at (5.5,1.25) {: $p$-valent};
\end{tikzpicture}\]
\caption[Two local configurations in $\vs{33p3p}$.]{\label{fig:33p3p}(A) Local configuration along an edge with edge-symbol $\esym{3,3}{5,5}$ in a face-homogeneous tessellation with valence sequence $\vs{3,3,p,3,p}$. (B) Local configuration in the same tessellation when rooted at a 3-valent vertex.}
\end{figure}

\begin{figure}[htbp]
\[\begin{tikzpicture}
[
    bvert/.pic = {\draw [fill=black] (0,0) circle [radius=0.06];},
    wvert/.pic = {\draw [fill=white] (0,0) circle [radius=0.06];},
]
\draw (0,0) circle [radius=1] pic {bvert};
\foreach \a in {0,120,240}{\draw (0,0) -- (\a:1);}
\foreach \a in {-1,1}{
    \draw (40*\a:1) -- (2,\a*.65) pic {bvert} 
        (2,0) -- (2,\a) -- (4,.9*\a) -- (4,0) 
        (2,\a) -- (3,0) 
        (2,\a) pic {wvert} -- (1.5,\a*1.25) 
        (4,.9*\a) pic {bvert} -- (3.5,\a*1.25)
        (4,\a*.5) pic {wvert};
    }
\draw (3,0) pic {bvert} -- (4,0) pic {bvert};
\foreach \a in {0,80,160,200,280}{\draw (\a:1) pic {bvert};}
\foreach \a in {40,120,240,320}{\draw (\a:1) pic {wvert};}
\draw (5,.25) pic {bvert}; \node [right] at (5,0.25) {: $3$-valent};
\draw (5,-.25) pic {wvert}; \node [right] at (5,-.25) {: $p$-valent};
\node [below right] at (0,0) {$v_0$};
\end{tikzpicture}\]
\caption[Non-concentricity of $\vs{3,3,p,3,p}$.]{\label{fig:33p3p-nonconc}Non-concentricity of $\vs{3,3,p,3,p}$ when rooted at a 3-valent vertex $v_0$.}
\end{figure}
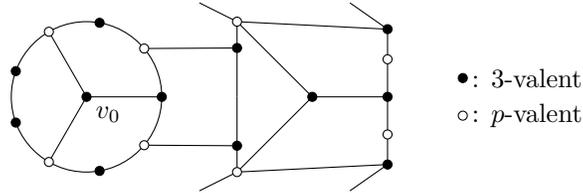

\medskip


\subsection{\label{sec:MainResult}The Main Result}
The following theorem establishes the so-called ``golden mean'' as the least rate of exponential growth for face-homogeneous tessellations with monomorphic valence sequences.

\begin{table}[htbp]
{\small
\[\begin{array}{c|cr||c|cr}
 \text{Class}          & \sigma & \gamma(T_\sigma)\approx & \text{Class} & \sigma & \gamma(T_\sigma)\approx \\ \hline
 \vs{p,p,p}            & \vs{7,7,7}            & 2.6180           & \vs{p,p,q,r,q}   & \vs{4,4,6,5,6}      &  6.6650 \\
 \vs{3,p,p}            & \vs{3,14,14}          & 2.6180           & \vs{3,p,q,q,p}   & \vs{3,4,6,6,4}      &  4.9911 \\
 \vs{p,p,q}            & \vs{6,6,7}            & 1.722            & \vs{p,q,r,s,t}   & \vs{4,6,10,12,8}    & 14.5753 \\
 \mathbf{\vs{4,p,q}}   & \mathbf{\vs{4,6,14}}  & \mathbf{1.6180}  & \vs{p,p,p,p,p,p} & \vs{4,4,4,4,4,4}    &  5.8284 \\  
 \vs{p,q,r}            & \vs{6,8,10}           & 3.4789           & \vs{p,p,q,p,p,q} & \vs{4,4,5,4,4,5}    &  7.1347 \\  
 \vs{p,p,p,p}          & \vs{5,5,5,5}          & 3.7320           & \vs{p,q,p,q,p,q} & \vs{4,5,4,5,4,5}    &  7.8729 \\  
 \vs{p,p,q,q}          & \vs{4,4,6,6}          & 3.4081           & \vs{p,q,q,p,r,r} & \vs{6,4,4,6,8,8}    & 13.1291 \\  
 \vs{3,p,3,p}          & \vs{3,7,3,7}          & 2.6180           & \vs{p,q,p,r,q,r} & \vs{4,5,4,6,5,6}    &  9.8115 \\  
 \vs{p,q,p,q}          & \vs{4,5,4,5}          & 2.6180           & \vs{p,q,r,p,q,r} & \vs{4,6,8,4,6,8}    & 13.5612 \\  
 \vs{3,p,4,p}          & \vs{3,6,4,6}          & 2.9655           & \vs{p,q,p,r,s,r} & \vs{4,5,4,6,7,6}    & 10.9033 \\  
 \mathbf{\vs{3,p,q,p}} & \mathbf{\vs{3,4,7,4}} & \mathbf{1.6180}  & \vs{p,q,r,p,s,t} & \vs{4,6,8,4,10,12}  & 18.1174 \\  
 \vs{p,q,p,r}          & \vs{4,5,4,6}          & 3.1462           & \vs{p,q,r,s,t,u} & \vs{4,6,10,14,12,8} & 23.9963 \\  
 \vs{p,q,r,s}          & \vs{4,6,10,8}         & 7.0367           & \vs{3,p,p,3,p,p} & \vs{3,4,4,3,4,4}    &  4.3306 \\  
 \vs{p,p,p,p,p}        & \vs{4,4,4,4,4}        & 3.7320           & \vs{3,p,3,p,3,p} & \vs{3,4,3,4,3,4}    &  3.7320 \\  
 \vs{3,3,3,3,p}        & \vs{3,3,3,3,7}        & 1.7553           & \vs{3,3,3,p,q,p} & \vs{3,3,3,4,5,4}    &  4.0265 \\  
 \vs{3,3,3,p,p}        & \vs{3,3,3,6,6}        & 3.0217           & \vs{3,p,q,3,q,p} & \vs{3,4,6,3,6,4}    &  6.8091 \\  
 \vs{3,3,p,3,p}        & \vs{3,3,5,3,5}        & 2.6180           & \vs{3,p,3,q,3,r} & \vs{3,4,3,5,3,6}    &  5.6723 \\  
 \vs{3,3,p,3,q}        & \vs{3,3,4,3,5}        & 1.9318           & \vs{3,p,q,r,q,p} & \vs{3,4,6,5,6,4}    &  8.0601 
\end{array}\]}
\caption[The least growth rate within each monomorphic class of valence sequences.]{\label{fig:LeastGrowthsTable}Table of the least exponential growth rate within each monomorphic class of valence sequences. All rates of growth have been truncated at four decimal places rather than being rounded.}
\end{table}
\begin{thm}[Least Exponential Growth Rate of Monomorphic Valence Sequences]\label{thm:LeastRateofGrowth} The least growth rate of a face-homogeneous tessellation with monomorphic valence sequence $\sigma$ such that $\eta(\sigma)>0$ is $\frac12(1+\sqrt{5})$ and is attained by exactly  the tessellations with valence sequences  $[4,6,14]$ and $[3,4,7,4]$.
\begin{proof} 
With respect to the partial order on valence sequences, if a valence sequence $\sigma$ has length at least $7$, then $[3,3,3,3,3,3,3]\leq \sigma$. A face-homogeneous tessellation $T_0$ with valence sequence $[3,3,3,3,3,3,3]$ is edge-homogeneous with edge-symbol $\langle3,3;7,7\rangle$ and so has growth rate $\gam(T_0)=\frac12(3+\sqrt{5})$ by \Cref{thm:EdgeGrowth}. But then if $\vs{3,3,3,3,3,3,3}<\sig$ and $T$ is a tessellation with monomorphic valence sequence $\sig$, then $\gam(T_0)\leq \gam(T)$, by \Cref{thm:GrowthOrder}. We proceed then by exhaustion: there are only finitely many forms of valence sequences of length at most 6. The Appendix contains an exhaustive classification of realizable valence sequences as monomorphic or polymorphic. For each form of monomorphic valence sequence, the least rate of growth is either determined or bounded below. The minimum growth rate of a minimal representative of each form is listed in \Cref{fig:LeastGrowthsTable}. Of these forms, $\vs{4,6,14}$ and $\vs{3,4,7,4}$ have the least rate of growth, shown to be $\frac12(1+\sqrt{5})$ in the Appendix.
\end{proof}
\end{thm}

\begin{rem}
It is interesting to observe that the two tessellations realizing the minimum exponential  growth rate are closely related. The face-homogeneous tessellation with valence sequence $\vs{4,6,14}$ can be realized by the classical tiling of the hyperbolic plane by triangles with interior angles $\frac{\pi}2$, $\frac{\pi}3$, and $\frac{\pi}7$.  Moreover, a face-homogeneous tessellation with valence sequence $\vs{3,4,7,4}$ is the subgraph of one with valence sequence $\vs{4,6,14}$ obtained by the deletion of all edges joining $6$-valent and $14$-valent vertices.  Many artistic renderings of these tilings exist, and can be found on web sites regarding the $(2,3,7)$-triangle group, the Order-7 triangular tiling, or triangular tilings of the hyperbolic plane, including Wikipedia.
\end{rem}

\section{Polymorphic Valence Sequences}
\subsection{Polymorphic Valence Sequences}\label{sec:PolymorphicVS}

With respect to the ordering of cyclic sequences, the least polymorphic valence sequence with positive angle excess is $[4,4,4,5]$; that is to say, every cyclic sequence  $\sigma$ such that ${\sigma<[4,4,4,5]}$ is either not realizable as a tessellation, is realizable only by a finite map or a Euclidean tessellation, or is monomorphic. 
While all valence sequence of length 3 are monomorphic, $k$-covalent polymorphic valence sequences abound for $k\geq 4$. The following theorem gives a simple sufficient condition under which a realizable valence sequence is polymorphic.

\begin{figure}[hbtp]
\[\newcommand{\hyt}{3}\newcommand{\cs}{1.5}
\begin{tikzpicture}
[
    bvert/.style = {fill=black, radius=2pt},
    vell/.pic = {\foreach \y in {-1,0,1} \draw [fill=black] (0,\y/10) circle [radius=.01];},
]
\draw (0,0) -- (0,\hyt) node [above] {$U_n$};
\draw (\cs,0) -- (\cs,\hyt) node [above] {$U_{n+1}$};
\fill (0,\hyt/4) circle [bvert] node [left] {$u_0$};
\fill (0,\hyt/2) circle [bvert] node [left] {$u_1$};
\fill (0,3*\hyt/4) circle [bvert] node [left] {$u_2$};
\foreach \a in {1,2,3,4,5}{    \fill (\cs, \a*\hyt/6) circle [bvert]; }
\draw (0,\hyt/4) -- (\cs,\hyt/6)
    (0,\hyt/2) -- (\cs,\hyt/3) node [right] {$v_1$}
    (0,\hyt/2) -- (\cs,\hyt/2) node [right] {$v_2$}
    (0,\hyt/2) -- (\cs,2*\hyt/3) node [right] {$v_r$}
    (0,3*\hyt/4) -- (\cs,5*\hyt/6);
\draw (\cs-.15,7*\hyt/12) pic {vell};
\node [left] at (\cs,5*\hyt/12) {$w$};
\path (\cs/2,\hyt/3) node {$b$} (\cs/2,2*\hyt/3) node {$b'$};
\end{tikzpicture}\]
\caption{\label{fig:PolySuff}A configuration of faces demonstrating polymorphicity.}
\end{figure}
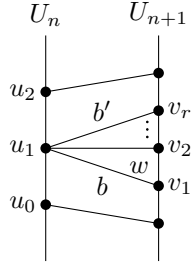
\begin{prop}\label{thm:AmbCond} Let $\sigma=[p_0,\ldots,p_{k-1}]$ be the valence sequence of a face-homogeneous tessellation $T\in\mathscr{G}_{4,4}\cup\mathscr{G}_{3+,5}$. If there exist distinct $i,j\in\{0,\ldots,k-1\}$ such that $p_i,p_{i+1}\geq 4$ and either
\begin{enumerate}
\item   $p_i=p_j$, $p_{i+1}=p_{j+1}$, and $p_{i+2}\neq p_{j+2}$, or
\item   $p_i=p_j$, $p_{i+1}=p_{j-1}$, and $p_{i+2}\neq p_{j-2}$,
\end{enumerate}
then $\sigma$ is polymorphic.
\end{prop}
\begin{proof}
As the only two forms of valence sequences of length $k=4$ that satisfy the hypothesis, namely $[p,p,p,q]$ and $[p,p,q,r]$, are polymorphic (see Appendix), we assume that $k\geq 5$. Also, since condition (2) is identical to (1) save for orientation within the cyclic sequence, it suffices to assume that there are distinct $i,j$ such that (1) holds. Furthermore, we may assume $i=0$ due to the rotational equivalence of valence sequences.

Since $k\geq5$, there exists for some $n$ a face in $F_n$ incident with three consecutive vertices $u_0,u_1,u_2\in U_n$ with valence $\rho(u_m)=p_m$ for $m=0,1,2$.  Let $b$ be the brick in $F_{n+1}$ incident with the edge $u_0u_1$, and let $b'$ be the brick (or perhaps notched brick if $p_2=3$) in $F_{n+1}$ incident with the edge $u_1u_2$.  Let $v_1,\ldots,v_r$ be the vertices in $U_{n+1}$ incident with $u_1$ in consecutive order, so that $v_1$ is incident with $b$ and $v_r$ is incident with $b'$. Thus $r=p_1-2\geq2$.  If $\sigma$ contains a subsequence $[q,p_1,p_2]$ with $q\neq p_0$, then $\rho(v_r)$ may equal either $p_0$ or $q$, resulting in a choice of face types for $b'$, and we're done.  Otherwise we must have $\rho(v_r)=p_0$, which forces the vertex $v_{r-2}$ and subsequent alternate neighbors of $u_1$ in $U_{n+1}$ also to be $p_0$-valent.

If $p_1$ is even, then $\rho(v_1)$ may equal either $p_2$ or $p_{j+2}$ in which case the wedge $w\in F_{n+1}$ incident with vertices $v_1,u_1,v_2$ may be of either type $\wft{2}$ or type $\wft{j+2}$, and $T$ is polymorphic. (See \Cref{fig:PolySuff}.)

If $p_1$ is odd, then working backward as in the even case forces $\rho(v_1)=p_0$, which implies that either $p_0=p_2$ or $p_0=p_{j+2}$, and without loss of generality, we assume the former.  Now we may assign $\rho(v_2)$ to be either $p_0$ or $p_{j+2}$, and the argument proceeds as in the even case.
\qquad\end{proof}

The existence of polymorphic valence sequences considerably complicates the computation of growth rates of face-homogeneous tessellations. The above proof suggests that, unlike in the monomorphic case, polymorphic valence sequences may admit many different accretion rules, as we illustrate in the next section.
\medskip


\subsection{\label{appx_a}Two non-isomorphic tessellations with the same valence sequence}
The minimal polymorphic valence sequence under the partial order on cyclic sequences, namely $\vs{4,4,4,5}$,  is unfortunately not amenable to study via our methods. In fact, there is no well-defined transition matrix between coronas, and this problem is shared by all valence sequences of the form $\vs{4,4,4,q}$ for $q>4$. However, $\vs{4,4,6,8}$ provides us with the opportunity to investigate two distinct (but related) accretion rules. 

The valence sequence $\vs{4,4,6,8}$ is representative of form $[p,p,q,r]$ discussed in the Appendix.  As every face is incident with a pair of adjacent 4-valent vertices, every realization of this valence sequence contains a countable infinity of pairwise-disjoint double rays, each induced exclusively by $4$-valent vertices.  \Cref{fig:4468-RayConfigs} (A) shows a strip-like patch bordering a double ray of $4$-valent vertices.  To obtain \Cref{fig:4468-RayConfigs} (B) from this (or vice versa), one can fix pointwise the half-plane on one side of the double ray while translating the half-plane on the other side along one edge of the double~ray.


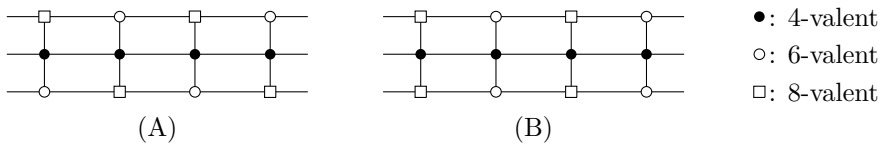
\begin{figure}[htbp]
{\centering\begin{tikzpicture}
[
    bvert/.pic={\fill [fill=black] (0,0) circle [radius=2pt];},
    wvert/.pic={\filldraw [draw=black, fill=white] (0,0) circle [radius=2pt];},
    sqvert/.pic={\filldraw [fill=white, draw=black] (-45:3pt) foreach \a in {45,135,225,315}{ -- (\a:3pt)};},
]
\draw foreach \a in {1,2,3,4}{ (\a,0) -- (\a,1/2) pic {bvert} -- (\a,1) } (0.5,1/2) -- (4.5,1/2);
\draw (0.5,1) foreach \a in {0,1}{-- (1+2*\a,1) pic {sqvert} -- (2+2*\a,1) pic {wvert}} -- (4.5,1);
\draw (0.5,0) foreach \a in {0,1}{-- (1+2*\a,0) pic {wvert} -- (2+2*\a,0) pic {sqvert}} -- (4.5,0);

\draw (5.5,1/2) foreach \a in {6,7,8,9}{ -- (\a,1/2) pic {bvert}} -- (9.5,1/2);
\draw foreach \a in {6,8}{
    (\a-.5,0) -- (\a+.5,0) (\a,0) pic {sqvert} -- (\a,1) pic {sqvert} (\a-.5,1) -- (\a+.5,1)
    (\a+.5,0) -- (\a+1.5,0) (\a+1,0) pic {wvert} --(\a+1,1) pic {wvert} (\a+.5,1) -- (\a+1.5,1)
};

\path (2.5,-.5) node {(A)} (7.5,-.5) node {(B)};
\draw (10.5,1) pic {bvert}; \node   [right] at (10.5,1) {: $4$-valent};
\draw (10.5,1/2) pic {wvert}; \node [right] at (10.5,1/2) {: $6$-valent};
\draw (10.5,0) pic {sqvert}; \node  [right] at (10.5,0) {: $8$-valent};
\end{tikzpicture}}
\caption{\label{fig:4468-RayConfigs}Two non-isomorphic patches of a tessellation with valence sequence $\vs{4,4,6,8}$, local to double rays of 4-valent vertices.}
\end{figure}

To construct  still other such (non-isomorphic) realizations, one can choose to ``translate'' along any one of these double rays by leaving fixed the half-plane on one side of the double ray but translating the half-plane on the other side one edge.  Since there exists a countable infinity of double rays along which one may choose to translate one or the other or neither of the adjacent half-planes, there exists an uncountable class of pairwise non-isomorphic tessellations that all have the same valence sequence $[4,4,6,8]$.

While one might expect that all tessellations having the same valence sequence always have the same growth rate, we show that this is not so.

We begin by observing that every $4$-valent vertex in a face-homogeneous tessellation with valence sequence $\vs{4,4,6,8}$ is adjacent to two other $4$-valent vertices and two vertices with valences $6$ or $8$; thus any given $4$-valent vertex either has exactly one $6$-valent and one $8$-valent neighbor, has two $6$-valent neighbors, or has two $8$-valent neighbors. Furthermore, every $4$-valent vertex lies on a double ray (two-way infinite path) of $4$-valent vertices; if one vertex along this path has a $6$-valent neighbor and an $8$-valent neighbor, then so does every other vertex along the double ray. This is the behavior demonstrated in \Cref{fig:4468-RayConfigs} (A).

If the local configuration specified in \Cref{fig:4468-RayConfigs} (A) is enforced along every double ray of $4$-valent vertices, then the tessellation obtained is unique; let this tessellation be $T_1$. We can then construct offspring diagrams for $T_1$ as given in \Cref{fig:4468_offdiag_1}. It is interesting to note that $T_1$ is the dual of the Cayley graph of the group with presentation 
\[G_1 = \left\langle a,b,c\mid a^2=b^2=c^2=(aba)^2=(bc)^3=(caba)^4=1\right\rangle.\]

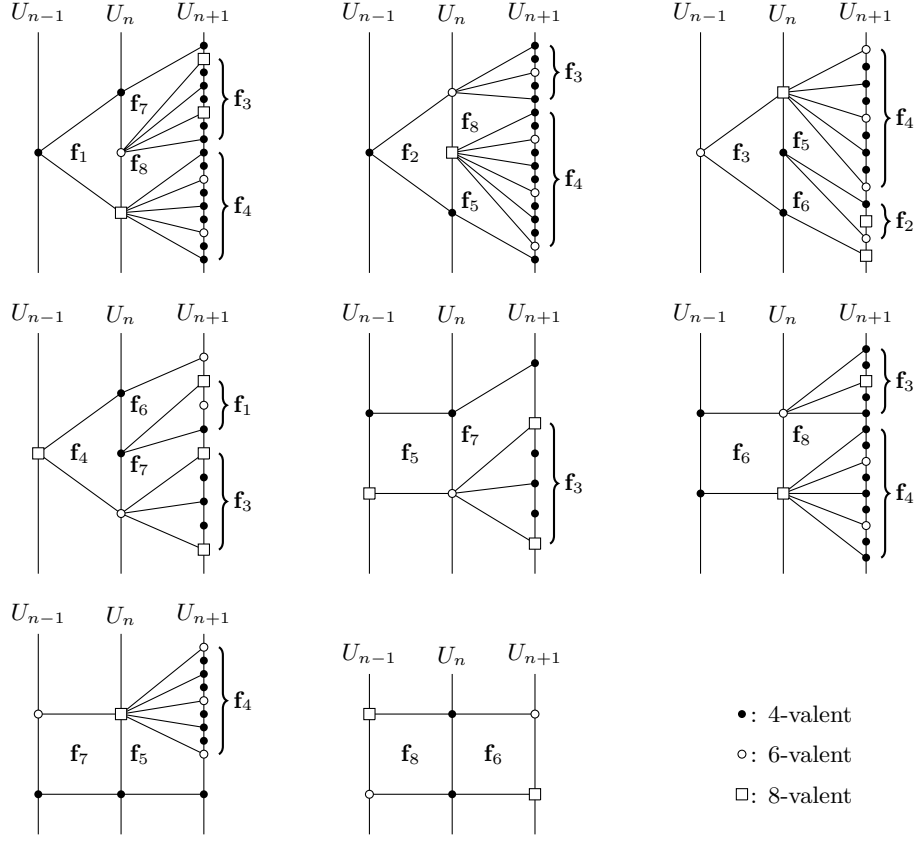
\begin{figure}[htbp]
{\centering\small 
\newcommand{\hy}{3.2}
\newcommand{\cs}{1.1}
\newcommand{\rw}{4}
\begin{tikzpicture}
[
    bvert/.pic={\fill [fill=black] (0,0) circle [radius=1.5pt];},
    wvert/.pic={\filldraw [draw=black, fill=white] (0,0) circle [radius=1.5pt];},
    sqvert/.pic={\filldraw [fill=white, draw=black] (-45:3pt) foreach \a in {45,135,225,315}{ -- (\a:3pt)};},
]
\foreach \a in {0,1,2}{
    \foreach \b in {1,2}{
        \draw (4*\a*\cs,\b*\rw) -- (4*\a*\cs,\b*\rw+\hy) node [above] {$U_{n-1}$};
        \draw ({(4*\a+1)*\cs},\b*\rw) -- ({(4*\a+1)*\cs},\b*\rw+\hy) node [above] {$U_n$};
        \draw ({(4*\a+2)*\cs},\b*\rw) -- ({(4*\a+2)*\cs},\b*\rw+\hy) node [above] {$U_{n+1}$};
        }
    }
\draw 
    foreach \a in {1/4,3/4}{(0,2*\rw+\hy/2) -- (\cs,2*\rw+\a*\hy)}
    foreach \a in {1,3,...,9}{(\cs,2*\rw+\hy/4) -- (2*\cs,2*\rw+\a*\hy/18)}
    foreach \a in {10,12,14,16}{(\cs,2*\rw+\hy/2) -- (2*\cs,2*\rw+\a*\hy/18)}
    (\cs,2*\rw+3*\hy/4) -- (2*\cs,2*\rw+17*\hy/18);
\path
    (0,2*\rw+\hy/2) pic {bvert} 
    (\cs,2*\rw+\hy/4) pic {sqvert}
    (\cs,2*\rw+\hy/2) pic {wvert}
    (\cs,2*\rw+3*\hy/4) pic {bvert}
    foreach \a in {1,2,4,5,6,8,9,10,11,13,14,15,17}{ (2*\cs,2*\rw+\a*\hy/18) pic {bvert}}
    foreach \a in {3,7}{ (2*\cs,2*\rw+\a*\hy/18) pic {wvert}}
    foreach \a in {12,16}{ (2*\cs,2*\rw+\a*\hy/18) pic {sqvert}};
\path 
    (\cs/2,2*\rw+\hy/2) node {$\ft_1$}
    (\cs,2*\rw+5*\hy/8) node [above right] {$\ft_7$}
    (\cs,2*\rw+3*\hy/8) node [above right] {$\ft_8$};
\draw [thick, decorate, decoration={brace, amplitude=3pt}]
    (.2+2*\cs,2*\rw+16*\hy/18) -- (.2+2*\cs,2*\rw+10*\hy/18) node [midway, right=2pt] {$\ft_3$};
\draw [thick, decorate, decoration={brace, amplitude=3pt}]
    (.2+2*\cs,2*\rw+9*\hy/18) -- (.2+2*\cs,2*\rw+\hy/18) node [midway, right=2pt] {$\ft_4$};
\draw
    foreach \a in {1/4,3/4}{(4*\cs,2*\rw+\hy/2) -- (5*\cs,2*\rw+\a*\hy)}
    (5*\cs,2*\rw+\hy/4) -- (6*\cs,2*\rw+\hy/18)
    foreach \a in {2,4,...,12}{ (5*\cs,2*\rw+\hy/2) -- (6*\cs,2*\rw+\a*\hy/18) }
    foreach \a in {13,15,17}{ (5*\cs,2*\rw+3*\hy/4) -- (6*\cs,2*\rw+\a*\hy/18) };
\path
    (4*\cs,2*\rw+\hy/2) pic {bvert}
    (5*\cs,2*\rw+\hy/4) pic {bvert}
    (5*\cs,2*\rw+\hy/2) pic {sqvert}
    (5*\cs,2*\rw+3*\hy/4) pic {wvert}
    foreach \a in {1,3,4,5,7,8,9,11,12,13,14,16,17}{ (6*\cs,2*\rw+\a*\hy/18) pic {bvert} }
    foreach \a in {2,6,10,15}{ (6*\cs,2*\rw+\a*\hy/18) pic {wvert} };
\path
    (4.5*\cs,2*\rw+\hy/2) node {$\ft_2$}
    (5*\cs,2*\rw+3*\hy/8) node [below right] {$\ft_5$}
    (5*\cs,2*\rw+5*\hy/8) node [right] {$\ft_8$};
\draw [thick, decorate, decoration={brace, amplitude=3pt}]
    (.2+6*\cs,2*\rw+17*\hy/18) -- (.2+6*\cs,2*\rw+13*\hy/18) node [midway, right=2pt] {$\ft_3$};
\draw [thick, decorate, decoration={brace, amplitude=3pt}]
    (.2+6*\cs,2*\rw+12*\hy/18) -- (.2+6*\cs,2*\rw+2*\hy/18) node [midway, right=2pt] {$\ft_4$};
\draw
    foreach \a in {1/4,3/4}{(8*\cs,2*\rw+\hy/2) -- (9*\cs,2*\rw+\a*\hy)}
    (9*\cs,2*\rw+\hy/4) -- (10*\cs,2*\rw+\hy/14)
    foreach \a in {2,4}{ (9*\cs,2*\rw+\hy/2) -- (10*\cs,2*\rw+\a*\hy/14) }
    foreach \a in {5,7,...,13}{ (9*\cs,2*\rw+3*\hy/4) -- (10*\cs,2*\rw+\a*\hy/14) };    
\path
    (8*\cs,2*\rw+\hy/2) pic {wvert}
    (9*\cs,2*\rw+\hy/4) pic {bvert}
    (9*\cs,2*\rw+\hy/2) pic {bvert}
    (9*\cs,2*\rw+3*\hy/4) pic {sqvert}
    foreach \a in {4,6,7,8,10,11,12}{ (10*\cs,2*\rw+\a*\hy/14) pic {bvert} }
    foreach \a in {2,5,9,13}{ (10*\cs,2*\rw+\a*\hy/14) pic {wvert} }
    foreach \a in {1,3}{ (10*\cs,2*\rw+\a*\hy/14) pic {sqvert} };
\path
    (8.5*\cs,2*\rw+\hy/2) node {$\ft_3$}
    (9*\cs,2*\rw+3*\hy/8) node [below right] {$\ft_6$}
    (9*\cs,2*\rw+5*\hy/8) node [below right] {$\ft_5$};
\draw [thick, decorate, decoration={brace, amplitude=3pt}]
    (.2+10*\cs,2*\rw+13*\hy/14) -- (.2+10*\cs,2*\rw+5*\hy/14) node [midway, right=2pt] {$\ft_4$};
\draw [thick, decorate, decoration={brace, amplitude=3pt}]
    (.2+10*\cs,2*\rw+4*\hy/14) -- (.2+10*\cs,2*\rw+2*\hy/14) node [midway, right=2pt] {$\ft_2$};
\draw
    foreach \a in {1/4,3/4}{(0,\rw+\hy/2) -- (\cs,\rw+\a*\hy)}
    foreach \a in {1,3,5}{ (\cs,\rw+\hy/4) -- (2*\cs,\rw+\a*\hy/10) }
    foreach \a in {6,8}{ (\cs,\rw+\hy/2) -- (2*\cs,\rw+\a*\hy/10) }
    (\cs,\rw+3*\hy/4) -- (2*\cs,\rw+9*\hy/10);
\path 
    (0,\rw+\hy/2) pic {sqvert}
    (\cs,\rw+\hy/4) pic {wvert}
    (\cs,\rw+\hy/2) pic {bvert}
    (\cs,\rw+3*\hy/4) pic {bvert}
    foreach \a in {1,5,8}{ (2*\cs,\rw+\a*\hy/10) pic {sqvert} }
    foreach \a in {2,3,4,6}{ (2*\cs,\rw+\a*\hy/10) pic {bvert} }
    foreach \a in {7,9}{ (2*\cs,\rw+\a*\hy/10) pic {wvert} };
\path
    (\cs/2,\rw+\hy/2) node {$\ft_4$}
    (\cs,\rw+3*\hy/8) node [above right] {$\ft_7$}
    (\cs,\rw+5*\hy/8) node [above right] {$\ft_6$};
\draw [thick, decorate, decoration={brace, amplitude=3pt}]
    (.2+2*\cs,\rw+8*\hy/10) -- (.2+2*\cs,\rw+6*\hy/10) node [midway, right=2pt] {$\ft_1$};
\draw [thick, decorate, decoration={brace, amplitude=3pt}]
    (.2+2*\cs,\rw+5*\hy/10) -- (.2+2*\cs,\rw+\hy/10) node [midway, right=2pt] {$\ft_3$};
\draw 
    foreach \a in {1,2}{ (4*\cs,\rw+\a*\hy/3) -- (5*\cs,\rw+\a*\hy/3)}
    foreach \a in {1,3,5}{ (5*\cs, \rw+\hy/3) -- (6*\cs,\rw+\a*\hy/8)}
    (5*\cs, \rw+2*\hy/3) -- (6*\cs,\rw+7*\hy/8);
\path
    (4*\cs,\rw+\hy/3) pic {sqvert}
    (4*\cs,\rw+2*\hy/3) pic {bvert}
    (5*\cs,\rw+\hy/3) pic {wvert}
    (5*\cs,\rw+2*\hy/3) pic {bvert}
    foreach \a in {1,5}{ (6*\cs,\rw+\a*\hy/8) pic {sqvert} }
    foreach \a in {2,3,4,7}{ (6*\cs,\rw+\a*\hy/8) pic {bvert} };
\path
    (4.5*\cs,\rw+\hy/2) node {$\ft_5$}
    (5*\cs,\rw+\hy/2) node [above right] {$\ft_7$};
\draw [thick, decorate, decoration={brace, amplitude=3pt}]
    (.2+6*\cs,\rw+5*\hy/8) -- (.2+6*\cs,\rw+\hy/8) node [midway, right=2pt] {$\ft_3$};
\draw 
    foreach \a in {1,2}{ (8*\cs,\rw+\a*\hy/3) -- (9*\cs,\rw+\a*\hy/3)}
    foreach \a in {1,3,...,9}{ (9*\cs,\rw+\hy/3) -- (10*\cs,\rw+\a*\hy/15) }
    foreach \a in {10,12,14}{ (9*\cs,\rw+2*\hy/3) -- (10*\cs,\rw+\a*\hy/15) };
\path
    foreach \a in {1,2}{ (8*\cs,\rw+\a*\hy/3) pic {bvert} }
    (9*\cs,\rw+\hy/3) pic {sqvert}
    (9*\cs,\rw+2*\hy/3) pic {wvert}
    foreach \a in {1,2,4,5,6,8,9,10,11,13,14}{ (10*\cs,\rw+\a*\hy/15) pic {bvert} }
    foreach \a in {3,7}{ (10*\cs,\rw+\a*\hy/15) pic {wvert} }
    (10*\cs,\rw+12*\hy/15) pic {sqvert};
\path
    (8.5*\cs,\rw+\hy/2) node {$\ft_6$}
    (9*\cs,\rw+\hy/2) node [above right] {$\ft_8$};
\draw [thick, decorate, decoration={brace, amplitude=3pt}]
    (.2+10*\cs,\rw+14*\hy/15) -- (.2+10*\cs,\rw+10*\hy/15) node [midway, right=2pt] {$\ft_3$};
\draw [thick, decorate, decoration={brace, amplitude=3pt}]
    (.2+10*\cs,\rw+9*\hy/15) -- (.2+10*\cs,\rw+\hy/15) node [midway, right=2pt] {$\ft_4$};
\draw
    (0,\hy/6) -- (0,\hy) node [above] {$U_{n-1}$}
    (\cs,\hy/6) -- (\cs,\hy) node [above] {$U_{n}$}
    (2*\cs,\hy/6) -- (2*\cs,\hy) node [above] {$U_{n+1}$}
    (0,\hy/3) -- (2*\cs,\hy/3)
    (0,2*\hy/3) -- (\cs,2*\hy/3)
    foreach \a in {17,15,13,11,9}{ (\cs,2*\hy/3) -- (2*\cs,\a*\hy/18) };
\path
    foreach \a in {0,1,2}{ (\a*\cs,\hy/3) pic {bvert} }
    (0,2*\hy/3) pic {wvert}
    (\cs,2*\hy/3) pic {sqvert}
    foreach \a in {9,13,17}{ (2*\cs,\a*\hy/18) pic {wvert} }
    foreach \a in {10,11,12,14,15,16}{ (2*\cs,\a*\hy/18) pic {bvert} };
\path
    (\cs/2,\hy/2) node {$\ft_7$}
    (\cs,\hy/2) node [right] {$\ft_5$};
\draw [thick, decorate, decoration={brace, amplitude=3pt}]
    (.2+2*\cs,17*\hy/18) -- (.2+2*\cs,9*\hy/18) node [midway, right=2pt] {$\ft_4$};
\draw
    (4*\cs,\hy/6) -- (4*\cs,5*\hy/6) node [above] {$U_{n-1}$}
    (5*\cs,\hy/6) -- (5*\cs,5*\hy/6) node [above] {$U_{n}$}
    (6*\cs,\hy/6) -- (6*\cs,5*\hy/6) node [above] {$U_{n+1}$}
    (4*\cs,\hy/3) -- (6*\cs,\hy/3)
    (4*\cs,2*\hy/3) -- (6*\cs,2*\hy/3);
\path
    (4*\cs,\hy/3) pic {wvert}
    (4*\cs,2*\hy/3) pic {sqvert}
    foreach \a in {1/3,2/3}{ (5*\cs,\a*\hy) pic {bvert} }
    (6*\cs,\hy/3) pic {sqvert}
    (6*\cs,2*\hy/3) pic {wvert};
\node at    (4.5*\cs,\hy/2)  {$\ft_8$};
\node at    (5.5*\cs,\hy/2) {$\ft_6$};
\draw (8.5*\cs,2*\hy/3) pic {bvert}; \node [right] at (8.5*\cs,2*\hy/3) {: $4$-valent};
\draw (8.5*\cs,\hy/2) pic {wvert}; \node [right] at (8.5*\cs,2*\hy/4) {: $6$-valent};
\draw (8.5*\cs,\hy/3) pic {sqvert}; \node [right] at (8.5*\cs,1*\hy/3) {: $8$-valent};
\end{tikzpicture}}
\caption{\label{fig:4468_offdiag_1}Offspring diagrams for $T_1$}
\end{figure}

Encoding the offspring diagrams into a matrix, we obtain the transition matrix $M_1$ of $T_1$ given below. The four entries underlined in the matrix are the only entries which change between this example and the next example, $T_2$, that we  construct.
\[M_1=\begin{bmatrix}
 0 & 0 & \underline{0} & \underline{1} & 0 & 0 & 0 & 0 \\
 0 & 0 & \underline{1} & \underline{0} & 0 & 0 & 0 & 0 \\
 3 & 1 & 0 & 1 & 1 & 1 & 0 & 0 \\
 2 & 5 & 2 & 0 & 0 & 2 & 2 & 0 \\
 0 & 1 & 1 & 0 & 0 & 0 & 1 & 0 \\
 0 & 0 & 1 & 1 & 0 & 0 & 0 & 1 \\
 1 & 0 & 0 & 1 & 1 & 0 & 0 & 0 \\
 1 & 1 & 0 & 0 & 0 & 1 & 0 & 0 \\
\end{bmatrix}\]
The characteristic polynomial of $M_1$ is \[f_1(z) = (z-1) (z+1) \left(z^2+3 z+1\right) \left(z^4-3 z^3-4 z^2-3 z+1\right),\] which in turn gives that the eigenvalue of maximum modulus of $M_1$ is \[\lambda_1 = \frac{1}{4} \left(3+\sqrt{33}+2 \sqrt{\frac{13}{2}+\frac{3 \sqrt{33}}{2}}\right) \approx 4.13016.\]

Considering again the double-rays of $4$-valent vertices, it is trivial to note that if a vertex on such a double ray has two $6$-valent neighbors in the tessellation, then both vertices adjacent to it in the double-ray have two $8$-valent neighbors. This local behavior is shown in \Cref{fig:4468-RayConfigs} (B).

If this pattern is extended to all such double rays we  obtain the tessellation $T_2$, which is also the dual of a Cayley graph. The underlying group of this Cayley graph is is 
\[G_2 = \left\langle a,b,c,d\mid a^2=b^2=c^2=d^2=(ab)^2=(ad)^2=(cd)^3=(bc)^4\right\rangle.\]
The growth behavior of $T_2$ differs from that of $T_1$ only in the offspring of faces of types ${\ft_3}$ and ${\ft_4}$, as shown in the offspring diagrams in \Cref{fig:4468_offdiag_2}.

\begin{figure}[htbp]
{\centering\small 
\newcommand{\hy}{3.2}
\newcommand{\cs}{1.1}
\newcommand{\rw}{4}
\begin{tikzpicture}
[
    bvert/.pic={\fill [fill=black] (0,0) circle [radius=1.5pt];},
    wvert/.pic={\filldraw [draw=black, fill=white] (0,0) circle [radius=1.5pt];},
    sqvert/.pic={\filldraw [fill=white, draw=black] (-45:3pt) foreach \a in {45,135,225,315}{ -- (\a:3pt)};},
]
\foreach \a in {0,1,2}{
    \foreach \b in {1,2}{
        \draw (4*\a*\cs,\b*\rw) -- (4*\a*\cs,\b*\rw+\hy) node [above] {$U_{n-1}$};
        \draw ({(4*\a+1)*\cs},\b*\rw) -- ({(4*\a+1)*\cs},\b*\rw+\hy) node [above] {$U_n$};
        \draw ({(4*\a+2)*\cs},\b*\rw) -- ({(4*\a+2)*\cs},\b*\rw+\hy) node [above] {$U_{n+1}$};
        }
    }
\draw 
    foreach \a in {1/4,3/4}{(0,2*\rw+\hy/2) -- (\cs,2*\rw+\a*\hy)}
    foreach \a in {1,3,...,9}{(\cs,2*\rw+\hy/4) -- (2*\cs,2*\rw+\a*\hy/18)}
    foreach \a in {10,12,14,16}{(\cs,2*\rw+\hy/2) -- (2*\cs,2*\rw+\a*\hy/18)}
    (\cs,2*\rw+3*\hy/4) -- (2*\cs,2*\rw+17*\hy/18);
\path
    (0,2*\rw+\hy/2) pic {bvert} 
    (\cs,2*\rw+\hy/4) pic {sqvert}
    (\cs,2*\rw+\hy/2) pic {wvert}
    (\cs,2*\rw+3*\hy/4) pic {bvert}
    foreach \a in {1,2,4,5,6,8,9,10,11,13,14,15,17}{ (2*\cs,2*\rw+\a*\hy/18) pic {bvert}}
    foreach \a in {3,7}{ (2*\cs,2*\rw+\a*\hy/18) pic {wvert}}
    foreach \a in {12,16}{ (2*\cs,2*\rw+\a*\hy/18) pic {sqvert}};
\path 
    (\cs/2,2*\rw+\hy/2) node {$\ft_1$}
    (\cs,2*\rw+5*\hy/8) node [above right] {$\ft_7$}
    (\cs,2*\rw+3*\hy/8) node [above right] {$\ft_8$};
\draw [thick, decorate, decoration={brace, amplitude=3pt}]
    (.2+2*\cs,2*\rw+16*\hy/18) -- (.2+2*\cs,2*\rw+10*\hy/18) node [midway, right=2pt] {$\ft_3$};
\draw [thick, decorate, decoration={brace, amplitude=3pt}]
    (.2+2*\cs,2*\rw+9*\hy/18) -- (.2+2*\cs,2*\rw+\hy/18) node [midway, right=2pt] {$\ft_4$};
\draw
    foreach \a in {1/4,3/4}{(4*\cs,2*\rw+\hy/2) -- (5*\cs,2*\rw+\a*\hy)}
    (5*\cs,2*\rw+\hy/4) -- (6*\cs,2*\rw+\hy/18)
    foreach \a in {2,4,...,12}{ (5*\cs,2*\rw+\hy/2) -- (6*\cs,2*\rw+\a*\hy/18) }
    foreach \a in {13,15,17}{ (5*\cs,2*\rw+3*\hy/4) -- (6*\cs,2*\rw+\a*\hy/18) };
\path
    (4*\cs,2*\rw+\hy/2) pic {bvert}
    (5*\cs,2*\rw+\hy/4) pic {bvert}
    (5*\cs,2*\rw+\hy/2) pic {sqvert}
    (5*\cs,2*\rw+3*\hy/4) pic {wvert}
    foreach \a in {1,3,4,5,7,8,9,11,12,13,14,16,17}{ (6*\cs,2*\rw+\a*\hy/18) pic {bvert} }
    foreach \a in {2,6,10,15}{ (6*\cs,2*\rw+\a*\hy/18) pic {wvert} };
\path
    (4.5*\cs,2*\rw+\hy/2) node {$\ft_2$}
    (5*\cs,2*\rw+3*\hy/8) node [below right] {$\ft_5$}
    (5*\cs,2*\rw+5*\hy/8) node [right] {$\ft_8$};
\draw [thick, decorate, decoration={brace, amplitude=3pt}]
    (.2+6*\cs,2*\rw+17*\hy/18) -- (.2+6*\cs,2*\rw+13*\hy/18) node [midway, right=2pt] {$\ft_3$};
\draw [thick, decorate, decoration={brace, amplitude=3pt}]
    (.2+6*\cs,2*\rw+12*\hy/18) -- (.2+6*\cs,2*\rw+2*\hy/18) node [midway, right=2pt] {$\ft_4$};
\draw
    foreach \a in {1/4,3/4}{(8*\cs,2*\rw+\hy/2) -- (9*\cs,2*\rw+\a*\hy)}
    (9*\cs,2*\rw+\hy/4) -- (10*\cs,2*\rw+\hy/14)
    foreach \a in {2,4}{ (9*\cs,2*\rw+\hy/2) -- (10*\cs,2*\rw+\a*\hy/14) }
    foreach \a in {5,7,...,13}{ (9*\cs,2*\rw+3*\hy/4) -- (10*\cs,2*\rw+\a*\hy/14) };    
\path
    (8*\cs,2*\rw+\hy/2) pic {wvert}
    (9*\cs,2*\rw+\hy/4) pic {bvert}
    (9*\cs,2*\rw+\hy/2) pic {bvert}
    (9*\cs,2*\rw+3*\hy/4) pic {sqvert}
    foreach \a in {4,6,7,8,10,11,12}{ (10*\cs,2*\rw+\a*\hy/14) pic {bvert} }
    foreach \a in {1,3,5,9,13}{ (10*\cs,2*\rw+\a*\hy/14) pic {wvert} }
    foreach \a in {2}{ (10*\cs,2*\rw+\a*\hy/14) pic {sqvert} };
\path
    (8.5*\cs,2*\rw+\hy/2) node {$\ft_3$}
    (9*\cs,2*\rw+3*\hy/8) node [below right] {$\ft_6$}
    (9*\cs,2*\rw+5*\hy/8) node [below right] {$\ft_5$};
\draw [thick, decorate, decoration={brace, amplitude=3pt}]
    (.2+10*\cs,2*\rw+13*\hy/14) -- (.2+10*\cs,2*\rw+5*\hy/14) node [midway, right=2pt] {$\ft_4$};
\draw [thick, decorate, decoration={brace, amplitude=3pt}]
    (.2+10*\cs,2*\rw+4*\hy/14) -- (.2+10*\cs,2*\rw+2*\hy/14) node [midway, right=2pt] {$\ft_1$};
\draw
    foreach \a in {1/4,3/4}{(0,\rw+\hy/2) -- (\cs,\rw+\a*\hy)}
    foreach \a in {1,3,5}{ (\cs,\rw+\hy/4) -- (2*\cs,\rw+\a*\hy/10) }
    foreach \a in {6,8}{ (\cs,\rw+\hy/2) -- (2*\cs,\rw+\a*\hy/10) }
    (\cs,\rw+3*\hy/4) -- (2*\cs,\rw+9*\hy/10);
\path 
    (0,\rw+\hy/2) pic {sqvert}
    (\cs,\rw+\hy/4) pic {wvert}
    (\cs,\rw+\hy/2) pic {bvert}
    (\cs,\rw+3*\hy/4) pic {bvert}
    foreach \a in {1,5,7,9}{ (2*\cs,\rw+\a*\hy/10) pic {sqvert} }
    foreach \a in {2,3,4,6}{ (2*\cs,\rw+\a*\hy/10) pic {bvert} }
    foreach \a in {8}{ (2*\cs,\rw+\a*\hy/10) pic {wvert} };
\path
    (\cs/2,\rw+\hy/2) node {$\ft_4$}
    (\cs,\rw+3*\hy/8) node [above right] {$\ft_7$}
    (\cs,\rw+5*\hy/8) node [above right] {$\ft_6$};
\draw [thick, decorate, decoration={brace, amplitude=3pt}]
    (.2+2*\cs,\rw+8*\hy/10) -- (.2+2*\cs,\rw+6*\hy/10) node [midway, right=2pt] {$\ft_2$};
\draw [thick, decorate, decoration={brace, amplitude=3pt}]
    (.2+2*\cs,\rw+5*\hy/10) -- (.2+2*\cs,\rw+\hy/10) node [midway, right=2pt] {$\ft_3$};
\draw 
    foreach \a in {1,2}{ (4*\cs,\rw+\a*\hy/3) -- (5*\cs,\rw+\a*\hy/3)}
    foreach \a in {1,3,5}{ (5*\cs, \rw+\hy/3) -- (6*\cs,\rw+\a*\hy/8)}
    (5*\cs, \rw+2*\hy/3) -- (6*\cs,\rw+7*\hy/8);
\path
    (4*\cs,\rw+\hy/3) pic {sqvert}
    (4*\cs,\rw+2*\hy/3) pic {bvert}
    (5*\cs,\rw+\hy/3) pic {wvert}
    (5*\cs,\rw+2*\hy/3) pic {bvert}
    foreach \a in {1,5}{ (6*\cs,\rw+\a*\hy/8) pic {sqvert} }
    foreach \a in {2,3,4,7}{ (6*\cs,\rw+\a*\hy/8) pic {bvert} };
\path
    (4.5*\cs,\rw+\hy/2) node {$\ft_5$}
    (5*\cs,\rw+\hy/2) node [above right] {$\ft_7$};
\draw [thick, decorate, decoration={brace, amplitude=3pt}]
    (.2+6*\cs,\rw+5*\hy/8) -- (.2+6*\cs,\rw+\hy/8) node [midway, right=2pt] {$\ft_3$};
\draw 
    foreach \a in {1,2}{ (8*\cs,\rw+\a*\hy/3) -- (9*\cs,\rw+\a*\hy/3)}
    foreach \a in {1,3,...,9}{ (9*\cs,\rw+\hy/3) -- (10*\cs,\rw+\a*\hy/15) }
    foreach \a in {10,12,14}{ (9*\cs,\rw+2*\hy/3) -- (10*\cs,\rw+\a*\hy/15) };
\path
    foreach \a in {1,2}{ (8*\cs,\rw+\a*\hy/3) pic {bvert} }
    (9*\cs,\rw+\hy/3) pic {sqvert}
    (9*\cs,\rw+2*\hy/3) pic {wvert}
    foreach \a in {1,2,4,5,6,8,9,10,11,13,14}{ (10*\cs,\rw+\a*\hy/15) pic {bvert} }
    foreach \a in {3,7}{ (10*\cs,\rw+\a*\hy/15) pic {wvert} }
    (10*\cs,\rw+12*\hy/15) pic {sqvert};
\path
    (8.5*\cs,\rw+\hy/2) node {$\ft_6$}
    (9*\cs,\rw+\hy/2) node [above right] {$\ft_8$};
\draw [thick, decorate, decoration={brace, amplitude=3pt}]
    (.2+10*\cs,\rw+14*\hy/15) -- (.2+10*\cs,\rw+10*\hy/15) node [midway, right=2pt] {$\ft_3$};
\draw [thick, decorate, decoration={brace, amplitude=3pt}]
    (.2+10*\cs,\rw+9*\hy/15) -- (.2+10*\cs,\rw+\hy/15) node [midway, right=2pt] {$\ft_4$};
\draw
    (0,\hy/6) -- (0,\hy) node [above] {$U_{n-1}$}
    (\cs,\hy/6) -- (\cs,\hy) node [above] {$U_{n}$}
    (2*\cs,\hy/6) -- (2*\cs,\hy) node [above] {$U_{n+1}$}
    (0,\hy/3) -- (2*\cs,\hy/3)
    (0,2*\hy/3) -- (\cs,2*\hy/3)
    foreach \a in {17,15,13,11,9}{ (\cs,2*\hy/3) -- (2*\cs,\a*\hy/18) };
\path
    foreach \a in {0,1,2}{ (\a*\cs,\hy/3) pic {bvert} }
    (0,2*\hy/3) pic {wvert}
    (\cs,2*\hy/3) pic {sqvert}
    foreach \a in {9,13,17}{ (2*\cs,\a*\hy/18) pic {wvert} }
    foreach \a in {10,11,12,14,15,16}{ (2*\cs,\a*\hy/18) pic {bvert} };
\path
    (\cs/2,\hy/2) node {$\ft_7$}
    (\cs,\hy/2) node [right] {$\ft_5$};
\draw [thick, decorate, decoration={brace, amplitude=3pt}]
    (.2+2*\cs,17*\hy/18) -- (.2+2*\cs,9*\hy/18) node [midway, right=2pt] {$\ft_4$};
\draw
    (4*\cs,\hy/6) -- (4*\cs,5*\hy/6) node [above] {$U_{n-1}$}
    (5*\cs,\hy/6) -- (5*\cs,5*\hy/6) node [above] {$U_{n}$}
    (6*\cs,\hy/6) -- (6*\cs,5*\hy/6) node [above] {$U_{n+1}$}
    (4*\cs,\hy/3) -- (6*\cs,\hy/3)
    (4*\cs,2*\hy/3) -- (6*\cs,2*\hy/3);
\path
    (4*\cs,\hy/3) pic {wvert}
    (4*\cs,2*\hy/3) pic {sqvert}
    foreach \a in {1/3,2/3}{ (5*\cs,\a*\hy) pic {bvert} }
    (6*\cs,\hy/3) pic {sqvert}
    (6*\cs,2*\hy/3) pic {wvert};
\node at    (4.5*\cs,\hy/2)  {$\ft_8$};
\node at    (5.5*\cs,\hy/2) {$\ft_6$};
\draw (8.5*\cs,2*\hy/3) pic {bvert}; \node [right] at (8.5*\cs,2*\hy/3) {: $4$-valent};
\draw (8.5*\cs,\hy/2) pic {wvert}; \node [right] at (8.5*\cs,2*\hy/4) {: $6$-valent};
\draw (8.5*\cs,\hy/3) pic {sqvert}; \node [right] at (8.5*\cs,1*\hy/3) {: $8$-valent};
\end{tikzpicture}}

\caption{\label{fig:4468_offdiag_2}Offspring diagrams for $T_2$}
\end{figure}
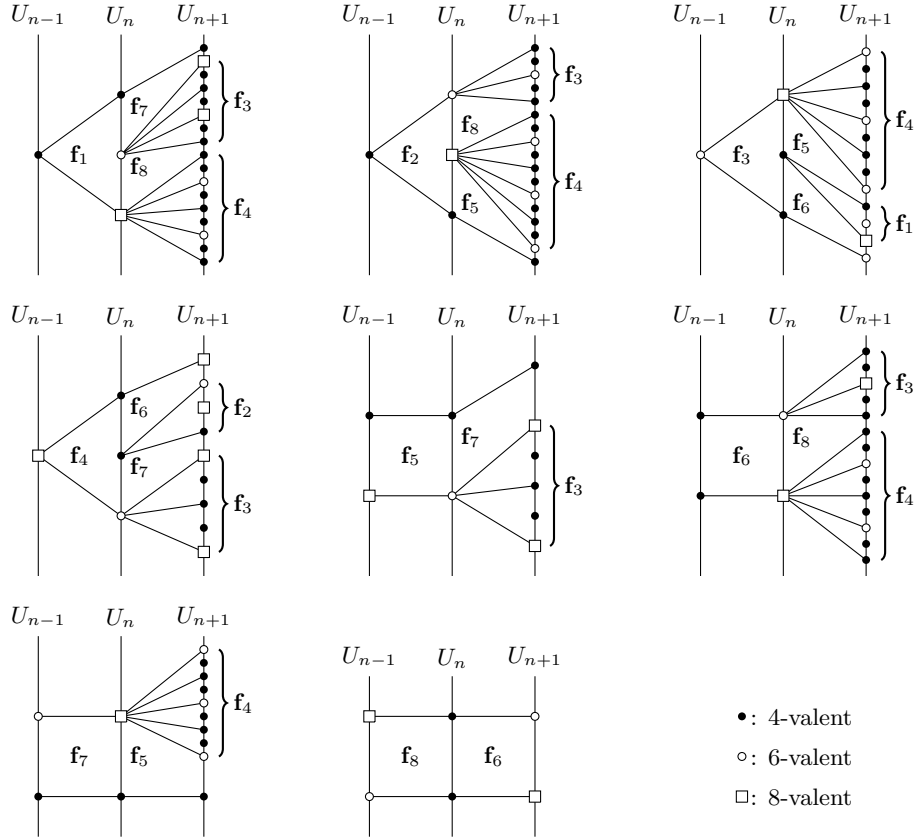

The effect of the change of offspring of types $\ft_2$ and $\ft_3$ in the transition matrix of $T_2$ lies only in the underlined $2\times 2$ submatrix of $M_1$, while the remainder of the matrix $M_2$ remains identical to $M_1$. Hence we have
 \[M_2=\begin{bmatrix}
 0 & 0 & \underline{1} & \underline{0} & 0 & 0 & 0 & 0 \\
 0 & 0 & \underline{0} & \underline{1} & 0 & 0 & 0 & 0 \\
 3 & 1 & 0 & 1 & 1 & 1 & 0 & 0 \\
 2 & 5 & 2 & 0 & 0 & 2 & 2 & 0 \\
 0 & 1 & 1 & 0 & 0 & 0 & 1 & 0 \\
 0 & 0 & 1 & 1 & 0 & 0 & 0 & 1 \\
 1 & 0 & 0 & 1 & 1 & 0 & 0 & 0 \\
 1 & 1 & 0 & 0 & 0 & 1 & 0 & 0 \\
\end{bmatrix}.\]
The characteristic polynomial of $M_2$ is
\[f_2(z)=(z-1)^2 \left(z^6+2 z^5-15 z^4-40 z^3-15 z^2+2 z+1\right).\] As polynomials of degree 6 are unfortunately not solvable by radicals, we obtain by approximation that the root of maximum modulus is $\lambda_2\approx 4.14659$. 

As these growth rates are nearly the same, there is only a small difference in corona sizes in the first several coronas. However, the size of the coronas and distribution of face types differs greatly farther from the root. To demonstrate this, \Cref{tab:4468-corsizes} gives corona sizes in Bilinski diagrams of $T_1$ and of $T_2$, both rooted at $4$-valent vertices. Note that the sizes of the coronas of $T_2$ dominate those of $T_1$ only after the $13^\text{th}$ corona.
\begin{table}[hbt]
{\small\[
\begin{array}{c|cc||c|cc}
 n & \abs{F_{1,n}} & \abs{F_{2,n}} & n & \abs{F_{1,n}} & \abs{F_{2,n}} \\ \hline
 1  & 4                     & 4                     & 29  & 1.20050\times 10^{18}  & 1.27748\times 10^{18} \\
 2  & 30                    & 28                    & 30  & 4.95826\times 10^{18}  & 5.29701\times 10^{18} \\
 3  & 110                   & 108                   & 31  & 2.04784\times 10^{19}  & 2.19652\times 10^{19} \\
 4  & 494                   & 468                   & 32  & 8.45791\times 10^{19}  & 9.10786\times 10^{19} \\
 5  & 1938                  & 1900                  & 33  & 3.49325\times 10^{20}  & 3.77673\times 10^{20} \\
 6  & 8272                  & 7956                  & 34  & 1.44277\times 10^{21}  & 1.56603\times 10^{21} \\
 7  & 33464                 & 32868                 & 35  & 5.95887\times 10^{21}  & 6.49377\times 10^{21} \\
 8  & 140046                & 136380                & 36  & 2.46111\times 10^{22}  & 2.69268\times 10^{22} \\
 9  & 573610                & 565956                & 37  & 1.01648\times 10^{23}  & 1.11655\times 10^{23} \\
 10 & 2.38167\times 10^6    & 2.34358\times 10^6    & 38  & 4.19821\times 10^{23}  & 4.62986\times 10^{23} \\
 11 & 9.80378\times 10^6    & 9.73259\times 10^6    & 39  & 1.73393\times 10^{24}  & 1.91983\times 10^{24} \\
 12 & 4.05773\times 10^7    & 4.02988\times 10^7    & 40  & 7.16140\times 10^{24}  & 7.96071\times 10^{24} \\
 13 & 1.67365\times 10^8    & 1.67318\times 10^8    & 41  & 2.95777\times 10^{25}  & 3.30099\times 10^{25} \\
 14 & \mathbf{6.91836\times 10^8}    & \mathbf{6.93034\times 10^8}    & 42  & 1.22161\times 10^{26}  & 1.36878\times 10^{26} \\
 15 & 2.85585\times 10^9    & 2.87639\times 10^9    & 43  & 5.04544\times 10^{26}  & 5.67580\times 10^{26} \\
 16 & 1.17992\times 10^{10} & 1.19181\times 10^{10} & 44  & 2.08385\times 10^{27}  & 2.35352\times 10^{27} \\
 17 & 4.87218\times 10^{10} & 4.94504\times 10^{10} & 45  & 8.60662\times 10^{27}  & 9.75910\times 10^{27} \\
 18 & 2.01257\times 10^{11} & 2.04947\times 10^{11} & 46  & 3.55467\times 10^{28}  & 4.04670\times 10^{28} \\
 19 & 8.31149\times 10^{11} & 8.50179\times 10^{11} & 47  & 1.46814\times 10^{29}  & 1.67800\times 10^{29} \\
 20 & 3.43297\times 10^{12} & 3.52419\times 10^{12} & 48  & 6.06363\times 10^{29}  & 6.95799\times 10^{29} \\
 21 & 1.41782\times 10^{13} & 1.46172\times 10^{13} & 49  & 2.50438\times 10^{30}  & 2.88520\times 10^{30} \\
 22 & 5.85596\times 10^{13} & 6.05990\times 10^{13} & 50  & 1.03435\times 10^{31}  & 1.19637\times 10^{31} \\
 23 & 2.41857\times 10^{14} & 2.51322\times 10^{14} & 60  & 1.49395\times 10^{37}  & 1.79797\times 10^{37} \\
 24 & 9.98918\times 10^{14} & 1.04199\times 10^{15} & 70  & 2.15777\times 10^{43}  & 2.70207\times 10^{43} \\
 25 & 4.12567\times 10^{15} & 4.32117\times 10^{15} & 80  & 3.11654\times 10^{49}  & 4.06079\times 10^{49} \\
 26 & 1.70397\times 10^{16} & 1.79166\times 10^{16} & 90  & 4.50134\times 10^{55}  & 6.10274\times 10^{55} \\
 27 & 7.03766\times 10^{16} & 7.42979\times 10^{16} & 100 & 6.50145\times 10^{61}  & 9.17148\times 10^{61} \\
 28 & 2.90667\times 10^{17} & 3.08066\times 10^{17} & 200 & 2.56861\times 10^{123} & 5.38996\times 10^{123} \\
\end{array}
\]}
\caption[Corona sizes in $T_1$ and $T_2$ both with $\sig=\vs{4,4,6,8}$]{\label{tab:4468-corsizes}Corona sizes in $T_1$ and $T_2$; emphasis on the $14^\text{th}$ corona beyond which the coronas of $T_2$ appear to exceed in size those of $T_1$.}
\end{table}


\subsection{\label{sec:Conjectures}Some conjectures}
Ideally, all tessellations realizing the same polymorphic valence sequence would have the same growth rate. The example of valence sequence $[4,4,6,8]$ illustrates that this is not so.  We propose the following definitions.

\begin{defn} Let $\sigma$ be some polymorphic valence sequence, and define $\mathscr{T}_\sigma$ to be the set of isomorphism classes of face-homogeneous tessellations with valence sequence $\sigma$. 
Let 
\begin{align}
\underline{\lambda}_\sigma & =\inf\{\gamma(T) : T\in \mathscr{T}_\sigma \}, \\
\overline{\lambda}_\sigma &= \sup\{\gamma(T) : T\in \mathscr{T}_\sigma \}, \\
\mathscr{L}_\sigma &= \{T:T\in \mathscr{T}_\sigma\text{ and }
    \gamma(T)=\underline{\lambda}_\sigma\},\text{ and } \\
\mathscr{H}_\sigma &= \{T:T\in \mathscr{T}_\sigma\text{ and }
    \gamma(T)=\overline{\lambda}_\sigma\}
\end{align}
\end{defn}

We conjecture that the lower and upper bounds $\underline{\lambda}_\sigma$ and $\overline{\lambda}_\sigma$ for any given valence sequence $\sigma$ are realized.

\begin{conj}\label{thm:ExtremaExist} Let $\sigma$ be a polymorphic valence sequence. Then $\mathscr{L}_\sigma$ and $\mathscr{H}_\sigma$ are nonempty.

\end{conj}

Bearing in mind the polymorphic valence sequence $[4,4,6,8]$ analyzed in \Cref{appx_a}, we propose as a conjecture the following sharper version of \Cref{thm:GrowthOrder}.

\begin{conj}\label{conj:StrongSkip} Let $\sigma_1$ and $\sigma_2$ be valence sequences such that $\sigma_1<\sigma_2$. Then 
\begin{equation}
\overline{\lambda}_{\sigma_1}\leq\underline{\lambda}_{\sigma_2}.
\end{equation}
\end{conj}

In the spirit of the famous quote of the late George P\'{o}lya \cite{Pol} (``If you can't solve a problem, then there is an easier problem you can solve: find it.''), we offer the following (perhaps) easier conjecture. 
\begin{conj} Let $\sigma_1$ and $\sigma_2$ be valence sequences with  $\sigma_1<\sigma_2$. Then 
\begin{equation}
\underline{\lambda}_{\sigma_1}\leq \underline{\lambda}_{\sigma_2}.
\end{equation}
\end{conj}

If \Cref{conj:StrongSkip} holds, then one could delete the condition of monomorphicity from the hypothesis of \Cref{thm:GrowthOrder} and therefore from \Cref{thm:LeastRateofGrowth} as well.  Moreover,  the Appendix could be much abbreviated.  For example, one could  eliminate the exhaustive consideration of the many forms of $6$-covalent face-homogeneous tessellations listed and treated there by observing that the least valence sequence $\sigma$ of length $6$ with $\eta(\sigma)>0$ is $[3,3,3,3,3,4]$. Thus, if any tessellation with the polymorphic valence sequence $[3,3,3,3,3,4]$ has growth rate greater than $\frac12(1+\sqrt{5})$, then so does every tessellation with valence sequence $\sigma\geq[3,3,3,3,3,4]$.

Beyond these conjectures, there are some open questions.  Consider the partially ordered set of valence sequences, and in particular, the poset consisting of the polymorphic valence sequences.
\begin{ques}\label{ques:longerIntervals} As one goes up a chain in the poset, do intervals of the form $\left[\lolam_\sig,\hilam_\sig\right]$ become (asymptotically) longer?
\end{ques}
\begin{ques}\label{ques:arbLongerIntervals}Do the intervals in the complement of \[\bigcup_{\sig}\left\{\left[\lolam_\sig,\hilam_\sig\right]:\sig\text{ is polymorphic}\right\}\] become arbitrarily long?
\end{ques}
If the answer to \Cref{ques:arbLongerIntervals} is negative, we pose the following.
\begin{ques}If $x$ is a sufficiently large real number, is there always some polymorphic valence sequence $\sig$ such that $\lolam_\sig\leq x \leq \hilam_\sig$?\end{ques}

\noindent Or, on the other hand,
\begin{ques}Do there exist polymorphic sequences $\sigma,\tau$ such that \[\left[\lolam_\sigma,\hilam_\sigma\right]\cap\left[\lolam_\tau,\hilam_\tau\right]\neq\emptyset?\]
\end{ques}

\bibliographystyle{amcjoucc}

\end{document}



\title{ Appendix to \\``Growth of Face-Homogeneous Tessellations''}


\author{Stephen J.\,Graves}
\address{Department of Mathematics\\
University of Texas at Tyler\\
Tyler, TX 75799}

\email{sgraves@uttyler.edu}

\author{Mark E.\,Watkins}
\address{Department of Mathematics\\
Syracuse University\\
Syracuse, NY 13244-1150}
\email{mewatkin@syr.edu}


\maketitle

\section{Introduction}

This article is the APPENDIX to our paper ``Growth of Face-Homogeneous Tessellations'' which will appear in {\em Ars Mathematica Contemporanea} \cite{GFHT}.  Its purpose is to make available to the interested reader many of the details of the numerous, often routine, calculations necessary to verify some of the results in the cited article.  These calculations also provide a classification of all general forms of valence sequences of lengths $3, 4, 5,$ and $6$. 

All external references in this APPENDIX are to items in \cite{GFHT}, and are cited with \cite{GFHT} preceding the reference, e.g., \cite{GFHT}Definition 1.9.  The reader should consult \cite{GFHT} for any unfamiliar terminology.  

{\bf Notational convention.}  Throughout this APPENDIX, and without further mention, the symbol $\sigma$ denotes a realizable valence sequence having positive angle excess, i.e.,  $\eta(\sigma)>0$ (see \cite{GFHT}Definition 1.9).  Moreover, unless explicitly stated otherwise, when $\sigma$ is given, the symbol $T$  denotes a face-homogeneous tessellation whose valence sequence is $\sigma$.

As explained in \cite{GFHT}, when the offspring diagrams for a given valence sequence $\sigma$ are determined solely by $\sigma$ itself, then $\sigma$ is {\em monomorphic}, and so the growth rates of such tessellations are explicit functions of the valence sequence.  In those cases where a functional value is not computationally feasible, a lower bound for the growth rates of all tessellations of the given form is determined by comparing the growth rates of the minimal valence sequences having the given form.   On the other hand, the possibility of two or more consistent offspring types of a given face type is sufficient for $\sigma$ to be {\em polymorphic}.

Many computations throughout were obtained or verified using {\em Wolfram Mathematica}.

\setcounter{thm}{0}
\begin{rem}\label{rem:PalinPoly} It is frequently the case that the characteristic polynomial of a transition matrix contains a factor of the form $z^4-a z^3-b z^2-a z+1$; these ``palindromic" polynomials have a root of maximum modulus
\[
\frac14\left(\left(a+\sqrt{a^2+4b+8}\right)+\sqrt{4b-8+a\left(a+\sqrt{a^2+4b+8}\right)}\right)
\]
when $a,b>0$. 
\end{rem}

\section[Face-homogeneous Tessellations of Covalence 3]{Classification of Face-Homogeneous Tessellations with Valence Sequences of Length 3}\label{ch:3cov}

Recall that all 3-covalent tessellations have monomorphic valence sequences by \cite{GFHT}Proposition 2.5.  There are exactly three such forms:  $[p,p,p],\,[p,p,q]$, and $[p,q,r]$ (cf. the last paragraph of \cite{GFHT}Section 1.3).  

\begin{prop}\label{thm:GrowthOf-ppp} If  $\sig=\vs{p,p,p}$, then $p\geq 7$ and 
\[ \gam(T) =  \frac12\left[(p-4)+\sqrt{(p-4)^2-4}\right]. \]
\begin{proof}
Since $\eta(\sigma)>0$, we must have $p\geq7$.  Moreover, $T$ is edge-homogeneous with edge-symbol $\esym{p,p}{3,3}$.  Hence by \cite{GFHT}{Proposition 1.17},
\begin{align*}
\gam(T) &= \frac12\left[(p-2)-2+\sqrt{(p-2)^2-4(p-2)}\right] \\
	&= \frac12\left[(p-4)+\sqrt{(p-4)^2-4}\right].
\end{align*}
\end{proof}\end{prop}

Often special attention is required when a term in the valence sequence equals ~3.

\begin{prop}\label{thm:GrowthOf-pp3} If $\sig=[p,p,3]$, then $p\geq 14$ and is even, and
\[ \gam(T)=\frac14\left(p-8+\sqrt{(p-8)^2-16}\right). \]
\begin{proof}
For $\sig$ to be realizable, $p$ must be even.  By  \cite{GFHT}Proposition 1.10, $p>12$. Hence $p\geq 14$. 
Take $T_0$ to be a face-homogeneous tessellation with valence sequence $\vs{p,p,3}$ labeled as a Bilinski diagram about a $p$-valent root. Let $v_0$ be an arbitrary $3$-valent vertex, and label the three $p$-valent vertices adjacent to $v_0$ by $v_1,v_2,$ and $v_3$. Then the induced graph of $\{v_0,v_1,v_2,v_3\}$ is isomorphic to the complete graph $K_4$. Since every face of $T_0$ is incident with exactly one $3$-valent vertex, each face $f$ lies in exactly one of these planar embeddings of $K_4$. 

Let $T_1$ be the tessellation remaining when all 3-valent vertices and their incident edges are removed from $T_0$. Then $T_1$ is face-homogeneous with valence sequence $\vs{\frac{p}2,\frac{p}2,\frac{p}2}$. (This is demonstrated for $p=16$ in \Cref{fig:pp3-center}.)
\begin{figure}[htbp]
\[
\begin{matrix}
\definecolor{vc0}{rgb}{1,1,1}
\definecolor{vc1}{rgb}{0,0,0}
\definecolor{ec0}{rgb}{0,0,0}
\definecolor{ec1}{rgb}{0,0,0}
\begin{tikzpicture}[xscale = 1.3, yscale=1.3]
\foreach \n/\x/\y in {v4_17/-1.75/-0.53, v5_30/1.85/-0.77, v5_31/1.96/-0.39, v5_19/-1.66/-1.11, v4_15/-1.82/0.18, v0_0/0.0/0.0, v4_14/-1.75/0.53, v4_31/1.82/-0.18, v4_30/1.75/-0.53, v4_12/-1.42/1.16, v4_11/-1.16/1.42, v4_10/-0.86/1.62, v4_18/-1.62/-0.86, v4_28/1.42/-1.16, v4_29/1.62/-0.86, v5_16/-2.0/0.0, v4_22/-0.53/-1.75, v4_23/-0.18/-1.82, v4_20/-1.16/-1.42, v4_21/-0.86/-1.62, v4_26/0.86/-1.62, v4_27/1.16/-1.42, v4_24/0.18/-1.82, v4_25/0.53/-1.75, v4_3/1.42/1.16, v4_2/1.62/0.86, v4_1/1.75/0.53, v4_0/1.82/0.18, v4_7/0.18/1.82, v4_6/0.53/1.75, v4_5/0.86/1.62, v4_4/1.16/1.42, v4_9/-0.53/1.75, v4_8/-0.18/1.82, v2_1/0.71/0.71, v2_0/1.0/0.0, v2_3/-0.71/0.71, v2_2/0.0/1.0, v2_5/-0.71/-0.71, v2_4/-1.0/0.0, v2_7/0.71/-0.71, v2_6/0.0/-1.0, v1_6/0.26/-0.62, v1_7/0.62/-0.26, v1_4/-0.62/-0.26, v1_5/-0.26/-0.62, v1_2/-0.26/0.62, v1_3/-0.62/0.26, v1_0/0.62/0.26, v1_1/0.26/0.62, v5_10/-0.77/1.85, v5_11/-1.11/1.66, v5_12/-1.41/1.41, v5_13/-1.66/1.11, v5_14/-1.85/0.77, v5_15/-1.96/0.39, v4_19/-1.42/-1.16, v5_17/-1.96/-0.39, v4_16/-1.82/-0.18, v5_18/-1.85/-0.77, v5_8/0.0/2.0, v5_9/-0.39/1.96, v5_21/-1.11/-1.66, v5_20/-1.41/-1.41, v5_23/-0.39/-1.96, v5_22/-0.77/-1.85, v5_25/0.39/-1.96, v5_24/0.0/-2.0, v5_27/1.11/-1.66, v5_26/0.77/-1.85, v5_29/1.66/-1.11, v5_28/1.41/-1.41, v5_0/2.0/0.0, v3_0/1.08/0.45, v3_1/0.45/1.08, v3_2/-0.45/1.08, v3_3/-1.08/0.45, v3_4/-1.08/-0.45, v3_5/-0.45/-1.08, v3_6/0.45/-1.08, v3_7/1.08/-0.45, v5_2/1.85/0.77, v5_3/1.66/1.11, v4_13/-1.62/0.86, v5_1/1.96/0.39, v5_6/0.77/1.85, v5_7/0.39/1.96, v5_4/1.41/1.41, v5_5/1.11/1.66} {    \node (\n) at (\x,\y) {}; };
\draw[ec0, thick]
(v1_0.center) edge[-] (v2_0.center)    (v1_0.center) edge[-] (v2_1.center)    (v1_1.center) edge[-] (v2_1.center)    (v1_1.center) edge[-] (v2_2.center)    (v1_2.center) edge[-] (v2_2.center)    (v1_2.center) edge[-] (v2_3.center)    (v1_3.center) edge[-] (v2_3.center)    (v1_3.center) edge[-] (v2_4.center)    (v1_4.center) edge[-] (v2_4.center)    (v1_4.center) edge[-] (v2_5.center)    (v1_5.center) edge[-] (v2_5.center)    (v1_5.center) edge[-] (v2_6.center)    (v1_6.center) edge[-] (v2_6.center)    (v1_6.center) edge[-] (v2_7.center)    (v1_7.center) edge[-] (v2_7.center)    (v1_7.center) edge[-] (v2_0.center)    (v4_0.center) edge[-] (v5_0.center)    (v4_0.center) edge[-] (v5_1.center)    (v4_1.center) edge[-] (v5_1.center)    (v4_1.center) edge[-] (v5_2.center)    (v4_2.center) edge[-] (v5_2.center)    (v4_2.center) edge[-] (v5_3.center)    (v4_3.center) edge[-] (v5_3.center)    (v4_3.center) edge[-] (v5_4.center)    (v4_4.center) edge[-] (v5_4.center)    (v4_4.center) edge[-] (v5_5.center)    (v4_5.center) edge[-] (v5_5.center)    (v4_5.center) edge[-] (v5_6.center)    (v4_6.center) edge[-] (v5_6.center)    (v4_6.center) edge[-] (v5_7.center)    (v4_7.center) edge[-] (v5_7.center)    (v4_7.center) edge[-] (v5_8.center)    (v4_8.center) edge[-] (v5_8.center)    (v4_8.center) edge[-] (v5_9.center)    (v4_9.center) edge[-] (v5_9.center)    (v4_9.center) edge[-] (v5_10.center)    (v4_10.center) edge[-] (v5_10.center)    (v4_10.center) edge[-] (v5_11.center)    (v4_11.center) edge[-] (v5_11.center)    (v4_11.center) edge[-] (v5_12.center)    (v4_12.center) edge[-] (v5_12.center)    (v4_12.center) edge[-] (v5_13.center)    (v4_13.center) edge[-] (v5_13.center)    (v4_13.center) edge[-] (v5_14.center)    (v4_14.center) edge[-] (v5_14.center)    (v4_14.center) edge[-] (v5_15.center)    (v4_15.center) edge[-] (v5_15.center)    (v4_15.center) edge[-] (v5_16.center)    (v4_16.center) edge[-] (v5_16.center)    (v4_16.center) edge[-] (v5_17.center)    (v4_17.center) edge[-] (v5_17.center)    (v4_17.center) edge[-] (v5_18.center)    (v4_18.center) edge[-] (v5_18.center)    (v4_18.center) edge[-] (v5_19.center)    (v4_19.center) edge[-] (v5_19.center)    (v4_19.center) edge[-] (v5_20.center)    (v4_20.center) edge[-] (v5_20.center)    (v4_20.center) edge[-] (v5_21.center)    (v4_21.center) edge[-] (v5_21.center)    (v4_21.center) edge[-] (v5_22.center)    (v4_22.center) edge[-] (v5_22.center)    (v4_22.center) edge[-] (v5_23.center)    (v4_23.center) edge[-] (v5_23.center)    (v4_23.center) edge[-] (v5_24.center)    (v4_24.center) edge[-] (v5_24.center)    (v4_24.center) edge[-] (v5_25.center)    (v4_25.center) edge[-] (v5_25.center)    (v4_25.center) edge[-] (v5_26.center)    (v4_26.center) edge[-] (v5_26.center)    (v4_26.center) edge[-] (v5_27.center)    (v4_27.center) edge[-] (v5_27.center)    (v4_27.center) edge[-] (v5_28.center)    (v4_28.center) edge[-] (v5_28.center)    (v4_28.center) edge[-] (v5_29.center)    (v4_29.center) edge[-] (v5_29.center)    (v4_29.center) edge[-] (v5_30.center)    (v4_30.center) edge[-] (v5_30.center)    (v4_30.center) edge[-] (v5_31.center)    (v4_31.center) edge[-] (v5_31.center)    (v4_31.center) edge[-] (v5_0.center)    ;
\draw[ec1]
(v0_0.center) edge[-] (v1_0.center)    (v0_0.center) edge[-] (v1_1.center)    (v0_0.center) edge[-] (v1_2.center)    (v0_0.center) edge[-] (v1_3.center)    (v0_0.center) edge[-] (v1_4.center)    (v0_0.center) edge[-] (v1_5.center)    (v0_0.center) edge[-] (v1_6.center)    (v0_0.center) edge[-] (v1_7.center)    (v0_0.center) edge[-] (v2_0.center)    (v0_0.center) edge[-] (v2_1.center)    (v0_0.center) edge[-] (v2_2.center)    (v0_0.center) edge[-] (v2_3.center)    (v0_0.center) edge[-] (v2_4.center)    (v0_0.center) edge[-] (v2_5.center)    (v0_0.center) edge[-] (v2_6.center)    (v0_0.center) edge[-] (v2_7.center)    (v2_0.center) edge[-] (v2_1.center)    (v2_0.center) edge[-] (v2_7.center)    (v2_0.center) edge[-] (v3_0.center)    (v2_0.center) edge[-] (v3_7.center)    (v2_0.center) edge[-] (v4_0.center)    (v2_0.center) edge[-] (v4_1.center)    (v2_0.center) edge[-] (v4_30.center)    (v2_0.center) edge[-] (v4_31.center)    (v2_0.center) edge[-] (v5_0.center)    (v2_0.center) edge[-] (v5_1.center)    (v2_0.center) edge[-] (v5_2.center)    (v2_0.center) edge[-] (v5_30.center)    (v2_0.center) edge[-] (v5_31.center)    (v2_1.center) edge[-] (v2_2.center)    (v2_1.center) edge[-] (v3_0.center)    (v2_1.center) edge[-] (v3_1.center)    (v2_1.center) edge[-] (v4_2.center)    (v2_1.center) edge[-] (v4_3.center)    (v2_1.center) edge[-] (v4_4.center)    (v2_1.center) edge[-] (v4_5.center)    (v2_1.center) edge[-] (v5_2.center)    (v2_1.center) edge[-] (v5_3.center)    (v2_1.center) edge[-] (v5_4.center)    (v2_1.center) edge[-] (v5_5.center)    (v2_1.center) edge[-] (v5_6.center)    (v2_2.center) edge[-] (v2_3.center)    (v2_2.center) edge[-] (v3_1.center)    (v2_2.center) edge[-] (v3_2.center)    (v2_2.center) edge[-] (v4_6.center)    (v2_2.center) edge[-] (v4_7.center)    (v2_2.center) edge[-] (v4_8.center)    (v2_2.center) edge[-] (v4_9.center)    (v2_2.center) edge[-] (v5_10.center)    (v2_2.center) edge[-] (v5_6.center)    (v2_2.center) edge[-] (v5_7.center)    (v2_2.center) edge[-] (v5_8.center)    (v2_2.center) edge[-] (v5_9.center)    (v2_3.center) edge[-] (v2_4.center)    (v2_3.center) edge[-] (v3_2.center)    (v2_3.center) edge[-] (v3_3.center)    (v2_3.center) edge[-] (v4_10.center)    (v2_3.center) edge[-] (v4_11.center)    (v2_3.center) edge[-] (v4_12.center)    (v2_3.center) edge[-] (v4_13.center)    (v2_3.center) edge[-] (v5_10.center)    (v2_3.center) edge[-] (v5_11.center)    (v2_3.center) edge[-] (v5_12.center)    (v2_3.center) edge[-] (v5_13.center)    (v2_3.center) edge[-] (v5_14.center)    (v2_4.center) edge[-] (v2_5.center)    (v2_4.center) edge[-] (v3_3.center)    (v2_4.center) edge[-] (v3_4.center)    (v2_4.center) edge[-] (v4_14.center)    (v2_4.center) edge[-] (v4_15.center)    (v2_4.center) edge[-] (v4_16.center)    (v2_4.center) edge[-] (v4_17.center)    (v2_4.center) edge[-] (v5_14.center)    (v2_4.center) edge[-] (v5_15.center)    (v2_4.center) edge[-] (v5_16.center)    (v2_4.center) edge[-] (v5_17.center)    (v2_4.center) edge[-] (v5_18.center)    (v2_5.center) edge[-] (v2_6.center)    (v2_5.center) edge[-] (v3_4.center)    (v2_5.center) edge[-] (v3_5.center)    (v2_5.center) edge[-] (v4_18.center)    (v2_5.center) edge[-] (v4_19.center)    (v2_5.center) edge[-] (v4_20.center)    (v2_5.center) edge[-] (v4_21.center)    (v2_5.center) edge[-] (v5_18.center)    (v2_5.center) edge[-] (v5_19.center)    (v2_5.center) edge[-] (v5_20.center)    (v2_5.center) edge[-] (v5_21.center)    (v2_5.center) edge[-] (v5_22.center)    (v2_6.center) edge[-] (v2_7.center)    (v2_6.center) edge[-] (v3_5.center)    (v2_6.center) edge[-] (v3_6.center)    (v2_6.center) edge[-] (v4_22.center)    (v2_6.center) edge[-] (v4_23.center)    (v2_6.center) edge[-] (v4_24.center)    (v2_6.center) edge[-] (v4_25.center)    (v2_6.center) edge[-] (v5_22.center)    (v2_6.center) edge[-] (v5_23.center)    (v2_6.center) edge[-] (v5_24.center)    (v2_6.center) edge[-] (v5_25.center)    (v2_6.center) edge[-] (v5_26.center)    (v2_7.center) edge[-] (v3_6.center)    (v2_7.center) edge[-] (v3_7.center)    (v2_7.center) edge[-] (v4_26.center)    (v2_7.center) edge[-] (v4_27.center)    (v2_7.center) edge[-] (v4_28.center)    (v2_7.center) edge[-] (v4_29.center)    (v2_7.center) edge[-] (v5_26.center)    (v2_7.center) edge[-] (v5_27.center)    (v2_7.center) edge[-] (v5_28.center)    (v2_7.center) edge[-] (v5_29.center)    (v2_7.center) edge[-] (v5_30.center)    (v3_0.center) edge[-] (v5_2.center)    (v3_1.center) edge[-] (v5_6.center)    (v3_2.center) edge[-] (v5_10.center)    (v3_3.center) edge[-] (v5_14.center)    (v3_4.center) edge[-] (v5_18.center)    (v3_5.center) edge[-] (v5_22.center)    (v3_6.center) edge[-] (v5_26.center)    (v3_7.center) edge[-] (v5_30.center)    (v5_0.center) edge[-] (v5_1.center)    (v5_0.center) edge[-] (v5_31.center)    (v5_1.center) edge[-] (v5_2.center)    (v5_10.center) edge[-] (v5_11.center)    (v5_10.center) edge[-] (v5_9.center)    (v5_11.center) edge[-] (v5_12.center)    (v5_12.center) edge[-] (v5_13.center)    (v5_13.center) edge[-] (v5_14.center)    (v5_14.center) edge[-] (v5_15.center)    (v5_15.center) edge[-] (v5_16.center)    (v5_16.center) edge[-] (v5_17.center)    (v5_17.center) edge[-] (v5_18.center)    (v5_18.center) edge[-] (v5_19.center)    (v5_19.center) edge[-] (v5_20.center)    (v5_2.center) edge[-] (v5_3.center)    (v5_20.center) edge[-] (v5_21.center)    (v5_21.center) edge[-] (v5_22.center)    (v5_22.center) edge[-] (v5_23.center)    (v5_23.center) edge[-] (v5_24.center)    (v5_24.center) edge[-] (v5_25.center)    (v5_25.center) edge[-] (v5_26.center)    (v5_26.center) edge[-] (v5_27.center)    (v5_27.center) edge[-] (v5_28.center)    (v5_28.center) edge[-] (v5_29.center)    (v5_29.center) edge[-] (v5_30.center)    (v5_3.center) edge[-] (v5_4.center)    (v5_30.center) edge[-] (v5_31.center)    (v5_4.center) edge[-] (v5_5.center)    (v5_5.center) edge[-] (v5_6.center)    (v5_6.center) edge[-] (v5_7.center)    (v5_7.center) edge[-] (v5_8.center)    (v5_8.center) edge[-] (v5_9.center)    ;
\path \foreach \n in {v0_0, v2_0, v2_1, v2_2, v2_3, v2_4, v2_5, v2_6, v2_7, v5_0, v5_1, v5_2, v5_3, v5_4, v5_5, v5_6, v5_7, v5_8, v5_9, v5_10, v5_11, v5_12, v5_13, v5_14, v5_15, v5_16, v5_17, v5_18, v5_19, v5_20, v5_21, v5_22, v5_23, v5_24, v5_25, v5_26, v5_27, v5_28, v5_29, v5_30, v5_31} { (\n) node [fill=vc0, draw=black, circle, inner sep=1pt, outer sep=0pt] {}};
\path \foreach \n in {v1_0, v1_1, v1_2, v1_3, v1_4, v1_5, v1_6, v1_7, v3_0, v3_1, v3_2, v3_3, v3_4, v3_5, v3_6, v3_7, v4_0, v4_1, v4_2, v4_3, v4_4, v4_5, v4_6, v4_7, v4_8, v4_9, v4_10, v4_11, v4_12, v4_13, v4_14, v4_15, v4_16, v4_17, v4_18, v4_19, v4_20, v4_21, v4_22, v4_23, v4_24, v4_25, v4_26, v4_27, v4_28, v4_29, v4_30, v4_31} { (\n) node [fill=vc1, draw=black, circle, inner sep=1pt, outer sep=0pt] {}};
\end{tikzpicture}
 &
\definecolor{vc0}{rgb}{1,1,1}
\definecolor{ec0}{rgb}{0,0,0}
\begin{tikzpicture}[xscale = 1.3, yscale=1.3]
\foreach \n/\x/\y in {v5_18/-1.85/-0.77, v5_19/-1.66/-1.11, v5_30/1.85/-0.77, v5_31/1.96/-0.39, v5_10/-0.77/1.85, v5_11/-1.11/1.66, v5_12/-1.41/1.41, v5_13/-1.66/1.11, v5_14/-1.85/0.77, v5_15/-1.96/0.39, v5_16/-2.0/0.0, v0_0/0.0/0.0, v5_21/-1.11/-1.66, v5_20/-1.41/-1.41, v5_23/-0.39/-1.96, v5_22/-0.77/-1.85, v5_25/0.39/-1.96, v5_24/0.0/-2.0, v5_27/1.11/-1.66, v5_26/0.77/-1.85, v5_29/1.66/-1.11, v5_28/1.41/-1.41, v5_2/1.85/0.77, v5_3/1.66/1.11, v5_17/-1.96/-0.39, v5_0/2.0/0.0, v5_1/1.96/0.39, v5_6/0.77/1.85, v5_7/0.39/1.96, v5_8/0.0/2.0, v5_9/-0.39/1.96, v5_4/1.41/1.41, v5_5/1.11/1.66, v2_1/0.71/0.71, v2_0/1.0/0.0, v2_3/-0.71/0.71, v2_2/0.0/1.0, v2_5/-0.71/-0.71, v2_4/-1.0/0.0, v2_7/0.71/-0.71, v2_6/0.0/-1.0} {    \node (\n) at (\x,\y) {}; };

\draw[ec0]
(v0_0.center) edge[-] (v2_0.center)    (v0_0.center) edge[-] (v2_1.center)    (v0_0.center) edge[-] (v2_2.center)    (v0_0.center) edge[-] (v2_3.center)    (v0_0.center) edge[-] (v2_4.center)    (v0_0.center) edge[-] (v2_5.center)    (v0_0.center) edge[-] (v2_6.center)    (v0_0.center) edge[-] (v2_7.center)    (v2_0.center) edge[-] (v2_1.center)    (v2_0.center) edge[-] (v2_7.center)    (v2_0.center) edge[-] (v5_0.center)    (v2_0.center) edge[-] (v5_1.center)    (v2_0.center) edge[-] (v5_2.center)    (v2_0.center) edge[-] (v5_30.center)    (v2_0.center) edge[-] (v5_31.center)    (v2_1.center) edge[-] (v2_2.center)    (v2_1.center) edge[-] (v5_2.center)    (v2_1.center) edge[-] (v5_3.center)    (v2_1.center) edge[-] (v5_4.center)    (v2_1.center) edge[-] (v5_5.center)    (v2_1.center) edge[-] (v5_6.center)    (v2_2.center) edge[-] (v2_3.center)    (v2_2.center) edge[-] (v5_10.center)    (v2_2.center) edge[-] (v5_6.center)    (v2_2.center) edge[-] (v5_7.center)    (v2_2.center) edge[-] (v5_8.center)    (v2_2.center) edge[-] (v5_9.center)    (v2_3.center) edge[-] (v2_4.center)    (v2_3.center) edge[-] (v5_10.center)    (v2_3.center) edge[-] (v5_11.center)    (v2_3.center) edge[-] (v5_12.center)    (v2_3.center) edge[-] (v5_13.center)    (v2_3.center) edge[-] (v5_14.center)    (v2_4.center) edge[-] (v2_5.center)    (v2_4.center) edge[-] (v5_14.center)    (v2_4.center) edge[-] (v5_15.center)    (v2_4.center) edge[-] (v5_16.center)    (v2_4.center) edge[-] (v5_17.center)    (v2_4.center) edge[-] (v5_18.center)    (v2_5.center) edge[-] (v2_6.center)    (v2_5.center) edge[-] (v5_18.center)    (v2_5.center) edge[-] (v5_19.center)    (v2_5.center) edge[-] (v5_20.center)    (v2_5.center) edge[-] (v5_21.center)    (v2_5.center) edge[-] (v5_22.center)    (v2_6.center) edge[-] (v2_7.center)    (v2_6.center) edge[-] (v5_22.center)    (v2_6.center) edge[-] (v5_23.center)    (v2_6.center) edge[-] (v5_24.center)    (v2_6.center) edge[-] (v5_25.center)    (v2_6.center) edge[-] (v5_26.center)    (v2_7.center) edge[-] (v5_26.center)    (v2_7.center) edge[-] (v5_27.center)    (v2_7.center) edge[-] (v5_28.center)    (v2_7.center) edge[-] (v5_29.center)    (v2_7.center) edge[-] (v5_30.center)    (v5_0.center) edge[-] (v5_1.center)    (v5_0.center) edge[-] (v5_31.center)    (v5_1.center) edge[-] (v5_2.center)    (v5_10.center) edge[-] (v5_11.center)    (v5_10.center) edge[-] (v5_9.center)    (v5_11.center) edge[-] (v5_12.center)    (v5_12.center) edge[-] (v5_13.center)    (v5_13.center) edge[-] (v5_14.center)    (v5_14.center) edge[-] (v5_15.center)    (v5_15.center) edge[-] (v5_16.center)    (v5_16.center) edge[-] (v5_17.center)    (v5_17.center) edge[-] (v5_18.center)    (v5_18.center) edge[-] (v5_19.center)    (v5_19.center) edge[-] (v5_20.center)    (v5_2.center) edge[-] (v5_3.center)    (v5_20.center) edge[-] (v5_21.center)    (v5_21.center) edge[-] (v5_22.center)    (v5_22.center) edge[-] (v5_23.center)    (v5_23.center) edge[-] (v5_24.center)    (v5_24.center) edge[-] (v5_25.center)    (v5_25.center) edge[-] (v5_26.center)    (v5_26.center) edge[-] (v5_27.center)    (v5_27.center) edge[-] (v5_28.center)    (v5_28.center) edge[-] (v5_29.center)    (v5_29.center) edge[-] (v5_30.center)    (v5_3.center) edge[-] (v5_4.center)    (v5_30.center) edge[-] (v5_31.center)    (v5_4.center) edge[-] (v5_5.center)    (v5_5.center) edge[-] (v5_6.center)    (v5_6.center) edge[-] (v5_7.center)    (v5_7.center) edge[-] (v5_8.center)    (v5_8.center) edge[-] (v5_9.center)    ;

\path \foreach \n in {v0_0, v2_0, v2_1, v2_2, v2_3, v2_4, v2_5, v2_6, v2_7, v5_0, v5_1, v5_10, v5_11, v5_12, v5_13, v5_14, v5_15, v5_16, v5_17, v5_18, v5_19, v5_2, v5_20, v5_21, v5_22, v5_23, v5_24, v5_25, v5_26, v5_27, v5_28, v5_29, v5_3, v5_30, v5_31, v5_4, v5_5, v5_6, v5_7, v5_8, v5_9} { (\n) node [fill=vc0, draw=black, circle, inner sep=1pt, outer sep=0pt] {}};
\end{tikzpicture} \\
\text{Tessellation }T_0 & \text{Tessellation }T_1
\end{matrix}
\]
\caption[Tessellations with valence sequences $\vs{16,16,3}$ and $\vs{8,8,8}$.]{\label{fig:pp3-center}Tessellation $T_0$ is face-homogeneous with valence sequence $\vs{16,16,3}$; $T_1$ is the tessellation remaining when all 3-valent vertices and their incident edges are removed from $T_0$ and is face-homogeneous with valence sequence $\vs{8,8,8}$. Edges belonging to both the coronas $\langle F_{0,n}\rangle$ and $\langle F_{0,n+1}\rangle$ of $T_0$ are darkened.}
\end{figure}
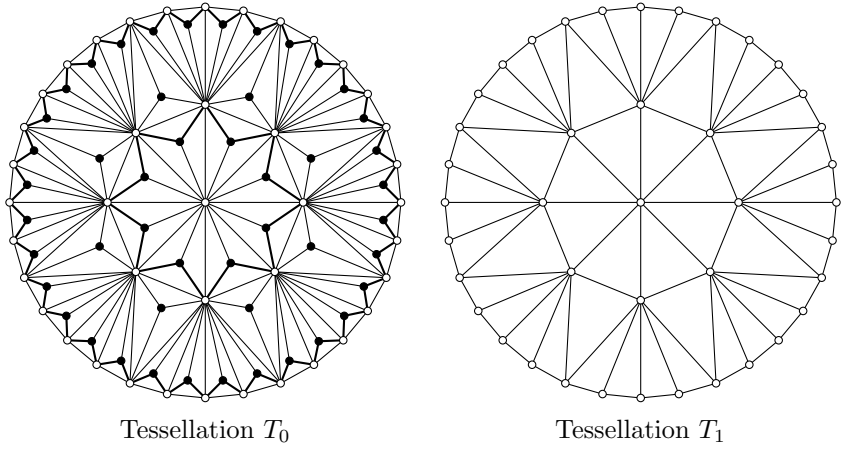

The corona boundaries in $T_0$ and $T_1$ are in close correspondence but are not identical, as $T_1$ is uniformly concentric while $T_0$ is not. Taking $F_{i,n}$ to be the $i^\textrm{th}$ corona of $T_i$, we see from the \Cref{fig:pp3-center} that there are two types of faces in $T_1$: those of type $\ft_1$ are incident with a single vertex in $U_{n-1}$ and two vertices in $U_n$; those of type $\ft_2$ are incident with two vertices in $U_{n-1}$ and one vertex in $U_n$. Each face of type $\ft_1$ in $T_1$ corresponds to exactly two faces of $T_0$, and each face of type $\ft_2$ in $T_1$ corresponds to exactly four faces of $T_0$. Since the transition matrix of $T_1$ is
\[ M = \begin{bmatrix}\frac{p}2-4 & -1 \\ 1 & 0\end{bmatrix}, \]
we have that
\[ \abs{F_{0,n}} = \begin{bmatrix}2\\4\end{bmatrix}\cdot M^{n-1}\begin{bmatrix}p/2\\0\end{bmatrix} \]
via the standard dot product.
So we have by \Cref{thm:GrowthOf-ppp},
\begin{align*}\gam(T_0) = \gam(T_1) &= \frac12\left[\frac{p}2-4+\sqrt{\left(\frac{p}2-4\right)^2-4}\right] \\
	&= \frac14\left[p-8+\sqrt{(p-8)^2-16}\right].
\end{align*}
\end{proof}
\end{prop}

\begin{prop}\label{thm:GrowthOf-ppq} If $\sig=\vs{p,p,q}$ for $q\geq4$, then $p\geq 6$ and is even, and
\[	\gam(T) = \frac14\left[\left(a+\sqrt{a^2+4b+8}\right)+\sqrt{4b-8+a\left(a+\sqrt{a^2+4b+8}\right)}\right] \]
where $a=\frac{p-4}2$ and $b=\frac{(p-4)(q-4)-4}2$.
\begin{proof} 
Since $\sigma$ is realizable, $p$ must be even and $\frac2p+\frac1q<\frac12$.  Hence $p\geq6$.  The offspring diagrams for this vertex sequence areshown in \Cref{fig:ppq-offspring}.
\begin{figure}[htbp]
\includegraphics{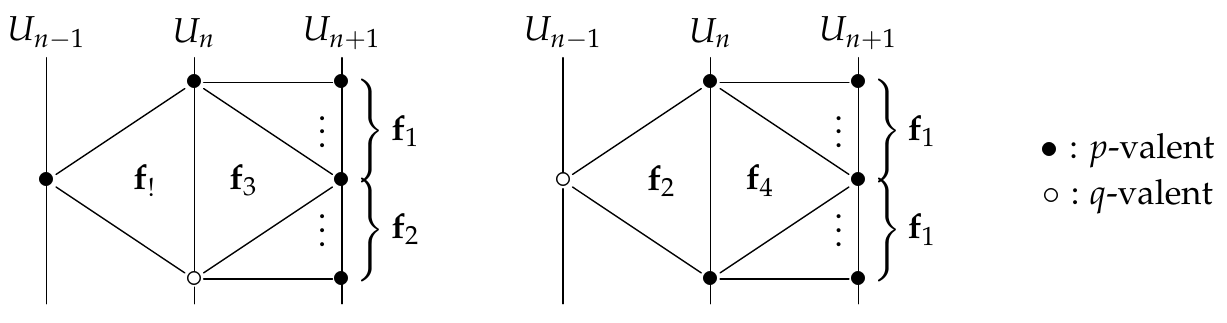}
\caption[Offspring diagrams for $\vs{p,p,q}$.]{\label{fig:ppq-offspring}Offspring diagrams for a face-homogeneous tessellation with valence sequence $\vs{p,p,q}$ and $q\geq4$.}
\end{figure}

We note that the presence of faces of types ${\ft_2}$ or ${\ft_3}$ in $F_n$ decreases by one the number of faces of each of  types ${\ft_0}$ and ${\ft_1}$ in $F_{n+1}$. This gives the transition matrix
\[ M= \left[\begin{matrix}
\frac12(p-4) & p-4 & -1 & 0 \\
\frac12(q-4) & 0 & 0 & -1 \\
1 & 0 & 0 & 0 \\
0 & 1 & 0 & 0 
\end{matrix} \right]. \]
The characteristic polynomial of $M$ is palindromic:
\[ \chi(z) = z^4 - \left(\frac{p-4}2\right) z^3 - \left(\frac{(p-4)(q-4)-4}2\right) z^2 - \left(\frac{p-4}2\right)  z + 1. \]
Hence by \Cref{rem:PalinPoly}, the growth rate of $T$ is as stated.
\end{proof}
\end{prop}

The following proposition proved in \cite{GFHT} as {Example 2.11} is restated here for the sake of completeness.

\begin{prop}\label{thm:GrowthOf-4pq}
If $\sig=[4,p,q]$, then
\[
\gam(T)=\frac14\sqrt{[2(p-4)(q-4)-16]+2\sqrt{(p-4)^2(q-4)^2-16(p-4)(q-4)}}.
\]
\hfill\qedsymbol
\end{prop}

Valence sequences of the form $\vs{p,q,r}$ introduce a computational difficulty, since the associated characteristic polynomial has degree 6 and is not solvable by radicals. Hence we are unable to determine a general formula for the growth rate of a face-homogeneous tessellation with valence sequence $\vs{p,q,r}$.  However, subject to the condition that the angle excess be positive, the least vertex sequence (in the partial order on cyclic sequences) having this form is $\vs{6,8,10}$. Thus if $T_0$ is a face-homogeneous tessellation with valence sequence $\vs{6,8,10}$ and $T$ is a face-homogeneous tessellation with valence sequence $\vs{p,q,r}\geq \vs{6,8,10}$, then by \cite{GFHT}{Theorem 2.7} we have $\gam(T_0)\leq\gam(T)$, since realizable valence sequences of length 3 are monomorphic.

This method of determining the rate of growth of the least element of a class of valence sequences as a lower bound for growth within the class will be used throughout this Appendix whenever the characteristic polynomial of a transition matrix is not ready solvable by radicals.

\begin{prop}\label{thm:GrowthOf-pqr} If $\sig=[p,q,r]$, then $\gam(T)\geq 3.4789$.
\begin{proof} 
As the terms are presumed to be distinct, they are all even.  By \Cref{thm:GrowthOf-4pq}, we may assume that they are all $\geq6$.  Then $T$ has offspring diagrams as shown in \Cref{fig:pqr-offspring}.
\begin{figure}[htbp]
\includegraphics[width=\linewidth]{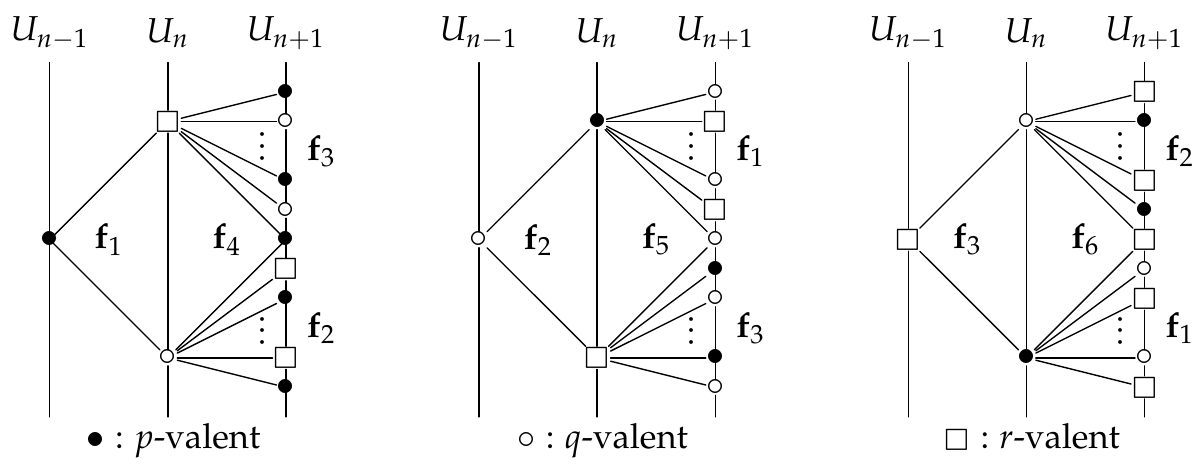}
\caption[Offspring diagrams for $\vs{p,q,r}$.]{\label{fig:pqr-offspring}Offspring diagrams for a face-homogeneous tessellation with valence sequence $\vs{p,q,r}$.}
\end{figure}

Thus the transition matrix is 
\[ M = \left[ \begin{matrix}
0 & \frac12(p-4) & \frac12(p-4) & -1 & 0 & 0 \\
\frac12(q-4) & 0 & \frac12(q-4) & 0 & -1 & 0 \\
\frac12(r-4) & \frac12(r-4) & 0 & 0 & 0 & -1 \\
1 & 0 & 0 & 0 & 0 & 0 \\
0 & 1 & 0 & 0 & 0 & 0 \\
0 & 0 & 1 & 0 & 0 & 0
\end{matrix}\right]\]
giving characteristic polynomial
\[ \chi(z) = z^6 -  az^4 - bz^3  - az^2+1, \]
with
\begin{align*}
	a&=\frac14\left( pq+pr+qr-8p-8q-8r+36 \right) \text{~and}\\
	b &=\frac14(p-4)(q-4)(r-4).
\end{align*}

The least valence sequence in this class is $\sig_1=\vs{6,8,10}$. If $T_1$ has valence sequence $\sig_1$, then numerical approximation of eigenvalues of $M$ gives that $\gam(T_1)\approx 3.4789$. The argument in the paragraph preceding this proposition gives that, if $T$ is a tessellation with valence sequence $\sig\geq\sig_1$, then $\gam(T)\geq3.4789.$
\end{proof}
\end{prop}


\section[Face-homogeneous Tessellations of Covalence 4]{Classification of Face-Homogeneous Tessellations with Valence Sequences of Length 4}\label{ch:4cov}

The ten forms of valence sequences of length $4$ are the following: $\vs{p,p,p,q}$ and $\vs{p,p,q,r}$, which are polymorphic, and $\vs{p,p,p,p}$, $\vs{p,p,q,q}$, $\vs{3,p,3,p}$, $\vs{p,q,p,q}$, $\vs{p,3,p,4}$, $\vs{p,3,p,r}$, $\vs{p,q,p,r}$, and $\vs{p,q,r,s}$, which are monomorphic.

\begin{prop}\label{thm:4cov-poly} The valence sequences $\sig_1=\vs{p,p,p,q}$ and $\sig_2=\vs{p,p,q,r}$ are polymorphic for all permitted distinct values of $p,q,r$ for which $\eta(\sig_1),\eta(\sig_2)>~0$.
\begin{proof}
If $p,q\geq 4$, then $\vs{p,p,p,q}$ is polymorphic by \cite{GFHT}{Proposition 3.1}. If $\sig_2=\vs{p,p,q,r}$ is realizable and $\eta(\sig_2)>0$, then the neighbors of an $r$-valent vertex must be alternately $p$-valent and $q$-valent; hence $r$ must be even. Similarly, $q$ must be even. The valences of neighbors of a $p$-valent vertex must be alternately $p$ and either $q$ or $r$. So $p$ is also even, and thus $p,q,r\geq 4$ as well. By \cite{GFHT}{Proposition 3.1}, $\sig_2$ is polymorphic.

The valence sequence $[p,p,p,3]$ is not subject to \cite{GFHT}{Proposition 3.1}.  In this case, we must have $p\geq 5$, and we refer to \Cref{fig:ppp3-ambig}. Since the types of the faces indicated in the figure with a question mark ({\bf ?}) are dependent upon the particular embedding rule chosen for $T$, $\vs{p,p,p,3}$ is polymorphic.
\begin{figure}[htbp]
\includegraphics{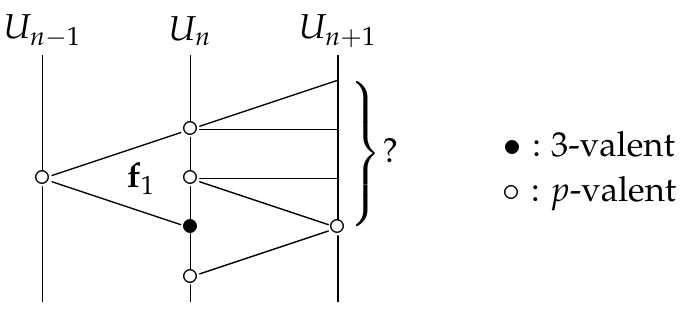}
\caption[Offspring diagrams showing $\vs{p,p,p,3}$ is polymorphic.]{\label{fig:ppp3-ambig}Offspring of a wedge of type ${\ft_0}$ in a face-homogeneous tessellation with valence sequence $\vs{p,p,p,3}$.}
\end{figure}
\end{proof}
\end{prop}

\begin{prop} If $\sig=\vs{p,p,p,p}$, then  $\sig$ is monomorphic, and
\[\gam(T) = (p-3)+\sqrt{(p-3)^2-1}. \]
\begin{proof} Clearly $p\geq 5$.
Since $T$ is edge-homogeneous with edge-symbol $\esym{p,p}{4,4}$, we apply \cite{GFHT}{Proposition 1.17}(1).
\end{proof}\end{prop}

\begin{prop}\label{thm:GrowthOf-ppqq} If $\sig=\vs{p,p,q,q}$, then $\sig$ is monomorphic, $p$ and $q$ are even, and
\[ \gam(T)=\frac18+\left(p+q+a-8+\sqrt{2(p^2+q^2)+36pq-112(p+q)+2a(p+q-8)}\right), \] where $a=\sqrt{p^2+34pq+q^2-96p-96q+256}$.
\begin{proof} The argument that $p$ and $q$ are even is similar to that of \Cref{thm:4cov-poly} above.
We see from the offspring diagrams in \Cref{fig:ppqq-offspring} that $\sig$ is monomorphic, 
\begin{figure}[htbp]
\includegraphics[width=\linewidth]{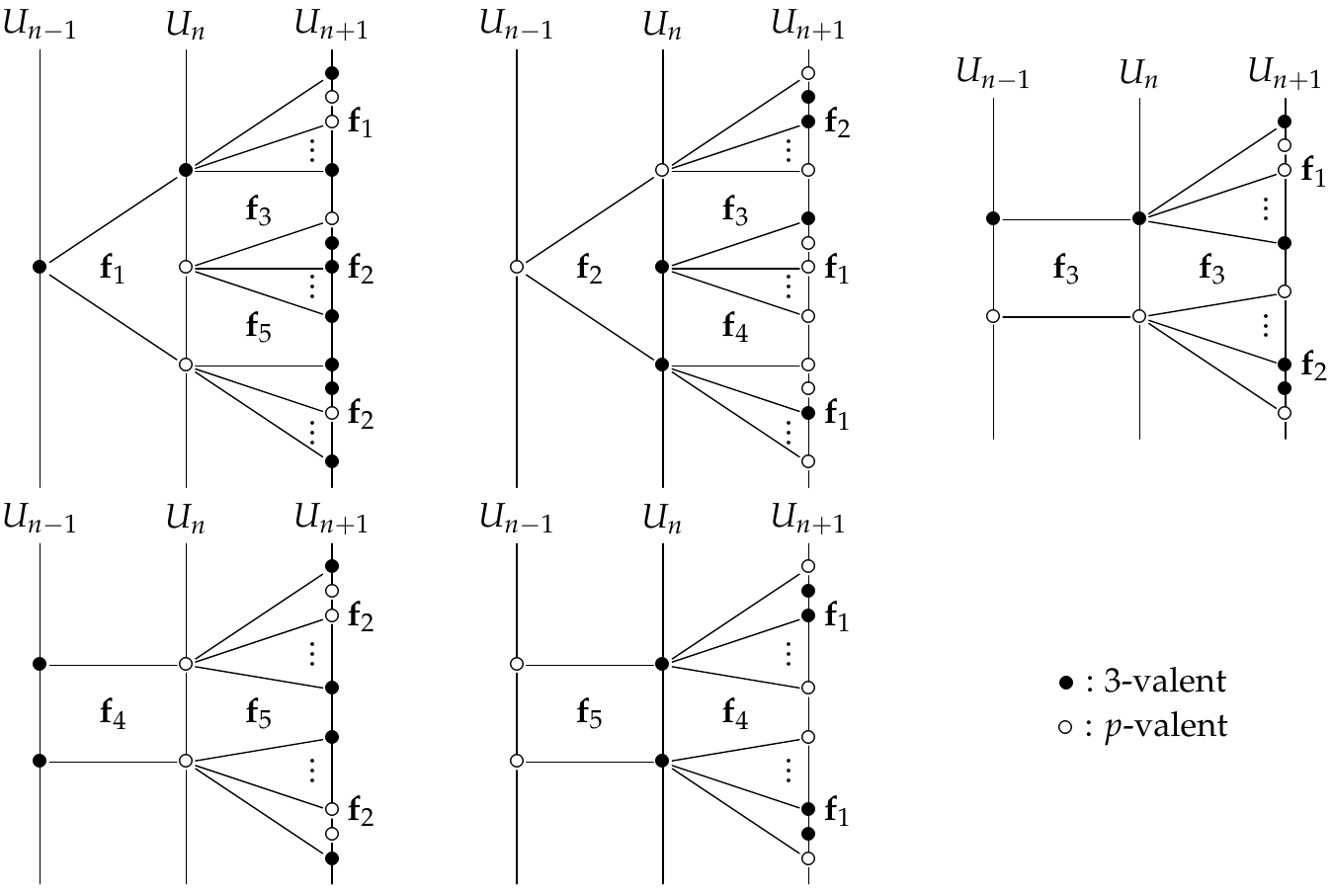}
\caption[Offspring diagrams for $\vs{p,p,q,q}$.]{\label{fig:ppqq-offspring}Offspring diagrams for a face-homogeneous tessellation with valence sequence $\vs{p,p,q,q}$.}
\end{figure}
yielding the transition matrix 
 \[ M = \left[\begin{matrix}
\frac{p-4}2 & \frac{3p-10}2 & \frac{p-4}2 & 0 & p-4 \\
\frac{3q-10}2 & \frac{q-4}2 & \frac{q-4}2 & q-4 & 0 \\
1 & 1 & 1 & 0 & 0 \\
0 & 1 & 0 & 0 & 1 \\
1 & 0 & 0 & 1 & 0
\end{matrix}\right]. \]
The characteristic polynomial of $M$ is
\[ \chi(z) = \frac12(z-1)(2z^4-az^3-bz^2-az+2), \]
where $a = p+q-8$ and $b = 4pq-10p-10q+20$.
Thus the growth rate of $T$ is as stated.
\end{proof}\end{prop}

\begin{prop}\label{thm:GrowthOf-3p3p} If  $\sig=\vs{p,q,p,q}$, then  $\sig$ is monomorphic, and
\[\gam(T)=\frac12\left(p-4+\sqrt{(p-4)^2-4}\right)\]
if $p=3$, while
\[\gam(T)=\frac12\left(p+q-6+\sqrt{(p+q-6)^2-4}\right)\]
if $p,q\geq4$.
\begin{proof} Since $\eta(\sig)>0$, we have $q\geq 7$ when $p=3$.  Otherwise we may assume $q>p\geq4$. 
Since $T$ is edge-homogeneous with edge-symbol $\esym{p,q}{4,4}$, we apply \cite{GFHT}{Proposition 1.17}(2) when $q=3$ and apply \cite{GFHT}{Proposition 1.17}(1) $p,q\geq~4$.
\end{proof}
\end{prop}

\begin{prop}\label{thm:GrowthOf-p3p4} If $\sig=\vs{3,p,4,p}$, then $\sig$ is monomorphic with $p\geq 5$ but not uniformly concentric, and
\[ \gam(T) = \frac14\left[
	p-3 + \sqrt{p^2-4p+1} + \sqrt{2p^2-10p-6+2(p-3)\sqrt{p^2-4p+1}}
	\right]. \]
\begin{proof}
The valence sequence $\sig=\vs{3,p,4,p}$ is not uniformly concentric, as seen in the offspring diagrams in \Cref{fig:p3p4-offspring} by the existence of type ${\ft_6}$ faces.  (This valence sequence shows that the sufficient conditions for non-concentricity in \cite{GFHT}{Theorem 1.4} are not necessary.)  
Two faces of type $\ft_1$ in $F_n$ are incident with a common $4$-valent vertex in $U_n$ whenever they are both adjacent to a face of type $\ft_6$ in $F_n$; the effect of this is that the $4$-valent vertex  is adjacent to no vertex in $U_{n+1}$. Hence the induced graph $\induced{U_{n+1}}$ must contain a pendant vertex whenever $F_n$ contains a face of type $\ft_6$. The inclusion of an $\ft_6$ face then removes two $\ft_4$ faces from $F_{n+1}$ to be replaced with an $\ft_7$ face, and converts two $\ft_1$ faces into $\ft_5$ faces. The faces of type $\ft_7$ in $F_n$ generate no offspring in $F_{n+1}$.

\begin{figure}[htbp]
\includegraphics[width=\linewidth]{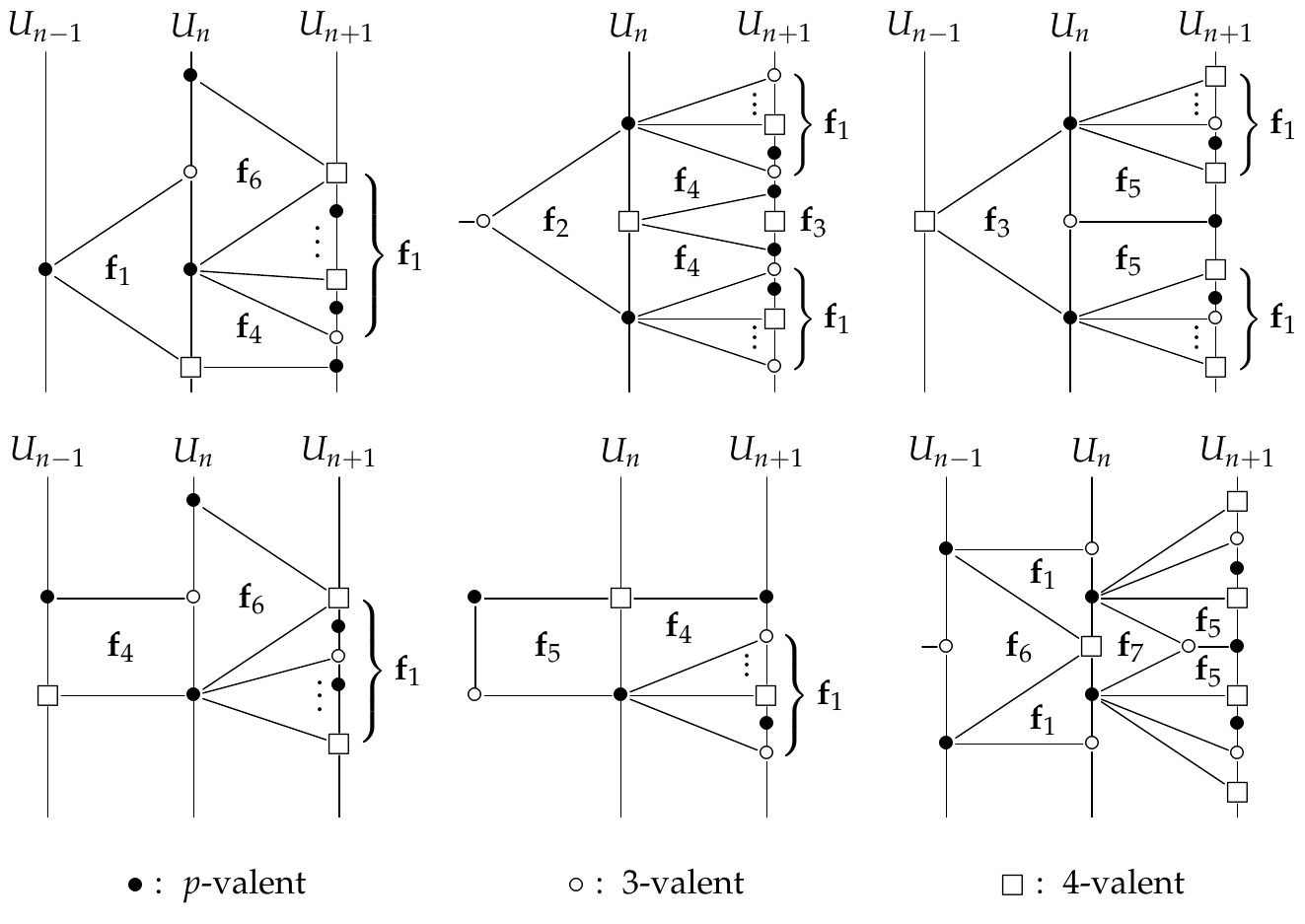}
\caption[Offspring diagrams for $\vs{3,p,4,p}$.]{\label{fig:p3p4-offspring}Offspring diagrams for a tessellation with valence sequence $\vs{3,p,4,p}$.}
\end{figure}
Subject to these considerations, the transition matrix $M$ of $T$ is
\[ M = \left[\begin{matrix}
p-3 & p-4 & p-4 & \frac12(p-4) & \frac12(p-4) & -2 & 0 \\
0 & 0 & 0 & 0 & 0 & 0 & 0 \\
0 & 1 & 0 & 0 & 0 & 0 & 0 \\
1 & 2 & 0 & 0 & 1 & -2 & 0 \\
0 & 0 & 2 & 0 & 0 & 2 & 0 \\
\frac12 & 0 & 0 & \frac12 & 0 & 0 & 0 \\
0 & 0 & 0 & 0 & 0 & 1 & 0
\end{matrix}\right].\] The characteristic polynomial of $M$ has the factorization 
\[ \chi(z) = z^3\left(z^4-(p-3)z^3-\frac12(p-8)z^2-(p-3)z+1\right). \]
As the quartic factor is palindromic, we apply {Remark 1.1}; the maximum modulus of an eigenvalue of $M$ can be computed by radicals and is given by
\[ \Lam  = \frac14\left[
	p-3 + \sqrt{p^2-4p+1} + \sqrt{2p^2-10p-6+2(p-3)\sqrt{p^2-4p+1}}
	\right].
\]
By \cite{GFHT}{Theorem 2.3} and \cite{GFHT}{Proposition 1.14}, the result holds.
\end{proof}
\end{prop}

\begin{prop}\label{thm:GrowthOf-3prp} If $\sig=\vs{3,p,q,p}$  with $q\geq 5$, then $\sig$ is monomorphic and
\[\gam(T) = \frac14\left(p-4+a+\sqrt{2p^2+2pq-18p-4q+16+a(2p-8)}\right),\]
where $a=\sqrt{p^2+2pq-10p-4q+16}$.
\begin{proof}
The offspring diagrams given in \Cref{fig:p3pr-offspring} show that $\sig$ is monomorphic.
\begin{figure}[htbp]
\includegraphics[width=\linewidth]{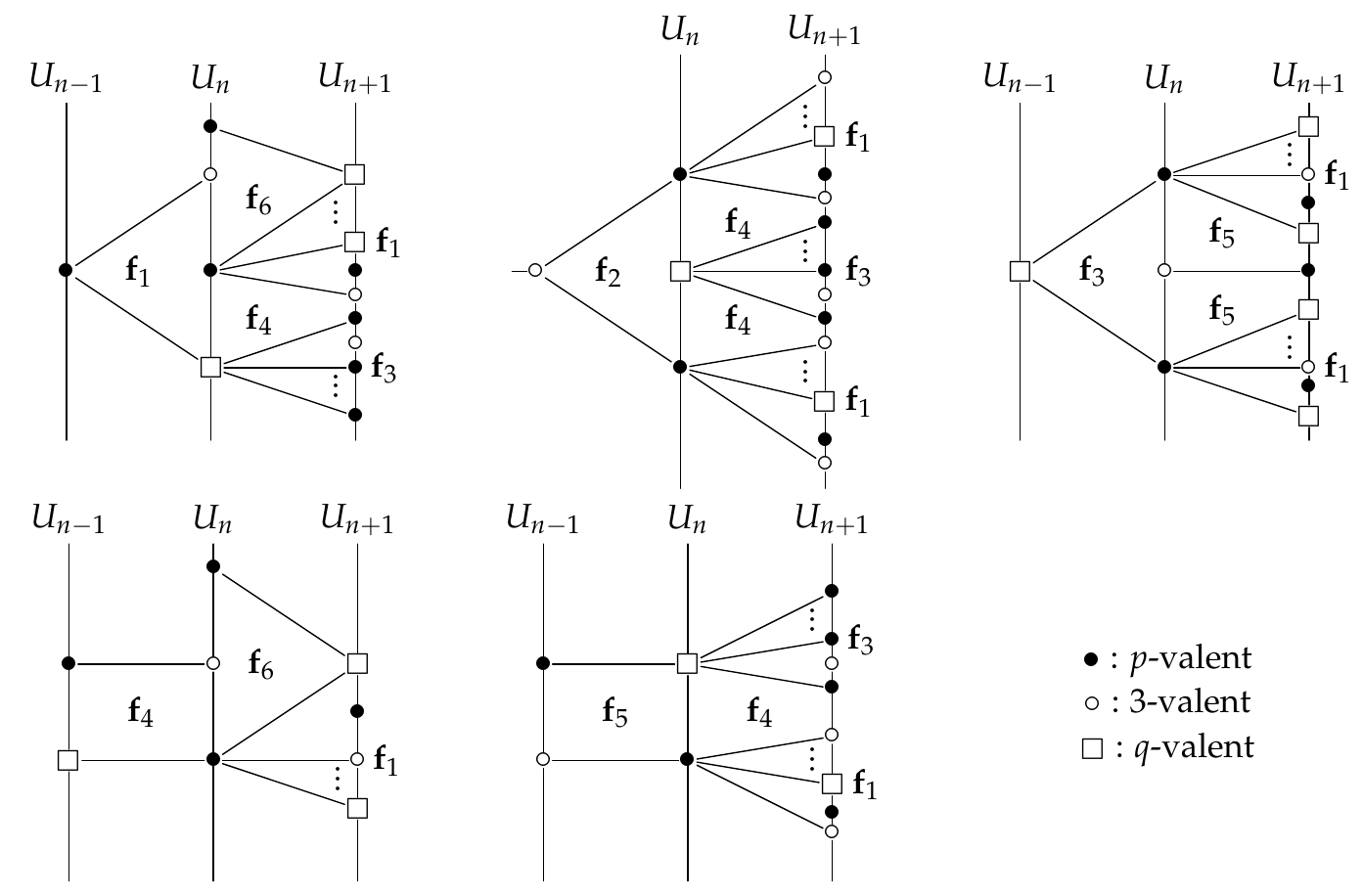}
\caption[Offspring diagrams for $\vs{3,p,q,p}$.]{\label{fig:p3pr-offspring}Offspring diagrams for a face-homogeneous tessellation with valence sequence $\vs{3,p,q,p}$ with $q\geq 5$.}
\end{figure} 
Although $\ft_6$ faces exhibit the ``collapsing'' behavior common to non-concentric tessellations, we note that the unique vertex in $F_{n+1}$ incident with an $\ft_6$ face in $F_n$ has valence $q\geq 5$. The effect of the $\ft_6$ face is to decrease the number of $\ft_3$ faces in $F_{n+1}$ by one. Hence the transition matrix of $T$ is 
\[ M=\left[\begin{matrix}
p-3 & p-4 & p-4 & \frac{p-4}2 & \frac{p-4}2 & 0 \\
0 & 0 & 0 & 0 & 0 & 0 \\
\frac{q-4}2 & q-3 & 0 & 0 & \frac{q-4}2 & -1 \\
1 & 2 & 0 & 0 & 1 & 0 \\
0 & 0 & 2 & 0 & 0 & 0 \\
\frac12 & 0 & 0 & \frac12 & 0 & 0 
\end{matrix}\right],\]
the characteristic polynomial of which is
\[ \chi(z) = z(z-1)
	\left[z^4-(p-4)z^3-\frac12(pq-p-2q-4)z^2-(p-4)z+1\right].\]
As the factor of degree $4$ is palindromic, \Cref{rem:PalinPoly} gives the desired result.
\end{proof}
\end{prop}
 
The following proposition, as compared with the previous one, illustrates  how much easier these calculations become when all valences are at least $4$. 
\begin{prop} If $\sig=\vs{p,q,p,r}$ and $q,r\geq 4$, then $\sig$ is monomorphic and
 {\small
\[\gam(T) = \frac14
	\left[p-4+a+\sqrt{2p^2+2pq+2pr+4qr-24p-16q-16r+64+a(2p-8)}\right], \]}
where $a=\sqrt{p^2+2pq+2pr+4qr-16p-16q-16r+64}$.
\begin{proof} Since $q\neq r$, it must be that $p$ is even for $\sig$ to be realizable. 
As the distribution of offspring shown in \Cref{fig:pqpr-offspring} is determined, $\sig$ is monomorphic.
\begin{figure}[htbp]
\includegraphics[width=\linewidth]{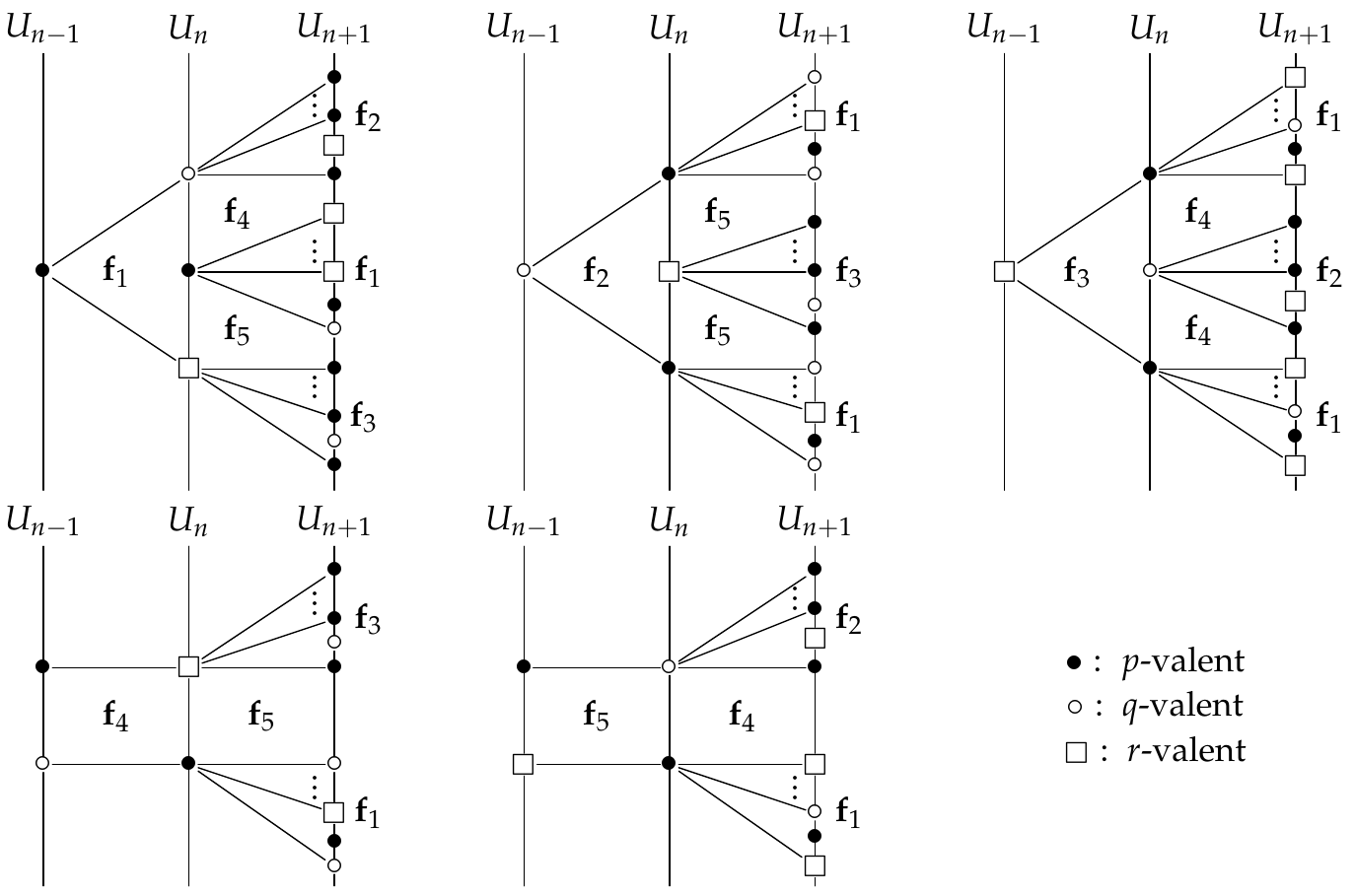}
\caption[Offspring diagrams for $\vs{p,q,p,r}$.]{\label{fig:pqpr-offspring}Offspring diagrams for a face-homogeneous tessellation with valence sequence $\vs{p,q,p,r}$ with $q,r\geq 4$.}
\end{figure}
From these diagrams we produce the transition matrix
\[ M = \left[\begin{matrix}
p-3 & p-4 & p-4 & \frac{p-4}2 & \frac{p-4}2 \\
\frac{q-4}2 & 0 & q-3 & 0 & \frac{q-4}2 \\
\frac{r-4}2 & r-3 & 0 & \frac{r-4}2 & 0 \\
1 & 0 & 2 & 0 & 1 \\
1 & 2 & 0 & 1 & 0
\end{matrix}\right].\]
The characteristic polynomial of $M$ is 
\[ \chi(z) = \frac12(z-1)(2z^4-bz^3-cz^2-bz+2),\] where
\begin{align*}
	b &= 2p-8\text{~and} \\
	c &= pq+pr+2qr-4p-8q-8r+20.
\end{align*} 
As in the previous proof, the fourth degree factor is palindromic, and we apply \Cref{rem:PalinPoly}.
\end{proof}
\end{prop}
 
\begin{prop} If $\sig=\vs{p,q,r,s}$, then $\sig$ is monomorphic, all terms are even, and $\gam(T)\geq 7.0367.$
\begin{proof} Since $p,q,r$ and $s$ are all distinct, they must be even by previous arguments.
We apply \cite{GFHT}{Lemma 2.9} to obtain the transition matrix
\[ M = \left[\begin{matrix}
0 & \frac{p-4}2 & p-3 & \frac{p-4}2   & 0 & 0 & \frac{p-4}2 & \frac{p-4}2 \\
\frac{q-4}2 & 0 & \frac{q-4}2 & q-3   & \frac{q-4}2 & 0 & 0 & \frac{q-4}2 \\
r-3 & \frac{r-4}2 & 0 & \frac{r-4}2   & \frac{r-4}2 & \frac{r-4}2 & 0 & 0 \\
\frac{s-4}2 & s-3 & \frac{s-4}2 & 0   & 0 & \frac{s-4}2 & \frac{s-4}2 & 0 \\
0 & 1 & 1 & 0   & 0 & 0 & 1 & 0 \\
0 & 0 & 1 & 1   & 0 & 0 & 0 & 1 \\
1 & 0 & 0 & 1   & 1 & 0 & 0 & 0 \\
1 & 1 & 0 & 0   & 0 & 1 & 0 & 0
\end{matrix}\right].\]
This matrix has characteristic polynomial 
\[ f(z) = (z-1)^2(z^6+2z^5-b z^4-c z^3-b z^2+2z+1), \]
with
\begin{align*}
b &= \frac14\left[(p+r)(q+s)+4(pr+qs)-16(p+q+r+s)+68\right]\text{~and}\\
c &= \frac14[2(pr(q+s)+qs(p+r)) -10(p+r)(q+s) \\
	&~\qquad -8(pr+qs) +32(p+q+r+s)-112].
\end{align*} 
As this polynomial cannot be factored further by elementary methods, we determine the least possible rate of growth of $T$ rather than attempt to calculate the growth rate in the general case.

There are three minimal valence sequences of the form $\vs{p,q,r,s}$, all pairwise incomparable under the partial order on cyclic sequences, namely: $\sig_1=\vs{4,6,8,10}$, $\sig_2=\vs{4,6,10,8}$, and $\vs{4,8,6,10}$. If $T_i$ is a tessellation with valence sequence $\sig_i$ for $i=1,2,3$, then the approximate growth rates of the three tessellations are
\begin{align*}
	\gam(T_1) &\approx 7.2091, \\
	\gam(T_2) &\approx 7.0367, \text{~and}\\
	\gam(T_3) &\approx 7.7106.
\end{align*} 
Hence $\gam(T)\geq \gam(T_2)\approx 7.0367.$
\end{proof}\end{prop}


\medskip

\section[Face-homogeneous Tessellations of Covalence 5]{Classification of Face-Homogeneous Tessellations with Valence Sequences of Length 5}\label{ch:5cov}

\begin{thm}\label{thm:5covAmb} Suppose that $\sig$ has length $5$. Then $\sig$ is polymorphic if and only if it is one of the following, where $p,q,r,s\geq 4$:
\begin{enumerate}
\item $\vs{p,p,p,p,q}$ with $\frac4p+\frac1q<\frac32$;
\item $\vs{p,p,p,p,3}$ for any integer $p\geq 4$;
\item $\vs{p,p,p,q,q}$ with $q$ even and $\frac3p+\frac2q<\frac32$;
\item $\vs{p,p,q,p,q}$ with $\frac3p+\frac2q<\frac32$;
\item $\vs{p,p,3,p,3}$ for any integer $p\geq 4$;
\item $\vs{p,p,p,q,r}$ with $q$ and $r$ both even, and $\frac3p+\frac1q+\frac1r<\frac32$;
\item $\vs{3,3,3,p,q}$ with $p$ and $q$ both even, and $\frac1p+\frac1q<\frac12$;
\item $\vs{p,p,q,p,r}$ with $\frac3p+\frac1q+\frac1r<\frac32$;
\item $\vs{p,p,3,p,q}$ with $\frac3p+\frac1q<\frac76$;
\item $\vs{p,p,q,q,r}$ with $p,q$, and $r$ even, and $\frac2p+\frac2q+\frac1r<\frac32$;
\item $\vs{p,q,p,q,r}$ with $r$ even, and $\frac2p+\frac2q+\frac1r<\frac32$;
\item $\vs{3,p,3,p,q}$ with $q$ even and $\frac2p+\frac1q<\frac56$;
\item $\vs{p,p,q,r,s}$ with $p,q,r$, and $s$ all even and $\frac2p+\frac1q+\frac1r+\frac1s<\frac32$;
\item $\vs{p,q,p,r,s}$ with $p,r$ and $s$ all even and $\frac2p+\frac1q+\frac1r+\frac1s<\frac32$.
\end{enumerate}
\begin{proof}
We establish by this theorem that each of these fourteen cases is polymorphic. That they are the only polymorphic valence sequences of length 5 is determined by the eight propositions which follow. 

Except for Case (7), all of these valence sequences are realized by tessellations in $\scr{G}_{3^+,5}$; as one sees by inspection, except for Cases (5), (7), and (12), all satisfy the further hypotheses of \cite{GFHT}{Theorem 3.1} and hence are polymorphic.   

\noindent \underline{Case \emph{(5)}} If $\sig=\vs{p,p,3,p,3}$, then there are two distinct offspring diagrams for a face of type ${\ft_1}$ as shown in \Cref{fig:pp3p3-ambig}. Thus, as in the previous case, $\sig$ is polymorphic.
\begin{figure}[htbp]
\includegraphics{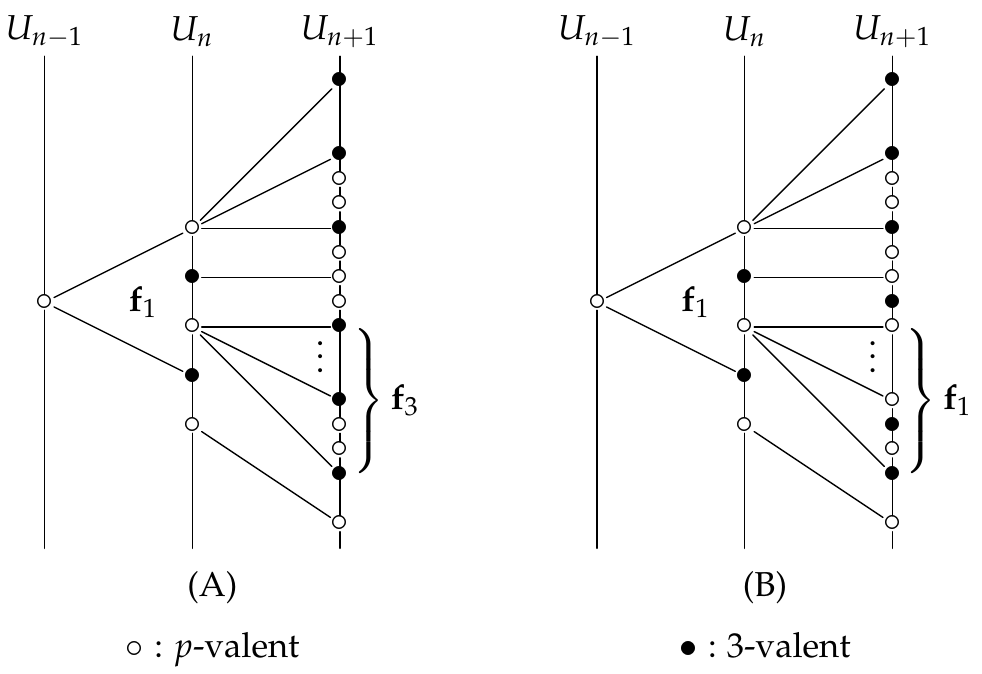}
\caption[Offspring diagrams showing $\vs{p,p,3,p,3}$ is polymorphic.]{\label{fig:pp3p3-ambig} Two distinct offspring diagrams for a face of type ${\ft_1}$ in a tessellation with valence sequence $\vs{p,p,3,p,3}$.}
\end{figure}

\noindent \underline{Case \emph{(7)}}  If $\sig=\vs{3,3,3,p,q}$, then both  $p$ and $q$ are even, and $\frac1p+\frac1q<\frac12$ holds. Without loss of generality  assume that $q\geq 6$. Then there are two sets of possible offspring for a face of type ${\ft_5}$, as shown in \Cref{fig:333pq-ambig}.
When adjacent faces of type ${\ft_5}$ occur in $F_n$ both incident to a common $3$-valent vertex in $U_n$, one must have offspring as specified in \Cref{fig:333pq-ambig}(A) and the other must have offspring as specified in \Cref{fig:333pq-ambig}(B). However, the configuration (as read clockwise about the common $q$-valent vertex in $U_{n-1}$) could be either (A)-(B) or (B)-(A).  As these are non-isomorphic configurations, $\sig$ is polymorphic.
\begin{figure}[htbp]
\includegraphics{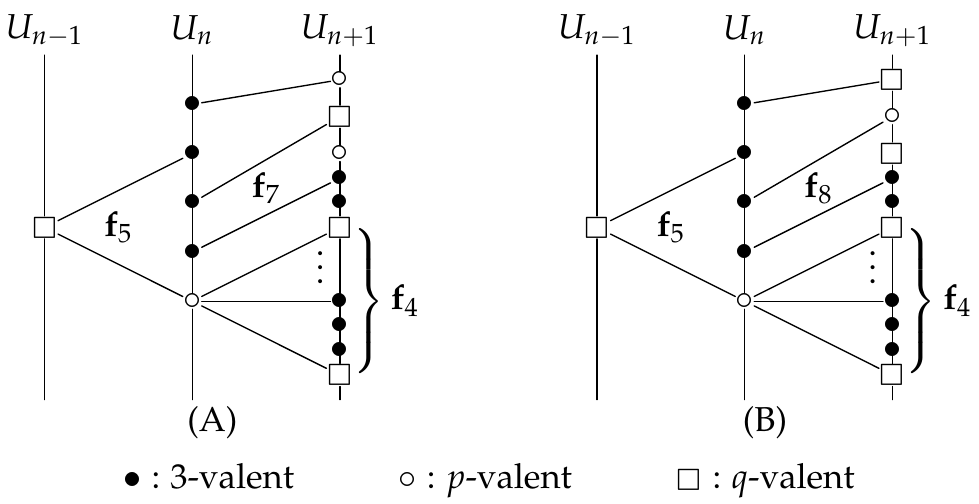}
\caption[Offspring diagrams showing $\vs{3,3,3,p,q}$ is polymorphic.]{\label{fig:333pq-ambig}Two distinct offspring diagrams for a face of type ${\ft_5}$ in a tessellation with valence sequence $\vs{3,3,3,p,q}$.}
\end{figure}

\noindent \underline{Case \emph{(12)}} If $\sig=\vs{3,p,3,p,q}$, then $q$ is even and $\frac2p+\frac1q<\frac56$ holds. The faces of type ${\ft_1}$ have at least two distinct offspring diagrams, depending partly upon one of the faces with which they are adjacent in $F_n$, as demonstrated in \Cref{fig:3p3pq-ambig}.
The faces labeled ${\ft_1},{\ft_2}$ can be of either type, thus demonstrating polymorphicity; also, the face labeled ${\ft_{10}}$ could have two distinct orientations.
\end{proof}
\end{thm}

\begin{figure}[htbp]
\includegraphics{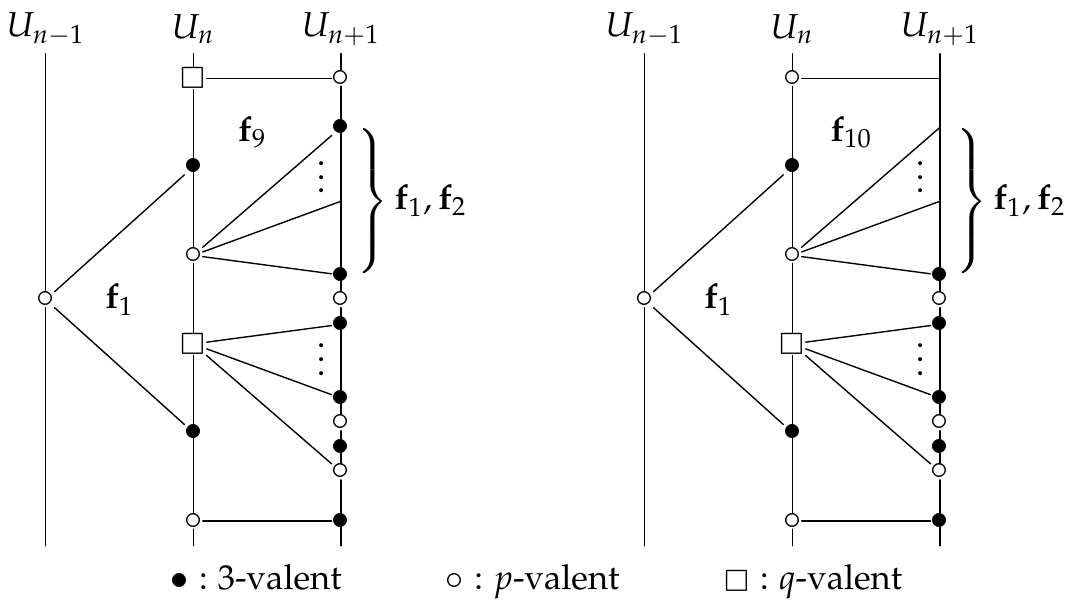}
\caption[Offspring diagrams showing $\vs{3,p,3,p,q}$ is polymorphic.]{\label{fig:3p3pq-ambig}Two distinct offspring diagrams for a face of type ${\ft_1}$ in a tessellation with valence sequence $\vs{3,p,3,p,q}$.}
\end{figure}

Having shown that these 14 valence sequences are polymorphic, we prove exhaustively via \Crefrange{thm:GrowthOf-ppppp}{thm:GrowthOf-pqrst} that all other valence sequences of length $5$ are monomorphic. In these eight cases, we also calculate the exact growth rates in terms of the valences wherever feasible; otherwise we demonstrate a lower bound for the growth rate of tessellations within the class as per the method of \Cref{thm:GrowthOf-pqr}.

\begin{prop}\label{thm:GrowthOf-ppppp} Let $\sig=\vs{p,p,p,p,p}$.  Then $p\geq 4$, $\sig$ is monomorphic, and
\[ \gam(T) = \frac12\left(3p-8+\sqrt{9p^2-48p+60}\right). \]
\begin{proof}
As $T$ is  edge-homogeneous with edge-symbol $\esym{p,p}{5,5}$, \cite{GFHT}{Proposition 1.17} gives the stated growth rate.
\end{proof}
\end{prop}

\begin{prop}\label{thm:GrowthOf-3333p} If $\sig=\vs{3,3,3,3,p}$, then $\sig$ is monomorphic with $p\geq 7$ and \[ \gam(T) = \frac12\sqrt{2(p-4)+2\sqrt{(p-5)(p-2)}}. \]
\begin{proof}
To show that $\sig$ is monomorphic, let $v_0$ be a $p$-valent root.  The first two coronas of $T$ must be as shown in \Cref{fig:3333p-central} (up to isomorphism). 
\begin{figure}[htbp]
\includegraphics[width=.5\linewidth]{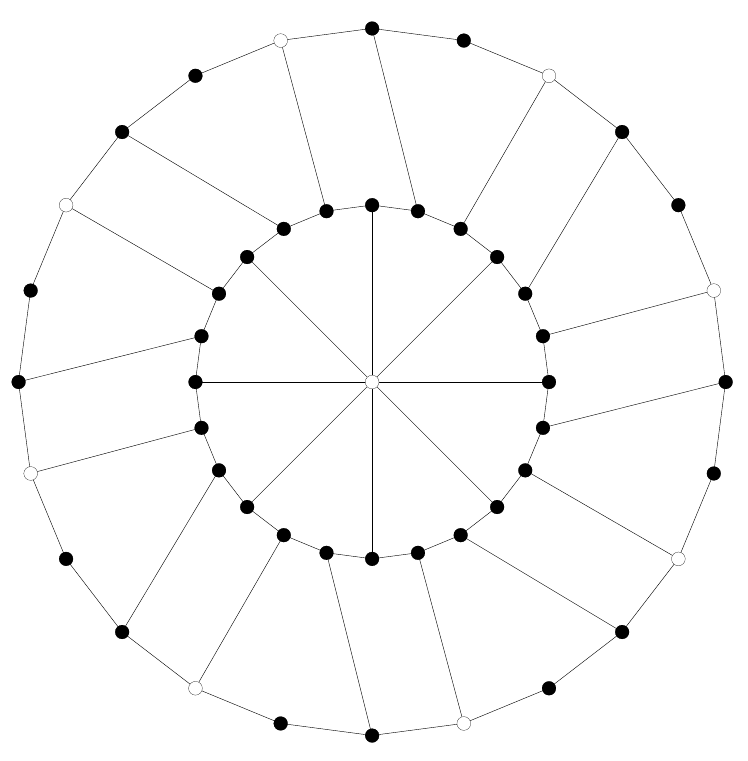}
\[\bull: 3\text{-valent}\qquad \circ: 8\text{-valent}\]
\caption[Root configuration in valence sequence $\vs{3,3,3,3,8}$]{\label{fig:3333p-central} First two coronas of a tessellation with valence sequence $\vs{3,3,3,3,8}$ when rooted at an $8$-valent vertex, exemplifying the pattern of valences at regional distance $1$ from a $p$-valent vertex in $\vs{3,3,3,3,p}$.}
\end{figure}
As the configuration of vertices at regional distance $2$ from any $p$-valent vertex is known, and every face is incident with exactly one $p$-valent vertex, we have fixed the configuration of the entire tessellation up to isomorphism; thus $\sig$ is monomorphic.

The offspring diagrams for $\vs{3,3,3,3,p}$ appear in \Cref{fig:3333p-offspring}, and give the matrix
\begin{figure}[htbp]
\includegraphics[width=\linewidth]{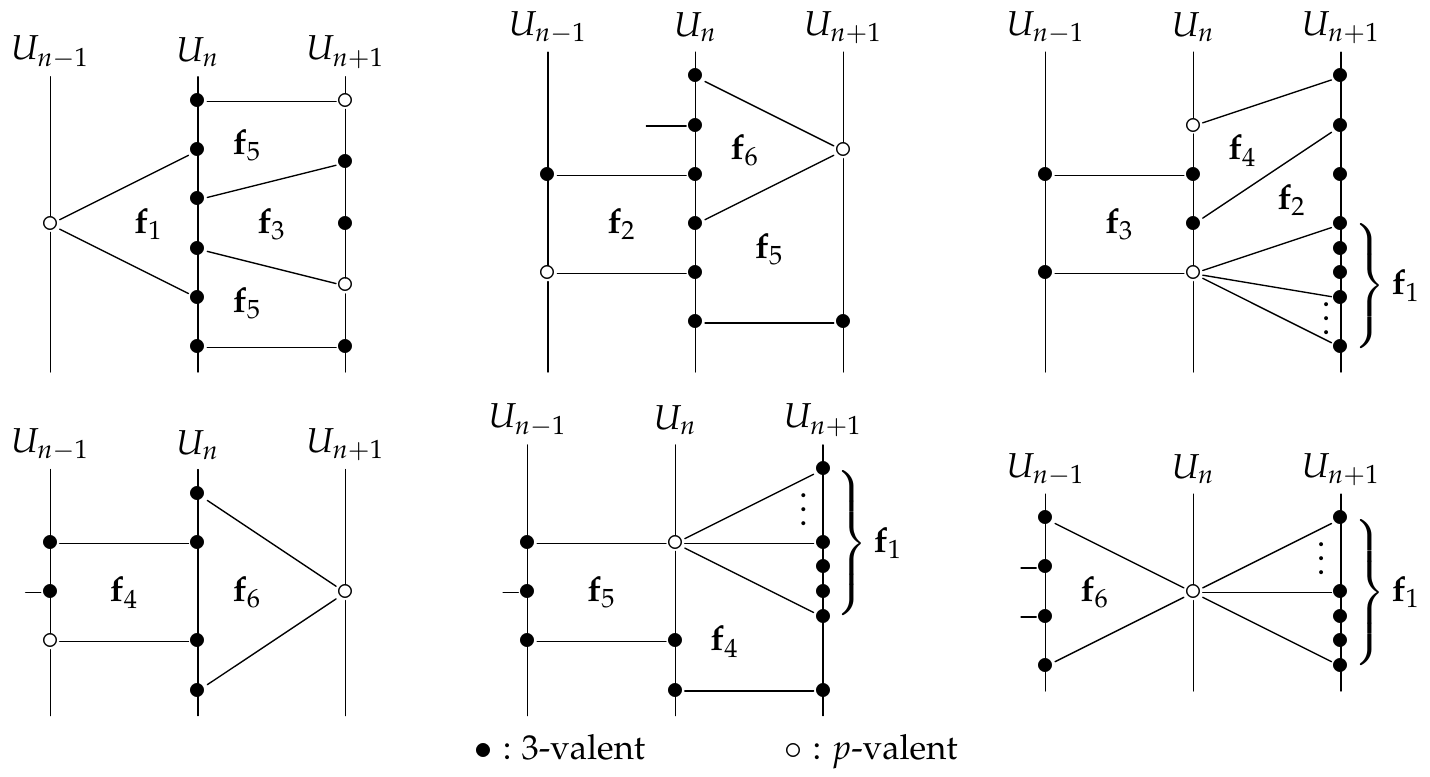}
\caption{\label{fig:3333p-offspring}Face types and offspring in $\vs{3,3,3,3,p}$.}
\end{figure}
\[ M = \left[\begin{matrix}
0 & 0 & \frac12(p-4) & 0 & \frac12(p-4) & -1 \\
0 & 0 & 1 & 0 & 0 & 0 \\
1 & 0 & 0 & 0 & 0 & 0 \\
0 & 0 & \frac12 & 0 & \frac12 & 0 \\
1 & \frac12 & 0 & 0 & 0 & 0 \\
0 & \frac12 & 0 & 0 & 0 & 0
\end{matrix}\right]. \]
This matrix has characteristic polynomial \[\chi(z) =  \frac14z^2(4z^4-4(p-4)z^2+(p-6)),\] and so the growth rate of a tessellation $T$ with valence sequence $\vs{3,3,3,3,p}$ is 
\[ \gam(T) = \frac12\sqrt{2(p-4)+2\sqrt{(p-5)(p-2)}}. \]
\end{proof}
\end{prop}

\begin{prop} If $\sig=\vs{3,3,3,p,p}$ with even $p\geq 6$, then $\sig$ is monomorphic and 
\[ \gam(T) = \frac18\left[p-4+\sqrt{p^2+32p-80}+\sqrt{2(p-4)(p+16+\sqrt{p^2+32p-80}}\right]. \]
\begin{proof} We again have that the faces within two coronas of a $p$-valent root vertex must be configured as shown (exemplified by $\vs{3,3,3,8,8}$) in \Cref{fig:33388-8local}. Since every face is incident with two adjacent $p$-valent vertices, this determines the accretion rule for the tessellation; hence $\sig$ is monomorphic. Offspring diagrams for the tessellation appear in \Cref{fig:333pp-offspring}.
\begin{figure}[htbp]
\includegraphics[width=.5\linewidth]{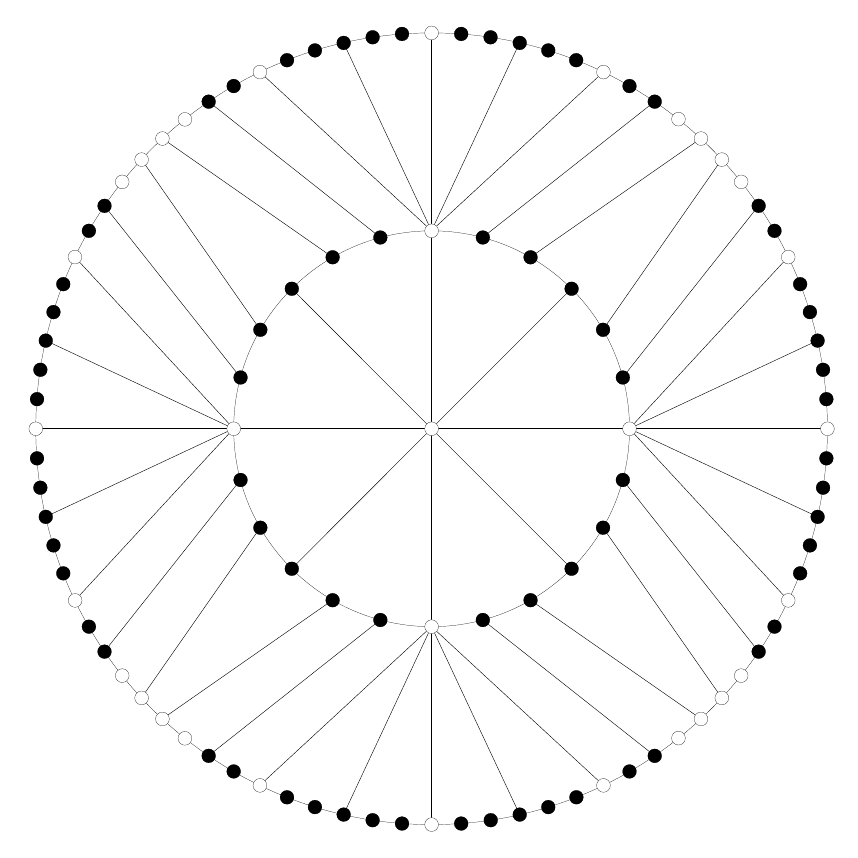}
\caption[Local Configuration of an $8$-valent vertex in $\vs{3,3,3,8,8}$.]{\label{fig:33388-8local}First two coronas of a tessellation with valence sequence $\vs{3,3,3,8,8}$ when rooted at an $8$-valent vertex.}
\end{figure}

\begin{figure}[htbp]
\includegraphics[width=\linewidth]{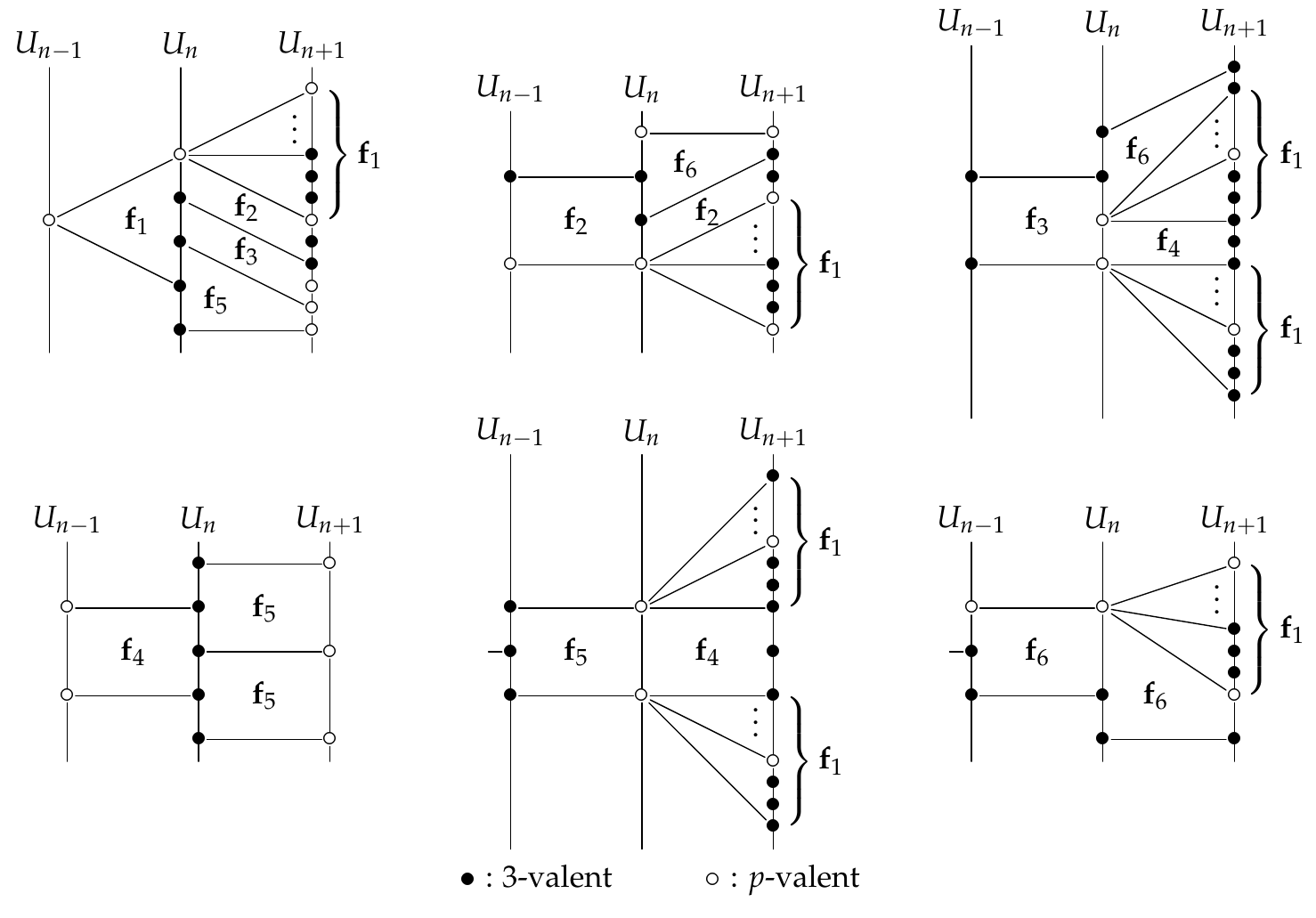}
\caption[Offspring diagrams for $\vs{3,3,3,p,p}$.]{\label{fig:333pp-offspring}Offspring diagrams for a tessellation with valence sequence $\vs{3,3,3,p,p}$.}
\end{figure}

We may assume by \cite{GFHT}{Proposition 1.14} that the Bilinski diagram of a tessellation $T$ with valence sequence $\sig$ is rooted at a $p$-valent vertex, and thereby is concentric. The offspring diagrams determine the transition matrix
\[ M = \left[\begin{matrix}
\frac{p-4}2 & \frac{p-4}2 & \frac{3p-10}2 & 0 & p-4 & \frac{p-4}2 \\
1 & 1 & 0 & 0 & 0 & 0 \\
1 & 0 & 0 & 0 & 0 & 0 \\
0 & 0 & 1 & 0 & 1 & 0 \\
\frac12 & 0 & 0 & 1 & 0 & 0 \\
0 & \frac12 & \frac12 & 0 & 0 & \frac12
\end{matrix}\right],\]
which has characteristic polynomial
\[ \chi(z) = \frac14(2z-1)(z-1)(2z^4-(p-4)z^3-(5p-16)z^2-(p-4)z+2). \]
Solving for the eigenvalue of maximum modulus gives 
\[ \gam(T) = \frac18\left[p-4+\sqrt{p^2+32p-80}+\sqrt{2(p-4)(p+16+\sqrt{p^2+32p-80}}\right]. \]
\end{proof}
\end{prop}
\begin{prop}\label{thm:GrowthOf-33p3p} If $\sig=\vs{3,3,p,3,p}$, then $\sig$ is monomorphic with $p\geq5$ and 
\[\gam(T)=\frac12\left(p-2+\sqrt{(p-2)^2-4}\right).\]
\begin{proof}
While $\sig$ is not uniformly concentric, choosing a $p$-valent root produces a concentric Bilinski diagram with exactly three face types, the offspring diagrams for which are shown in \Cref{fig:33p3p-offspring}. The transition matrix 
\[ M = \begin{bmatrix} 
p-3 & \frac{p-4}{2} & \frac{p-4}{2} \\
1 & 1 & 0 \\
1 & \frac{1}{2} & \frac{1}{2}
\end{bmatrix}\] 
has a cubic characteristic polynomial with factorization
\[\chi(z) = \frac12(2z-1)(z^2-(p-2)z+1).\] 
By \cite{GFHT}{Theorem 2.3} and \cite{GFHT}{Proposition 1.14}, 
\[\gam(T) = \frac12\left((p-2)+\sqrt{(p-2)^2-4}\right).\] 
\begin{figure}[htbp]
\includegraphics[width=\linewidth]{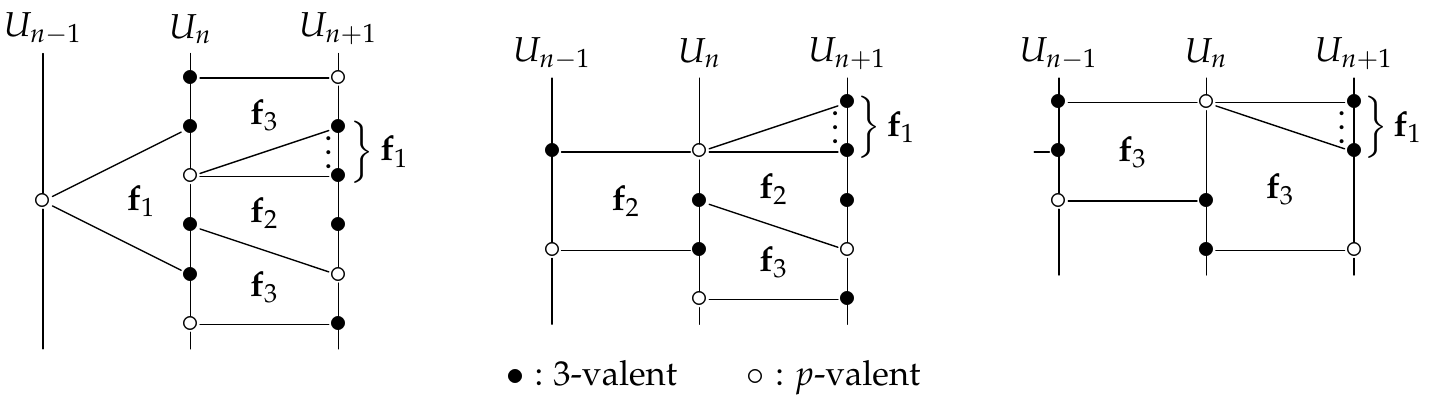}
\caption[Offspring diagrams for $\vs{3,3,p,3,p}$.]{\label{fig:33p3p-offspring}Offspring diagrams for a tessellation with valence sequence $\vs{3,3,p,3,p}$ when rooted at a $p$-valent vertex.}
\end{figure}
\end{proof}
\end{prop}

\begin{prop}\label{thm:GrowthOf-33p3q} If $\sig=\vs{3,3,p,3,q}$, then $\sig$ is monomorphic with $p,q\geq 4$ and
\[ \gam(T) = \frac12 \sqrt{2(p-2)(q-2)-4+2\sqrt{(p-2)^2(q-2)^2-4(p-2)(q-2)}}. \]
\begin{proof}
Suppose a $3$-valent vertex in $V(T)$ is adjacent to at least two other $3$-valent vertices. These three consecutive $3$-valent vertices must all be incident with a common face $f$; but then $f$ does not have valence sequence $\sig$. By contradiction, each $3$-valent vertex in $V(T)$ must be adjacent to at most one other $3$-valent vertex.

On the other hand, assume that $v_1$ is a $3$-valent vertex adjacent to no other $3$-valent vertices. Then there is a face $f$ incident with $v_1$ which is incident with consecutive vertices with valences $(p,3,p)$ or $(q,3,q)$. In neither case does $f$ have valence sequence $\sig$. Thus each $3$-valent vertex in $V(T)$ must be adjacent to exactly one other $3$-valent vertex. This forces all edges with edge-symbol $\esym{3,3}{5,5}$ to have the configuration shown in \Cref{fig:33p3q-EdgeLocal}.

\begin{figure}[htbp]
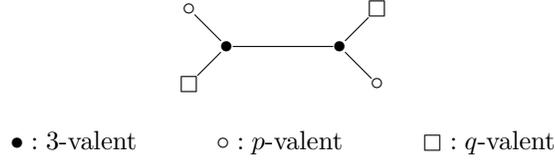

\[\xy<1cm,0cm>:
(1.5,0)*{\bullet}="x1"; (3,0)*{\bullet}="x2" **@{-};
(1,0.5)*{\circ}="x3"; "x1" **@{-};
(1,-0.5)*{\boxempty}="x4"; "x1" **@{-};
(3.5,0.5)*{\boxempty}="x5"; "x2" **@{-};
(3.5,-0.5)*{\circ}="x6"; "x2" **@{-};
\endxy\]\medskip
\[\begin{matrix}
\bullet: 3\text{-valent} &\quad& 
\circ: p\text{-valent} &\quad&
\boxempty: q\text{-valent}
\end{matrix}\]
\caption[An edge-local configuration in $\vs{3,3,p,3,q}$.]{\label{fig:33p3q-EdgeLocal}The local configuration in a tessellation with valence sequence $\vs{3,3,p,3,q}$ at an edge with symbol $\esym{3,3}{5,5}$.}
\end{figure}
As with $\vs{3,3,p,3,p}$, this configuration determines the entire tessellation, and so  $\sig$ is monomorphic.  Taking a $p$-valent vertex to be the root of the Bilinski diagram gives face types and offspring diagrams as shown in \Cref{fig:33p3q-offspring}.

\begin{figure}[htbp]
\includegraphics[width=\linewidth]{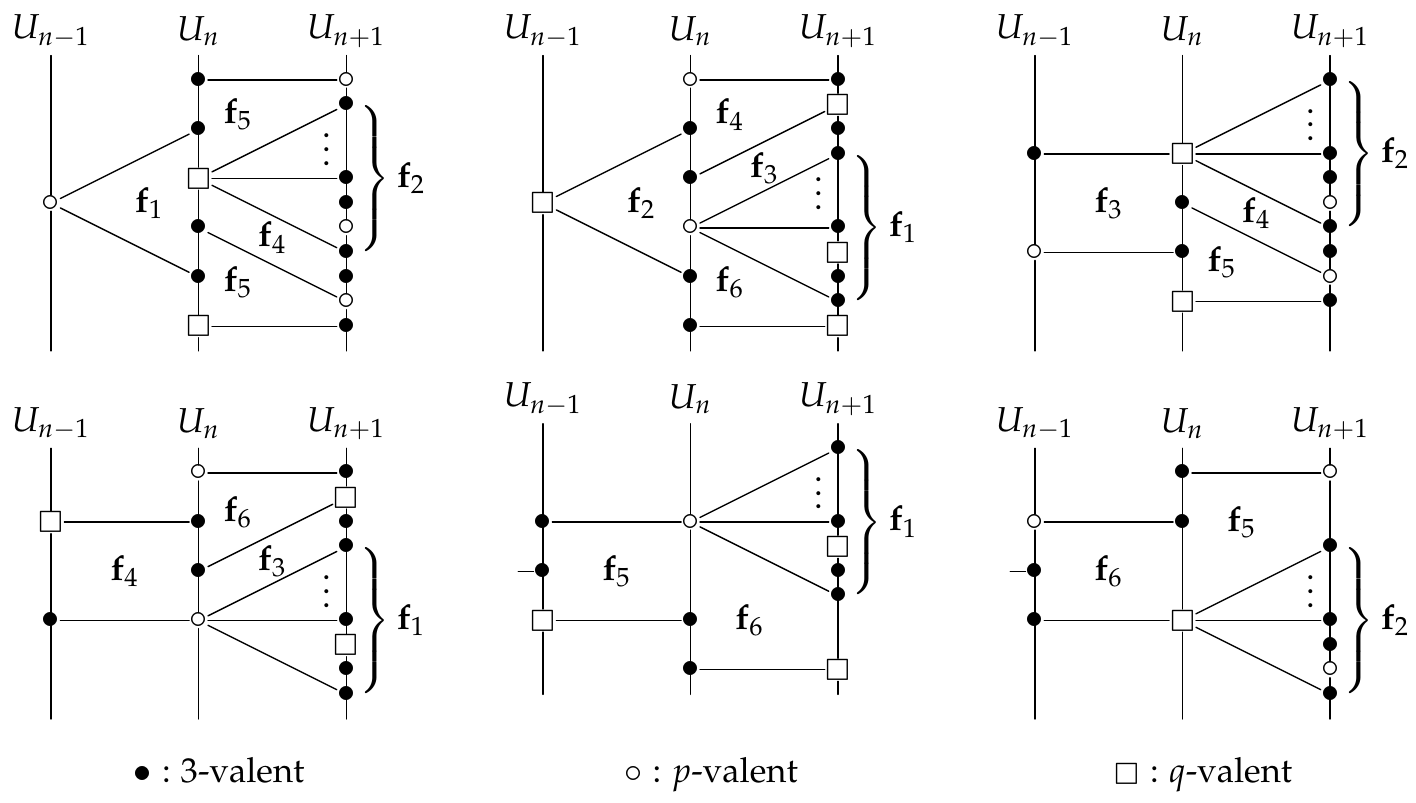}
\caption[Offspring diagrams for $\vs{3,3,p,3,q}$.]{\label{fig:33p3q-offspring}Offspring diagrams for a tessellation with valence sequence $\vs{3,3,p,3,q}$.}
\end{figure}

In turn these give the transition matrix
\[ M = \left[\begin{matrix}
0 & p-3 & 0 & \frac{p-4}2 & \frac{p-4}2 & 0 \\
q-3 & 0 & \frac{q-4}2 & 0 & 0 & \frac{q-4}2 \\
0 & 1 & 0 & 1 & 0 & 0 \\
1 & 0 & 1 & 0 & 0 & 0 \\
1 & 0 & \frac12 & 0 & 0 & \frac12 \\
0 & 1 & 0 & \frac12 & \frac12 & 0
\end{matrix}\right], \]
which has characteristic polynomial
\begin{align*}
\chi(z) &= z^6-\left(pq-2p-2q+\frac94\right)z^4 +\frac14(pq-2p-2q+6)-\frac14 \\
	&= \frac14(2z-1)(2z+1)(z^4+(pq-2p-2q+2)z^2-1).
\end{align*}
 Thus the growth rate of $T$  is 
\[ \gam(T) = \frac12 \sqrt{2(p-2)(q-2)-4+2\sqrt{(p-2)^2(q-2)^2-4(p-2)(q-2)}}.\]
\end{proof}
\end{prop}
\begin{prop}\label{thm:GrowthOf-ppqrq} Let $\sig=\vs{p,p,q,r,q}$ with $p,q,r\geq 4$.  Then $p$ and $q$ are even and 
\[\gam(T)\geq \frac14\left(3+\sqrt{113}+\sqrt{106+6\sqrt{113}}\right) \approx 6.6650.\]
\begin{proof} Assuming that $T$ is labeled as a Bilinski diagram rooted at a $p$-valent vertex, the offspring diagrams and face types are as shown in \Cref{fig:ppqrq-offspring} and yield the transition matrix
\[ M = \left[\begin{matrix}
\frac{p-4}2 & \frac{3p-10}2 & 2p-6 & \frac{p-4}2 & 0 & \frac{3p-10}2 \\
\frac{3q-10}2 & q-3 & q-4 & q-3 & q-4 & \frac{q-4}2 \\
r-3 & \frac{r-4}2 & 0 & \frac{r-4}2 & r-3 & 0 \\
1 & 1 & 2 & 1 & 0 & 1 \\
0 & 1 & 1 & 0 & 0 & 1 \\
2 & 1 & 0 & 1 & 2 & 0 
\end{matrix}\right],\]
which has characteristic polynomial
\begin{align*}
\chi(z) &= z^6-\left(\frac12p+q-4\right)z^5-\left(\frac74pq+2pr+
		\frac12qr-9p-6p-6r+25\right)z^4 \\
	&~\qquad -\left(\frac34pqr-4pq-5pr-2qr+17p+10q+12r-40\right)z^3 \\
	&~\qquad -\left(\frac74pq+2pr+\frac12qr-9p-6p-6r+25\right)z^2 
		-\left(\frac12p+q-4\right)z+1.
\end{align*}

\begin{figure}[htbp]
\includegraphics[width=\linewidth]{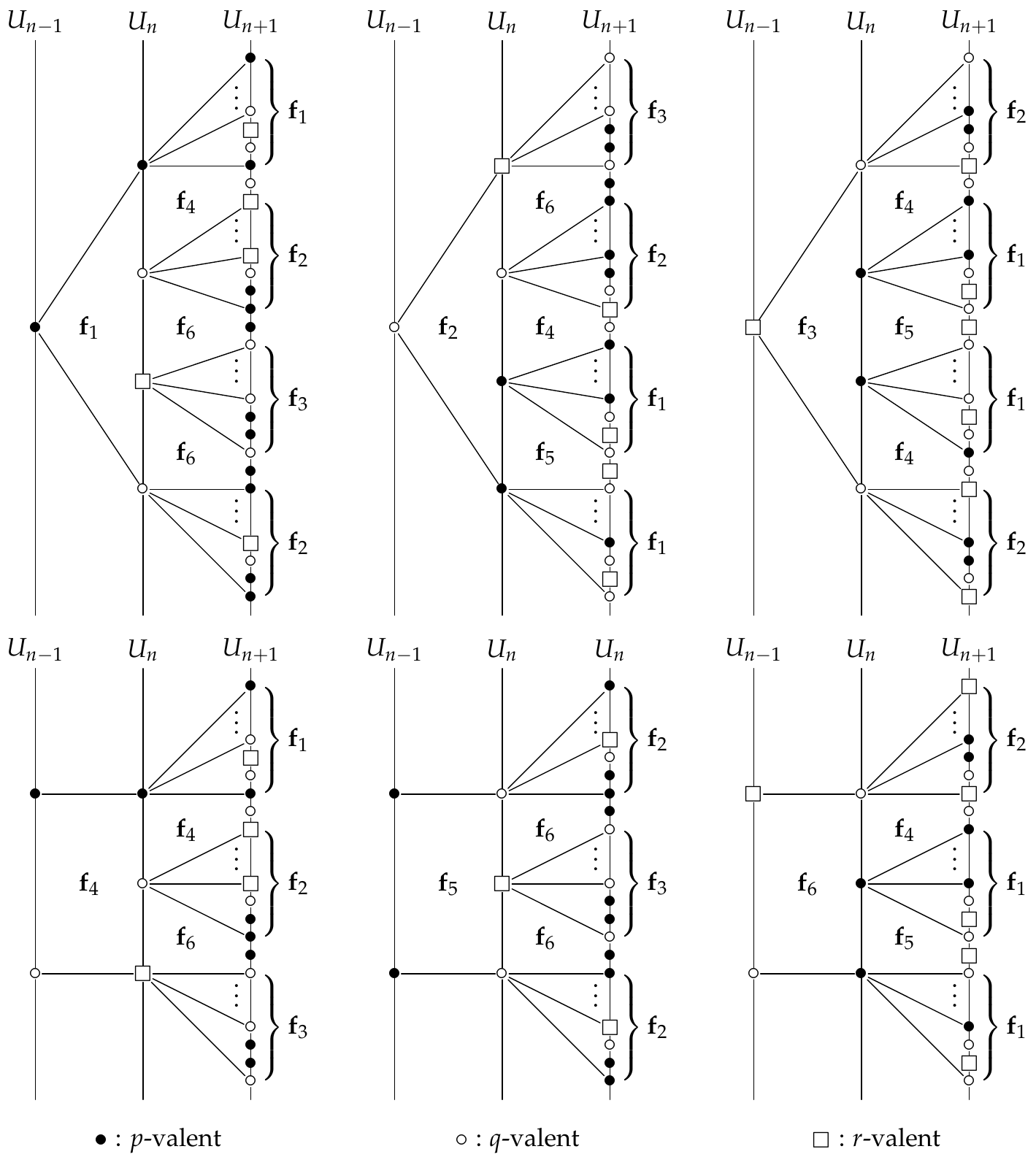}
\caption[Offspring diagrams for $\vs{p,p,q,r,q}$.]{\label{fig:ppqrq-offspring}Offspring diagrams for a tessellation with valence sequence $\vs{p,p,q,r,q}$.}
\end{figure}

There are four valence sequences in this class which are minimal under the partial order on valence sequences: $\sig_1=\vs{4,4,6,5,6}$, $\sig_2=\vs{6,6,4,5,4}$, $\sig_3=\vs{6,6,8,4,8}$, and $\sig_4=\vs{8,8,6,4,6}$. If $T_i$ is a face-homogeneous tessellation with valence sequence $\sig_i$, then the growth rates of $T_i$ for $i=1,2,3,4$ are 

\begin{align*}
\gam(T_1)=\frac14\left(3+\sqrt{113}+\sqrt{106+6\sqrt{113}}\right)&\approx 6.6650
	\\
\gam(T_2)=2+\sqrt{2}+\sqrt{5+4\sqrt{2}} &\approx 6.6786 \\
\gam(T_3)=\frac12\left(3+\sqrt{65}+\sqrt{70+6\sqrt{65}}\right)&\approx 10.9711 \\
\gam(T_4) &\approx 10.5872
\end{align*}
Thus by \cite{GFHT}{Theorem 2.7}, if $\sig\geq\sig_i$ for some $i\in\{1,2,3,4\}$, then $\gam(T)\geq\geq6.6650$.
\end{proof}
\end{prop}
\begin{figure}[htbp]
\includegraphics[width=\linewidth]{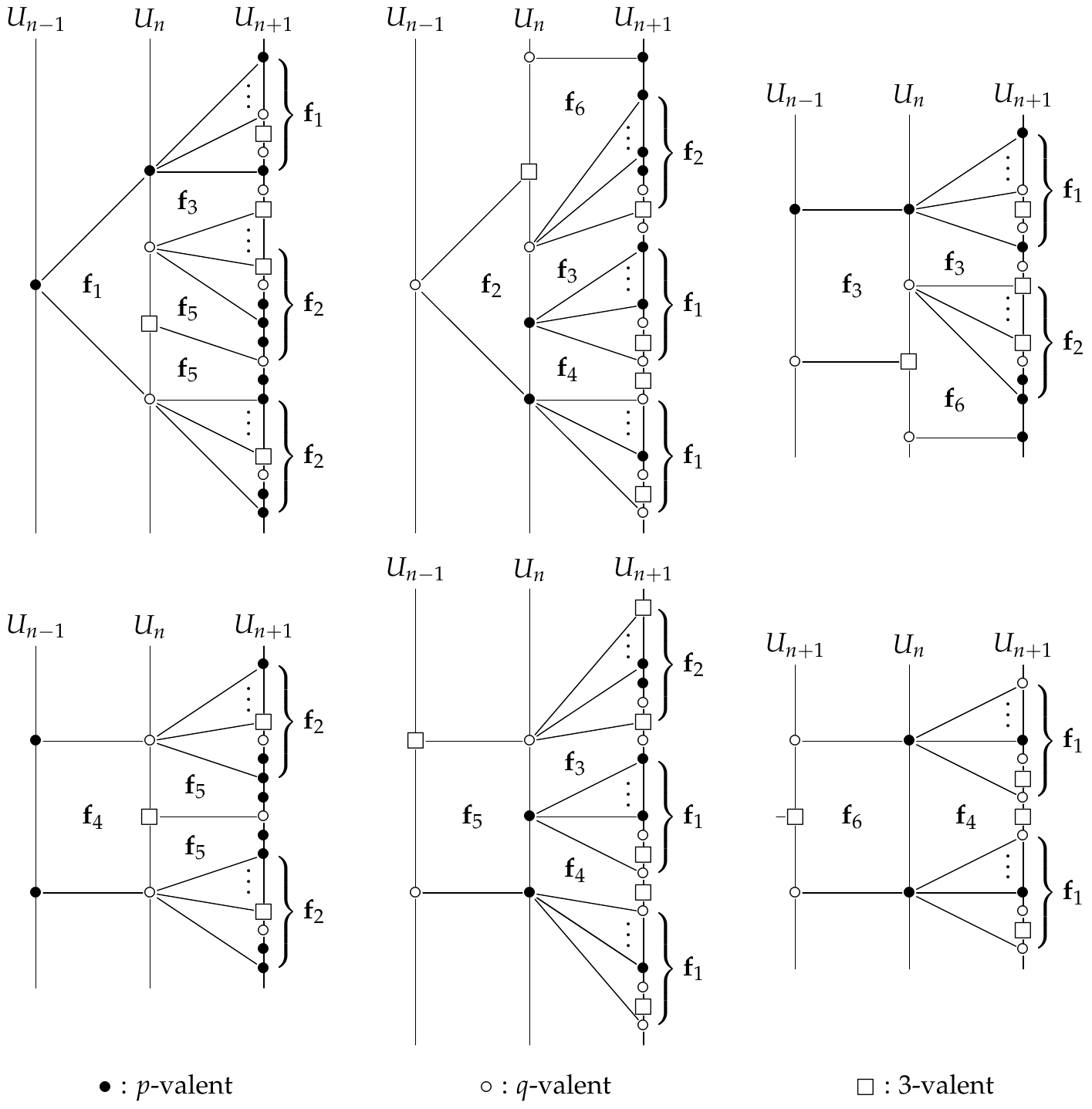}
\caption[Offspring diagrams for $\vs{p,p,q,3,q}$.]{\label{fig:ppq3q-offspring}Offspring diagrams for a tessellation with valence sequence $\vs{p,p,q,3,q}$.}
\end{figure}

\begin{prop}\label{thm:GrowthOf-ppq3q} Let $\sig=\vs{p,p,q,3,q}$ where $p,q\geq 4$.  Then both $p$ and $q$ are even, $\frac1p+\frac1q<\frac7{12}$,  $\sig$ is monomorphic, and 
$\gam(T)\geq 4.9911.$
\begin{proof} Suppose $T$ is labeled as a Bilinski diagram with a $p$-valent root. Then the face types and offspring diagrams for $T$ are described in \Cref{fig:ppq3q-offspring}.
These offspring diagrams give us the transition matrix 
\[ M = \left[\begin{matrix}
\frac{p-4}2 & \frac{3p-10}2 & \frac{p-4}2 & 0 & \frac{3p-10}2 & p-4 \\
\frac{3q-10}2 & q-3 & q-3 & q-4 & \frac{q-4}2 & 0 \\
1 & 1 & 1 & 0 & 1 & 0 \\
0 & 1 & 0 & 0 & 1 & 1 \\
2 & 0 & 0 & 2 & 0 & 0 \\
0 & \frac12 & \frac12 & 0 & 0 & 0
\end{matrix}\right].\]
So we have the characteristic polynomial
\begin{align*}
\chi(z) &= z^6 - \left(\frac12p+q-4\right)z^5 
		-\left(\frac74pq-3p-\frac92q+7\right)z^4 \\
	&~\quad +\left(\frac74pq-2p-4q+4\right)z^3
		-\left(\frac74pq-3p-\frac92q+7\right)z^2 \\
	&~\quad - \left(\frac12p+q-4\right)z +1.
\end{align*}
We again reduce to the minimal sequences in this class, namely ${\sig_1=\vs{4,4,6,3,6}}$ and ${\sig_2=\vs{6,6,4,3,4}}$. If $T_i$ is a tessellation with valence sequence $\sig_i$ for $i=1,2$, then numerical approximation of the eigenvalues of the transition matrices gives that
\begin{align*}
\gam(T_1) &\approx 5.4330 \text{ and}\\ \gam(T_2) &\approx 4.9911,
\end{align*}
so $\gam(T)\geq 4.9911$ by \cite{GFHT}{Theorem 2.7} whenever $\sig\geq \sig_1$ or $\sig\geq\sig_2$.\end{proof}
\end{prop}
\begin{prop}\label{thm:GrowthOf-pqrst} If $\sig=\vs{p,q,r,s,t}$, then all terms are even, $\frac1p+\frac1q+\frac1r+\frac1s+\frac1t<\frac32$ holds, $\sig$ is monomorphic, and $\gam(T)\geq 14.5753$.
\begin{proof} Applying {Lemma 2.9}, we calculate the transition matrix for a tessellation $T$ with valence sequence $\sig$ to be 
{\fontsize{9}{11}\selectfont\[ M =  \left[ \begin {array}{cccccccccc}
0 & \frac{p-4}2 & p-3 & p-3 & \frac{p-4}2 & 0 & 0 & \frac{p-4}2 & p-3 & \frac{p-4}2\\
\frac{q-4}2 & 0 & \frac{q-4}2 & q-3 & q-3 & \frac{q-4}2 & 0 & 0 & \frac{q-4}2 & q-3\\
r-3 & \frac{r-4}2 & 0 & \frac{r-4}2 & r-3 & r-3 & \frac{r-4}2 & 0 & 0 & \frac{r-4}2\\
s-3 & s-3 & \frac{s-4}2 & 0 & \frac{s-4}2 & \frac{s-4}2 & s-3 & \frac{s-4}2 & 0 & 0\\
\frac{t-4}2 & t-3 & t-3 & \frac{t-4}2 & 0 & 0 & \frac{t-4}2 & t-3 & \frac{t-4}2 & 0\\
0 & 1 & 1 & 1 & 0 & 0 & 0 & 1 & 1 & 0\\
0 & 0 & 1 & 1 & 1 & 0 & 0 & 0 & 1 & 1\\
1 & 0 & 0 & 1 & 1 & 1 & 0 & 0 & 0 & 1\\
1 & 1 & 0 & 0 & 1 & 1 & 1 & 0 & 0 & 0\\
1 & 1 & 1 & 0 & 0 & 0 & 1 & 1 & 0 & 0
\end {array} \right] .\]}

There are $\frac12(5-1)!=12$ distinct valences sequences in this class which are minimal under the partial order on cyclic sequences.  These sequences together with the approximate growth rates of associated tessellations are tabulated in \Cref{fig:pqrst-MinSeqs}.  An arbitrary valence sequence $\sigma$ in this class  dominates at least one of these $12$ minimal valence sequences, and so $\gam(T)\geq 14.5753$ by \cite{GFHT}{Theorem 2.7}.

\begin{table}
\[\begin{array}{c|c} \sig_i & \gam(T_i) \\ \hline\hline
\vs{4, 6, 8, 10, 12} & 14.8673 \\ 
\vs{4, 6, 8, 12, 10} & 14.6868 \\ 
\vs{4, 6, 10, 8, 12} & 15.1157 \\ 
\mathbf{\vs{4, 6, 10, 12, 8}} & \mathbf{14.5753} \\ 
\vs{4, 6, 12, 8, 10} & 15.0199 \\ 
\vs{4, 6, 12, 10, 8} & 14.6594 \\ 
\vs{4, 8, 6, 10, 12} & 15.1938 \\ 
\vs{4, 8, 6, 12, 10} & 15.0988 \\ 
\vs{4, 8, 10, 6, 12} & 15.5192 \\ 
\vs{4, 8, 12, 6, 10} & 15.3442 \\ 
\vs{4, 10, 6, 8, 12} & 15.5443 \\ 
\vs{4, 10, 8, 6, 12} & 15.6248 \\ 
\end{array}\]
\caption[Table of minimal sequences in class $\vs{p,q,r,s,t}$ and their approximate growth rates.]{\label{fig:pqrst-MinSeqs} Table of minimal sequences in class $\vs{p,q,r,s,t}$ and their approximate growth rates.}
\end{table}
\end{proof}
\end{prop}

\section[Face-homogeneous Tessellations of Covalence 6]{Classification of Face-Homogeneous Tessellations with Valence Sequences of Length 6}\label{ch:6cov}

There are 32 distinct classes of valence sequences of length $6$, and every face-homogeneous tessellation with covalence at least 6 is uniformly concentric by \cite{GFHT}Proposition 1.3. However, the inclusion of 3-valent vertices and their concomitant notched bricks warrant a separate treatment in \Cref{sec:6cov-3val}. We further extend the convention in \cite{GFHT} (following {Definition 1.7}) that all letters such as $p$, $q$, etc. appearing in valence sequences (not numerals such as 3, 4, etc.) represent integers $\geq 4$.

\subsection{Valence sequences of length 6 with all vertices of valence at least 4}

The first result is an immediate consequence of \cite{GFHT}{Proposition 3.1}.

\begin{cor} Let $\sig=\vs{p_0,\ldots,p_5}$ with all terms at least $4$. If any of the following conditions holds, then $\sig$ is polymorphic.
\begin{enumerate}
\item $\sig=\vs{p,p,p,p,p,q}$, with $\frac5p+\frac1q<2$;
\item $\sig=\vs{p,p,p,p,q,q}$, with $q$ even and $\frac2p-\frac1q<1$;
\item $\sig=\vs{p,p,p,q,p,q}$, with $\frac2p+\frac1q<1$;
\item $\sig=\vs{p,p,p,q,q,q}$, with $\frac1p+\frac1q<\frac23$;

\item $\sig=\vs{p,p,q,p,q,q}$, with $\frac1p+\frac1q<\frac23$;
\item $\sig=\vs{p,p,p,p,q,r}$, with $\frac4p+\frac1q+\frac1r<2$;
\item $\sig=\vs{p,p,p,q,p,r}$, with $\frac4p+\frac1q+\frac1r<2$;
\item $\sig=\vs{p,p,q,p,p,r}$, with $p$ even and $\frac4p+\frac1q+\frac1r<2$;

\item $\sig=\vs{p,p,p,q,q,r}$, with $q$ and $r$ even and $\frac3p+\frac2q+\frac1r<2$;
\item $\sig=\vs{p,p,p,q,r,q}$, with $q$ even and $\frac3p+\frac2q+\frac1r<2$;
\item $\sig=\vs{p,p,q,p,q,r}$, with $r$ even and $\frac3p+\frac2q+\frac1r<2$;
\item $\sig=\vs{p,q,p,p,q,r}$, with $r$ even and $\frac3p+\frac2q+\frac1r<2$;

\item $\sig=\vs{p,p,q,q,p,r}$, with $q$ even and $\frac3p+\frac2q+\frac1r<2$;
\item $\sig=\vs{p,p,q,q,r,r}$, with $p,q$, and $r$ all even and $\frac1p+\frac1q+\frac1r<1$;
\item $\sig=\vs{p,p,q,r,q,r}$, with $p$ even and $\frac1p+\frac1q+\frac1r<1$;
\item $\sig=\vs{p,p,p,q,r,s}$, with $q,r$, and $s$ all even and $\frac3p+\frac1q+\frac1r+\frac1s<2$;

\item $\sig=\vs{p,p,q,p,r,s}$, with $r$ and $s$ even and $\frac3p+\frac1q+\frac1r+\frac1s<2$;
\item $\sig=\vs{p,q,p,r,p,s}$, with $\frac3p+\frac1q+\frac1r+\frac1s<2$;
\item $\sig=\vs{p,p,q,q,r,s}$, with $p,q,r$, and $s$ all even and $\frac2p+\frac2q+\frac1r+\frac1s<2$;
\item $\sig=\vs{p,p,q,r,r,s}$, with $p,q,r$, and $s$ all even and $\frac2p+\frac2q+\frac1r+\frac1s<2$;

\item $\sig=\vs{p,q,r,p,q,s}$, with $p,q,r$, and $s$ all even and $\frac2p+\frac2q+\frac1r+\frac1s<2$;
\item $\sig=\vs{p,p,q,r,s,t}$, with $p,q,r,s,$ and $t$ all even and $\frac2p+\frac1q+\frac1r+\frac1s+\frac1t<2$;
\item $\sig=\vs{p,q,p,r,s,t}$ with $p,r,s,$ and $t$ all even and $\frac2p+\frac1q+\frac1r+\frac1s+\frac1t<2$;
\end{enumerate}
\end{cor}

The following propositions give the rate of growth of the nine monomorphic valence sequences of length 6 all of whose terms are at least $4$.

\begin{prop} If $\sig=\vs{p,p,p,p,p,p}$ with $p\geq 4$, then $\sig$ is monomorphic and \[ \gam(T) = 2p-5+2\sqrt{p^2-5p+6}. \]
\begin{proof} Clearly $T$ is edge-homogeneous with edge-symbol $\esym{p,p}{6,6}$, and so one may apply  \cite{GFHT}{Proposition 1.17}.
\end{proof}
\end{prop}
\begin{prop}\label{thm:GrowthOf-ppqppq} If $\sig=\vs{p,p,q,p,p,q}$, then $\frac2p+\frac1q<1$ holds, $\sig$ is monomorphic with $p$ even, and 
\[\gam(T) = \frac18\left(a+b+\sqrt{2 a b+2 b^2+32 p q-112 p-96 q+256}\right) \]
where $a = \sqrt{(p+2 q-8)(25 p+2 q-72)}$ and $b = 5 p+2 q-16$.

\begin{proof} 
This valence sequence gives rise to only four face types, as shown in \Cref{fig:ppqppq-offspring}, which in turn yields the transition matrix 
\[ M = \left[\begin{matrix}
\frac{5p-16}2 & 3p-10 & 2p-7 & 2p-6 \\
\frac{3q-10}2 & q-3 & q-3 & q-4 \\
3 & 2 & 2 & 2 \\
1 & 2 & 1 & 1 
\end{matrix}\right]. \]
\begin{figure}[htbp]
\includegraphics[width=\linewidth]{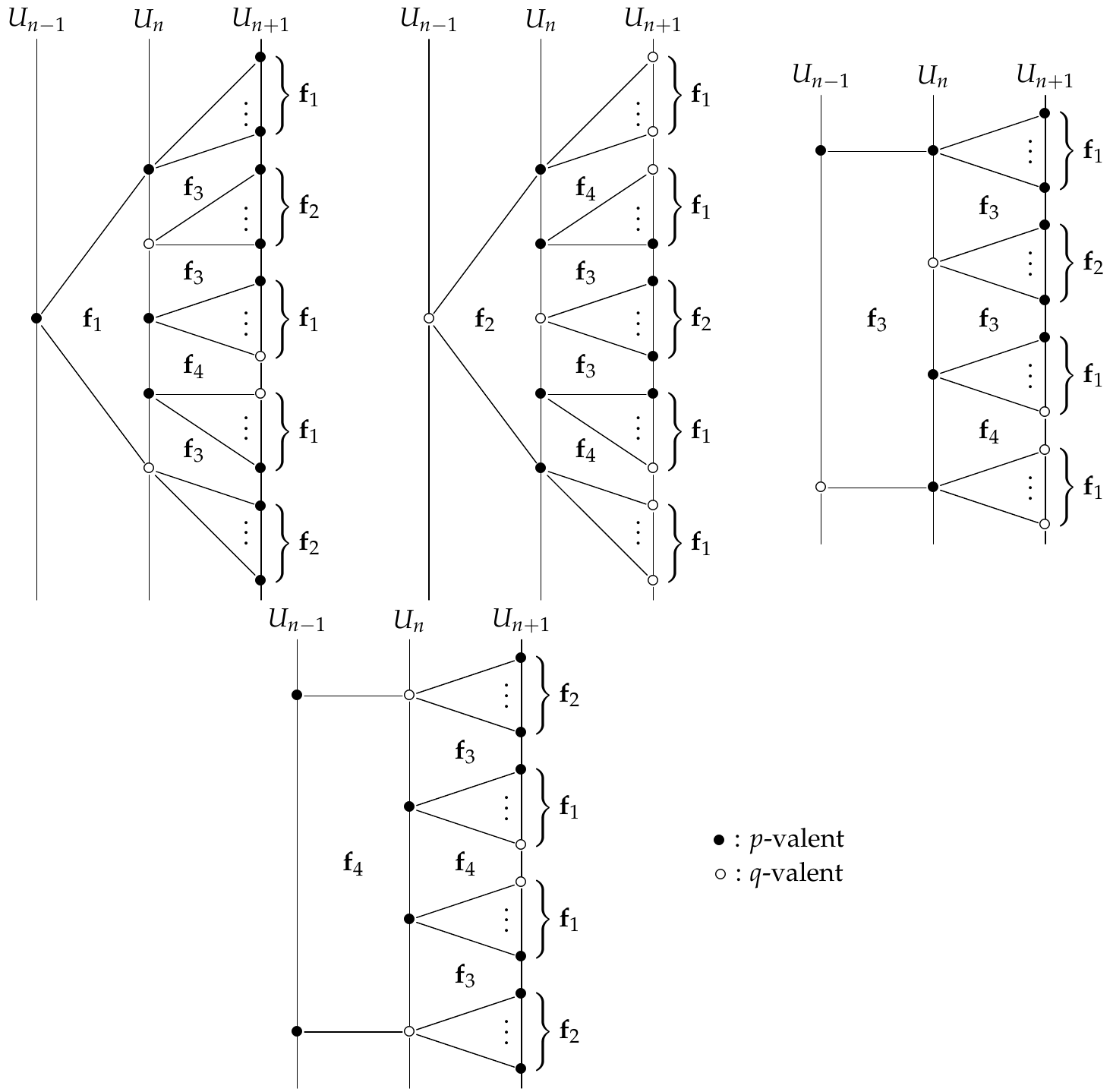}
\caption[Offspring diagrams for $\vs{p,p,q,p,p,q}$.]{\label{fig:ppqppq-offspring}Offspring diagrams for a face-homogeneous tesssellation with valence sequence $\vs{p,p,q,p,p,q}$.}
\end{figure}
The characteristic polynomial of $M$ is
\[\chi(z) = z^4-\left(\frac52p+q-8\right)z^3-(2pq-7p-6q+18)z^2-\left(\frac52p+q-8\right)z+1,\]
which is palindromic, and so \Cref{rem:PalinPoly} applies.
\end{proof}
\end{prop}

\begin{prop} If $\sig=\vs{p,q,p,q,p,q}$, then  $\frac1p+\frac1q<\frac23$ holds, $\sig$ is monomorphic,  and
\[ \gam(T) = p+q-5+\sqrt{(p+q-4)(p+q-6)}. \]
\begin{proof}
Again $T$ is edge-homogeneous with edge-symbol $\esym{p,q}{6,6}$, and so one again applies  \cite{GFHT}{Proposition 1.17}.
\end{proof}
\end{prop}

\begin{prop} If $\sig=\vs{p,q,q,p,r,r}$, then $\frac1p+\frac1q+\frac1r<1$ holds, $\sig$ is monomorphic with $p$, $q$, and $r$ all even, and the following hold:
\begin{enumerate}
\item $\sig_1 = \vs{4,6,6,4,8,8}$, $\sig_2 = \vs{6,4,4,6,8,8}$, and $\sig_3 = \vs{8,4,4,8,6,6}$ are pairwise incomparable minimal elements in this class. Their transition matrices yield, respectively, the characteristic polynomials
\begin{align*}
\chi_1(z) &= {\left(z^{6} - 5 \, z^{5} - 111 \, z^{4} - 34 \, z^{3} - 111 \, z^{2} - 5 \, z + 1\right)} {\left(z - 1\right)}, \\
\chi_2(z) &= {\left(z^{6} - 6 \, z^{5} - 87 \, z^{4} - 80 \, z^{3} - 87 \, z^{2} - 6 \, z + 1\right)} {\left(z - 1\right)}, \\
\intertext{and}
\chi_3(z) &= {\left(z^{6} - 7 \, z^{5} - 79 \, z^{4} - 94 \, z^{3} - 79 \, z^{2} - 7 \, z + 1\right)} {\left(z - 1\right)}.
\end{align*}
\item If $\sig_i$ is the valence sequence of the tessellation $T_i$, $(i=1,2,3)$, then 
\begin{align*}
\gam(T_1) &\approx 13.4721, & \gam(T_2) &\approx 13.1291, & \gam(T_3) &\approx 13.4341
\end{align*}
\item For any tessellation $T$ with valence sequence $\sig=\vs{p,q,q,p,r,r}$, we have $\gam(T)\geq 13.1291$.
\end{enumerate}
\end{prop}
\begin{proof}
The offspring diagrams given in \Cref{fig:pqqprr-offspring} demonstrate that $\sig$ is monomorphic with transition matrix
\[ M = \left[\begin{matrix}
p-3 & \frac{3p-10}2 & \frac{3p-10}2 & p-3 & p-3 & p-4 & p-4 \\
\frac{3q-10}2 & \frac{q-4}2 & 2q-6 & \frac{3q-10}2 & \frac{q-4}2 & 0 & 2q-6 \\
\frac{3r-10}2 & 2r-6 & \frac{r-4}2 & \frac{r-4}2 & \frac{3r-10}2 & 2r-6 & 0 \\
1 & 2 & 1 & 1 & 1 & 2 & 0 \\
1 & 1 & 2 & 1 & 1 & 0 & 2 \\
1 & 0 & 1 & 1 & 0 & 0 & 1 \\
1 & 1 & 0 & 0 & 1 & 1 & 0
\end{matrix}\right]. \] 
\begin{figure}[htbp]
\includegraphics[width=\linewidth]{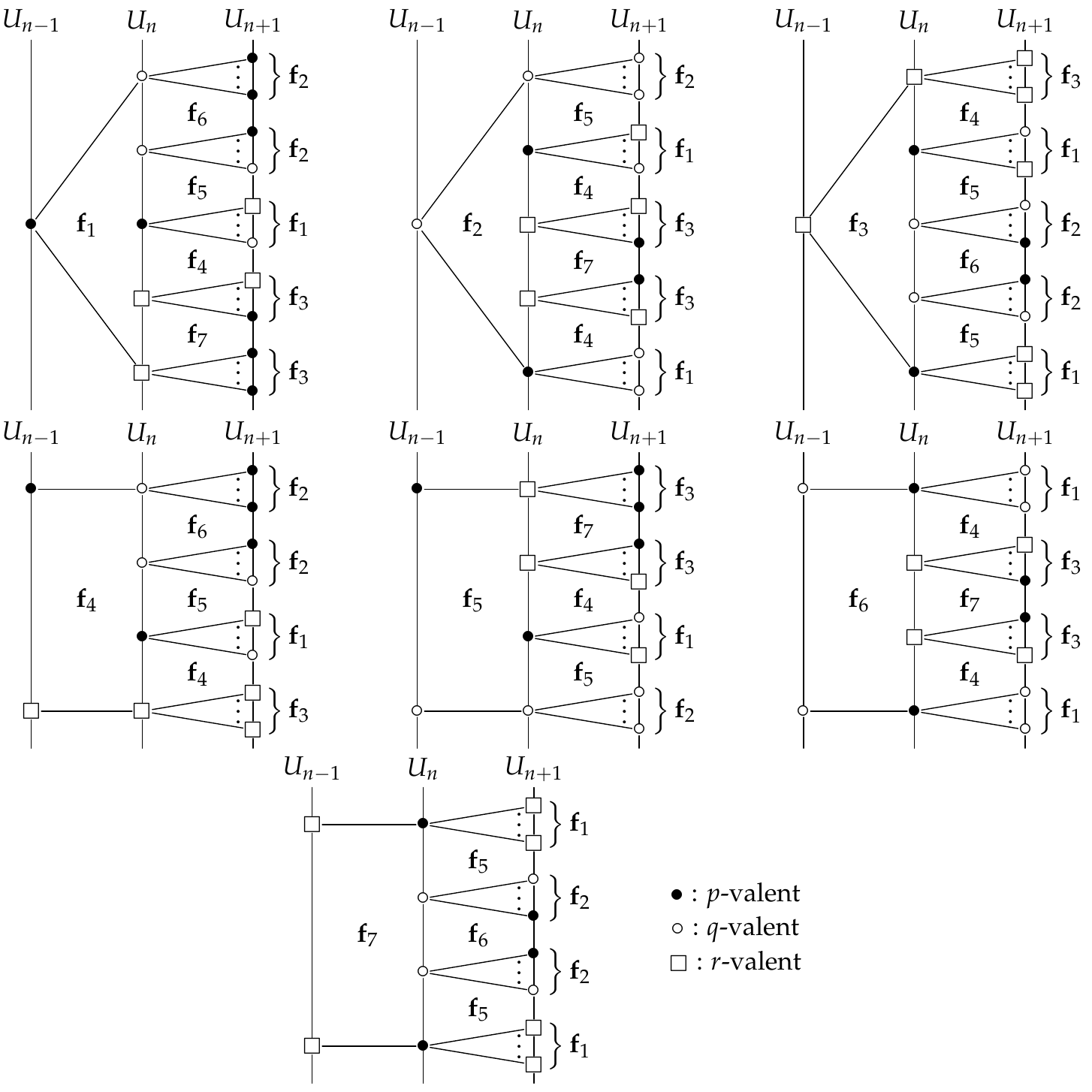}
\caption[Offspring diagrams for $\vs{p,q,q,p,r,r}$.]{\label{fig:pqqprr-offspring} Offspring diagrams for a face-homogeneous tessellation with valences sequence $\vs{p,q,q,p,r,r}$.}
\end{figure}
The characteristic polynomial of $M$ (which is palindromic in $-z$) is
\[ \chi(z) = z^7-az^6-bz^5-cz^4+cz^3+bz^2+az-1, \]
where
\begin{align*}
a &= \frac12(2 p+q+r-10), \\
b &= \frac14(7 p q+7 p r-36 p+15 q r-50 q-50 r+156), \text{ and}\\
c &= \frac14(12 p q r-33 p q-33 p r+88 p-57 q r+140 q+140 r-340)
\end{align*}
The minimal valence sequences in this class are as given; numerical approximation of eigenvalues provides the desired result.
\end{proof}

\begin{prop}\label{thm:GrowthOf-pqprqr} If $\sig=\vs{p,q,p,r,q,r}$, then $\sig$ is monomorphic with $p$ and $r$  even,  $\frac1p+\frac1q+\frac1r<1$ holds, and $\gam(T)\geq 9.8115$.
\begin{proof}
The offspring diagrams for $\sig$ appear in \Cref{fig:pqprqr-offspring}. They in turn give the transition matrix
\[ M = \left[\begin{matrix}
p-3 & p-4 & \frac{3p-10}2 & 2p-6 & p-3 & \frac{p-4}2 & \frac{3p-10}2 \\
\frac{q-4}2 & 0 & q-3 & q-3 & \frac{q-4}2 & 0 & q-3 \\
\frac{3r-10}2 & 2r-6 & r-3 & r-4 & r-3 & \frac{3r-10}2 & \frac{r-4}2 \\
q-3 & q-3 & \frac{q-4}2 & 0 & \frac{q-4}2 & q-3 & 0 \\
1 & 2 & 1 & 2 & 1 & 1 & 1 \\
1 & 0 & 2 & 2 & 1 & 0 & 2 \\
2 & 2 & 1 & 0 & 1 & 2 & 0
\end{matrix}\right]. \] 
The characteristic polynomial is
\[ \chi(z) = z^7-az^6-bz^b-cz^4+cz^3+bz^2+az-1,\]
with
\begin{align*}
	a &= p+r-5, \\
	b &= \frac14(10 p q+5 p r-36 p+4 q^2+10 q r-64 q-36 r+156), \text{ and} \\
	c &= \frac14\left(4pq^2+4qr^2+4pqr-38pq-19pr-38qr\right.\\
	  &\phantom{=} \left.\phantom{x^2}+88p+192q+88r-340\right).
\end{align*}
\begin{figure}[htbp]
\includegraphics[width=\linewidth]{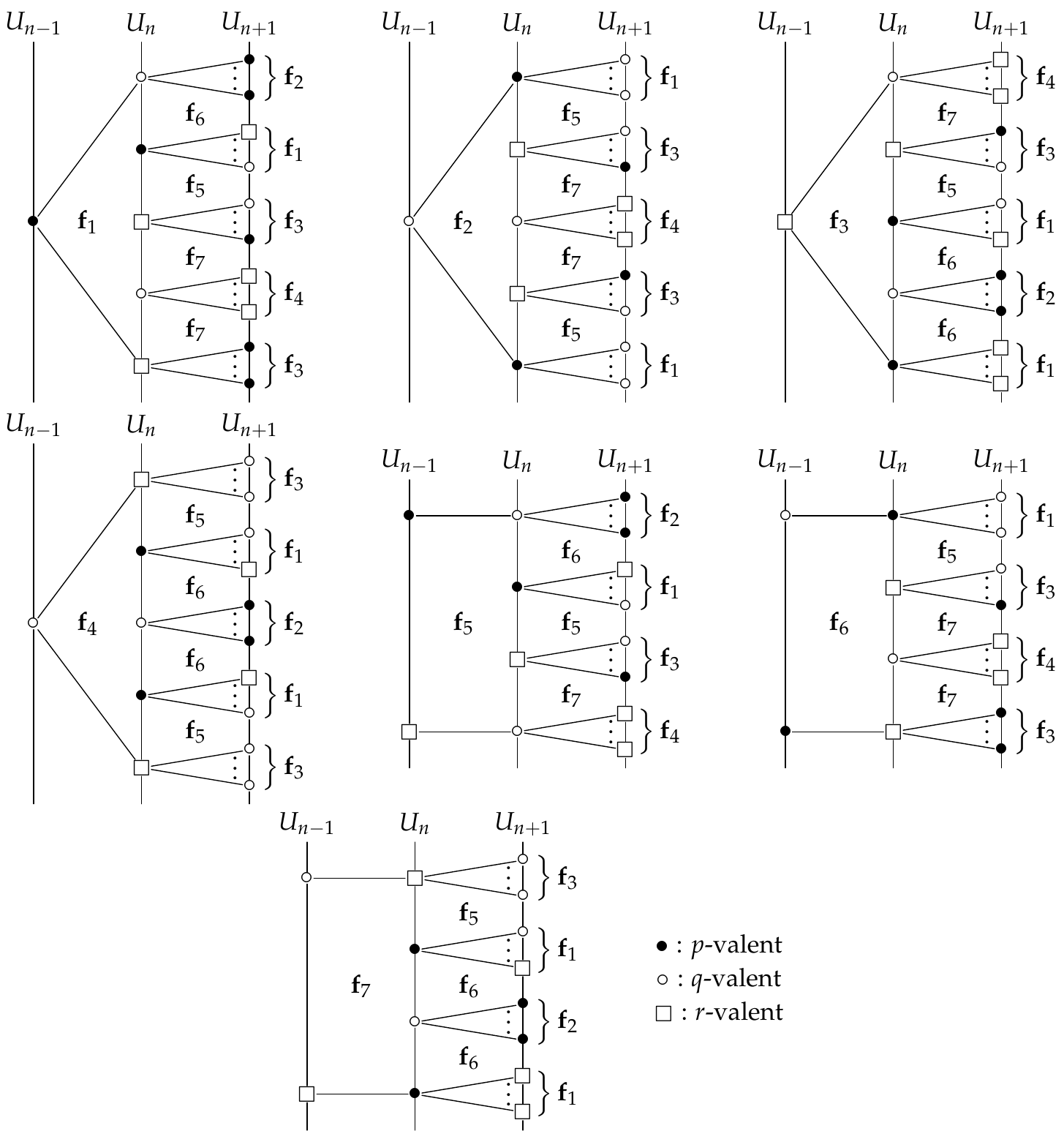}
\caption[Offspring diagrams for $\vs{p,q,p,r,q,r}$.]{\label{fig:pqprqr-offspring}Offspring diagrams for a face-homogeneous tessellation with valence sequence $\vs{p,q,p,r,q,r}$.}
\end{figure}
There are two minimal valence sequences in this class, namely  ${\sig_1=\vs{4,5,4,6,5,6}}$ and ${\sig_2=\vs{6,4,6,8,4,8}}$. As $\chi(z)$ does not factor, numerical approximation of the eigenvalues in the respective cases of $T_1$ and $T_2$ yields
$\gam(T_1) \approx 9.8115$ and
$\gam(T_2) \approx 13.5789$
Thus  for any tessellation $T$ with valence sequence $\sig\geq \sig_1$ and $\sig\geq\sig_2$, we have $\gam(T)\geq 9.8115$.
\end{proof}
\end{prop}


\begin{prop}\label{thm:GrowthOf-pqrpqr} If $\sig=\vs{p,q,r,p,q,r}$, then $\frac1p+\frac1q+\frac1r<1$ holds, $\sig$ is monomorphic with all valences even, and $\gam(T) \geq 13.5612$.
\begin{proof}
The offspring diagrams in \Cref{fig:pqrpqr-offspring} show that $\sig$ is monomorphic and give the transition matrix
\[ M = \left[\begin{matrix}
p-3 & \frac{3p-10}2 & \frac{3p-10}2 & p-3 & p-3 & p-4 \\
\frac{3q-10}2 & q-3 & \frac{3q-10}2 & q-4 & q-3 & q-3 \\
\frac{3r-10}2 & \frac{3r-10}2 & r-3 & r-3 & r-4 & r-3 \\
1 & 2 & 1 & 1 & 1 & 1 \\
1 & 1 & 2 & 1 & 1 & 1 \\
2 & 1 & 1 & 1 & 1 & 1
\end{matrix}\right].\]

\begin{figure}[htbp]
\includegraphics[width=\linewidth]{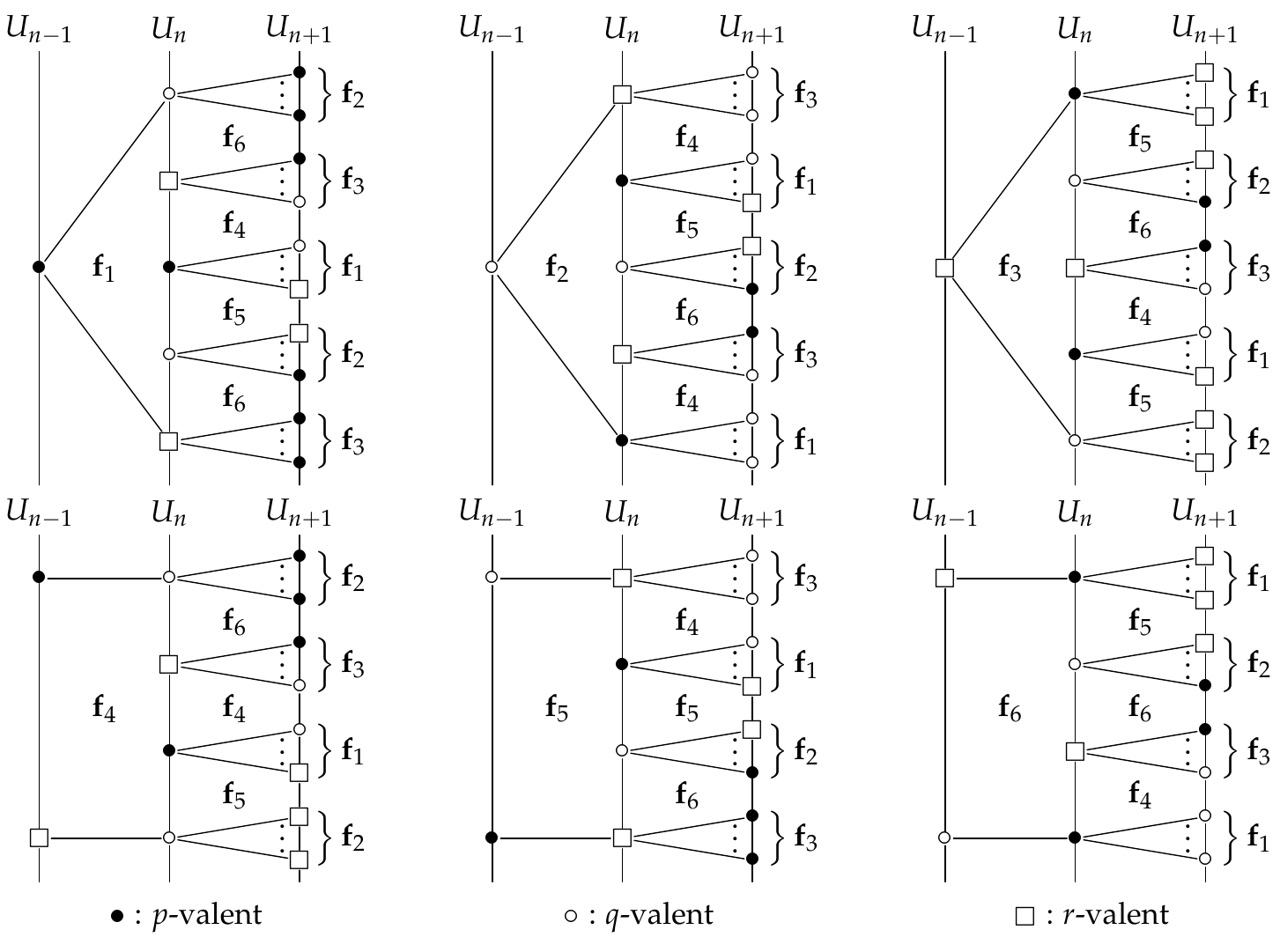}
\caption[Offspring diagrams for $\vs{p,q,r,p,q,r}$.]{\label{fig:pqrpqr-offspring}Offspring diagrams for a face-homogeneous tessellation with valence sequence $\vs{p,q,r,p,q,r}$.}
\end{figure}

\noindent The characteristic polynomial of $M$ is
\[ \chi(z) = z^6-az^5-bz^4-cz^3-bz^2-az+1,\] where
\begin{align*}
	a &= p+q+r-6, \\
	b &= \frac14\left(5(pq+pr+qr)-32(p+q+r)+132\right),\text{ and} \\
	c &= \frac12\left(2pqr-7(pq+pr+qr)+28(p+q+r)-104\right).
\end{align*}
Since $p,q$, and $r$ must all be even, the unique minimal valence sequence in this class is $\sig'=\vs{4,6,8,4,6,8}$. If $T'$ is a face-homogeneous tessellation with valence sequence $\sig'$, then numerical approximation of eigenvalues gives that 
$\gam(T')\approx 13.5612$.  By arguments similar to those above, we have $\gam(T)\geq 13.5612.$
\end{proof}
\end{prop}

\begin{prop}\label{thm:GrowthOf-pqprsr} If $\sig=\vs{p,q,p,r,s,r}$, then $\frac2p+\frac1q+\frac2r+\frac1s<2$ holds, $\sig$ is monomorphic with $p$ and $r$ even, and $\gam(T)\geq 10.9033$.

\begin{proof}
That $\sig$ is monomorphic is shown by the offspring diagrams in \Cref{fig:pqprsr-offspring}, which in turn yield  the transition matrix
\[ M = \left[\begin{matrix}
p-3 & p-4 & \frac{3p-10}2 & 2p-6 & p-3 & \frac{p-4}2 & \frac{3p-10}2 \\
\frac{q-4}2 & 0 & q-3 & q-3 & \frac{q-4}2 & 0 & q-3 \\
\frac{3r-10}2 & 2r-6 & r-3 & r-4 & r-3 & \frac{3r-10}2 & \frac{r-4}2 \\
s-3 & s-3 & \frac{s-4}2 & 0 & \frac{s-4}2 & s-3 & 0 \\
1 & 2 & 1 & 2 & 1 & 1 & 1 \\
1 & 0 & 2 & 2 & 1 & 0 & 2 \\
2 & 2 & 1 & 0 & 1 & 2 & 0
\end{matrix}\right].\]
The characteristic polynomial of $M$ is
\[\chi(z) = z^7-az^6-bz^5-cz^4+cz^3+bz^2+az-1,\] where
\begin{align*}
	a &= p+r-5, \\
	b &= \frac12pq+\frac54pr+2ps+2qr+qs+\frac12rs-9p-8q-9r-8s+39 \text{,~and} \\
	c &= \frac12pqr+pqs+\frac12prs+qrs-\frac72pq-\frac{19}4pr-6ps-6qr-7qs- 
		\frac72rs\\
	  &~\quad +22p+24q+22r+24s-85.
\end{align*}

The four minimal valence sequences in this class are $\sig_1=\vs{4,5,4,6,7,6}$, $\sig_2=\vs{4,7,4,6,5,6}$, $\sig_3=\vs{6,4,6,8,5,8}$, and $\sig_4=\vs{6,5,6,8,4,8}$. Numerical approximation of eigenvalues of their respective transition matrices gives the following growth rates, where $T_i$ is a face-homogeneous tessellation with valence sequence $\sig_i$:
\begin{align*}
	\gam(T_1) &\approx 10.9033; \\
	\gam(T_2) &\approx 11.2073; \\
	\gam(T_3) &\approx 14.2638;\text{ and} \\
	\gam(T_4) &\approx 14.4008.
\end{align*}
By \cite{GFHT}{Theorem 2.7}, we have $\gam(T)\geq 10.9033$.
\end{proof}

\begin{figure}
\includegraphics[width=\linewidth]{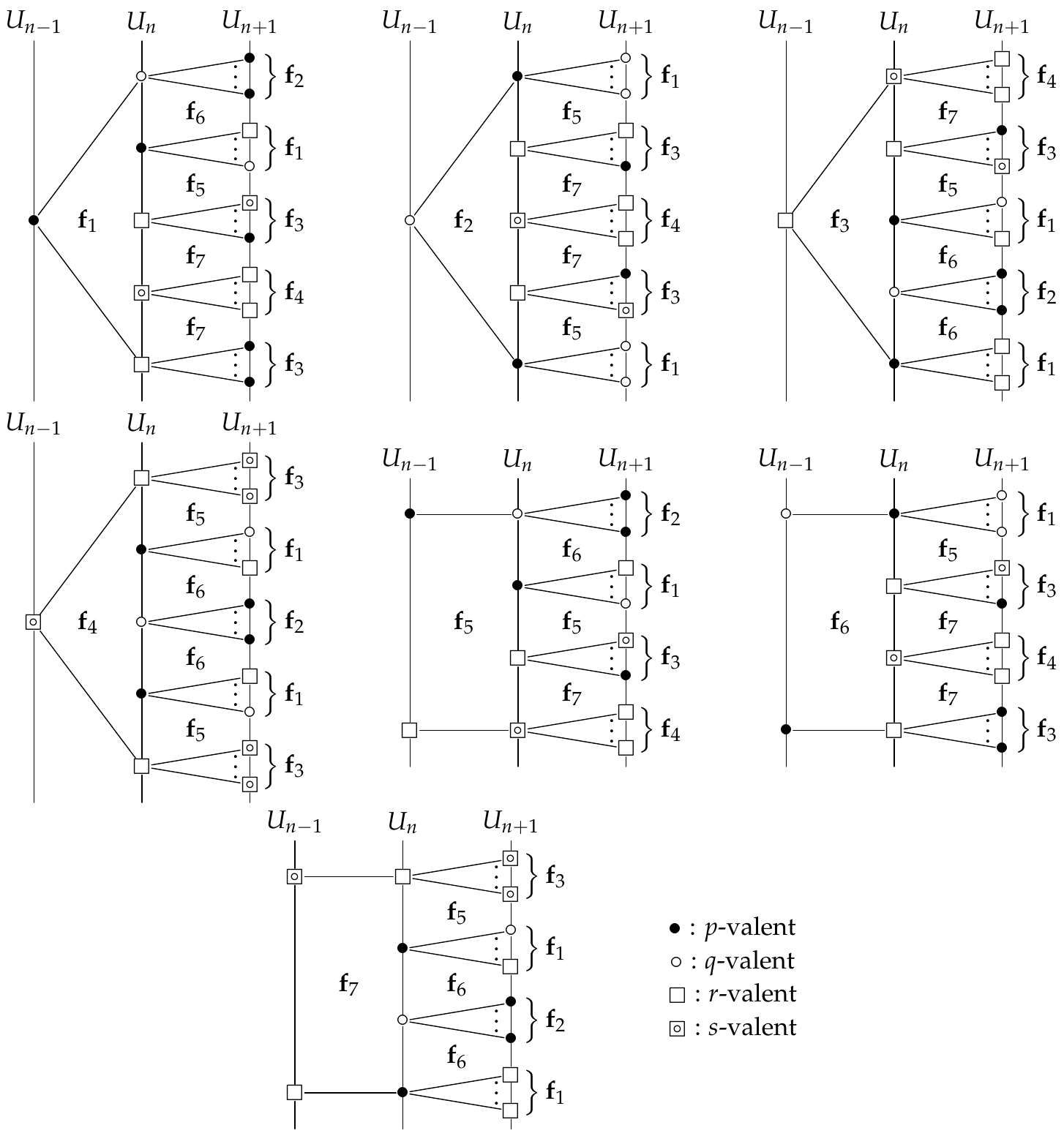}
\caption[Offspring diagrams for $\vs{p,q,p,r,s,r}$.]{\label{fig:pqprsr-offspring}Offspring diagrams for a face-homogeneous tessellation with valence sequence $\vs{p,q,p,r,s,r}$.}
\end{figure}
\end{prop}

\begin{prop} If $\sig=\vs{p,q,r,p,s,t}$, then $\frac2p+\frac1q+\frac1r+\frac1s+\frac1t<2$ holds, $\sig$ is monomorphic with all terms even, and $\gam(T)\geq 18.1174$.
\begin{proof}
As $T$ is monomorphic and belongs to $\scr{G}_{4,4}$, we again appeal to \cite{GFHT}{Lemma 2.9} to obtain the $12\times 12$ transition matrix 
{\fontsize{7}{9}\selectfont\[ M = \left[\begin{array}{cccccccccccc}
0 & \frac{p-4}2 & p-3 & p-3 & p-3 & \frac{p-4}2 & 0 & 0 & \frac{p-4}2 & p-3 & p-3 & \frac{p-4}2 \\
\frac{q-4}2 & 0 & \frac{q-4}2 & q-3 & q-3 & q-3 & \frac{q-4}2 & 0 & 0 & \frac{q-4}2 & q-3 & q-3 \\
r-3 & \frac{r-4}2 & 0 & \frac{r-4}2 & r-3 & r-3 & r-3 & \frac{r-4}2 & 0 & 0 & 
\frac{r-4}2 & r-3 \\
p-3 & p-3 & \frac{p-4}2 & 0 & \frac{p-4}2 & p-3 & p-3 & p-3 & \frac{p-4}2 & 0 &
0 & \frac{p-4}2 \\
s-3 & s-3 & s-3 & \frac{s-4}2 & 0 & \frac{s-4}2 & \frac{s-4}2 & s-3 & s-3 & 
\frac{s-4}2 & 0 & 0 \\
\frac{t-4}2 & t-3 & t-3 & t-3 & \frac{t-4}2 & 0 & 0 & \frac{t-4}2 & t-3 & 
t-3 & \frac{t-4}2 & 0 \\ 
0 & 0 & 1 & 1 & 1 & 1 &		0 & 0 & 1 & 1 & 1 & 0 \\
1 & 0 & 0 & 1 & 1 & 1 & 		0 & 0 & 0 & 1 & 1 & 1 \\
1 & 1 & 0 & 0 & 1 & 1 &		1 & 0 & 0 & 0 & 1 & 1 \\
1 & 1 & 1 & 0 & 0 & 1 &		1 & 1 & 0 & 0 & 0 & 1 \\
1 & 1 & 1 & 1 & 0 & 0 &		1 & 1 & 1 & 0 & 0 & 0 \\
0 & 1 & 1 & 1 & 1 & 0 &		0 & 1 & 1 & 1 & 0 & 0 
\end{array}\right]. \]}
The minimal valence sequences in this class are $\sig_1=\vs{4,6,8,4,10,12}$ and $\sig_2=\vs{6,4,8,6,10,12}$.  Continuing previous notational convention, we have by numerical approximation
\[ \gam(T_1) \approx 18.1174\quad\text{and}\quad\gam(T_2)\approx19.4819.\]
Thus by \cite{GFHT}{Theorem 2.7}, we have $\gam(T)\geq 18.1174$.
\end{proof}
\end{prop}

\begin{prop}\label{thm:GrowthOf-pqrstu} If $\sig=\vs{p,q,r,s,t,u}$, then $\frac1p+\frac1q+\frac1r+\frac1s+\frac1t+\frac1u<2$ holds, $\sig$ is monomorphic with all terms even, and $\gam(T)\geq 23.9963$.
\begin{proof}
As all the valences are distinct, $\sig$ is monomorphic, and so  \cite{GFHT}{Lemma 2.9} applies to determine directly the transition matrix{\fontsize{7}{9}\selectfont 
\[ M = \left[\begin{array}{cccccccccccc}
0 & \frac{p-4}2 & p-3 & p-3 & p-3 & \frac{p-4}2 &
0 & 0 & \frac{p-4}2 & p-3 & p-3 & \frac{p-4}2 \\
\frac{q-4}2 & 0 & \frac{q-4}2 & q-3 & q-3 & q-3 &
\frac{q-4}2 & 0 & 0 & \frac{q-4}2 & q-3 & q-3 \\
r-3 & \frac{r-4}2 & 0 & \frac{r-4}2 & r-3 & r-3 &
r-3 & \frac{r-4}2 & 0 & 0 & \frac{r-4}2 & r-3 \\
s-3 & s-3 & \frac{s-4}2 & 0 & \frac{s-4}2 & s-3 &
s-3 & s-3 & \frac{s-4}2 & 0 & 0 & \frac{s-4}2 \\
t-3 & t-3 & t-3 & \frac{t-4}2 & 0 & \frac{t-4}2 &
\frac{t-4}2 & t-3 & t-3 & \frac{t-4}2 & 0 & 0 \\
\frac{u-4}2 & u-3 & u-3 & u-3 & \frac{u-4}2 & 0 &
0 & \frac{u-4}2 & u-3 & u-3 & \frac{u-4}2 & 0 \\ 
0 & 1 & 1 & 1 & 1 & 0 &		0 & 0 & 1 & 1 & 1 & 0 \\
0 & 0 & 1 & 1 & 1 & 1 &		0 & 0 & 0 & 1 & 1 & 1 \\
1 & 0 & 0 & 1 & 1 & 1 &		1 & 0 & 0 & 0 & 1 & 1 \\
1 & 1 & 0 & 0 & 1 & 1 &		1 & 1 & 0 & 0 & 0 & 1 \\
1 & 1 & 1 & 0 & 0 & 1 &		1 & 1 & 1 & 0 & 0 & 0 \\
1 & 1 & 1 & 1 & 0 & 0 &		0 & 1 & 1 & 1 & 0 & 0 
\end{array}\right].\]}

As all the valences must also be even, all the minimal valence sequences in this class are permutations on $(4,6,8,10,12,14)$, of which there are exactly $\frac12(6-1)!=60$.  Among these, the valence sequence in this class with the least rate of growth is $\vs{4,6,10,14,12,8}$, as shown in \Cref{fig:pqrstu-MinSeqs}, with growth rate approximately $23.9963$.

\begin{table}[htbp]
\[\begin{array}{cc|c@{\hspace{1em}}||@{\hspace{1em}}c|c}
&\sig_i&\gam(T_i) & \sig_i&\gam(T_i)\\ \hline\hline
&\vs{4, 6, 8, 10, 12, 14} & 24.41712251 & \vs{4, 8, 10, 6, 12, 14} & 24.87880785\\
&\vs{4, 6, 8, 10, 14, 12} & 24.21960545 & \vs{4, 8, 10, 6, 14, 12} & 24.80899840\\
&\vs{4, 6, 8, 12, 10, 14} & 24.61270141 & \vs{4, 8, 10, 12, 6, 14} & 25.35303706\\
&\vs{4, 6, 8, 12, 14, 10} & 24.08285752 & \vs{4, 8, 10, 14, 6, 12} & 25.18185562\\
&\vs{4, 6, 8, 14, 10, 12} & 24.48957224 & \vs{4, 8, 12, 6, 10, 14} & 25.03584620\\ \hline
&\vs{4, 6, 8, 14, 12, 10} & 24.15679180 & \vs{4, 8, 12, 6, 14, 10} & 24.90132812\\
&\vs{4, 6, 10, 8, 12, 14} & 24.58261968 & \vs{4, 8, 12, 10, 6, 14} & 25.23325447\\
&\vs{4, 6, 10, 8, 14, 12} & 24.45332058 & \vs{4, 8, 12, 14, 6, 10} & 24.77199967\\
&\vs{4, 6, 10, 12, 8, 14} & 24.85506399 & \vs{4, 8, 14, 6, 10, 12} & 25.01874720\\
&\vs{4, 6, 10, 12, 14, 8} & 24.04416760 & \vs{4, 8, 14, 6, 12, 10} & 24.95377767\\ \hline
&\vs{4, 6, 10, 14, 8, 12} & 24.68019759 & \vs{4, 8, 14, 10, 6, 12} & 24.96611444\\
\star&\mathbf{\vs{4, 6, 10, 14, 12, 8}} & \mathbf{23.99630569} & \vs{4, 8, 14, 12, 6, 10} & 24.67581103\\
&\vs{4, 6, 12, 8, 10, 14} & 24.80458415 & \vs{4, 10, 6, 8, 12, 14} & 24.98236812\\
&\vs{4, 6, 12, 8, 14, 10} & 24.54658791 & \vs{4, 10, 6, 8, 14, 12} & 24.82334169\\
&\vs{4, 6, 12, 10, 8, 14} & 24.88247982 & \vs{4, 10, 6, 12, 8, 14} & 25.58786488\\ \hline
&\vs{4, 6, 12, 10, 14, 8} & 24.27525170 & \vs{4, 10, 6, 14, 8, 12} & 25.52589967\\
&\vs{4, 6, 12, 14, 8, 10} & 24.41387775 & \vs{4, 10, 8, 6, 12, 14} & 24.98070079\\
&\vs{4, 6, 12, 14, 10, 8} & 24.06126298 & \vs{4, 10, 8, 6, 14, 12} & 24.89788355\\
&\vs{4, 6, 14, 8, 10, 12} & 24.78717089 & \vs{4, 10, 8, 12, 6, 14} & 25.70437512\\
&\vs{4, 6, 14, 8, 12, 10} & 24.65801011 & \vs{4, 10, 8, 14, 6, 12} & 25.58679602\\ \hline
&\vs{4, 6, 14, 10, 8, 12} & 24.74623904 & \vs{4, 10, 12, 6, 8, 14} & 25.18902550\\
&\vs{4, 6, 14, 10, 12, 8} & 24.33423320 & \vs{4, 10, 12, 8, 6, 14} & 25.30889683\\
&\vs{4, 6, 14, 12, 8, 10} & 24.45269393 & \vs{4, 10, 14, 6, 8, 12} & 25.05471909\\
&\vs{4, 6, 14, 12, 10, 8} & 24.16830858 & \vs{4, 10, 14, 8, 6, 12} & 25.04268293\\
&\vs{4, 8, 6, 10, 12, 14} & 24.71739812 & \vs{4, 12, 6, 8, 10, 14} & 25.38856036\\ \hline
&\vs{4, 8, 6, 10, 14, 12} & 24.50356321 & \vs{4, 12, 6, 10, 8, 14} & 25.71589810\\
&\vs{4, 8, 6, 12, 10, 14} & 24.99665188 & \vs{4, 12, 8, 6, 10, 14} & 25.24072240\\
&\vs{4, 8, 6, 12, 14, 10} & 24.43259325 & \vs{4, 12, 8, 10, 6, 14} & 25.77151261\\
&\vs{4, 8, 6, 14, 10, 12} & 24.93312365 & \vs{4, 12, 10, 6, 8, 14} & 25.29247408\\
&\vs{4, 8, 6, 14, 12, 10} & 24.58265856 & \vs{4, 12, 10, 8, 6, 14} & 25.49621988\\ \hline
\end{array}
\]
\caption[Table of minimal sequences in the class $\vs{p,q,r,s,t,u}$, with approximate growth rates.]{\label{fig:pqrstu-MinSeqs}Table of minimal sequences in the class $\vs{p,q,r,s,t,u}$, with approximate growth rates. See \Cref{thm:GrowthOf-pqrstu}.}
\end{table}
Thus by \cite{GFHT}{Theorem 2.7}, we have $\gam(T)\geq 23.9963$.
\end{proof}
\end{prop}

\smallskip

\subsection{\label{sec:6cov-3val} Valence sequences of length 6 with at least one 3-valent vertex}
We first classify those valence sequences which are polymorphic. Complete sets of offspring diagrams are generally not needed and hence not provided, as polymorphicity becomes evident by considering offspring of only some of the face types.

\begin{thm} If $\sig$ is one of the following 23 valence sequences where $p,q,r,s\geq 4$, then $\sig$ is polymorphic: 
\[\begin{array}{cc|cc|cc}
\text{Case} & \text{Sequence} & \text{Case} & \text{Sequence} & \text{Case} & \text{Sequence} \\ \hline
\text{(1)} & \vs{3,3,3,3,3,p} & \text{(9)}  & \vs{3,3,3,p,3,q} & \text{(17)} & \vs{p,p,q,q,p,3} \\
\text{(2)} & \vs{p,p,p,p,p,3} & \text{(10)} & \vs{p,p,p,3,p,q} & \text{(18)} & \vs{p,p,3,q,3,q} \\
\text{(3)} & \vs{3,3,3,3,p,p} & \text{(11)} & \vs{p,p,3,p,p,q} & \text{(19)} & \vs{3,3,3,p,q,r} \\
\text{(4)} & \vs{3,3,3,p,3,p} & \text{(12)} & \vs{3,3,3,p,p,q} & \text{(20)} & \vs{p,p,3,p,q,r} \\
\text{(5)} & \vs{p,p,p,3,p,3} & \text{(13)} & \vs{p,p,p,q,3,q} & \text{(21)} & \vs{p,3,p,q,p,r} \\
\text{(6)} & \vs{3,3,3,p,p,p} & \text{(14)} & \vs{p,p,3,p,3,q} & \text{(22)} & \vs{p,3,p,q,r,s} \\
\text{(7)} & \vs{3,3,p,3,p,p} & \text{(15)} & \vs{p,3,p,p,3,q}  \\
\text{(8)} & \vs{3,3,3,3,p,q} &\text{(16)} & \vs{3,3,p,p,3,q} \\
\end{array}\]
\begin{proof}
Exactly twelve of the $23$ listed cases satisfy the hypothesis of \cite{GFHT}{Proposition 3.1}, including belonging to $\scr{G}_{3^+,5}$.  For these cases it suffices to identify the triples satisfying the proposition's hypotheses. Valence sequences (2), (5), (6), and (10) contain triples $(p,p,3)$ as well as $(p,p,p)$; valence sequences (11), (12), (14), (17), and (18) contain both $(p,p,3)$ and $(p,p,q)$; sequence (13) contains both $(p,p,p)$ and $(p,p,q)$;  sequence (20) contains both $(p,p,3)$ and $(p,p,r)$;  finally, sequence (21) contains both $(q,p,3)$ and $(q,p,r)$.
We treat the remaining eight cases in turn.
\begin{enumerate}[(1)~]
\item $\vs{3,3,3,3,3,p}$:     We have distinct possible sets of offspring of a face of type ${\ft_5}$ as shown in \Cref{fig:33333p-ambig}. Thus $\sig$ is polymorphic.

\begin{figure}
\includegraphics{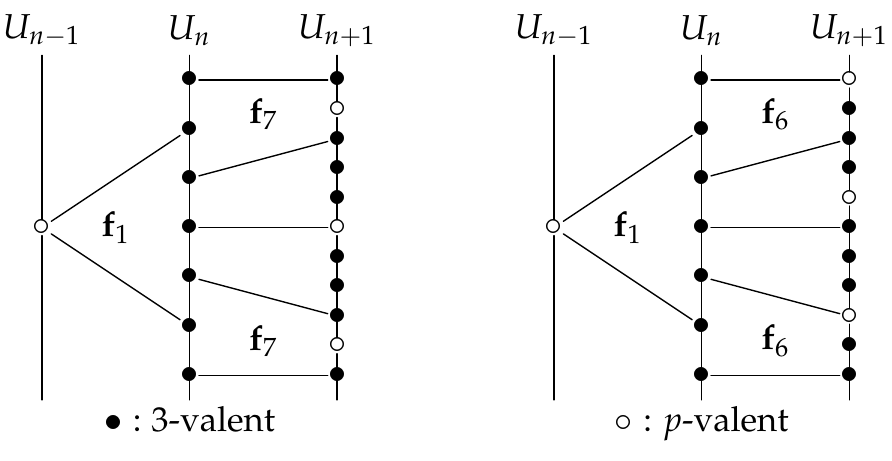}

\caption[Offspring diagrams showing $\vs{3,3,3,3,3,p}$ is polymorphic.]{\label{fig:33333p-ambig}Distinct possible offspring diagrams for a face of type ${\ft_1}$ in tessellations with valence sequence $\vs{3,3,3,3,3,p}$.}
\end{figure}

\stepcounter{enumi}
\item $\vs{3,3,3,3,p,p}$:   Thus $p$ is even.  The $3$-valent vertices fall into two classes:  those adjacent only to other $3$-valent vertices, and those adjacent to two $3$-valent vertices and one $p$-valent vertex. Once the neighbors of an arbitrary $3$-valent vertex are determined, the vertices at regional distance $2$ from that $3$-valent vertex can be identified by the configurations given in \Cref{fig:3333pp-3Root}.

\begin{figure}[htbp]
\[\begin{matrix}
\includegraphics[width=4cm]{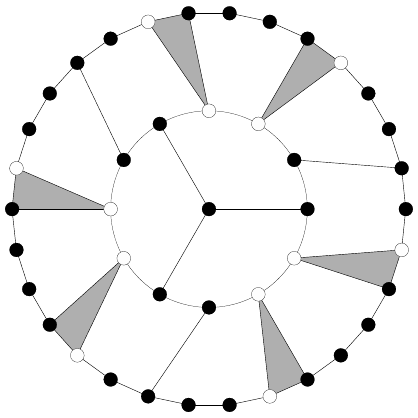} & \quad &
\includegraphics[width=4cm]{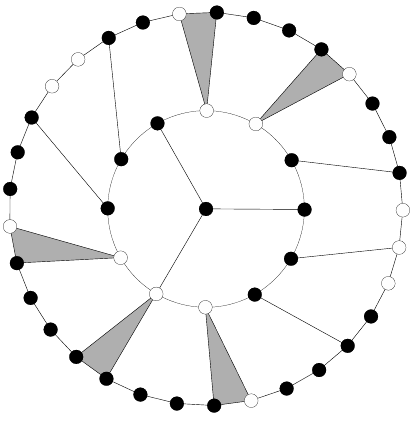} \\
\text{(A)} & & \text{(B)}
\end{matrix}\]
\caption[First two coronas of valence sequence $\vs{3,3,3,3,p,p}$ with $3$-valent root.]{\label{fig:3333pp-3Root} First two coronas of tessellations with valence sequence $\vs{3,3,3,3,p,p}$ when rooted at a $3$-valent vertex. Figure (A) depicts the root as a $3$-valent vertex adjacent only to $3$-valent vertices; figure (B) depicts the root as a $3$-valent vertex adjacent to two $3$-valent vertices and one $p$-valent vertex. Shaded regions represent $p-3$ or $p-4$ adjacent wedges all of the same type.}
\end{figure}

Both of these configurations admit wedges in $F_2$ incident with a $p$-valent vertex in $U_1$. The configuration of descendants of wedges in $F_n$ with a $p$-valent vertex in $U_{n-1}$ is not uniquely determined. Let these faces be denoted as of type ${\ft_3}$.

Faces of type ${\ft_3}$ in $F_n$ have two distinct possible sets of offspring, as shown in \Cref{fig:3333pp-ft3-ambig}.
\begin{figure}[htbp]
\includegraphics{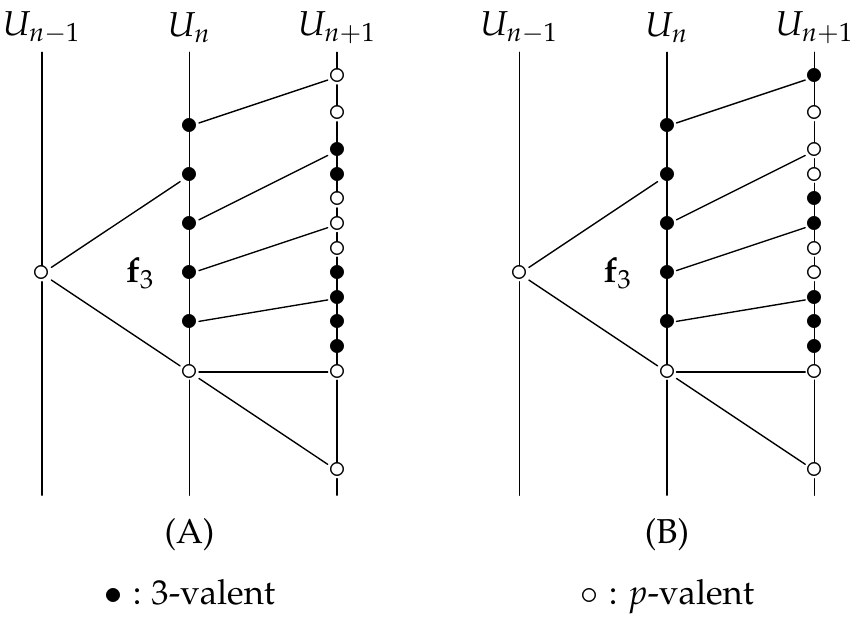}
\caption[Offspring diagrams showing $\vs{3,3,3,3,p,p}$ is polymorphic.]{\label{fig:3333pp-ft3-ambig}Distinct sets of offspring for a face of type ${\ft_3}$ in a tessellation with valence sequence $\vs{3,3,3,3,p,p}$.}
\end{figure}
When adjacent faces of type ${\ft_3}$ occur in $F_n$, both incident to a common $3$-valent vertex in $U_n$, one must have offspring as specified in \Cref{fig:3333pp-ft3-ambig}(A) and the other must have offspring as specified in \Cref{fig:3333pp-ft3-ambig}(B).   However, the arrangement (in clockwise order about the common $p$-valent vertex in $U_{n-1}$ may be (A)-(B) or (B)-(A). The choice of one arrangement at one point in the tessellation in no way restricts the choices of the other arrangement elsewhere; hence the valence sequence is polymorphic.

\item $\vs{3,3,3,p,3,p}$:   Wedges in $F_n$ with a $p$-valent vertex in $U_{n-1}$ have four possible configurations, partially determined by the vertices incident with the face, as shown in \Cref{fig:333p3p-ambig}. Face types ${\ft_1}$ and ${\ft_2}$ have their offspring completely determined by these neighboring vertices. However, a face in $F_n$ which has as offspring wedges each incident with a $p$-valent vertex in $U_n$ admits multiple distinct sets of offspring, as various accretion rules can specify varying numbers of faces of types ${\ft_1},{\ft_2},{\ft_3}$, and ${\ft_4}$ as offspring. Therefore $\sig$ is polymorphic.
\begin{figure}
\includegraphics{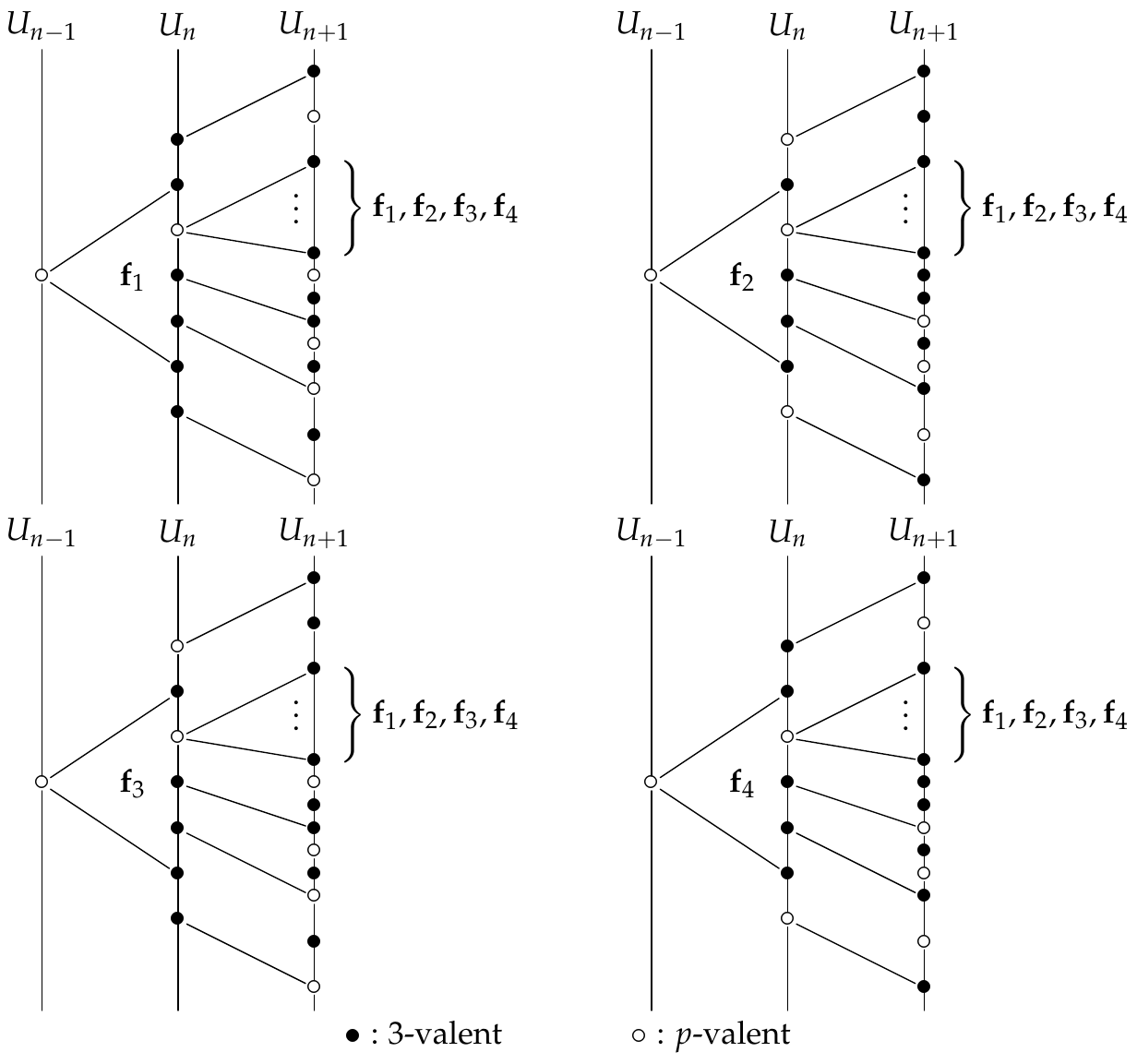}
\caption[Offspring diagrams showing $\vs{3,3,3,p,3,p}$ is polymorphic.]{\label{fig:333p3p-ambig}Offspring diagrams for wedges in $F_n$ of a face-homogeneous tessellation with valence sequence $\vs{3,3,3,p,3,p}$, with $p$-valent vertex in $U_{n-1}$.}
\end{figure}

\stepcounter{enumi}
\stepcounter{enumi}

\item $\vs{3,3,p,3,p,p}$:    As shown in \Cref{fig:33p3pp-pWedge}, there are at least two distinct types of wedges which occur in $F_n$ incident with a $p$-valent vertex in $U_{n-1}$:  types $\ft_1$ and $\ft_2$. Pairs of adjacent wedges occurring as offspring of a face in $F_n$ may thus be $\ft_1\ft_1$-pairs or $\ft_2,\ft_2$-pairs with no effect on the surrounding offspring. When $p=4$, these adjacent wedges occur only in $F_1$; if $p\geq 5$, the pairs occur in $F_n$ for all $n>1$. Therefore $\sig$ is polymorphic. 

\begin{figure}
\includegraphics{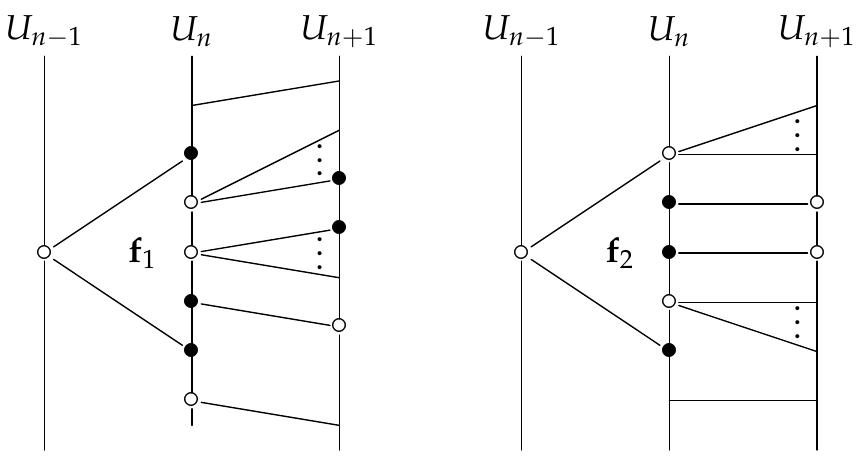}
\caption[Offspring diagrams showing $\vs{3,3,p,3,3,p}$ is polymorphic.]{\label{fig:33p3pp-pWedge}Wedges of types ${\ft_1}$ and ${\ft_2}$ in a face-homogeneous tessellation with valence sequence $\vs{3,3,p,3,p,p}$.}
\end{figure}

\item $\vs{3,3,3,3,p,q}$:   Both $p$ and $q$ are even, and $\frac1p+\frac1q<\frac23$ holds. The wedge type ${\ft_1}$ in a tessellation with valence sequence $\sig$ has the configuration shown in \Cref{fig:3333pq-ambig}. 
In the figure, the only vertices drawn in $U_{n+1}$ are those which are explicitly determined by $\sig$ as offspring of an $\ft_1$ face; all other vertices incident with the offspring of an $\ft_1$ face in $F_n$ are dependent upon the choice of accretion rule.
\begin{figure}
\includegraphics{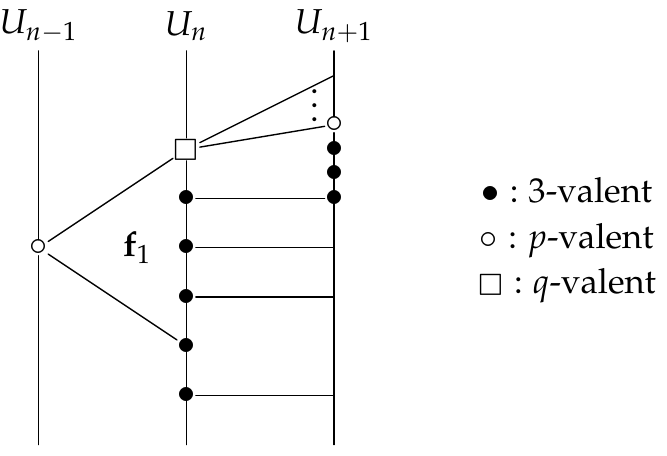}
\caption[Offspring diagrams showing $\vs{3,3,3,3,p,q}$ is polymorphic.]{\label{fig:3333pq-ambig} Offspring of a wedge of type ${\ft_1}$ in a face-homogeneous tessellation with valence sequence $\vs{3,3,3,3,p,q}$.}
\end{figure}

\item $\vs{3,3,3,p,3,q}$:   Thus $\frac1p+\frac1q<\frac23$ holds. Wedges of type ${\ft_1}$ have the configuration shown in \Cref{fig:333p3q-ambig-appx}, where the only vertices drawn in $U_{n+1}$ are those which are explicitly determined by the valence sequence in the offspring of a $\ft_1$ face; all other vertices incident with the offspring of an $\ft_3$ face in $F_n$ are determined by the choice of accretion rule.
\begin{figure}[htbp]
\includegraphics{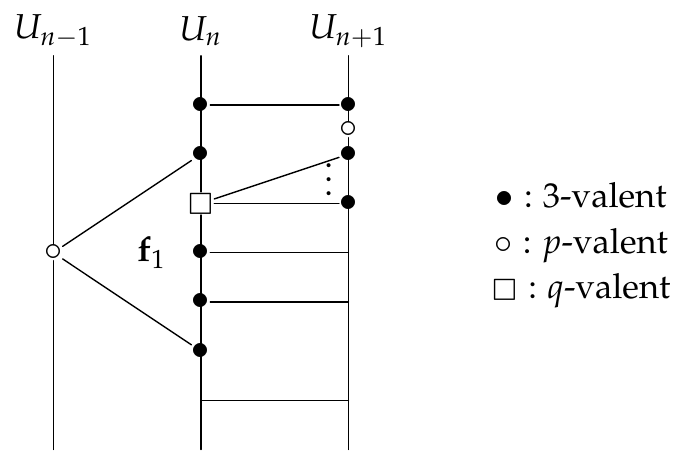}
\caption[Offspring diagrams showing $\vs{3,3,3,p,3,q}$ is polymorphic.]{\label{fig:333p3q-ambig-appx}Offspring of a wedge of type ${\ft_1}$ in a face-homogeneous tessellation with valence sequence $\vs{3,3,3,p,3,q}$.}
\end{figure}

\stepcounter{enumi}
\stepcounter{enumi}
\stepcounter{enumi}
\stepcounter{enumi}
\stepcounter{enumi}
\item $\vs{p,3,p,p,3,q}$:   $p$ and $q$ both even and $\frac3p+\frac1q<\frac43$ holds. The wedge type ${\ft_1}$ shown in \Cref{fig:p3pp3q-ambig} has distinct sets of offspring dependent upon the chosen accretion rule, so $\vs{p,3,p,p,3,q}$ is polymorphic.
\begin{figure}[htbp]
\includegraphics{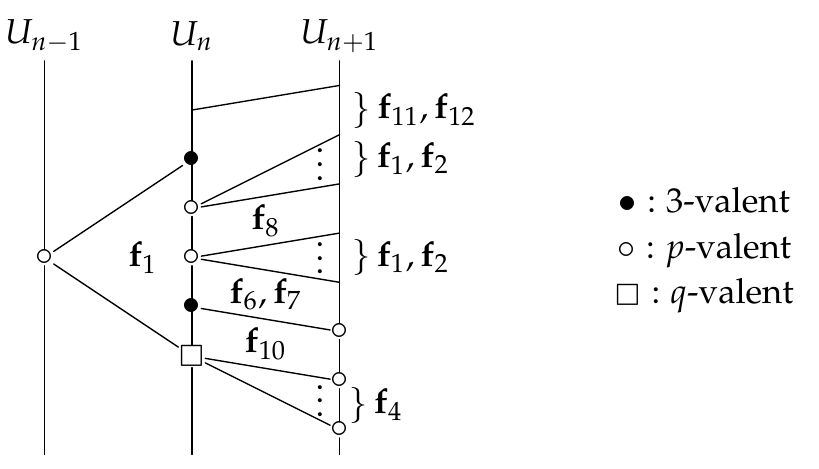}
\caption[Offspring diagrams showing $\vs{p,3,p,p,3,q}$ is polymorphic.]{\label{fig:p3pp3q-ambig}Offspring diagram for wedge of type ${\ft_1}$ in a face-homogeneous tessellation with valence sequence $\vs{p,3,p,p,3,q}$.}
\end{figure}

\item $\vs{3,3,p,p,3,q}$:   Thus $p$ is even and $\frac2p+\frac1q<1$ holds. The wedge  of type ${\ft_1}$ illustrated in \Cref{fig:33pp3q-ambig} admits multiple distinct offspring diagrams depending upon the choice of accretion rule. Hence $\vs{3,3,p,p,3,q}$ is polymorphic.
\begin{figure}[htbp]
\includegraphics{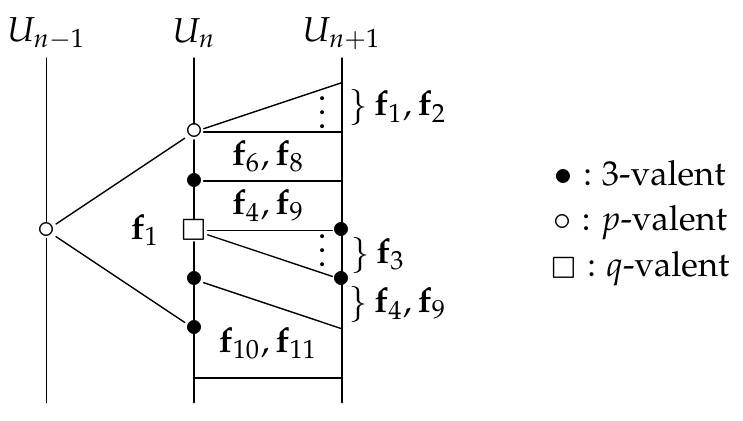}
\caption[Offspring diagrams showing $\vs{3,3,p,p,3,q}$ is polymorphic.]{\label{fig:33pp3q-ambig}Offspring diagram for a wedge of type ${\ft_1}$ in a face-homogeneous tessellation with valence sequence $\vs{3,3,p,p,3,q}$.}
\end{figure}

\stepcounter{enumi}
\stepcounter{enumi}

\item $\vs{3,3,3,p,q,r}$:   $p$, $q$, and $r$ all even, $\frac1p+\frac1q+\frac1r<1$. Wedges of type ${\ft_1}$ have one of two configurations of offspring in a tessellation with valence sequence $\vs{3,3,3,p,q,r}$, as shown in \Cref{fig:333pqr-ambig}. The two diagrams differ in the face type marked with an arrow.
\begin{figure}[htbp]
\includegraphics{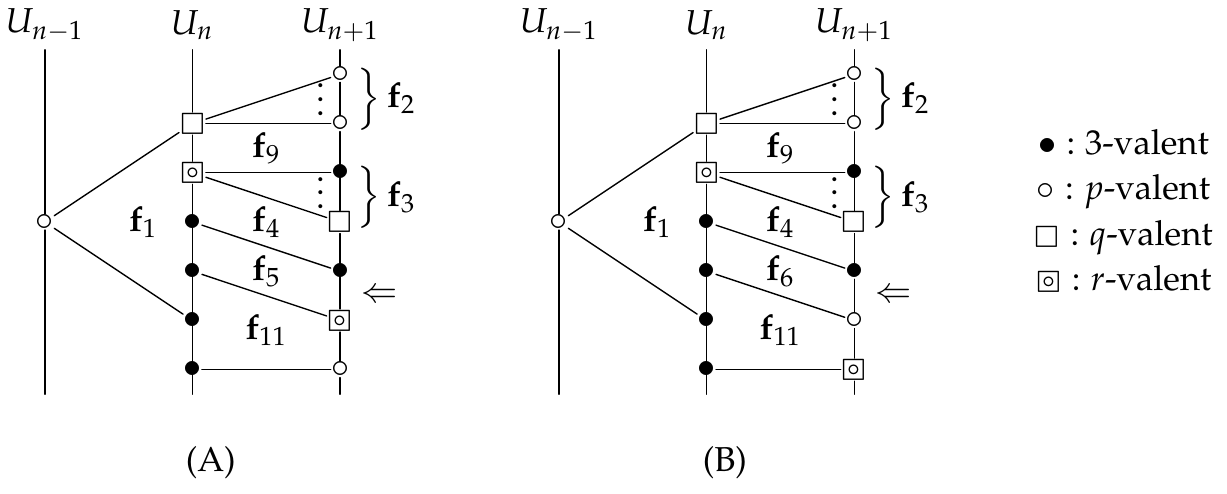}
\caption[Offspring diagrams showing $\vs{3,3,3,p,q,r}$ is polymorphic.]{\label{fig:333pqr-ambig}Distinct offspring diagrams for a wedge of type ${\ft_1}$ in a face-homogeneous tessellation with valence sequence $\vs{3,3,3,p,q,r}$.}
\end{figure}
Whenever two faces of type $\ft_1$ are incident with a common edge with edge-symbol $\esym{3,p}{6,6}$, the two configurations (A) and (B) shown in the figure will be paired; the choice of clockwise orientation of ``(A) then (B)" versus ``(B) then (A)" demonstrates that $\sig$ is polymorphic.

\stepcounter{enumi}

\stepcounter{enumi}

\item $\vs{p,3,p,q,r,s}$:   $p,q,r,$ and $s$ even. There are two distinct sets of offspring for a wedge of type ${\ft_1}$, as shown in \Cref{fig:p3pqrs-ambig}.  Once again, the choice of clockwise orientation about a $p$-valent vertex in $U_{n-1}$ determines the pattern ``(A) then (B)" versus ``(B) then (A)," and gives that $\sig$ is polymorphic.
\begin{figure}[htbp]
\includegraphics{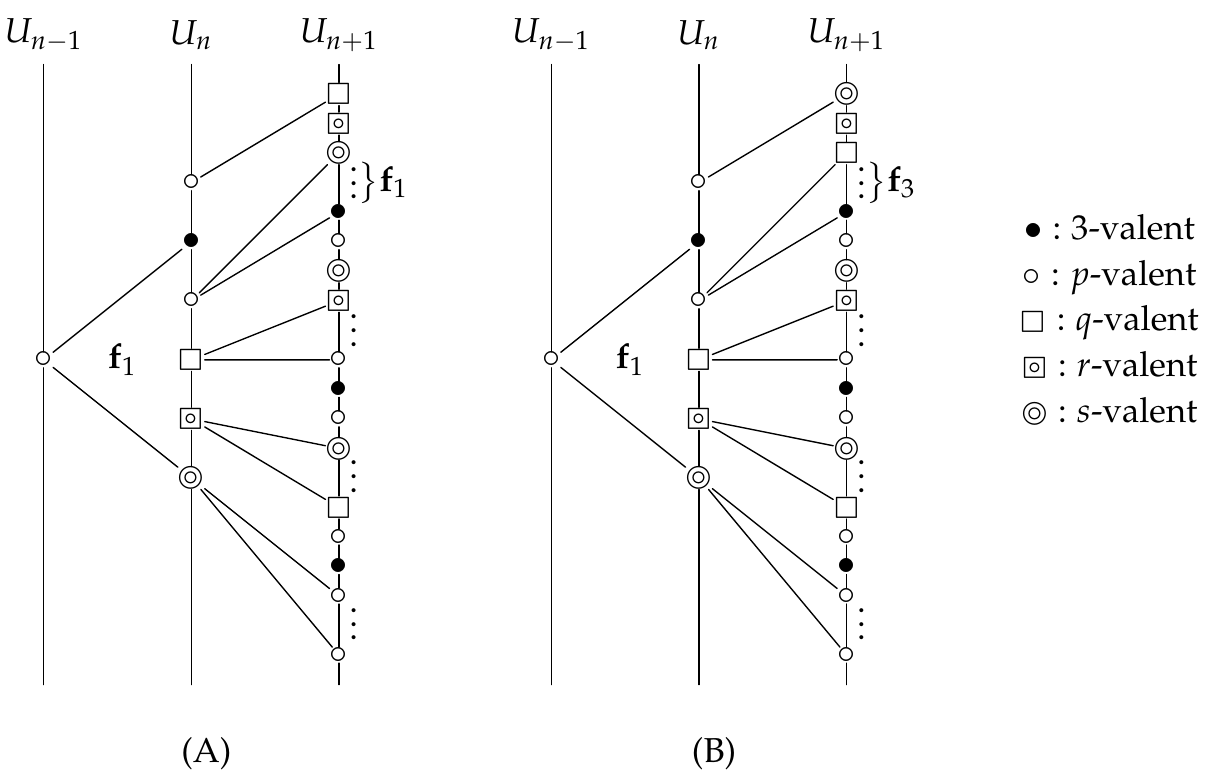}
\caption[Offspring diagrams showing $\vs{p,3,p,q,r,s}$ is polymorphic.]{\label{fig:p3pqrs-ambig}Offspring diagram for a wedge of type ${\ft_1}$ in a face-homogeneous tessellation with valence sequence $\vs{p,3,p,q,r,s}$. Only those offspring which may vary in type are labeled.}
\end{figure}

\end{enumerate}
This concludes the proof of the theorem.
\end{proof}
\end{thm}

\begin{prop}\label{thm:GrowthOf-pp3pp3} If $\sig=\vs{p,p,3,p,p,3}$, then $\sig$ is monomorphic with $p$ even, and 
\[ \gam(T) = \frac18\left[ 5p-10+a+\sqrt{50p^2+(10a-216)p-20a+168}\right], \]
where $a=\sqrt{25p^2-116+132}$.
\begin{proof} \sloppypar The offspring diagrams shown in \Cref{fig:pp3pp3-offspring} demonstrate that $\sig$ is monomorphic with transition matrix
\[ M = \left[\begin{matrix}
\frac{5p-16}2 & 3p-10 & 2p-7 & 2p-6 & p-4 \\
0 & 0 & 0 & 0 & 0 \\
2 & 2 & 2 & 0 & 2 \\
1 & 0 & 1 & 1 & 0 \\
\frac12 & 0 & 0 & 1 & 0 
\end{matrix}\right].\]
Its characteristic polynomial is 
\[ \chi(z) = z^5+\left(5-\frac52p\right)z^4+pz^3+\left(5-\frac52p\right)z^2+z, \] 
which has a palindromic quartic factor, and we apply \Cref{rem:PalinPoly}.
\end{proof}
\end{prop}

\begin{figure}[htbp]
\includegraphics[width=\linewidth]{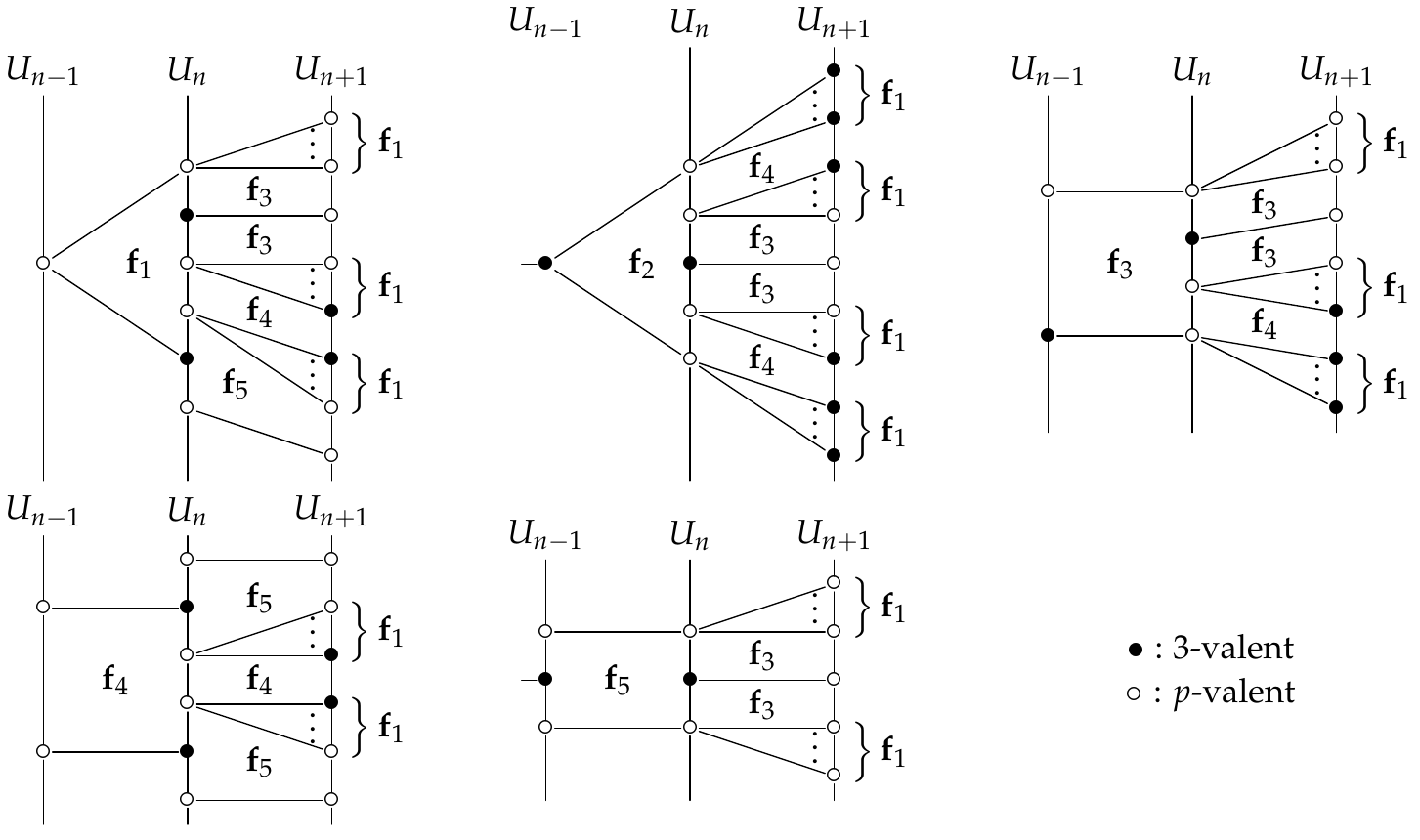}
\caption[Offspring diagrams for $\vs{p,p,3,p,p,3}$.]{\label{fig:pp3pp3-offspring}Offspring diagrams for a face-homogeneous tessellation with valence sequence $\vs{p,p,3,p,p,3}$.}
\end{figure}

\begin{prop} If $\sig=\vs{3,p,3,p,3,p}$, then $\sig$ is monomorphic and \[\gam(T)=p-2+\sqrt{(p-2)^2-1}.\]
\begin{proof}
As $T$ is edge-homogeneous with edge-symbol $\esym{3,p}{6,6}$, the result follows from \cite{GFHT}{Proposition 1.17}.
\end{proof}\end{prop}

\begin{prop}\label{thm:GrowthOf-333pqp} If $\sig=\vs{3,3,3,p,q,p}$, then $p$ is even, $\sig$ is monomorphic, and \[ \gam(T) \geq \frac14\left(1+\sqrt{57}+\sqrt{42+2\sqrt{57}}\right)\approx4.0265. \]
\begin{proof} Clearly $p,q\geq 4$. The unique set of offspring diagrams for $\sig$ is depicted in \Cref{fig:333pqp-offspring}. Taking 
\[ M = \left[\begin{matrix}
p-3 & p-4 & p-3 & \frac{3p-10}2 & \frac{p-4}2 & p-3 & p-4 \\
\frac{q-4}2 & 0 & \frac{q-4}2 & q-3 & 0 & \frac{q-4}2 & q-3 \\
1 & 2 & 1 & 0 & 1 & 0 & 0 \\
1 & 2 & 0 & 0 & 1 & 0 & 0 \\
1 & 0 & 1 & 2 & 0 & 1 & 2 \\
0 & 0 & \frac12 & \frac12 & 0 & \frac12 & 0 \\
\frac12 & 0 & 0 & 0 & \frac12 & 0 & 0 
\end{matrix}\right]\]
to be the transition matrix, we obtain
\[ \chi(z) = \frac14(2z-1)\left(2z^6-az^5-bz^4-cz^3-bz^2-az+2\right),\]
where
\begin{align*}
	a &= 2p-4, \\
	b &= pq+p+2q-14,\text{~and}\\
	c &= 2pq-10p-12q+40.
\end{align*}
The  minimal valence sequences in this class are $\sig_1=\vs{3,3,3,4,5,4}$ and $\sig_2=\vs{3,3,3,6,4,6}$; respectively, tessellations $T_1$ and $T_2$ with these valence sequences have growth rates
\begin{align*}
	\gam(T_1) &= \frac14\left(1+\sqrt{57}+\sqrt{42+2\sqrt{57}}\right)
		\approx 4.0265\quad\text{and} \\
	\gam(T_2) &= 3+2\sqrt{2} \approx 5.8284.
\end{align*}
By \cite{GFHT}{Theorem 2.7}, the growth rate of a face-homogeneous tessellation $T$ with valence sequence $\sig$ such that $\sig\geq\sig_1$ or $\sig\geq\sig_2$ satisfies the stated inequality.
\end{proof}
\end{prop}

\begin{figure}[htbp]
\includegraphics[width=\linewidth]{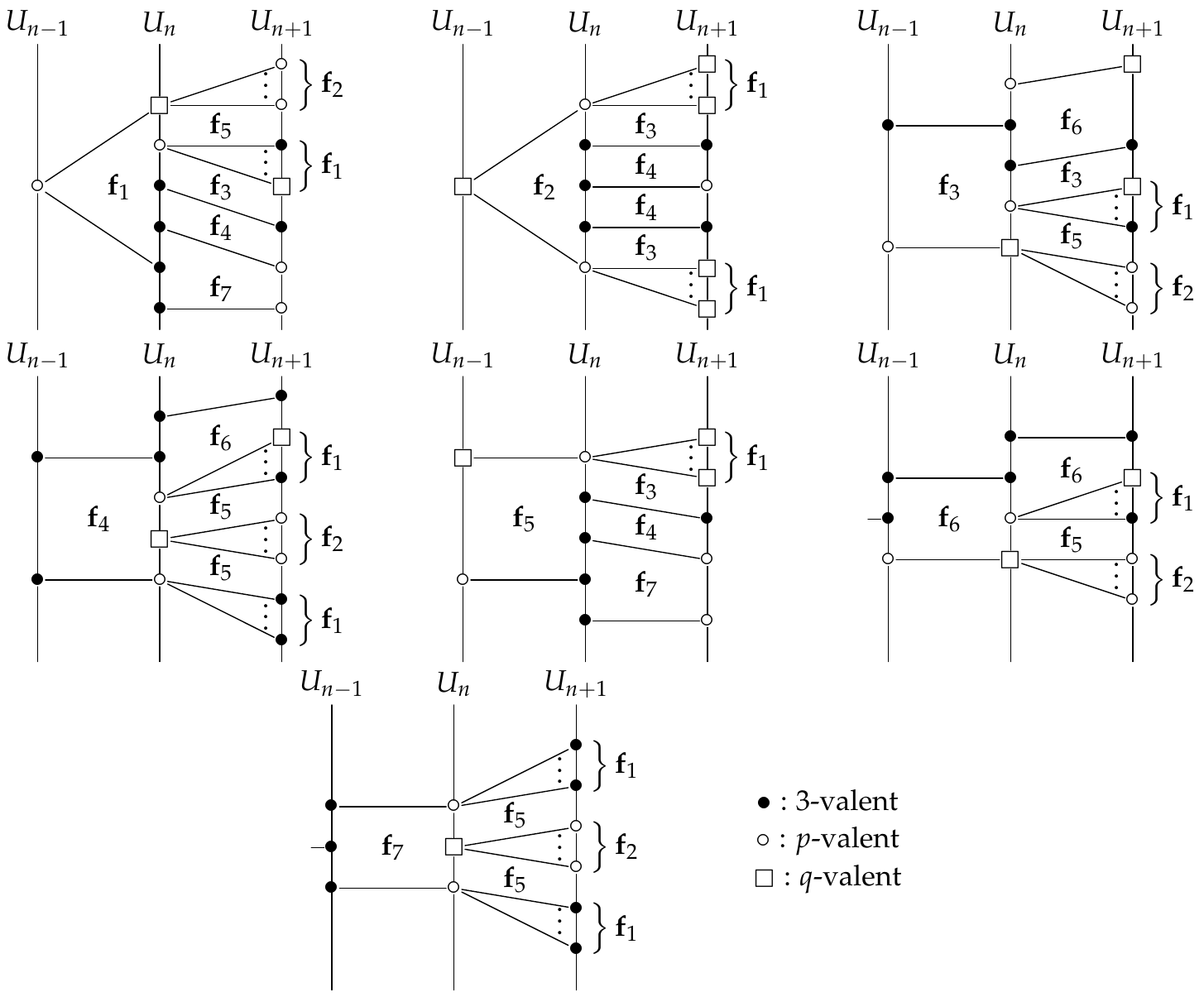}
\caption[Offspring diagrams for $\vs{3,3,3,p,q,p}$.]{\label{fig:333pqp-offspring}Offspring diagrams for a face-homogeneous tessellation with valence sequence $\vs{3,3,3,p,q,p}$. See \Cref{thm:GrowthOf-333pqp}.}
\end{figure}


\begin{figure}[htbp]
\includegraphics[width=\linewidth]{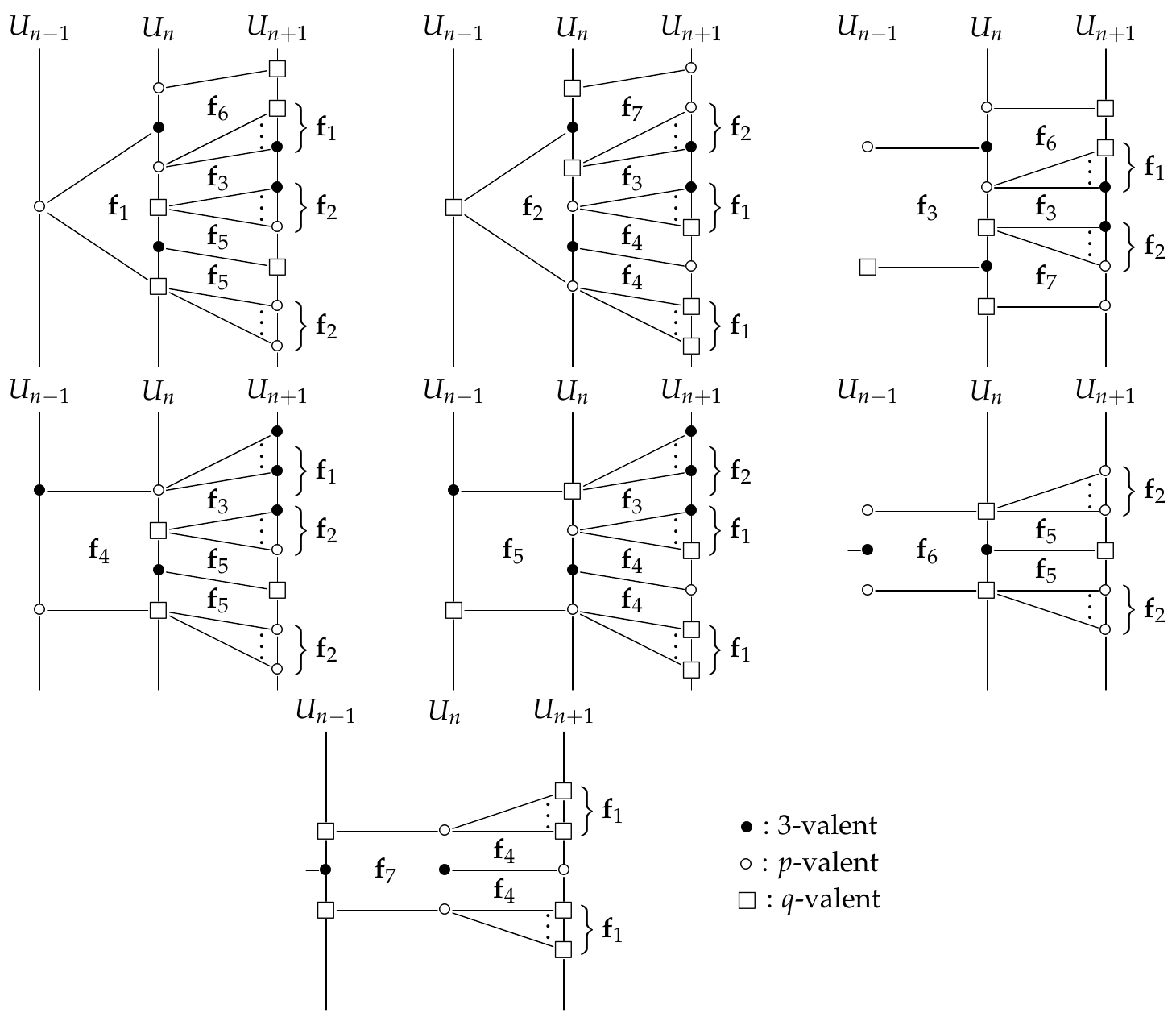}
\caption[Offspring diagrams for $\vs{p,3,p,q,3,q}$.]{\label{fig:p3pq3q-offspring}Offspring diagrams for a face-homogeneous tessellation with valence sequence $\vs{p,3,p,q,3,q}$.}
\end{figure}
\begin{prop} If $\sig=\vs{p,3,p,q,3,q}$, then $\sig$ is monomorphic with  both $p$ and $q$ even, and $\gam(T)\geq 6.8091$.
\begin{proof}
The offspring diagrams for a tessellation $T$ with valence sequence $\sig$ are shown in \Cref{fig:p3pq3q-offspring}. Hence $\sig$ is monomorphic, with 
\[ M = \left[\begin{matrix}
p-3 & \frac{3p-10}2 & p-3 & \frac{p-4}2 & \frac{3p-10}2 & 0 & p-4 \\
\frac{3q-10}2 & q-3 & q-3 & \frac{3q-10}2 & \frac{q-4}2 & q-4 & 0 \\
1 & 1 & 1 & 1 & 1 & 0 & 0 \\
0 & 2 & 0 & 0 & 2 & 0 & 2 \\
2 & 0 & 0 & 2 & 0 & 2 & 0 \\
\frac12 & 0 & \frac12 & 0 & 0 & 0 & 0 \\
0 & \frac12 & \frac12 & 0 & 0 & 0 & 0
\end{matrix}\right].\]
This transition matrix has the characteristic polynomial 
\[ \chi(z) = \frac14(z-1)(4z^6-az^5-bz^4+cz^3-bz^2-az+4), \]
where
\begin{align*}
	a &= 4p+4q-24,\\
	b &= 5pq-2p-2q-24,\text{~and}\\
	c &= 2pq-8p-8q+40.
\end{align*}

As the characteristic polynomial is not solvable, we consider the minimal valence sequence in this class, $\sig_1=\vs{4,3,4,6,3,6}$. If $T_1$ is a face-homogeneous tessellation with valence sequence $\sig_1$, we have by numerical approximation of eigenvalues that $\gam(T_1)\approx6.8091$. Hence by {Theorem 2.7}, if $T$ is a face-homogeneous tessellation with valence sequence $\sig\geq\sig_1$,  then $\gam(T)\geq 6.8091$.
\end{proof}\end{prop}


\begin{prop}\label{thm:GrowthOf-3p3q3r} If $\sig=\vs{3,p,3,q,3,r}$, then $\sig$ is monomorphic and  $\gam(T)\geq 5.6429.$
\begin{proof}
Clearly each $3$-valent vertex of $T$ has a neighbor of each other valence.   This yields the offspring diagrams shown in \Cref{fig:3p3q3r-offspring} showing $\sig$ to be monomorphic with transition matrix
{\fontsize{7}{10}\selectfont
\[M = \left[\begin{array}{ccc|ccc|ccc|ccc}
0 & p-3 & p-3 & \frac{p-4}2 & 0 & 0 & \frac{p-4}2 & p-3 & p-3 & 0 & 0 & p-3 \\
q-3 & 0 & q-3 & q-3 & q-3 & \frac{q-4}2 & 0 & 0 & \frac{q-4}2 & q-3 & 0 & 0 \\
r-3 & r-3 & 0 & 0 & \frac{r-4}2 & r-3 & r-3 & \frac{r-4}2 & 0 & 0 & r-3 & 0 \\
\hline
0 & 0 & 1 & 	1 & 0 & 0 &		0 & 0 & 1 &		0 & 0 & 0 \\
1 & 0 & 0 &		0 & 1 & 1 & 	0 & 0 & 0 & 	0 & 0 & 0 \\
0 & 1 & 0 & 	0 & 0 & 0 &		1 & 1 & 0 & 	0 & 0 & 0 \\
\hline
0 & 0 & 1 & 	1 & 0 & 0 & 	0 & 0 & 1 & 	0 & 0 & 0 \\
1 & 0 & 0 & 	0 & 1 & 1 & 	0 & 0 & 0 & 	0 & 0 & 0 \\
0 & 1 & 0 &		0 & 0 & 0 &		1 & 1 & 0 & 	0 & 0 & 0 \\
\hline
0 & 1 & 0 &		0 & 0 & 0 & \frac12 & \frac12 & 0 & 	0 & \frac12 & \frac12 \\
0 & 0 & 1 & 	\frac12 & 0 & 0 & 0 & 0 & \frac12 & 	\frac12 & 0 & \frac12 \\
1 & 0 & 0 & 	0 & \frac12 & \frac12 & 0 & 0 & 0 & 	\frac12 & \frac12 & 0 \\
\end{array}\right]\]}
The characteristic polynomial of $M$ is 
\[ \chi(z) = \frac14 z^3(z-1)(2z+1)^2(z^6-az^4-bz^3-az^2+1), \]
with 
\begin{align*}
	a &= pq+pr+qr-4p-4q-4r+9\text{~and} \\
	b &= 2pqr-4pq-4pr-4qr+8p+8q+8r-16.
\end{align*}
	
The unique minimal member is ${\sig_1=\vs{3,4,3,5,3,6}}$. By numerical approximation of eigenvalues, if $T_1$ is a tessellation with valence sequence $\sig_1$, then $\gam(T_1)\approx 5.6723$.  Hence $\gam(T)\geq \gam(T_1)\approx5.6723$ by \cite{GFHT}{Theorem 2.7}.
\end{proof}\end{prop}
\begin{figure}[htbp]
\includegraphics[width=\linewidth]{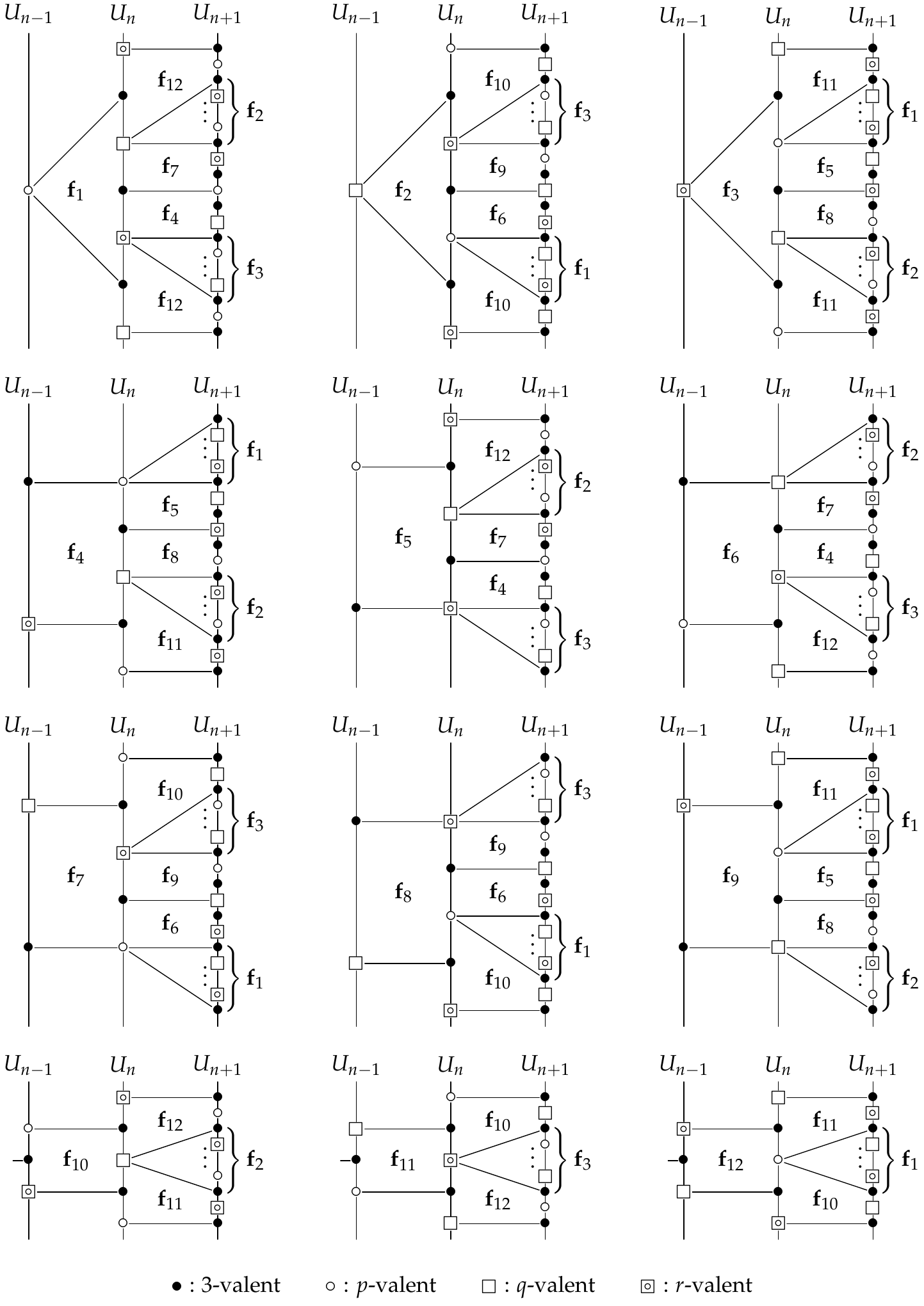}
\caption[Offspring diagrams for $\vs{3,p,3,q,3,r}$.]{\label{fig:3p3q3r-offspring}Offspring diagrams for a face-homogeneous tessellation with valence sequence $\vs{3,p,3,q,3,r}$. See \Cref{thm:GrowthOf-3p3q3r}.}
\end{figure}


\begin{prop} If $\sig=\vs{p,3,p,q,r,q}$, then $\sig$ is monomorphic with $p$ and $q$ even and $\gam(T)\geq 8.0601$.
\begin{proof} The offspring diagrams given in \Cref{fig:p3pqrq-offspring} show that $\sig$ is monomorphic. Using those diagrams to compute the transition matrix, we obtain 
\[ M = \left[\begin{matrix}
p-3 & \frac{3p-10}2 & 2p-6 & p-3 & \frac{p-4}2 & \frac{3p-10}2 & 0 \\
\frac{3q-10}2 & q-3 & q-4 & q-3 & \frac{3q-10}2 & \frac{q-4}2 & q-4 \\
r-3 & \frac{r-4}2 & 0 & \frac{r-4}2 & r-3 & 0 & r-3 \\
1 & 1 & 2 & 1 & 1 & 1 & 0 \\
0 & 2 & 2 & 0 & 0 & 2 & 0 \\
2 & 1 & 0 & 1 & 2 & 0 & 2 \\
\frac12 & 0 & 0 & \frac12 & 0 & 0 & 0
\end{matrix}\right].\]
\begin{figure}[htbp]
\includegraphics[width=\linewidth]{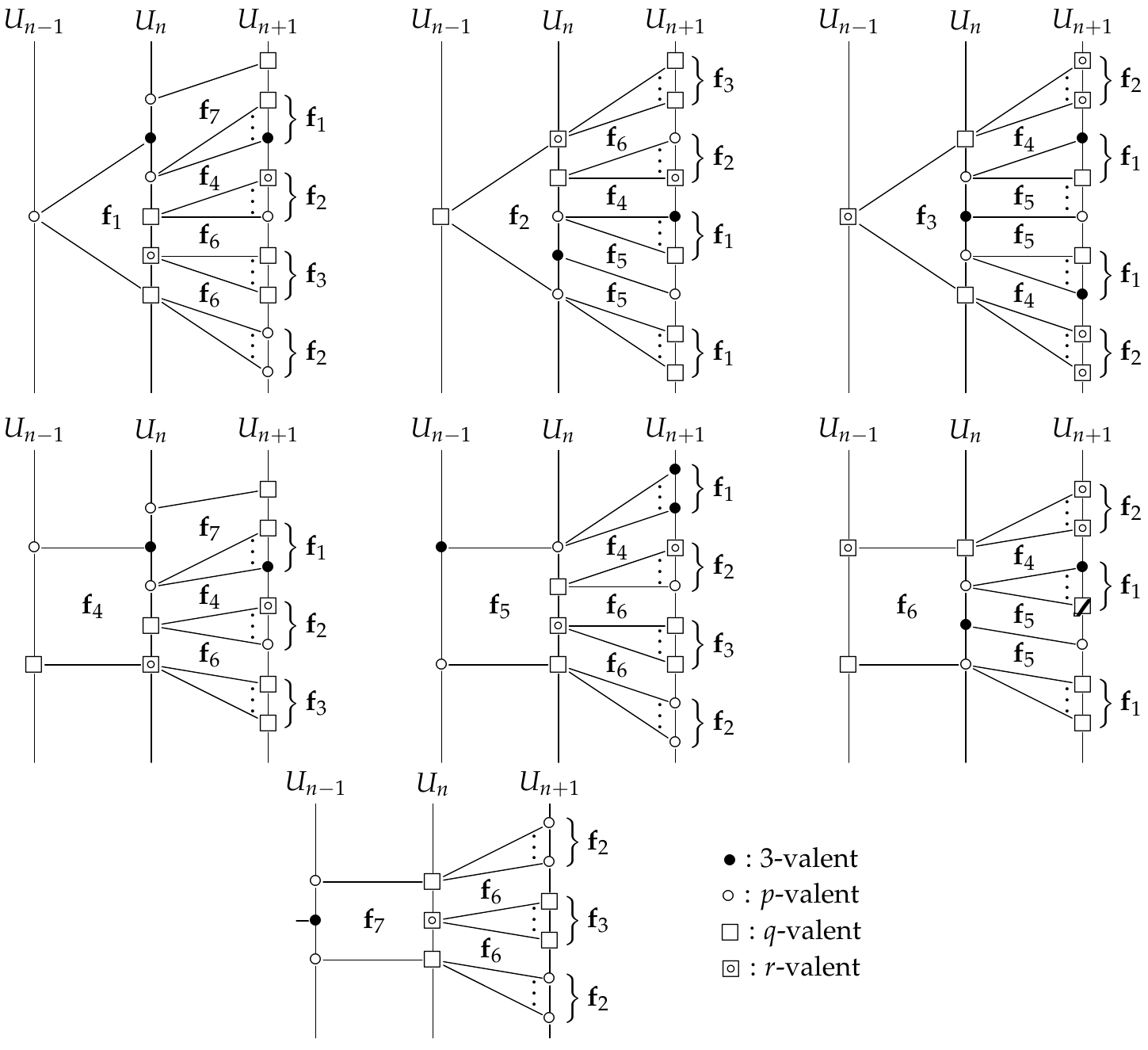}
\caption[Offspring diagrams for $\vs{p,3,p,q,r,q}$.]{\label{fig:p3pqrq-offspring}Offspring diagrams for a face-homogeneous tessellation with valence sequence $\vs{p,3,p,q,r,q}$.}
\end{figure}

In this class are four minimal valence sequences: $\sig_1=\vs{4,3,4,6,5,6}$, $\sig_2=\vs{6,3,6,4,5,4}$, $\sig_3=\vs{6,3,6,8,4,8}$, and $\sig_4=\vs{8,3,8,6,4,6}$. If $T_i$ is a tessellation with valence sequence $\sig_i$ for $i=1,2,3,4$, then numerical approximation of eigenvalues gives 
\begin{align*}
	\gam(T_1) &\approx 8.0601, \\ 
	\gam(T_2) &\approx 8.4980, \\ 
	\gam(T_3) &\approx 12.6691,\text{~and} \\
	\gam(T_4) &\approx 12.8316.
\end{align*}
If $T$ is a tessellation with valence sequence $\sig\geq\sig_i$, for some $i\in\{1,2,3,4\}$,  then $\gam(T)\geq 8.0601$ by \cite{GFHT}{Theorem 2.7}.
\end{proof}
\end{prop}

\bibliographystyle{amsplain}